\documentclass[12pt]{article}
\usepackage{amsmath}
\usepackage{amssymb}
\usepackage{mathrsfs}
\usepackage{bm}
\usepackage[top=3cm,bottom=3cm,right=2.5cm,left=2.5cm]{geometry}
\newcommand{\Z}{\mathbb{Z}}
\newcommand{\R}{\mathbb{R}}
\newcommand{\N}{\mathbb{N}}
\newcommand{\Q}{\mathbb{Q}}
\newcommand{\SL}{{\rm SL}}
\newcommand{\SO}{{\rm SO}}

\newcommand{\A}{\mathcal{A}}
\newcommand{\D}{\mathcal{D}}
\newcommand{\RP}{\mathbb{RP}}
\newcommand{\supp}{{\rm supp}}
\newcommand{\diam}{{\rm diam}}
\newcommand{\B}{\mathcal{B}}
\newcommand{\I}{\mathcal{I}}
\usepackage{amsthm}
\newtheorem{prop}{Proposition}[section]
\newtheorem{thm}[prop]{Theorem}
\newtheorem{cor}[prop]{Corollary}
\newtheorem{lem}[prop]{Lemma}
\newtheorem{prob}[prop]{Problem}
\newtheorem*{claim}{Claim}
\theoremstyle{definition}
\newtheorem{dfn}{Definition}[section]

\newtheorem*{rem}{Remark}

\usepackage{graphicx}
\usepackage{float}

\usepackage{bbm}

\title{On the $L^q$ dimension of stationary measures for Möbius iterated function systems}
\author{Shunsuke Usuki}

\begin{document}
	
\maketitle
	
\begin{abstract}
We study the $L^q$ dimension $D(\nu,q)\ (q>1)$ of stationary measures $\nu$ for Möbius iterated function systems on $\R$ satisfying the strongly Diophantine condition, and try the extension of Shmerkin's result \cite[Theorem 6.6]{Shm19}. As the result, we show that there is the dichotomy: the $L^q$ spectrum $\tau(\nu,q)=(q-1)D(\nu,q)$ is equal to the desired value $\min\{\widetilde{\tau}(\nu,q),q-1\}$ for any $q>1$, where $\widetilde{\tau}(\nu,q)$ is the zero of the canonical pressure function, or there exist $q_0>1$ and $0<\alpha<1$ such that $\tau(\nu,q)=\min\{\widetilde{\tau}(\nu,q),q-1\}$ for $1<q<q_0$ and $\tau(\nu,q)=\alpha q$ for $q\geq q_0$. In addition, we give examples of Möbius iterated function systems which show the latter case by giving an affirmative answer to Solomyak's question \cite[Question 2]{Sol24}.
\end{abstract}
	
\tableofcontents

\section{Introduction}\label{section_introduction}

\subsection{Background on the dimension of stationary measures for IFSs}\label{section_background}

When $U\subset\R$ is a non-empty compact interval and $\Phi=\{\varphi_i\}_{i\in\I}\ (|\I|<\infty)$\footnote{In general, when we write $\Phi=\{\varphi_i\}_{i\in\I}$, we permit the case that $\varphi_i=\varphi_j$ for some $i\neq j$. However, separation conditions which will be introduced later will rule this out.} is a non-empty finite family of $C^{1+\gamma}$-maps ($\gamma>0$) on $U$, $\varphi_i: U\to U$, such that $0<|\varphi_i'(x)|<1$ for any $x\in U$, we call $\Phi$ an {\it iterated function system (IFS)} on $U$. Then, there is a unique non-empty compact subset $K\subset U$ such that
\begin{equation*}
K=\bigcup_{i\in\I}\varphi_i(K).
\end{equation*}
This $K$ is called the {\it attractor} (or the self-conformal set) of $\Phi$.
Furthermore, for a probability vector $p=(p_i)_{i\in\I}$, there is a unique Borel probability measure $\nu$ on $U$ such that
\begin{equation*}
\nu=\sum_{i\in\I}p_i\varphi_i\nu.
\end{equation*}
We have $\nu(K)=1$ and call this $\nu$ the {\it stationary measure} (or the self-conformal measure) of $\Phi$ and $p$. These are central objects of interest in the field of fractal geometry. In particular, the dimension of attractors and stationary measures for IFSs has been actively studied for a long time.

If an IFS $\Phi=\{\varphi_i\}_{i\in\I}$ satisfies the {\it strong separation condition (SSC)}, that is, if the images of the attractor $K$: $\varphi_i(K)\ (i\in\I)$ are pairwise disjoint, the dimension of $K$ and the stationary measure $\nu$ is well-understood.
Before 2010's, little was known beyond the cases of the SSC or some variants of the SSC (e.g., the open set condition), and hence the dimension related to IFSs {\it with overlaps}, that is, IFSs such that $\varphi_i(K)\ (i\in\I)$ have significant overlaps, remained unclear.

However, this area has been dramatically progressed in the last decade. The series of the progress began from Hochman's groundbreaking work \cite{Hoc14} for linear IFSs, that is, IFSs $\Phi=\{\varphi_i\}_{i\in\I}$ such that all $\varphi_i\ (i\in\I)$ are contracting affine maps on $\R$: $\varphi_i(x)=\lambda_ix+a_i\ (0<|\lambda_i|<1, a_i\in\R)$, and the Hausdorff dimension.
In the following, for $i=(i_1,\dots,i_n)\in\I^*=\bigcup_{n\in\N}\I^n$, we write $\varphi_i=\varphi_{i_1}\circ\cdots\circ \varphi_{i_n}$.
We define the distance $d(\varphi_1,\varphi_2)$ between two affine maps $\varphi_1(x)=\lambda_1x+a_1$ and $\varphi_2(x)=\lambda_2x+a_2$ by
\begin{equation*}
d(\varphi_1,\varphi_2)=
\begin{cases}1&\text{if }\lambda_1\neq\lambda_2,\\
|a_1-a_2|&\text{if }\lambda_1=\lambda_2.
\end{cases}
\end{equation*}
\begin{dfn}[Exponential separation condition along a subsequence for linear IFSs]
For a linear IFS $\Phi=\{\varphi_i(x)=\lambda_ix+a_i\}_{i\in\I}$ on $\R$, we say that $\Phi$ satisfies the {\it exponential separation condition along a subsequence} if
there is $c>0$ such that
\begin{equation*}
	i,j\in\I^n, i\neq j \implies d(\varphi_i,\varphi_j)\geq c^n
\end{equation*}
holds for infinitely many $n\in\N$.
\end{dfn}
We write $\dim_H$ for the Hausdorff dimension of a set or a Borel probability measure\footnote{In this paper, we define $\dim_H\nu$ for a Borel probability measure $\nu$ by $\dim_H\nu=\inf\{\dim_H E\left|\ E\subset\R: \text{Borel measurable}, \nu(E)>0\right.\}$. But, if $\nu$ is a stationary measure for an IFS, it is known that $\nu$ is exact dimensional, and hence $\dim_H\nu$ coincides with the pointwise dimension at $\nu$-almost every point and $\inf\{\dim_H E\left|\ E\subset\R: \text{Borel measurable}, \nu(\R\setminus E)=0\right.\}$.} on $\R$.

\begin{thm}[{\cite[Theorem 1.1 and Corollary 1.2]{Hoc14}}]\label{dim_linear_IFS_with_exponental_separation}
Let $\Phi=\{\varphi_i(x)=\lambda_ix+a_i\}_{i\in\I}$ be a linear IFS on $\R$ and assume that $\Phi$ satisfies the exponential separation condition along a subsequence. 
Then, if $p=(p_i)_{i\in\I}$ is a probability vector and we write $\nu$ for the stationary measure (self-similar measure) of $\Phi$ and $p$, we have
\begin{equation*}
\dim_H\nu=\min\left\{\frac{H(p)}{\chi(\Phi,p)},1\right\},
\end{equation*}
where $H(p)=-\sum_{i\in\I}p_i\log p_i$ is the entropy of $p$ and $\chi(\Phi,p)=-\sum_{i\in\I}p_i\log|\lambda_i|$ is the Lyapunov exponent of $\Phi$ and $p$.
Furthermore, for the attractor $K$ of $\Phi$, we have
\begin{equation*}
\dim_HK=\min\left\{s(\Phi),1\right\},
\end{equation*}
where $s(\Phi)$ is the unique $s\geq0$ such that
\begin{equation*}
\sum_{i\in\I}|\lambda_i|^s=1.
\end{equation*}
\end{thm}

Theorem \ref{dim_linear_IFS_with_exponental_separation} is today extended in several ways within the range of linear IFSs (see \cite{Hoc15}, \cite{Var19}, \cite{Rap22}, \cite{RV24} and \cite{FF24}).

The natural problem is the extension of Theorem \ref{dim_linear_IFS_with_exponental_separation} to non-linear IFSs. In \cite{HS17}, Hochman and Solomyak showed that Theorem \ref{dim_linear_IFS_with_exponental_separation} is extended to {\it Möbius iterated function systems}, that is, IFSs $\Phi=\{\varphi_i\}_{\in\I}$ such that all $\varphi_i\ (i\in\I)$ are Möbius transformations on $\R\sqcup\{\infty\}$.
As we will see in Section \ref{section_main_problem}, Möbius IFSs can be seen as the natural actions of some kind of finite families $\A=\{A_i\}_{i\in\I}$ of elements of $\SL_2(\R)$ on the one-dimensional real projective space $\RP^1$, and the result in \cite{HS17} is indeed more general, regarding more general $\SL_2(\R)$ random matrix products. The corresponding to the exponential separation condition for linear IFSs in this case is the following {\it strongly Diophantine condition}.
We equip a left-invariant Riemannian metric on the Lie group $G=\SL_2(\R)$ and write $d_G$ for the metric determined by this Riemannian metric.
For $\A=\{A_i\}_{i\in\I}$ and $i=(i_1,\dots,i_n)\in\I^*$, we write $A_i=A_{i_1}\cdots A_{i_n}$.

\begin{dfn}\label{dfn_strongly_Diophantine_condition}
Let $\A=\{A_i\}_{i\in\I}$ be a non-empty finite family of elements of $\SL_2(\R)$.
Then, we say that $\A$ is {\it strongly Diophantine} if there exists a constant $c>0$ such that, for each $n\in\N$, we have
\begin{equation*}
	i,i'\in\I^n,i\neq i'\implies d_G(A_i,A_{i'})\geq c^n.
\end{equation*}
\end{dfn}

We can see that this condition is independent of the choice of a left-invariant Riemannian metric on $\SL_2(\R)$ (see \cite[Section 2.3]{HS17}).

\begin{thm}[{\cite[Theorem 1.1]{HS17}}]\label{Hausdorff_dim_of_Furstenberg_measures}
Let $\A$ be a non-empty finite family of elements of $\SL_2(\R)$ and assume that $\A$ generates an unbounded subgroup in $\SL_2(\R)$ and is totally irreducible\footnote{We say that $\A$ is totally irreducible if $\A$ does not preserve any non-empty finite subset of $\RP^1$.}. We take a probability measure $\mu$ on $\SL_2(\R)$ with $\supp\ \mu=\A$ and write $\nu$ for the unique stationary measure of $\mu$ (called the Furstenberg measure) on $\RP^1$. If $\A$ is strongly Diophantine, we have
\begin{equation*}
\dim_H\nu=\min\left\{\frac{H(\mu)}{2\chi(\mu)},1\right\},
\end{equation*}
where $H(\mu)$ is the entropy of $\mu$ and $\chi(\mu)>0$ is the Lyapunov exponent of $\mu$.
\end{thm}

In addition, the result on the Hausdorff dimension of the attractors of Möbius IFSs satisfying the strongly Diophantine condition was obtained in \cite{ST21}.

Recently, Rapaport showed in \cite{Rap24} the complete extension of Theorem \ref{dim_linear_IFS_with_exponental_separation} to real analytic IFSs, that is, IFSs $\Phi=\{\varphi_i\}_{i\in\I}$ such that all $\varphi_i\ (i\in\I)$ are real analytic.

We notice that the above extensions of Theorem \ref{dim_linear_IFS_with_exponental_separation}
are about the Hausdorff dimension. However, as another way of extending Theorem \ref{dim_linear_IFS_with_exponental_separation}, we can consider replacing the Hausdorff dimension with more strong notion of dimension. Here, we consider the {\it $L^q$ dimension} for stationary measures. To define this,
we take the $2^{-m}$ dyadic partition $\D_m=\left\{[2^{-m}k,2^{-m}(k+1))\right\}_{k\in\Z}$ of $\R$ for $m\in\N$. In the following and throughout this paper, the base of $\log$ is $2$.

\begin{dfn}\label{dfn_L^d_dimension}
For $q>1$ and a Borel probability measure $\nu$ on $\R$ with the compact support and $m\in\N$, the {\it $2^{-m}$ $L^q$ norm} of $\nu$ is
\begin{equation*}
\|\nu^{(m)}\|_q=\left(\sum_{I\in\D_m}\nu(I)^q\right)^{1/q}.
\end{equation*}
The {\it $L^q$ spectrum} of $\nu$ is\footnote{The last inclusion follows from Hölder's inequality.}
\begin{equation*}
\tau(q)=\tau(\nu,q)=\liminf_{m\to\infty}\left(-\frac{1}{m}\log\|\nu^{(m)}\|_q^q\right)=\liminf_{m\to\infty}\left(-\frac{1}{m}\log\sum_{I\in\D_m}\nu(I)^q\right)\in[0,q-1],
\end{equation*}
and the {\it $L^q$ dimension} of $\nu$ is
\begin{equation*}
	D(q)=D(\nu,q)=\frac{\tau(\nu,q)}{q-1}\in[0,1].
\end{equation*}
\end{dfn}

We can see that the $L^q$ norm measures the “smoothness” of $\nu$: the smaller the value $\|\nu^{(m)}\|_q$ is, the more uniform the distribution $\nu(I)\ (I\in\D_m)$ is, and the closer to $1$ the value $\|\nu^{(m)}\|_q$ is, the more concentrated $\nu$ is on a single interval of $\D_m$.

It can be easily seen from Hölder's inequality that $\tau(q)\ (q>1)$ is non-decreasing, concave, and hence continuous in $q$.
It holds that $D(q)\ (q>1)$ is non-increasing in $q$, and $\lim_{q\searrow1}D(q)\leq \dim_H\nu$ (see \cite{FLR02}).
In particular, if $\nu$ is a stationary measure for an IFS, it is shown in \cite{PS00} that the limit in the definition of the $L^q$ spectrum exists. We will see this fact in Section \ref{subsection_stationary_measures} for our case of Möbius IFSs. Furthermore, we have
\begin{equation*}
\lim_{q\searrow1}D(q)=\dim_H\nu
\end{equation*}
(see \cite[Theorem 5.1 and Remark 5.2]{SS16}).

The $L^q$ dimension is also important in terms of “local information” of measures which it tells.
If $\nu$ is a stationary measure for an IFS, it is known (\cite{FH09}) that $\nu$ is exact dimensional. Hence, if $s<\dim_H\nu$, we have $\nu(B_r(s))\leq r^s$ for $\nu$-a.e. $x$ and sufficiently small $r>0$ depending on $x$\footnote{Here, we write $B_r(x)$ for the open ball in $\R$ of radius $r$ and center $x$.}. However, what the $L^q$ dimension tells is more strong: if $s<D(q)$, it is easily seen from the definition of $D(q)$ that $\nu(B_r(x))\leq r^{(1-1/q)s}$ holds for all $x$ and sufficiently small $r>0$ independent of $x$ (which is stated as Lemma 1.7 in \cite{Shm19}).

Classically, the $L^q$ spectrum is studied regarding {\it multifractal formalism}. This is the heuristic principle saying that, for a “natural” measure $\nu$, the dimension spectrum $\dim_H\left\{x\in\R\right|$ $\left.\lim_{r\searrow0}\log\nu(B_r(x))/\log r=\alpha\right\}\ (\alpha\in[0,1])$ is given by the Legendre transform $\tau^*(\alpha)$ of the $L^q$ spectrum $\tau(q)$ of $\nu$. For this topic, e.g., see \cite{LN99}, \cite{Pat97}, \cite{Fen07} and references therein. Actually, the heuristic idea of multifractal formalism is important in the argument in \cite{Shm19} and this paper, and we will notice it again in Section \ref{subsection_multifractal_structure}.

In the landmark paper \cite{Shm19}, Shmerkin proved the following extension of Theorem \ref{dim_linear_IFS_with_exponental_separation} to the $L^q$ dimension.

\begin{thm}[{\cite[Theorem 6.6]{Shm19}}]\label{L^q_dim_of_self_similar_measures}
Let $\Phi=\left\{\varphi_i(x)=\lambda_ix+a_i\right\}_{i\in\mathcal{I}}\ (0<|\lambda_i|<1)$ be a linear IFS on $\R$ satisfying the exponential separation condition along a subsequence and $p=(p_i)_{i\in\mathcal{I}}$ be a probability vector. We write $\nu$ for the stationary measure (self-similar measure) of $\Phi$ and $p$.
Then, for $q>1$, we have $D(\nu,q)=\min\{\widetilde{\tau}(q)/(q-1),1\}$, where $\widetilde{\tau}(q)$ is the unique solution to
\begin{equation}\label{L^q_Bowen_formula_linear_IFS}
\sum_{i\in\mathcal{I}}p_i^q|\lambda_i|^{-\widetilde{\tau}(q)}=1.
\end{equation}
\end{thm}

This result has an application to Bernoulli convolutions about the regularity of its density (\cite[Theorem 1.3]{Shm19}). Shmerkin's result is indeed more general, regarding wider class of measures called dynamically-driven self-similar measures. Remarkably, the corresponding result to Theorem \ref{L^q_dim_of_self_similar_measures} for some dynamically-driven self-similar measures is applied to prove famous Furstenberg's $\times p,\times q$ intersection conjecture (\cite[Theorem 1.2]{Shm19})\footnote{The solution of this conjecture is simultaneous and independent achievement by Shmerkin and Wu (\cite{Wu19}), and their strategies are quite different. Today, the short proof of this conjecture is known (\cite{Aus22}).}.

Recently, Shmerkin's result was extended to higher dimensional linear cases by Corso and Shmerkin (\cite{CS24}).

\subsection{M\"{o}bius IFSs and the main problem}\label{section_main_problem}

{\it The extension of Theorem \ref{L^q_dim_of_self_similar_measures} to non-linear IFSs} is a natural problem. In this paper, we consider this problem, more precisely, {\it the extension of Theorem \ref{L^q_dim_of_self_similar_measures} to Möbius IFSs}.

First of all, we formalize the problem. We consider a Möbius IFS $\Phi=\{\varphi_i\}_{i\in\I}$, that is, a non-empty finite family of Möbius transformations $\varphi_i\ (i\in\I)$ on $\R\sqcup\{\infty\}$ restricted to an invariant compact interval $U\subset\R$ such that $0<|\varphi_i'(x)|<1$ for any $x\in U$ and $i\in\I$. However, as we mentioned in Section \ref{section_background}, such an IFS can be treated as the natural action of a non-empty finite family $\A=\{A_i\}_{i\in\I}$ of elements of $\SL_2(\R)$ on $\RP^1$ with the property called {\it uniform hyperbolicity}. Here, the natural action of $\SL_2(\R)$ on $\RP^1$ is
\begin{equation*}
\begin{array}{ccc}
\SL_2(\R)\times \RP^1&\longrightarrow&\RP^1\\
\rotatebox{90}{$\in$}&&\rotatebox{90}{$\in$}\\
(A,[v])&\longmapsto&[Av],
\end{array}
\end{equation*}
where, for $v\in\R^2\setminus\{0\}$, $[v]$ is the image of $v$ in $\RP^1$. We first see this correspondence of a Möbius IFS $\Phi$ to a uniformly hyperbolic family $\A\subset\SL_2(\R)$.

We write $G=\SL_2(\R)$ and take a non-empty finite family $\A=\left\{A_i\right\}_{i\in\mathcal{I}}$ of elements of $G$.
Let $\|A\|=\max_{v\in\R^2,\|v\|\leq1}\|Av\|$\footnote{In this paper, we take the Euclidean norm $\|^{t}(v_1,v_2)\|=\sqrt{v_1^2+v_2^2}$ on $\R^2$.} be the operator norm of $A\in {\rm M}_2(\R)$.

\begin{dfn}[Uniform hyperbolicity]\label{definition_uniformly_hyperbolic}
	We say that $\A$ is {\it uniformly hyperbolic} if there exist constants $c>0$ and $r>1$ such that
	\begin{equation*}
		\|A_{i_1}\cdots A_{i_n}\|\geq cr^n\quad\text{for any }n\in\N, (i_1,\dots,i_n)\in\I^n.
	\end{equation*}
\end{dfn}

Then, the following fact holds.

\begin{prop}[{\cite[Theorem 2.2]{ABY10}}]\label{proposition_Möbius_uniformly_hyperbolic_SL_2(R)}
A non-empty finite family $\A$ of elements of $G$ is uniformly hyperbolic if and only if there exists a non-empty open subset $U\subsetneq\RP^1$ with finite connected components having disjoint closures such that
\begin{equation*}
A\overline{U}\subset U\quad\text{for any}\ A\in\A.
\end{equation*}
\end{prop}

If $\Phi=\{\varphi_i\}_{i\in\I}$ is a Möbius IFS, then, for each $i\in\I$, there is $A_i\in G$ such that
\begin{equation*}
\varphi_i(x)=\frac{a_ix+b_i}{c_ix+d_i}\quad\text{for }x\in\R\sqcup\{\infty\},\quad\text{where }A_i=
\begin{pmatrix}
a_i&b_i\\c_i&d_i
\end{pmatrix},
\end{equation*}
and $\varphi_i:\R\sqcup\{\infty\}\to\R\sqcup\{\infty\}$ and the action of $A_i$ on $\RP^1$ are conjugate via the identification $F:\R\sqcup\{\infty\}\stackrel{\sim}{\longrightarrow}\RP^1$ defined by $F(x)=[x:1]$ for $x\in\R$ and $F(\infty)=[1:0]$. We can take a bounded open interval $U\subset\R$ such that $\varphi_i(\overline{U})\subset U$ for any $i\in\I$. Then, the family $\A=\{A_i\}_{i\in\I}$ satisfies the condition of Proposition \ref{proposition_Möbius_uniformly_hyperbolic_SL_2(R)} for $F(U)\subsetneq\RP^1$, and hence $\A$ is uniformly hyperbolic.
In the following and throughout this paper, we treat non-empty, finite and uniformly hyperbolic families of elements of $G$ and their natural actions on $\RP^1$, instead of Möbius IFSs on $\R$ themselves\footnote{This is because (of the author's taste and) we want to cite the argument directly from \cite{HS17}, in which the action of $G$ on $\RP^1$ is treated, rather than mathematical reasons.}.

We take a non-empty finite family $\A=\{A_i\}_{i\in\I}\subset G$ which is uniformly hyperbolic. Let $U\subsetneq \RP^1$ be an open subset as in Proposition \ref{proposition_Möbius_uniformly_hyperbolic_SL_2(R)}. Then, the following holds as same as IFSs on $\R$ by the contraction property of the action of $\A$ which will be seen in Section \ref{subsection_action_of_uniformly_hyperbolic_subset}.
There exists a unique non-empty compact subset $K\subset U$, called the attractor of $\A$, such that
\begin{equation*}
	K=\bigcup_{i\in\I}A_iK.
\end{equation*}
The attractor $K$ is independent of the choice of $U$.

For finite Borel measures $\mu$ and $\nu$ on $G$ and $\RP^1$ respectively, we write $\mu{\bm .}\nu$ for the push-forward measure on $\RP^1$ of $\mu\times\nu$ by the action $G\times\RP^1\ni(g,x)\mapsto gx\in\RP^1$.
We take a non-degenerate probability vector $p=(p_i)_{i\in\I}$ and define the probability measure 
\begin{equation*}
	\mu=\sum_{i\in\I}p_i\delta_{A_i}
\end{equation*}
on $G$ such that $\supp\ \mu=\A$.
Then, there exists a unique stationary measure $\nu$ of $\mu$ on $\RP^1$, that is, a Borel probability measure on $\RP^1$ with
\begin{equation*}
	\nu=\mu{\bm .}\nu=\sum_{i\in\I}p_iA_i\nu,
\end{equation*}
such that $\nu(U)=1$. Furthermore, we have $\supp\ \nu=K$ and $\nu$ is independent of the choice of $U$.

For a Borel probability measure $\nu$ on $\RP^1$ and $q>1$, the $2^{-m}$ $L^q$ norm $\|\nu^{(m)}\|_q$, the $L^q$ spectrum $\tau(\nu,q)$ and the $L^q$ dimension $D(\nu,q)$ are also defined in the same way as Borel probability measures on $\R$ with the compact supports, by just replacing the $2^{-m}$ dyadic partition of $\R$ with that of $\RP^1$ (which will be defined in Section \ref{subsection_stationary_measures}).

If $\A=\{A_i\}_{i\in\I}$ corresponds to a Möbius IFS $\Phi=\{\varphi_i\}_{i\in\I}$, then, of course, these attractor and stationary measure for $\A$ coincide to those for $\Phi$ introduced in Section \ref{section_background}, via the identification $F:\R\sqcup\{\infty\}\stackrel{\sim}{\longrightarrow}\RP^1$ as above. Moreover, the values of the $L^q$ spectrum and the $L^q$ dimension are also equal (by Lemma \ref{L^q_norms_of_two_partitions}).

To formalize our problem, we need to define the natural corresponding things of the notions in Theorem \ref{L^q_dim_of_self_similar_measures} for a finite uniformly hyperbolic family of elements of $G$. The corresponding of the exponential separation condition along a subsequence is the strongly Diophantine condition (Definition \ref{dfn_strongly_Diophantine_condition})\footnote{It may be more natural to consider the strongly Diophantine condition “along a subsequence”, but we consider this one in this paper.}.
Here, we consider the corresponding of (\ref{L^q_Bowen_formula_linear_IFS}) for the non-empty finite and uniformly hyperbolic family $\A=\{A_i\}_{i\in\I}\subset G$ and a non-degenerate probability vector $p=(p_i)_{i\in\I}$.

Since the case $|\I|=1$ is trivial, we assume $|\I|\geq2$.
We recall that, for $i=(i_1,\dots,i_n)\in\I^*$, we write $A_i=A_{i_1}\cdots A_{i_n}$, and we also write $p_i=p_{i_1}\cdots p_{i_n}$.
For $q>1$, we define the pressure function $\Psi_q:\R_{>0}\rightarrow\R$ by
\begin{equation*}
	\Psi_q(s)=\lim_{n\to\infty}\frac{1}{n}\log\sum_{i\in\I^n}p_i^q\|A_i\|^{2s}.
\end{equation*}
The existence of the limit is due to the subadditivity.\footnote{The reason why $\|A_i\|^2$ appears is that, as seen in Section \ref{subsection_G_action_on_RP^1}, the “essential contraction ratio” of the action of $A\in G$ on $\RP^1$ is $\|A\|^{-2}$.}
It is easily seen that $\Psi_q$ is non-decreasing and convex, and hence continuous\footnote{Actually, we will see in Section \ref{subsection_incompatibility} that $\Psi_q(s)$ can be defined for $(q,s)\in\R^2$ and is analytic on $\R^2$.}. By the uniform hyperbolicity of $\A$, we have $\Psi_q(s)\to\infty$ as $s\to\infty$. Furthermore, if $s>0$ is so small that $\|A_i\|^{2s}<p_i^{1-q}$ for each $i\in\I$ (we notice that $0<p_i<1$ for each $i\in\I$, because $p$ is non-degenerate and $|\I|\geq 2$), we have $\Psi_q(s)<0$. From these facts, we can see that $\Psi_q$ has the unique zero on $\R_{>0}$.
We write $\widetilde{\tau}(q)>0$ for this zero.\footnote{The definition of $\widetilde{\tau}(q)$ is more involved than that in (\ref{L^q_Bowen_formula_linear_IFS}). This is because, unlike linear IFSs, the (“essential”) contraction ratio $\|A\|^{-2}$ of the action of $A\in G$ on $\RP^1$ is not commutative under compositions.}
We notice a non-trivial (but maybe well-known) fact on $\widetilde{\tau}(q)$.

\begin{prop}\label{prop_analyticity_of_zero_of_pressure_function}
The function $\widetilde{\tau}:\R_{>0}\to\R$ is analytic.
\end{prop}

For the stationary measure $\nu$ of $\mu=\sum_{i\in\I}p_i\delta_{A_i}$, it always holds that
\begin{equation*}
D(\nu,q)\leq\min\left\{\frac{\widetilde{\tau}(q)}{q-1},1\right\},
\end{equation*}
and, as we will see in Section \ref{subsection_essential_part_main_theorem}, the proof is not so hard. It is well-known that, if the action of $\A$ on $\RP^1$ satisfies the strong separation condition, we have $D(\nu,q)=\widetilde{\tau}(q)/(q-1)$ (see \cite[Lemma 5.1]{Pat97}).

We point out one more thing for our problem. We should assume that the attractor $K$ is not a singleton.
If otherwise, the stationary measure $\nu$ is a one-point mass and hence $D(\nu,q)=0$ trivially.\footnote{In the case of linear IFSs satisfying the exponential separation condition along a subsequence, this assumption trivially holds. It is because, if two affine maps share a fixed point, then they are commutative, so the separation condition ensures it. }
In some important works for non-linear IFSs (for example, \cite{HS17}, \cite{ST21} and \cite{Rap24}), this assumption is naturally assumed.

Here, we state our problem: the natural extension of of Theorem \ref{L^q_dim_of_self_similar_measures} to actions of finite and uniformly hyperbolic families of elements of $G$ on $\RP^1$.

\begin{prob}\label{main_problem_L^q_dim_of_Furstenberg_measures}
Let $\A=\{A_i\}_{i\in\I}\ (|\I|\geq 2)$ be a non-empty finite family of elements of $G$ and assume that $\A$ is uniformly hyperbolic and strongly Diophantine and the attractor $K$ of $\A$ is not a singleton. We take a probability measure $\mu=\sum_{i\in\I}p_i\delta_{A_i}$ on $G$ such that $\supp\ \mu=\A$ and write $\nu$ for the stationary measure of $\mu$. Then, for $q>1$, does it hold that
\begin{equation*}
D(\nu,q)=\min\left\{\frac{\widetilde{\tau}(q)}{q-1},1\right\}\quad\text{?}
\end{equation*}
\end{prob}

However, as we will see just below, {\it THE ANSWER IS NO} in general.

\subsection{The main theorem}\label{section_main_theorem}

In this section, we state the main theorem in this paper. As we stated above, the natural extension of Theorem \ref{L^q_dim_of_self_similar_measures} does not hold in general.
In fact, the following theorem tells us that a problem occurs when two distinct elements of $\A$ share a common fixed point.
We notice that this case is ruled out for linear IFSs satisfying exponential separation condition along a subsequence, because two affine maps sharing a fixed point must commute.

\begin{thm}[Counterexamples to Problem \ref{main_problem_L^q_dim_of_Furstenberg_measures}]\label{counterexample_to_natural_extension}
Let $\A=\{A_i\}_{i\in\I}\subset G$ be a family satisfying the conditions of Problem \ref{main_problem_L^q_dim_of_Furstenberg_measures}. Furthermore, we assume that there are distinct $i_0,j_0\in\I$ such that $A_{i_0}$ and $A_{j_0}$ share a common fixed point in $K$.
If we take a non-degenerate probability vector $p=(p_i)_{i\in\I}$ so that $p_{i_0}=p_{j_0}=p_0$ and $0<p_0<1/2$ is sufficiently close to $1/2$ in terms of $\A$, then the stationary measure $\nu$ of $\mu=\sum_{i\in\I}p_i\delta_{A_i}$ satisfies
\begin{equation*}
D(\nu,q)<\min\left\{\frac{\widetilde{\tau}(q)}{q-1},1\right\}
\end{equation*}
for sufficiently large $q>1$.
Examples of such $\A$ indeed exist.
\end{thm}

This fact is remarkable because we know that Theorem \ref{dim_linear_IFS_with_exponental_separation} for linear IFSs regarding the Hausdorff dimension is completely extended to Möbius IFSs (Theorem \ref{Hausdorff_dim_of_Furstenberg_measures}). Hence, it can be said that the difference of properties about dimension between linear and Möbius IFSs appears in the $L^q$ dimension for the first time.
In Section \ref{section_counterexamples}, we will see the proof of Theorem \ref{counterexample_to_natural_extension}.
The proof of the first statement is not difficult, and the examples are given by resolving Solomyak's question \cite[Question 2]{Sol24} (Theorem \ref{freeness_of_Solomyak_example}).

Although Problem \ref{main_problem_L^q_dim_of_Furstenberg_measures} is not true in general, {\it there is a strong dichotomy on the shape of the graph of the $L^q$ spectrum $\tau(\nu,q)\ (q>1)$}. The following is the main theorem of this paper.

\begin{thm}[The main theorem]\label{the_main_theorem_L^q_dim_Mobius_IFS}
Let $\A=\{A_i\}_{i\in\I}\ (|\I|\geq2)$ be a non-empty finite family of elements of $G$ and assume that $\A$ is uniformly hyperbolic and strongly Diophantine and the attractor $K$ of $\A$ is not a singleton. We take a probability measure $\mu$ on $G$ such that $\supp\ \mu=\A$ and write $\nu$ for the stationary measure of $\mu$.
Then, either the following (I) or (II), only one of them, holds:
\begin{enumerate}
\renewcommand{\labelenumi}{(\Roman{enumi})}
\item we have
\begin{equation*}
\tau(\nu,q)=\min\left\{\widetilde{\tau}(q),q-1\right\}\quad\text{for any }q>1,
\end{equation*}
\item there exist $q_0>1$ and $0<\alpha<1$ such that
\begin{equation*}
\tau(\nu,q)=
\begin{cases}
\min\left\{\widetilde{\tau}(q),q-1\right\}&\text{if }1<q<q_0,\\
\alpha q&\text{if }q\geq q_0.
\end{cases}
\end{equation*}
\end{enumerate}
\end{thm}

As a corollary of Theorem \ref{the_main_theorem_L^q_dim_Mobius_IFS}, we can understand the shape of the graph of $\tau(\nu,q)$ for a stationary measure $\nu$ as in Theorem \ref{counterexample_to_natural_extension}.

\begin{cor}\label{shape_of_L^q_spectrum_of_counterexamples}
Let $\nu$ be a stationary measure as in Theorem \ref{counterexample_to_natural_extension}. Then, the case (II) of Theorem \ref{the_main_theorem_L^q_dim_Mobius_IFS} holds for $\nu$.
\end{cor}

To the author's knowledge, this is the first example of a non-trivial stationary measure for an IFS whose $L^q$ spectrum has such the form.

We can think a stationary measure $\nu$ of the case (II) of Theorem \ref{the_main_theorem_L^q_dim_Mobius_IFS} is “singular” in the following sense.\footnote{Here, we write $B_r(x)$ for the open ball in $\RP^1$ of radius $r$ and center $x$. The metric on $\RP^1$ will be introduced in Section \ref{subsection_G_action_on_RP^1}. We also recall that $\D_m$ is the $2^{-m}$ dyadic partition of $\RP^1$, which will be defined in Section \ref{subsection_stationary_measures}.}

\begin{prop}\label{justification_heuristic_singularity}
Let $\nu$ be a stationary measure as in Theorem \ref{the_main_theorem_L^q_dim_Mobius_IFS} and assume that the case (II) holds for $\nu$. Then, the following holds.
\begin{enumerate}
\renewcommand{\labelenumi}{(\roman{enumi})}
\item We have
\begin{equation*}
\alpha=\min\left\{\alpha'>0\left|\ \text{there exists some $x\in K$ such that }\liminf_{r\searrow 0}\right.\frac{\log\nu(B_r(x))}{\log r}\leq\alpha'\right\}.
\end{equation*}
		\item If we write
		\begin{equation*}
			E_\alpha=\left\{x\in K\left|\ \lim_{r\searrow0}\frac{\log\nu(B_r(x))}{\log r}=\alpha\right.\right\},
		\end{equation*}
		then $E_\alpha$ is dense in $K$ and satisfies
		\begin{equation*}
			\dim_HE_\alpha=0.
		\end{equation*}
		\item For any $q\geq q_0$ and $\varepsilon>0$, there exists $\delta=\delta(\mu, q,\varepsilon)>0$ such that, for sufficiently large $m\in\N$ in terms of $\mu,q,\varepsilon,\delta$, we have
		\begin{equation*}
			\left|\left\{I\in\D_m\left|\ 2^{-(\alpha+\varepsilon)m}\leq\nu(I)\leq 2^{-(\alpha-\varepsilon/q)m}\right.\right\}\right|=
			\left|\left\{I\in\D_m\left|\ 2^{-(\alpha+\varepsilon)m}\leq\nu(I)\right.\right\}\right|\leq 2^{(q+1)\varepsilon m}
		\end{equation*}
		and
		\begin{equation*}
			\sum_{I\in\D_m,\ \nu(I)<2^{-(\alpha+\varepsilon)m}}\nu(I)^q\leq 2^{-\delta m}\sum_{I\in\D_m}\nu(I)^q.
		\end{equation*}
	\end{enumerate}
\end{prop}

The statements (i) and (ii) tell us that $\alpha$ is the minimum value of the pointwise dimension of $\nu$ and the points at which the pointwise dimension of $\nu$ is $\alpha$ are dense in the attractor $K$, but the Hausdorff dimension is zero.
The statement (iii) tells that the $L^q$ norm of $\nu$ restricted to the $2^{-m}$ intervals whose $\nu$-masses are strictly less that $2^{-\alpha m}$ is much smaller than the entire $L^q$ norm, and hence a much small number of $2^{-m}$ intervals whose $\nu$-masses are about $2^{-\alpha m}$ gives the essential contribution to $\sum_{I\in\D_m}\nu(I)^q\sim 2^{-\alpha mq}$. These statements represent the “singularity” in some sense.
 This proposition will be discussed in Section \ref{subsection_singularity}.

In the end of this section, we give remarks on the main Theorem \ref{the_main_theorem_L^q_dim_Mobius_IFS}.

\begin{rem}
\begin{enumerate}
\renewcommand{\labelenumi}{(\roman{enumi})}
\item Theorem \ref{the_main_theorem_L^q_dim_Mobius_IFS} gives no information to determine which (I) or (II) occurs for a given stationary measure and we can't obtain general methods to do this in this paper. The author thinks that it is much meaningful to give such methods.
\item The author thinks that, for a stationary measure $\nu$, even if one can know that $\nu$ is of the case (II), it is difficult to know the exact value of $\alpha$ or $q_0$. It seems that the case (II) occurs because of complicated overlaps of the IFS. Hence, to know the value of $\alpha$ or $q_0$, we have to study such overlaps in much detail and he does not know how to do that.
\item The author expects the similar result as Theorem \ref{the_main_theorem_L^q_dim_Mobius_IFS} for general analytic IFSs on $\R$.
\end{enumerate}
\end{rem}

\subsection{About the proof: the $L^q$ norm flattening theorem}\label{section_L^q_norm_flattening}

In the following two sections, we see the essential steps of the proof of the main Theorem \ref{the_main_theorem_L^q_dim_Mobius_IFS}. At first sight, the strategy of the proof would be trying to follow the argument of \cite{Shm19}. In \cite{Shm19}, the most essential is {\it the $L^q$ norm flattening theorem} \cite[Theorem 5.1]{Shm19}\footnote{This is indeed the result for a general dynamically-driven self-similar measure, but here we state it only for a self-similar measure.}.
For two finite Borel measures $\mu$ and $\nu$ on $\R$, we write $\mu*\nu$ for the linear convolution on $\R$ of $\mu$ and $\nu$, that is, the push-forward measure of $\mu\times\nu$ by $\R\times\R\ni (x,y)\mapsto x+y\in\R$.

\begin{thm}[{\cite[Theorem 5.1]{Shm19}}]\label{L^q_norm_flattening_theorem_for_self_similar_measure}
Let $\nu$ be a stationary measure (self-similar measure) of a linear IFS $\Phi$ on $\R$ and a probability vector $p$ (not assuming the exponential separation condition along a subsequence on $\Phi$) and $\tau(q)\ (q>1)$ be the $L^q$ spectrum of $\nu$. We take $q>1$ and assume that $\tau(q)$ is differentiable at $q$ and $\tau(q)<q-1$.
Then, for any $0<\sigma<1$, there is $\varepsilon=\varepsilon(\Phi,p,q,\sigma)>0$ such that the following holds for sufficiently large $m\in\N$ in terms of $\Phi,p,q,\sigma,\varepsilon$:
for any Borel probability measure $\rho$ on $[0,1]$ such that $\|\rho^{(m)}\|_q^q\leq 2^{-\sigma m}$, we have
\begin{equation}\label{flattening_L^q_norm_for_self_similar_measure}
\|(\rho*\nu)^{(m)}\|_q^q\leq 2^{-(\tau(q)+\varepsilon)m}.
\end{equation}
\end{thm}

For a Borel probability measure $\nu$ on $\R$, if we take the convolution $\rho*\nu$ of $\nu$ with another Borel probability measure $\rho$ on $[0,1]$, then it is easily seen that the value of $\|(\rho*\nu)^{(m)}\|_q^q$ does not essentially increase from $\|\nu^{(m)}\|_q^q$ in general. However, (\ref{flattening_L^q_norm_for_self_similar_measure}) tells us that, if $\nu$ is a self-similar measure and $\tau(q)<q-1$\footnote{The assumption that $\tau(q)$ is differentiable at $q$ is also needed. We notice that, since $\tau(q)$ is concave, it is differentiable at every $q>1$ except for at most countable points.}, the value of $\|(\rho*\nu)^{(m)}\|_q^q$ strictly and uniformly decreases from $\|\nu^{(m)}\|_q^q$ on $\rho$ having a flat $L^q$ norm in terms of $\sigma$.

We extend Theorem \ref{L^q_norm_flattening_theorem_for_self_similar_measure} to a stationary measure for a finite and uniformly hyperbolic family $\A\subset G$ and the action of $G$ on $\RP^1$.
To state the result, we need to prepare some notions.
We use the {\it dyadic-like partitions of $G$} from \cite[Section 4.2]{HS17}. That is, we take a family of Borel partitions $\D^G_m\ (m\in\N)$ such that, for some constant $M>1$, the following holds:
\begin{enumerate}
	\renewcommand{\labelenumi}{(\roman{enumi})}
	\item $\D_{m+1}^G$ refines $\D_m^G$,
	\item each $\xi\in\D_m^G$ contains at most $M$ elements of $\D_{m+1}^G$,
	\item for each $\xi\in\D_m^G$, $\xi$ contains an open ball of radius $M^{-1}2^{-m}$ and $\diam\ \xi\leq M2^{-m}$.
\end{enumerate}
In (iii), we recall that we have equipped the metric on $G$ by the left-invariant Riemannian metric.
For a finite Borel measure $\theta$ on $G$ with the compact support, $m\in\N$ and $q>1$, we define the {\it $\D^G_m$ $L^q$ norm} of $\theta$ by
\begin{equation*}
\|\theta^{(m)}\|_q=\left(\sum_{\xi\in\D^G_m}\theta(\xi)^q\right)^{1/q}.
\end{equation*}

Let $\A$ be a non-empty finite family of elements of $G$ which is uniformly hyperbolic. Then, by Proposition \ref{proposition_Möbius_uniformly_hyperbolic_SL_2(R)} and shrinking the open subset a little, we obtain two non-empty open subsets $U_1\subset U_0\subsetneq\RP^1$ with finite connected components having disjoint closures such that
\begin{equation}\label{domain_U1_the_IFS_acts_on}
\overline{U_1}\subset U_0,\quad A\overline{U_0}\subset U_1\quad\text{for every }A\in\A.
\end{equation}
We notice that the attractor $K$ of $\A$ is in $U_1$.

These notions are naturally considered to be necessary for the extension of Theorem \ref{L^q_norm_flattening_theorem_for_self_similar_measure} to a stationary measure for $\A$ and the $G$-action on $\RP^1$.
However, to establish it, we need to assume one more thing which does not appear in Theorem \ref{L^q_norm_flattening_theorem_for_self_similar_measure}: {\it the positivity of the Legendre transform of the $L^q$ spectrum}. In fact, the necessity of this assumption is the direct reason why the main Theorem \ref{the_main_theorem_L^q_dim_Mobius_IFS} splits into two cases and the case (II) emerges.

Let $\nu$ be a stationary measure for $\A$.
Then, for the $L^q$ spectrum $\tau(q)\ (q>1)$ of $\nu$, which is non-decreasing, concave and continuous, the {\it Legendre transform $\tau^*(\alpha)$ of $\tau$} is defined by
\begin{equation*}
	\tau^*(\alpha)=\inf_{q\in(1,\infty)}\left(\alpha q-\tau(q)\right)\in[-\infty,\infty),\quad\alpha\in\R.
\end{equation*}
For $q>1$, let $\tau'^{,+}(q)$ and $\tau'^{,-}(q)$ be the right and left derivative of $\tau$ at $q$. By the concavity of $\tau$, they exist and are finite values and, if we re-write $\alpha^+=\tau'^{,+}(q)$ and $\alpha^-=\tau'^{,-}(q)$, we have
\begin{equation}\label{Legendre_transform_at_left_right_derivative}
	\tau^*(\alpha^+)=\alpha^+q-\tau(q),\quad \tau^*(\alpha^-)=\alpha^-q-\tau(q).
\end{equation}
Hence, if $\tau$ is differentiable at $q>1$ and we write $\alpha=\tau'(q)$, then
\begin{equation}\label{Legendre_transform_at_derivative}
	\tau^*(\alpha)=\alpha q-\tau(q).
\end{equation}
We notice that, by the concavity, the differentiability of $\tau$ at $q$ always holds except for at most countable $q$.
Actually, $\tau^*(\alpha)$ is non-negative, which will be seen from Corollary \ref{tau^*(alpha)_is_nonnegative} and is stated in \cite[Proposition 3.5]{LN99}. Furthermore, once we have $\tau^*(\alpha)=0$ for $q_0>1$ and $\alpha=\tau'(q_0)$, it holds that
\begin{equation*}
\tau(q)=\alpha q\quad\text{for any }q\geq q_0
\end{equation*}
(Lemma \ref{tau(q)_coincides_the_line_through_the_origin_proof_of_the_main_theorem}).

Then, we state our $L^q$ norm flattening theorem.
Here and throughout this paper, for parameters or objects $a, b,c,\dots$ and $X>0$, we write $X\gg_{a,b,c,\dots}1$ or $X\ll_{a,b,c,\dots}1$ if $X$ is sufficiently large or small only in terms of $a,b,c,\dots$, respectively.\footnote{This notation may be a little confusing, because sometimes we write $X\ll_{a,b,c,\dots}Y$ meaning that there is a constant $C>0$ determined only by $a,b,c,\dots$ such that $X\leq CY$. However, in this paper, we always take the above meaning whenever we write $X\gg_{a,b,c,\dots}1$ or $X\ll_{a,b,c,\dots}1$. When we express that $X\leq CY$ for some constant $C>0$ determined only by $a,b,c,\dots$, we write $X\leq O_{a,b,c,\dots}(1)Y$ or $X= O_{a,b,c,\dots}(1)Y$.}

\begin{thm}[The $L^q$ norm flattening theorem]\label{L^q_norm_flattening_theorem}
Let $\A$ be a non-empty finite family of elements of $G$ and assume that it is uniformly hyperbolic.
We take a probability measure $\mu$ on $G$ such that $\supp\ \mu=\A$ and write $\nu$ for the stationary measure of $\mu$. We also take an open subset $U_1\subsetneq\RP^1$ as in (\ref{domain_U1_the_IFS_acts_on}).
We write $\tau(q)\ (q>1)$ for the $L^q$ spectrum of $\nu$.
Let $q>1$ be such that $\tau(q)$ is differentiable at $q$ and assume that
\begin{equation*}
\tau(q)<q-1
\end{equation*}
and, for $\alpha=\tau'(q)$,
\begin{equation*}
\tau^*(\alpha)=\alpha q-\tau(q)>0.
\end{equation*}
Let $C,L>1$ be constants. Then, for any $\sigma>0$, there exists $\varepsilon=\varepsilon(M,\mu,q,\sigma)>0$ such that the following holds for sufficiently large $m\in\N,\ m\gg_{M,\mu,q,\sigma,\varepsilon,C,L}1$. Let $\theta$ be a Borel probability measure on $G$ and $r\in\N$ such that $\diam\ \supp\ \theta\leq L$, $C^{-1}2^r\leq\|g\|^2\leq C2^r$ and $u^-_g\notin U_1$ for every $g\in\supp\ \theta$. Assume that
\begin{equation*}
\|\theta^{(m)}\|_q^q\leq 2^{-\sigma m}.
\end{equation*}
Then, we have\footnote{The reason why $\|(\theta{\bm. }\nu)^{(m+r)}\|_q^q$ appears, not $\|(\theta{\bm .}\nu)^{(m)}\|_q^q$, is because we take the “essential contraction ratio” $\|g\|^{-2}\sim2^{-r}$ of the action of $g\in\supp\ \theta$ on $\RP^1$ into account (we will see it in Section \ref{subsection_G_action_on_RP^1}).}
\begin{equation*}
\|(\theta{\bm .}\nu)^{(m+r)}\|_q^q\leq 2^{-(\tau(q)+\varepsilon)m}.
\end{equation*}
\end{thm}
In the above, $u^-_g$ for $g\in G$ is the contracting singular direction of $g$, which will be defined in Section \ref{subsection_G_action_on_RP^1}.
We notice that, once we obtain Theorem \ref{L^q_norm_flattening_theorem}, for the proof of the main Theorem \ref{the_main_theorem_L^q_dim_Mobius_IFS}, we can adopt the similar argument as that of \cite{Shm19} to show Theorem \ref{L^q_dim_of_self_similar_measures} from Theorem \ref{L^q_norm_flattening_theorem_for_self_similar_measure}.

\subsection{About the proof: linearization and the $L^q$ norm porosity lemma}\label{section_L^q_norm_porosity}

Here, we consider how to prove the $L^q$ norm flattening Theorem \ref{L^q_norm_flattening_theorem}. We first go back to the proof in \cite{Shm19} of Theorem \ref{L^q_norm_flattening_theorem_for_self_similar_measure}. As well-known in this field, this is a great achievement by applications of tools in {\it additive combinatorics}, which follows the work of Hochman \cite{Hoc14}.
More precisely, Shmerkin showed {\it the inverse theorem for $L^q$ norms of linear convolutions} \cite[Theorem 2.1]{Shm19} (Theorem \ref{the_inverse_theorem_for_L^q_norms} in this paper) by using sophisticated tools in additive combinatorics, and deduced Theorem \ref{L^q_norm_flattening_theorem_for_self_similar_measure} from the combination of it and the “self-similarity” of a self-similar measure in terms of $L^q$ norms.

We want to prove Theorem \ref{L^q_norm_flattening_theorem} by the similar method as that of \cite{Shm19}.
However, we can't directly apply it to our case. It is because the action of $G$ on $\RP^1$ is not linear, and hence the convolution $\theta{\bm .}\nu$ of probability measures on $G$ and $\RP^1$ is a non-linear convolution. On the other hand, it is much hard (seems impossible) to extend the inverse theorem for $L^q$ norms to non-linear convolutions, because, as we remarked above, the inverse theorem is proved by additive combinatorics, for which it is essential that convolutions are linear.
Therefore, to prove Theorem \ref{L^q_norm_flattening_theorem}, we have to overcome this gap.

Then, we turn to the work of Hochman and Solomyak \cite{HS17}, and recall their result Theorem \ref{Hausdorff_dim_of_Furstenberg_measures}.
The essential for the proof of Theorem \ref{Hausdorff_dim_of_Furstenberg_measures} is “the entropy growth theorem” under the $G$-action on $\RP^1$ \cite[Theorem 5.13]{HS17}, which plays the corresponding role to the $L^q$ norm flattening theorem for entropy.
This theorem is essentially due to the inverse theorem for {\it entropy} of linear convolutions, which is shown in the groundbreaking work of Hochman \cite{Hoc14}.
However, this inverse theorem is also based on the ideas of additive combinatorics,
so it is much hard to extend itself to non-linear convolutions.

Hochman and Solomyak overcame this obstruction by {\it linearization} of the action $G\times\RP^1\ni(g,x)\mapsto gx\in\RP^1$. Linearization is the idea of restricting the action to much small components (on both $G$ and $\RP^1$) and approximating it by a linear action on each component. Their idea is roughly, for a Borel probability measures $\theta$ on $G$ and the stationary measure $\nu$, (i) decomposing $\theta{\bm .}\nu$ into the average of convolutions $\theta_{g,i}{\bm .}{\nu_{x,i}}$ of small component measures $\theta_{g,i}$ and $\nu_{x,i}$ (where $i$ represents the scale), (ii) approximating each $\theta_{g,i}{\bm .}{\nu_{x,i}}$ by a linear convolution $\theta_{g,i}{\bm.}\delta_x*S\nu_{x,i}$ (where $S$ is some scale change), (iii) applying the inverse theorem for entropy of linear convolutions to each $\theta_{g,i}{\bm.}\delta_x*S\nu_{x,i}$ and obtain some entropy growth for each of them, and (iv) obtaining the entire entropy growth of $\theta{\bm .}\nu$ as the average of growth at each component.

Hochman and Solomyak's idea was realized by some nice properties of entropy. To recover the entire entropy from entropy of small components, they used local entropy average formulae:
\begin{equation}\label{local_entropy_average}
	\frac{1}{n}H(\nu,\D_n)=\mathbb{E}_{1\leq i\leq n}\left(\frac{1}{m}H(\nu_{x,i},\D_{i+m})\right)+O\left(\frac{m}{n}\right)
\end{equation}
and
\begin{equation}\label{local_entropy_average_convolution}
	\frac{1}{n}H(\theta*\nu,\D_n)\geq \mathbb{E}_{1\leq i\leq n}\left(\frac{1}{m}H(\theta_{g,i}*\nu_{x,i},\D_{i+m})\right)-O\left(\frac{1}{m}+\frac{m}{n}\right)
\end{equation}
(see \cite[Lemmas 5.1 and 5.2]{HS17}) and the analogy of (\ref{local_entropy_average_convolution}) for non-linear convolutions (\cite[Lemma 5.3]{HS17}).
At the step (iii) above, {\it entropy porosity} of the stationary measure $\nu$ plays an important role, that is, the property that most of the small components $\nu_{i,x}$ have essentially the same entropy as the entire $\nu$. This property enables us to get entropy growth from the inverse theorem for most of the components.

Our strategy to prove Theorem \ref{L^q_norm_flattening_theorem} is to overcome the gap stated in the earlier of this section using linearization like \cite{HS17}. However, this is not straightforward. It is because, for $L^q$ norms, there are no corresponding properties to the nice ones of entropy, e.g. local entropy average formulae (\ref{local_entropy_average}) and (\ref{local_entropy_average_convolution}) or entropy porosity. Hence, we need to establish some different method to linearize the $G$-action on $\RP^1$ in a suitable manner for $L^q$ norms so that we can get the situation to which Shmerkin's inverse theorem for $L^q$ norms of linear convolutions is applicable.

{\it The} following {\it $L^q$ norm porosity lemma} enables us to realize the linearization method stated as above. We first prepare the notation.
For a Borel probability measure $\theta$ on a second countable and locally compact space and a Borel subset $E$ such that $\theta(E)>0$, we write $\theta|_E$ for the (non-normalized) restriction of $\theta$ to $E$ and $\theta_E=\theta|_E/\theta(E)$ for the component measure of $\theta$ on $E$, that is, the normalized restriction of $\theta$ to $E$.
For a Borel probability measure $\nu$ on $\RP^1$ and $I\in\D_s$ such that $\nu(I)>0$, we write
\begin{equation*}
	\widehat{\nu_I}=\frac{1}{\nu(2I)}\nu|_I=\frac{\nu(I)}{\nu(2I)}\nu_I,
\end{equation*}
where $2I$ is the $2$ times extension of the $2^{-s}$ dyadic interval $I$ in $\RP^1$ with the same center.
In addition, for $x\in\RP^1$, we define $f_x:G\rightarrow \RP^1$ by
\begin{equation*}
	f_x(g)=gx,\quad g\in G.
\end{equation*}

Here, we state the $L^q$ norm porosity lemma. We notice that it requires the assumption of the positivity of the Legendre transform (\ref{condition_tau^*(apha)_is_positive}), and this is the reason why the $L^q$ norm flattening Theorem \ref{L^q_norm_flattening_theorem} requires this assumption.

\begin{lem}[The $L^q$ norm porosity lemma]\label{L^q_norm_porosity}
Let $\A$ be a non-empty finite family of elements of $G$ and assume that $\A$ is uniformly hyperbolic. We take a probability measure $\mu$ on $G$ such that $\supp\ \mu=\A$ and write $\nu$ for the stationary measure of $\mu$. We also take an open subset $U_1\subsetneq\RP^1$ as in (\ref{domain_U1_the_IFS_acts_on}).\footnote{Since $\A$, $\nu$, $\tau$ and $U_1$ can be thought to be determined by $\mu$ (we notice thar $\supp\ \mu=\A$), in the following statement, we write $\mu$ for dependence on some of these things all together.}
We write $\tau(q)\ (q>1)$ for the $L^q$ spectrum of $\nu$.
Let $q>1$ and assume that $\tau(q)$ is differentiable at $q$. We write $\alpha=\tau'(q)$ and further assume that
\begin{equation}\label{condition_tau^*(apha)_is_positive}
\tau^*(\alpha)=\alpha q-\tau(q)>0.
\end{equation}
Let $0<\sigma<1$ be a constant and $0<\delta<\sigma, \delta\ll_{M,\mu,q,\sigma} 1$ be a sufficiently small constant. Furthermore, let $0<\varepsilon<\delta, \varepsilon\ll_{M,\mu,q,\sigma,\delta}1$ be a sufficiently small constant and  $C,L>1$ be constants.
Then, there exist a small constant $\kappa=\kappa(M,\mu,q,\sigma)>0$ (not depending on $\delta$ and $\varepsilon$) and
a constant $D=D(M,\mu,q,\sigma,\delta,\varepsilon,C)\in\N$ such that the following holds for sufficiently large $l\in\N$, $l\gg_{M,\mu,q,\sigma,\kappa,\delta,\varepsilon,D,C,L}1$.
Let $m=lD$ and $\theta$ be a Borel probability measure on $G$ and $r\in\N$ such that $\diam\ \supp\ \theta\leq L$, $C^{-1}2^r\leq\|g\|^2\leq C2^r$ and $u_g^-\notin U_1$ for every $g\in\supp\ \theta$.
Assume that
\begin{enumerate}
\renewcommand{\labelenumi}{(\roman{enumi})}
\item $\|\theta^{(m)}\|_q^q\leq 2^{-\sigma m}$,
\item $2^{-\varepsilon m}\|\nu^{(m)}\|_q^q\leq \|(\theta{\bm .}\nu)^{(m+r)}\|_q^q$.
\end{enumerate}
We write $n=\lfloor \delta l \rfloor$. Then, there exist $i\in\N$ with $n<i\leq l-n$, $\xi\in\D^G_{iD}$ with $\theta(\xi)>0$, $I_0\in\D_{iD}$ with $\nu(I_0)>0$, a Borel probability measure $\rho_\xi$ on $\xi$ with $\rho_\xi\ll\theta_\xi$ and $x_0\in I_0\cap K$ (where $K$ is the attractor of $\A$) such that
\begin{enumerate}
\renewcommand{\labelenumi}{(\Roman{enumi})}
\item $\|(f_{x_0}\rho_\xi)^{((i+n)D+r)}\|_q^q\leq 2^{-\kappa nD}$,
\item $2^{-(\tau(q)+\sqrt{\delta}/2)nD}\leq \|({\rho_\xi}{\bm .}\widehat{\nu_{I_0}})^{((i+n)D+r)}\|_q^q$.
\end{enumerate}
\end{lem}

We will deduce the $L^q$ norm flattening Theorem \ref{L^q_norm_flattening_theorem} from this Lemma \ref{L^q_norm_porosity} as follows.
If we deny the conclusion of Theorem \ref{L^q_norm_flattening_theorem}, we have a Borel probability measure $\theta$ on $G$ satisfying the assumption of Lemma \ref{L^q_norm_porosity}. Then, by this lemma, we can obtain a small component $\nu_{I_0}$ of $\nu$ and a small fraction $\rho_\xi$ of $\theta$ satisfying (I) and (II). But, the non-linear convolution $\rho_\xi{\bm .}\nu_{I_0}$ can be approximated by the linear convolution $f_{x_0}\rho_\xi*S_{2^{-r}}\nu_{I_0}$ ($S_{2^{-r}}$ is the scale change by $2^{-r}$). Hence, we can apply Shmerkin's inverse theorem for $L^q$ norms of linear convolutions to $f_{x_0}\rho_\xi*S_{2^{-r}}\nu_{I_0}$, and, by the similar method as the proof of \cite[Theorem 5.1]{Shm19}, we will be able to get a contradiction.

The name of “porosity” represents that some small component $\nu_{I_0}$ of $\nu$ and some small fraction $\rho_\xi$ of $\theta$ have essentially the same properties (I) and (II) as the properties (i) and (ii) of the entire $\nu$ and $\theta$. However, we emphasize that the property (I) is up-graded from (i). That is, not only that $\rho_\xi$ has the flat $L^q$ norm on $G$, we can also have that its “projection” on $\RP^1$ by $f_{x_0}:G\ni g\mapsto gx_0\in\RP^1$ has the flat $L^q$ norm. This will be much important for the application of Shmerkin's inverse theorem for $L^q$ norms. We notice in advance that the assumption (\ref{condition_tau^*(apha)_is_positive}) and the parametrization of $n=\lfloor \delta l\rfloor$ which may seem strange at the first sight work to get this up-grading from (i) to (I) (see (\ref{separation_of_D'xi_proof_of_porosity})). In particular, the assumption (\ref{condition_tau^*(apha)_is_positive}) ensures $\kappa$ obtained from our proof to be positive.

\subsection{Organization of the paper}\label{section_organization}

In Section \ref{section_preliminaries}, we confirm basic facts on the action of a finite uniformly hyperbolic family $\A\subset G$ on $\RP^1$ and a stationary measure for $\A$. We also see multifractal properties of a stationary measure as in \cite{Shm19}.
Many of them are from \cite{HS17} and \cite{Shm19}.

In Section \ref{section_counterexamples}, we focus on counterexamples to Problem \ref{main_problem_L^q_dim_of_Furstenberg_measures}. In particular, we show Theorem \ref{counterexample_to_natural_extension}. As we saw in Section \ref{section_main_theorem}, according to the main Theorem \ref{the_main_theorem_L^q_dim_Mobius_IFS}, such counterexamples must be “singular”, so we also show Proposition \ref{justification_heuristic_singularity} in this section.

In Section \ref{section_proof_of_L^q_norm_porosity}, we prove the $L^q$ norm porosity Lemma \ref{L^q_norm_porosity}. It can be said that the most essential novelty of this paper is in this section.

Using Lemma \ref{L^q_norm_porosity}, we prove the $L^q$ norm flattening Theorem \ref{L^q_norm_flattening_theorem} in Section \ref{section_proof_of_L^q_norm_flattening_theorem}.
We prove the main Theorem \ref{the_main_theorem_L^q_dim_Mobius_IFS} in the final Section \ref{section_proof_of_the_main_theorem}.
\subsection{Notation}\label{section_notation}

As we said in Section \ref{section_L^q_norm_flattening}, for parameters or objects $a, b,c,\dots$ and $X>0$, we write $X\gg_{a,b,c,\dots}1$ or $X\ll_{a,b,c,\dots}1$ if $X$ is sufficiently large or small only in terms of $a,b,c,\dots$, respectively.

For $X,Y>0$, we write $X=O_{a,b,c,\dots}(1)Y$ or $X\leq O_{a,b,c\dots}(1)Y$ if there is a constant $C>0$ depending only on $a,b,c,\dots$ such that $X\leq CY$. In this case, we also write $Y=\Omega_{a,b,c\dots}(1)X$ or $Y\geq\Omega_{a,b,c,\dots}(1)X$. If $X\leq O_{a,b,c,\dots}(1)Y$ and $Y\leq O_{a,b,c,\dots}(1)X$, we write $X=\Theta_{a,b,c,\dots}(1)Y$.
In the same sense, for $A>0$, we write $O_{a,b,c,\dots}(A)$ for some constant $C\in\R$ such that $|C|\leq O_{a,b,c,\dots}(1)A$ and $\Omega_{a,b,c\dots}(1)$ for some constant $\eta>0$ such that $\eta\geq \Omega_{a,b,c,\dots}(1)A$.

In the following table, we summarize main notational conventions in this paper.

\vspace{\baselineskip}
\begin{tabular}{ll}
\hline
$\N$&The set of natural numbers $\{1,2,3,\dots\}$\\
$[l]$& $\{0,1,\dots,l-1\}$\\
$\I,\I^*$&Index set ($|\I|<\infty$) and word set $\I^*=\bigcup_{n\in\N}\I^n$\\
$G,d_G$&The linear group $\SL_2(\R)$ with the metric $d_G$ on $G$ determined by a fixed\\
&left-invariant Riemannian metric\\
$B^G_r(g)$& Open ball in $G$ of radius $r$ and center $x$\\
$\|\cdot\|$& Operator norm for on $M_2(\R)$\\
$\A=\{A_i\}_{i\in\I}$&Non-empty and finite family of elements of $G$ which is uniformly hyperbolic\\
$\mu$&Probability measure on $G$ such that $\supp\ \mu=\A$\\
$p=(p_i)_{i\in\I}$&Non-degenerate probability vector such that $\mu=\sum_{i\in\I}p_i\delta_{A_i}$\\
$\nu$& The stationary measure on $\RP^1$ of $\mu$\\
$\nu|_I$& Non-normalized restriction of $\nu$ to $I$\\
$\nu_I$& Component measure of $\nu$ on $I$, $\nu_I=\nu|_I/\nu(I)$\\
$aI$& Expansion or contraction of an interval $I$ by $a>0$ with the same center\\
$\widehat{\nu_I}$&$\widehat{\nu_I}=\nu|_I/\nu(2I)$\\
$K$& The attractor of $\A$\\
$\Psi_q(s), \widetilde{\tau}(q)$& The canonical pressure function and its zero\\
$U,U_1,U_0$& Open subsets of $\RP^1$ associated to the action of $\A$ (see (\ref{domain_U1_the_IFS_acts_on}) and (\ref{domains_the_Möbius_IFS_acts_on_preliminaries}))\\
$C_1,C_2$& Constants associated to the contraction property of the action of $\A$\\
&(see Corollaries \ref{contraction_on_U_by_A_preliminaries} and \ref{Lipschitz_continuity_of_the_action})\\
$\dim_H$& Hausdorff dimension\\
$\D_m$& $2^{-m}$ dyadic partition of $\RP^1$ or $\R/\pi/Z$\\
$\D_{u,m}$&$2^{-m}$ dyadic partition of $\RP^1$ or $\R/\pi\Z$ with the base $u$\\
$\D'(E)$& Set of atoms of $\D'\subset\D_{u,m}$ intersecting a set $E$\\
$\|\nu^{(m)}\|_q$& $2^{-m}$ $L^q$ norm for a Borel probability measure $\nu$ on $\RP^1$\\
$\tau(\nu,q)=\tau(q)$& The $L^q$ spectrum of $\nu$\\
$D(\nu,q)=D(q)$& The $L^q$ dimension of $\nu$\\
$\tau^*(\alpha)$& The Legendre transform of $\tau$\\
$\D^G_m,M$& $2^{-m}$ dyadic-like partition of $G$ and its associated constant\\
$\theta$&Borel probability measure on $G$ with the compact support\\
$\theta|_\xi,\theta_\xi$& Non-normalized restriction and component measure of $\theta$ on $\xi$\\
$\|\theta^{(m)}\|_q$& $2^{-m}$ $L^q$ norm of a Borel probability measure $\theta$ on $G$\\
$f,f_x$& The maps $G\times\RP^1\ni(g,x)\mapsto gx\in\RP^1$ and $G\ni g\mapsto gx\in\RP^1$ for $x\in\RP^1$\\
$\theta{\bm .}\nu$& Push-forward measure of $\theta\times\nu$ by $f$\\
$\theta'*\nu'$& Linear convolution of two finite Borel measures $\theta',\nu'$ on $\R/\pi\Z$\\
$\angle_u(x)$&Identification of $\RP^1$ with $\R/\pi\Z$ by the angle between $x$ and $u$\\
&(by default $u=[1:0]$)\\
$\angle$&$\angle=\angle_{[1:0]}$\\
$d_{\RP^1}$&Metric on $\RP^1$ by angles\\
$B_r(x)$& Open ball in $\RP^1$ (or $\R/\pi\Z$) of radius $r$ and center $x$\\
$\lambda_g^+,\lambda_g^-$& Singular values of $g\in G$ ($\lambda_g^+\geq1\geq\lambda_g^->0$)\\
$u_g^+,u_g^-$& Singular vectors of $g$ corresponding to $\lambda_g^+,\lambda_g^-$, respectively\\
$v_g^+,v_g^-$& $v_g^+=gu_g^+/\lambda_g^+,v_g^-=gu_g^-/\lambda_g^-$\\
\hline
\end{tabular}

\begin{tabular}{ll}
\hline
$\pi$& Coding map from $\I^\N$ to $K$\\
$\Omega_m$& $2^m$ stopping word set (see Section \ref{subsection_stationary_measures})\\
$\mu_m$& $\mu_m=\sum_{i\in\Omega_m}p_i\delta_{A_i}$\\
$\pi_m$& $2^m$ stopping coding map (see Section \ref{subsection_stationary_measures})\\
$P$\qquad\qquad\qquad\quad&Bernoulli measure on $\I^\N$ associated to $p=(p_i)_{i\in\I}$\\
$F$&Identification of $\R\sqcup\{\infty\}$ with $\RP^1$\\
$\Pi$&Canonical map from $\R$ to $\R/\pi\Z$\\
$S_a$&Multiplication on $\R/\pi\Z$ by $a$\\
$T_u$&Translation on $\R/\pi\Z$ by $u$\\
\hline
\end{tabular}

\subsection*{Acknowledgements}
The author is grateful to Yuki Takahashi and Yuto Nakajima. He organized the seminar on Shmerkin's paper \cite{Shm19} with them, and then this problem was brought to him by Yuki Takahashi. Takahashi discussed this problem with him and helped him with this work.
Nakajima talked with him on both mathematical and non-mathematical things and it was a big support for him.

He is also grateful to Masayuki Asaoka, Mitsuhiro Shishikura, Peter Varjú and Johannes Jaerisch for helpful discussions, and Pablo Shmerkin and Boris Solomyak for helpful comments.

This work is supported by JSPS KAKENHI Grant Number JP23KJ1211.

\section{Preliminaries}\label{section_preliminaries}

In this section, we see basic facts on the action of a finite uniformly hyperbolic family $\A\subset G$ on $\RP^1$ and a stationary measure for $\A$, and multifractal properties of a stationary measure as in \cite{Shm19}.
We notice that many of them are from \cite{HS17} and \cite{Shm19}, but Section \ref{subsection_effective_biLipschitz} is not in them and will play an important role in the proof of the $L^q$ norm porosity Lemma \ref{L^q_norm_porosity}.

\subsection{Geometry of the $\SL_2(\R)$-action on $\RP^1$}\label{subsection_G_action_on_RP^1}

In this section, we see basic facts on geometry of the action of $G=\SL_2(\R)$ on $\RP^1$. All of the contents in this section are from \cite[Section 2]{HS17}.

We define the metric $d_{\RP^1}$ on $\RP^1$, where, for $x, y\in\RP^1$, $d_{\RP^1}(x,y)$ is the angle between the two lines corresponding to $x$ and $y$ in $\R^2$ (taking the value in $[0,\pi/2]$). For a fixed $u\in\RP^1$, we write $\angle_u(x)\ (x\in\RP^1)$ for the angle from $u$ to $x$ in the counterclockwise orientation (taking the value in $[0,\pi)$). If we consider $\angle_u$ as the $\R/\pi\Z$-valued function, this is a diffeomorphism between $\RP^1$ and $\R/\pi\Z$.
We also notice that $d_{\RP^1}(x,y)=|\angle_u(x)-\angle_u(y)|$, where the right-hand side is the canonical metric on $\R/\pi\Z$.

For a matrix $g\in G$, we write ${\lambda_g^+}^2,{\lambda_g^-}^2\ (\lambda_g^+\geq 1\geq\lambda_g^->0)$ for the two eigenvalues of the symmetric and positive definite matrix $g^*g$ (where $g^*$ is the transpose of $g$) and $u_g^+, u_g^-\in\R^2\setminus\{0\}$ for the corresponding normalized eigenvectors (determined up to multiplication by $\pm 1$). We call $\lambda_g^+,\lambda_g^-$ the {\it singular values} of $g$ and $u_g^+,u_g^-$ the {\it singular vectors} of $g$. We notice that $\|g\|=\lambda_g^+$ and $\lambda_g^-=(\lambda_g^+)^{-1}$. 
We also write $v_g^+=gu_g^+/\lambda_g^+, v_g^-=gu_g^-/\lambda_g^-$. Then $\{u_g^+,u_g^-\}$ and $\{v_g^+,v_g^-\}$ are the orthonormal bases in $\R^2$ and
\begin{equation*}
	gu_g^+=\lambda_g^+v_g^+,\quad gu_g^-=\lambda_g^-v_g^-.
\end{equation*}

Here, we see the action of $g$ on $\RP^1$ through the coordinates $\angle_{u_g^+}$ and $\angle_{v_g^+}$. We take $u_g^+,u_g^-\in\R^2\setminus\{0\}$ as $\angle_{u_g^+}(u_g^-)=\pi/2$.
Then, for $\theta\in\R/\pi\Z$, we have
\begin{align*}
	\angle_{v_g^+}\circ g\circ\angle_{u_g^+}^{-1}(\theta)=&\ \angle_{v_g^+}\circ g(\cos\theta\cdot u_g^++\sin\theta\cdot u_g^-)\\
	=&\ \angle_{v_g^+}(\lambda_g^+\cos\theta\cdot v_g^++\lambda_g^-\sin\theta\cdot v_g^-)\\
	=&\ \arctan\left(\frac{1}{{\lambda_g^+}^2}\tan\theta\right),
\end{align*}
where, in the third line, we used the fact that $g\in G=\SL_2(\R)$ does not change the orientation.
The derivative is
\begin{equation}\label{derivative_action_of_G_preliminaries}
	(\angle_{v_g^+}\circ g\circ\angle_{u_g^+}^{-1})'(\theta)=\frac{1}{{\lambda_g^+}^2\cos^2\theta+{\lambda_g^+}^{-2}\sin^2\theta}.
\end{equation}
By using this description, we can see the following. Here, $B_\varepsilon(x)=\{y\in\RP^1\left|\ d_{\RP^1}(x,y)<\varepsilon\right.\}$ is an open ball in $\RP^1$.

\begin{lem}[{\cite[Lemma 2.4]{HS17}}]\label{contraction_bounded_distortion}
	For $0<\varepsilon<\pi/2$, there exists a constant $C_\varepsilon>1$ such that, for every $g\in G$ and $x,y\in\RP^1\setminus B_\varepsilon(u_g^-)$, we have
	\begin{equation*}
		\frac{C_\varepsilon^{-1}}{{\lambda_g^+}^2}d_{\RP^1}(x,y)\leq d_{\RP^1}(gx,gy)\leq\frac{C_\varepsilon}{{\lambda_g^+}^2}d_{\RP^1}(x,y).
	\end{equation*}
\end{lem}

This can be seen as a contracting property and a bounded distortion property for the $G$-action on $\RP^1$, which are important in studying fractals or measures generated by an IFS.

\subsection{The action of a uniformly hyperbolic family in $\SL_2(\R)$}\label{subsection_action_of_uniformly_hyperbolic_subset}

We take a non-empty finite family $\A=\{A\}_{i\in\I}$ of elements of $G$ which is uniformly hyperbolic. By the definition of the uniform hyperbolicity, there are constants $c>0$ and $r>1$ such that
\begin{equation}\label{def_of_uniform_hyperbolicity_of_A_preliminaries}
	\|A_{i_1}\cdots A_{i_n}\|\geq cr^n\quad\text{for every }n\in\N, (i_1,\dots,i_n)\in\I^n.
\end{equation}
We fix $\A$ throughout Section \ref{section_preliminaries} (not assuming to be strongly Diophantine). In this section, we see properties of the action of $\A$ on $\RP^1$.

The fundamental fact is Proposition \ref{proposition_Möbius_uniformly_hyperbolic_SL_2(R)}. This gives us an open subset $U_0\subsetneq \RP^1$.
As we did in (\ref{domain_U1_the_IFS_acts_on}), we take three non-empty open subsets $U\subset U_1\subset U_0\subsetneq \RP^1$ with finite connected components having disjoint closures such that
\begin{equation}\label{domains_the_Möbius_IFS_acts_on_preliminaries}
	\overline{U}\subset U_1,\quad\overline{U_1}\subset U_0,\quad \overline{AU_0}\subset U\quad\text{for every }A\in\A,
\end{equation}
and fix them for $\A$ throughout Section \ref{section_preliminaries}.
We need the following lemma to see that the action of $\A$ on $U$ is really contracting.

\begin{lem}\label{center_of_expansion_notin_U}
	There exists $n_0\in\N$ such that, for every $n\geq n_0$ and $i\in\I^n$, we have
	\begin{equation*}
		u_{A_i}^-\notin U_1.
	\end{equation*}
\end{lem}

\begin{proof}
	We take large $n\in\N$ (specified later) and $i\in \I^n$. Then, by (\ref{def_of_uniform_hyperbolicity_of_A_preliminaries}), ${\lambda_{A_i}^+}^2=\|A_i\|^2\geq c^2r^{2n}$ and this is large. Assume that $u_{A_i}^-\in U_1$. Then, from $\overline{U_1}\subset U_0$, there is $0<\varepsilon<1$ determined by $U_1\subset U_0$ such that $\overline{B}_\varepsilon(u_{A_i}^-)=\left\{y\in\RP^1\left|\ d_{\RP^1}(u_{A_i}^-,y)\leq\varepsilon\right.\right\}\subset U_0$. If we take $0<\delta<1$ such that $\RP^1\setminus U_0$ contains an interval of the length $\delta$, then, by taking $n$ sufficiently large for $\varepsilon$ and $\delta$, we can see from Lemma \ref{contraction_bounded_distortion} that
	\begin{equation*}
		A_i\left(\RP^1\setminus B_\varepsilon(u_{A_i}^-)\right)\subset B_{\delta/4}(v_{A_i}^+).
	\end{equation*}
	This tells us that $\overline{B}_\varepsilon(u_{A_i}^-)$ is much expanded by $A_i$, and the two end points of the large interval $A_i\left(\overline{B}_\varepsilon(u_{A_i}^-)\right)$ is $\delta/2$-close. However, from $A\overline{U_0}\subset U_0$ for every $A\in\A$, we have $A_i\left(\overline{B}_\varepsilon(u_{A_i}^-)\right)\subset U_0$ and this contradicts the choice of $\delta$. So we showed that $u_{A_i}^-\notin U_1$.
\end{proof}

By Lemmas \ref{contraction_bounded_distortion} and \ref{center_of_expansion_notin_U}, we obtain the following corollary.

\begin{cor}\label{contraction_on_U_by_A_preliminaries}
	There exists a constant $C_1>1$ determined by $\A$ such that, for any $x,y\in U$ and any $g\in G$ such that $u^-_g\notin U_1$, we have
	\begin{equation*}
		\frac{C_1^{-1}}{\|g\|^2}d_{\RP^1}(x,y)\leq d_{\RP^1}(gx,gy)\leq \frac{C_1}{\|g\|^2}d_{\RP^1}(x,y).
	\end{equation*}
	In particular, for any $i\in\I^*$ and $x,y\in U$, we have
	\begin{equation*}
		\frac{C_1^{-1}}{\|A_i\|^2}d_{\RP^1}(x,y)\leq d_{\RP^1}(A_ix,A_iy)\leq\frac{C_1}{\|A_i\|^2}d_{\RP^1}(x,y).
	\end{equation*}
\end{cor}

\begin{proof}
	The first statement immediately follows from Lemma \ref{contraction_bounded_distortion} and $d_{\RP^1}(U,\RP^1\setminus U_1)>0$.
	The second statement for $i\in\bigcup_{n\geq n_0}\I^n$ follows from the first one and Lemma \ref{center_of_expansion_notin_U}. Since $\A$ is finite, by taking $C_1$ sufficiently large, we can make the second statement valid for any $i\in\I^*$.
\end{proof}

The above contraction property of the action of $\A$ enables us to apply the basic theory of contracting IFSs to this action. In particular, as stated in Section \ref{section_main_problem}, there exists a unique non-empty compact subset $K\subset U_0$, called the attractor of $\A$, such that
\begin{equation*}
	K=\bigcup_{i\in\I}A_iK.
\end{equation*}
From Lemmas \ref{contraction_bounded_distortion} and \ref{center_of_expansion_notin_U}, we can see that $K$ is independent of an open subset $U_0$ obtained from Proposition \ref{proposition_Möbius_uniformly_hyperbolic_SL_2(R)} (we omit the detail).

Let $x_0\in K$. Then we can define the map $\pi: \I^\N\rightarrow K$ by
\begin{equation*}
	\pi(i)=\lim_{n\to\infty}A_{i_1}\cdots A_{i_n}x_0,\quad i=(i_1,i_2,\dots)\in \I^\N,
\end{equation*}
which is independent of the choice of $x_0$. This $\pi$ is surjective (not injective in general), and called the {\it coding map}. By Corollary \ref{contraction_on_U_by_A_preliminaries}, for $i=(i_1, i_2,\dots), j=(j_1,j_2,\dots)\in \I^\N$ and $n\in\N$, we have
\begin{equation*}\label{continuity_of_pi}
	d_{\RP^1}(\pi(i),\pi(j))\leq\frac{C_1\pi}{\|A_{i_1}\cdots A_{i_n}\|^2}\quad\text{if }\ (i_1,\dots,i_n)=(j_1,\dots,j_n),
\end{equation*}
and this implies that $\pi:\I^\N\rightarrow K$ is H\"{o}lder continuous.

\subsection{Properties of stationary measures}\label{subsection_stationary_measures}

We take a non-degenerate probability vector $p=(p_i)_{i\in\I}$ and define the probability measure $\mu=\sum_{i\in\I}p_i\delta_{A_i}$ such that $\supp\ \mu=\A$. Then, as stated in Section \ref{section_main_problem},
from the contraction property of the action of $\A$ and the basic theory of IFSs, we can see that there exists a unique stationary measure $\nu$ of $\mu$ on $\RP^1$:
\begin{equation*}
	\nu=\mu{\bm .}\nu
\end{equation*}
such that $\nu(U_0)=1$, and it follows that $\supp\ \nu=K$\footnote{From $\nu(K)=1$, we can easily see that $\nu$ is independent of an open set $U_0$ obtained from Proposition \ref{proposition_Möbius_uniformly_hyperbolic_SL_2(R)}.}. In this section, we see properties of such the stationary measure $\nu$.
In particular, we will see that the limit in the definition of the $L^q$ spectrum exists for $\nu$.

Let $\pi:\I^\N\rightarrow K$ be the coding map and $P$ be the Bernoulli measure on $\I^\N$ determined from the probability measure $\sum_{i\in\I}p_i\delta_i$ on $\I$. Then, we can see that the push-forward measure $\pi P$ is $\mu$-stationary and $\pi P(K)=1$. Hence, by the uniqueness, we have
\begin{equation*}\label{nu_is_push_forward_of_Bernoulli_measure_preliminaries}
	\nu=\pi P.
\end{equation*}

For $m\in\N$, we define the set of {\it $2^m$ stopping words} by
\begin{equation*}
	\Omega_m=\left\{i=(i_1,\dots,i_n)\in \I^*\left|\ \|A_{i_1}\|^2<2^m,\dots,\|A_{i_1}\cdots A_{i_{n-1}}\|^2<2^m,\|A_{i_1}\cdots A_{i_n}\|^2\geq 2^m\right.\right\}.
\end{equation*}
By Corollary \ref{contraction_on_U_by_A_preliminaries}, for $i\in\Omega_m$ and $x,y\in U$, we have
\begin{equation}\label{contraction_2^{-m}}
	C_1^{-1}2^{-m}d_{\RP^1}(x,y)\leq d_{\RP^1}(A_ix,A_iy)\leq C_12^{-m}d_{\RP^1}(x,y)
\end{equation}
(where we have re-taken $C_1$ large).
We consider the vector $p=(p_i)_{i\in\Omega_m}$. This is a probability vector, because
\begin{align*}
	1&=\sum_{i_1\in\I}p_{i_1}\\&=\sum_{i_1\in \I,\ \|A_{i_1}\|^2\geq 2^m}p_{i_1}+\sum_{i_1\in \I,\ \|A_{i_1}\|^2< 2^m}p_{i_1}\\
	&=\sum_{i_1\in \I,\ \|A_{i_1}\|^2\geq 2^m}p_{i_1}+\sum_{i_1\in \I,\ \|A_{i_1}\|^2< 2^m}\sum_{i_2\in \I}p_{i_1}p_{i_2}\\
	&=\sum_{i_1\in\I,\ \|A_{i_1}\|^2\geq 2^m}p_{i_1}+\sum_{(i_1,i_2)\in \I^2,\  \|A_{i_1}\|^2<2^m,\ \|A_{i_1}A_{i_2}\|^2\geq2^m}p_{i_1}p_{i_2}+\sum_{(i_1,i_2)\in \I^2,\  \|A_{i_1}\|^2<2^m,\ \|A_{i_1}A_{i_2}\|^2<2^m}p_{i_1}p_{i_2}
\end{align*}
\begin{flalign*}
	\ \ &\quad\vdots&\\
	&=\sum_{i\in\Omega_m}p_i.&
\end{flalign*}
We define the finitely supported probability measure $\mu_m$ on $G$ by
\begin{equation*}
	\mu_m=\sum_{i\in\Omega_m}p_i\delta_{A_i}.
\end{equation*}
Then, we can extend $\nu=\mu{\bm .}\nu$ by the same way as above, and obtain that
\begin{equation}\label{nu_Omega_m}
	\nu=\mu_m{\bm.}\nu.
\end{equation}
For $i\in \I^\N$, we write $n(m,i)\in\N$ for the smallest integer such that $\|A_{i_1}\cdots A_{i_{n(m,i)}}\|^2\geq2^m$. That is, $n(m,i)$ is the unique integer such that $(i_1,\dots,i_{n(m,i)})\in\Omega_m$. We define the {\it $2^m$ stopping coding map} $\pi_m: \I^\N\rightarrow K$ by
\begin{equation*}
	\pi_m(i)=A_{i_1}\dots A_{i_{n(m,i)}}x_0,\quad i\in \I^\N.
\end{equation*}
Then, it is easily seen that
\begin{equation*}
	\pi_mP=\mu_m{\bm .}\delta_{x_0}.
\end{equation*}
Furthermore, by Corollary \ref{contraction_on_U_by_A_preliminaries}, we have
\begin{equation}\label{pi_pi_m}
	d_{\RP^1}(\pi(i),\pi_m(i))\leq\frac{C_1\pi}{\|A_{i_1}\dots A_{i_{n(m,i)}}\|^2}\leq C_1\pi 2^{-m}.
\end{equation}

For $m\in\N$, let
\begin{equation*}
	\D_m=\left\{\angle_{[1:0]}^{-1}[\pi k2^{-m},\pi(k+1)2^{-m})\right\}_{k=0}^{2^m-1}
\end{equation*}
be the $2^{-m}$ dyadic partition of $\RP^1$ (via the identification $\angle_{[1:0]}:\RP^1\cong\R/\pi\Z$). We recall that, for $q>1$, the $L^q$ norm of $\nu$ is
\begin{equation*}
	\|\nu^{(m)}\|_q=\left(\sum_{I\in\D_m}\nu(I)^q\right)^{1/q}
\end{equation*}
and the $L^q$ spectrum of $\nu$ is
\begin{equation*}
	\tau(q)=\liminf_{m\to\infty}\left(-\frac{1}{m}\log\|\nu^{(m)}\|_q^q\right)=\liminf_{m\to\infty}\left(-\frac{1}{m}\log\sum_{I\in\D_m}\nu(I)^q\right).
\end{equation*}
Here, we show the following.

\begin{prop}\label{existence_of_limit_of_L^q_dim}
	For any $q>1$ and the stationary measure $\nu$ as above, the limit in the definition of the $L^q$ spectrum exists, that is,
	\begin{equation*}
		\tau(q)=\lim_{m\to\infty}\left(-\frac{1}{m}\log\|\nu^{(m)}\|_q^q\right)=\lim_{m\to\infty}\left(-\frac{1}{m}\log\sum_{I\in\D_m}\nu(I)^q\right).
	\end{equation*}
\end{prop}

This proposition corresponds to \cite[Proposition 4.6]{Shm19} and is a special case of \cite[Theorem 1.1]{PS00}, but we give a proof for completeness. We first see the following basic property of the $L^q$ norm from \cite{Shm19}, which will be repeatedly used in this note.

\begin{lem}[{\cite[Lemma 4.1]{Shm19}}]\label{L^q_norms_of_two_partitions}
	Let $(X,\mathscr{B},\mu)$ be a finite measure space. Suppose $\mathcal{P}, \mathcal{Q}$ are two finite families of measurable subsets of $X$ and $M\geq 1$ such that each element of $\mathcal{P}$ is covered by at most $M$ elements of $\mathcal{Q}$ and each element of $\mathcal{Q}$ intersects with at most $M$ elements of $\mathcal{P}$. Then, for every $q>1$, we have
	\begin{equation*}
		\sum_{P\in\mathcal{P}}\mu(P)^q\leq M^q\sum_{Q\in\mathcal{Q}}\mu(Q)^q.
	\end{equation*}
\end{lem}

The proof is simple, so we skip it.

\begin{proof}[Proof of Proposition \ref{existence_of_limit_of_L^q_dim}]
	Let $q>1$.
	It is sufficient to show that $(\log(C'\|\nu^{(m)}\|_q^q))_{m=1}^\infty$ is subadditive for some constant $C'>0$, that is, there exists a constant $C'>0$ such that, for $m,n\in\N$,
	\begin{equation*}
		\|\nu^{(m+n)}\|_q^q\leq C'\|\nu^{(m)}\|_q^q\|\nu^{(n)}\|_q^q.
	\end{equation*}
	
	We recall $f_{x_0}:G\ni g\mapsto gx_0\in \RP^1$.
	By (\ref{nu_Omega_m}), we have for $m,n\in\N$ that
	\begin{align*}
		\|\nu^{(m+n)}\|_q^q&=\sum_{I\in\D_{m+n}}\nu(I)^q\\
		&=\sum_{I\in\D_{m+n}}\left(\mu_m{\bm .}\nu(I)\right)^q\\
		&=\sum_{I\in\D_{m+n}}\left(\int_G\nu(A^{-1}I)\ d\mu_m(A)\right)^q\\
		&=\sum_{I\in\D_{m+n}}\left(\sum_{J\in\D_m}\int_{f_{x_0}^{-1}J}\nu(A^{-1}I)\ d\mu_m(A)\right)^q.
	\end{align*}
	Here, for each $I\in\D_{m+n}$, assume that $J\in\D_m$ satisfies $\int_{f_{x_0}^{-1}J}\nu(A^{-1}I)\ d\mu_m(A)\neq0$. Since $\supp\ \nu=K$, this tells us that there is $i\in\Omega_m$ such that $f_{x_0}(A_i)=A_ix_0\in J$ and $A_i^{-1}I\cap K\neq\emptyset$. Hence, we have $A_iK\cap J\neq\emptyset$ and $A_iK\cap I\neq\emptyset$. By (\ref{contraction_2^{-m}}), we have $\diam\ A_iK\leq C_1\pi 2^{-m}$. So the number of $J\in\D_m$ such that $A_iK\cap J\neq\emptyset$ and $A_iK\cap I\neq\emptyset$ for some $i\in\Omega_m$ is at most $O_\A(1)$, and hence we have
	\begin{equation*}
		\left(\sum_{J\in\D_m}\int_{f_{x_0}^{-1}J}\nu(A^{-1}I)\ d\mu_m(A)\right)^q\leq O_{\A,q}(1)\sum_{J\in\D_m}\left(\int_{f_{x_0}^{-1}J}\nu(A^{-1}I)\ d\mu_m(A)\right)^q.
	\end{equation*}
	Therefore, we obtain that
	\begin{equation}\label{estimate_1_existence_limit_Lq_norm}
		\|\nu^{(m+n)}\|_q^q\leq O_{\A,q}(1)\sum_{J\in\D_m}\sum_{I\in\D_{m+n}}\left(\int_{f_{x_0}^{-1}J}\nu(A^{-1}I)\ d\mu_m(A)\right)^q.
	\end{equation}
	
	For each $J\in\D_m$ and $I\in\D_{m+n}$, by Hölder's inequality, we have
	\begin{align*}
		\int_{f_{x_0}^{-1}J}\nu(A^{-1}I)\ d\mu_m(A)&\leq\left(\int_{f_{x_0}^{-1}J}\nu(A^{-1}I)^q\ d\mu_m(A)\right)^{1/q}\left(\int_{f_{x_0}^{-1}J}1^{q/(q-1)}\ d\mu_m(A)\right)^{(q-1)/q}\\
		&=\mu_m\left(f_{x_0}^{-1}J\right)^{(q-1)/q}\left(\int_{f_{x_0}^{-1}J}\nu(A^{-1}I)^q\ d\mu_m(A)\right)^{1/q}.
	\end{align*}
	From this and (\ref{estimate_1_existence_limit_Lq_norm}), it follows that
	\begin{align}\label{estimate_2_existence_limit_Lq_spectrum}
		\|\nu^{(m+n)}\|_q^q&\leq O_{\A,q}(1)\sum_{J\in\D_m}\sum_{I\in\D_{m+n}}\mu_m(f_{x_0}^{-1}J)^{q-1}\int_{f_{x_0}^{-1}J}\nu(A^{-1}I)^q\ d\mu_m(A)\nonumber\\
		&=O_{\A,q}(1)\sum_{J\in\D_m}\mu_m(f_{x_0}^{-1}J)^{q-1}\int_{f_{x_0}^{-1}J}\sum_{I\in\D_{m+n}}\nu(A^{-1}I)^q\ d\mu_m(A).
	\end{align}
	
	Here, for each $i\in\Omega_m$, we see $\sum_{I\in\D_{m+n}}\nu(A_i^{-1}I)^q=\sum_{I\in\D_{m+n}}\nu(A_i^{-1}I\cap K)^q=\sum_{I'\in A_i^{-1}\D_{m+n}\cap K}\nu(I')^q$. For each $I\in\D_{m+n}$ and $x,y\in A_i^{-1}I\cap K$, we have from (\ref{contraction_2^{-m}}) that
	\begin{equation*}
		C_1^{-1}2^{-m}d_{\RP^1}(x,y)\leq d_{\RP^1}(A_ix, A_iy)\leq \pi2^{-m-n},
	\end{equation*}
	and hence
	\begin{equation*}
		\diam\ A_i^{-1}I\cap K\leq C_1\pi2^{-n}.
	\end{equation*}
	So the number of elements of $\D_n$ which intersects with $A_i^{-1}I\cap K$ is at most $O_\A(1)$. 
	Furthermore, we notice that $A_i^{-1}\D_{m+n}$ is a partition of $\RP^1$ by intervals and $U$ is a union of finite proper open intervals with disjoint closures. Hence, for every $A_i^{-1}I\ (I\in\D_{m+n})$ but finite ones which contain endpoints of $U$, we have $A_i^{-1}I\subset U$ or $A_i^{-1}I\cap U=\emptyset$. For each $A_i^{-1}I$ which is contained in $U$, if we take $x,y\in A_i^{-1}I$ such that $A_ix$ and $A_iy$ are close to the two endpoints of $I$, we have from (\ref{contraction_2^{-m}}) that
	\begin{equation*}
		2^{-m-n}\leq d_{\RP^1}(A_ix,A_iy)\leq C_12^{-m}d_{\RP^1}(x,y),
	\end{equation*}
	and hence
	\begin{equation*}
		d_{\RP^1}(x,y)\geq C_1^{-1}2^{-n}.
	\end{equation*}
	Therefore, every $A_i^{-1}I\ (I\in\D_{m+n})$ contained in $U$ contains an open interval of length $\Omega_\A(2^{-n})$. So it follows that, for each $I'\in \D_n$, $I'\cap K\ (\subset U)$ intersects at most $O_\A(1)$ elements of $A_i^{-1}\D_{m+n}$. Hence, by Lemma \ref{L^q_norms_of_two_partitions}, we have
	\begin{equation*}
		\sum_{I\in\D_{m+n}}\nu(A_i^{-1}I)^q\leq O_{\A,q}(1)\sum_{I'\in\D_n}\nu(I')^q=O_{\A,q}(1)\|\nu^{(n)}\|_q^q.
	\end{equation*}
	From this and (\ref{estimate_2_existence_limit_Lq_spectrum}), we obtain
	\begin{align}\label{estimate_3_existence_limit_Lq_spectrum}
		\|\nu^{(m+n)}\|_q^q&=O_{\A,q}(1)\sum_{J\in\D_m}\mu_m(f_{x_0}^{-1}J)^{q-1}\int_{f_{x_0}^{-1}J}\sum_{I\in\D_{m+n}}\nu(A^{-1}I)^q\ d\mu_m(A)\nonumber\\
		&\leq O_{\A,q}(1)\|\nu^{(n)}\|_q^q\sum_{J\in\D_m}\mu_m(f_{x_0}^{-1}J)^q.
	\end{align}
	
	Finally, by $\nu=\pi P$, $f_{x_0}\mu_m=\mu_m{\bm .}\delta_{x_0}=\pi_mP$ and (\ref{pi_pi_m}), we have from Lemma \ref{L^q_norms_of_two_partitions} that
	\begin{equation*}
		\sum_{J\in\D_m}\mu_m(f_{x_0}^{-1}J)^q=\sum_{J\in\D_m}P(\pi_m^{-1}J)^q\leq O_{\A,q}(1)\sum_{J\in\D_m}P(\pi^{-1}J)^q =O_{\A,q}(1)\|\nu^{(m)}\|_q^q.
	\end{equation*}
	Hence, by this and (\ref{estimate_3_existence_limit_Lq_spectrum}), we obtain
	\begin{equation*}
		\|\nu^{(m+n)}\|_q^q\leq O_{\A,q}(1)\|\nu^{(n)}\|_q^q\|\nu^{(m)}\|_q^q,
	\end{equation*}
	which is what we want to show. So we complete the proof.
\end{proof}

\subsection{Effective separation of $\SL_2(\R)$ by the action on $\RP^1$}\label{subsection_effective_biLipschitz}

We need to see some effective Lipschitz property of the $G$-action on $\RP^1$. We first see the following fact (omit the proof). Here, we introduce the metric on $(\RP^1)^3$ by $d_{(\RP^1)^3}((x_1,x_2,x_3),(x'_1,x'_2,x'_3))=\sqrt{\sum_{i=1}^3d_{\RP^1}(x_i,x'_i)^2}$,
and, for $g\in G$ and $x=(x_1,x_2,x_3)\in (\RP^1)^3$, we write $gx=(gx_1,gx_2,gx_3)\in(\RP^1)^3$.
For $g\in G$ and $r>0$, $B^G_r(g)$ denotes the open ball in $G$ of radius $r$ and center $g$ with respect to the metric determined by the fixed left-invariant Riemannian metric on $G$.

\begin{prop}[{\cite[Lemma 2.5]{HS17}}]\label{bi_Lipschitz_of_the_action_HS17}
	We fix $0<\eta_0<1$. Then, there exist sufficiently small $0<r_{\eta_0}<1$ and $C_{\eta_0}>1$ such that the following holds. Let $g_0\in G$ and $x=(x_1,x_2,x_3)\in (\RP^1)^3$ be $\eta_0$-separated (that is, $d_{\RP^1}(x_i,x_j)\geq\eta_0$ for $i\neq j$), and assume that $x_i\notin B_{\eta_0}(u_{g_0}^-)$ for $i=1,2,3$. Then, for any $g\in B^G_{r_{\eta_0}}(g_0)$, we have
	\begin{equation*}
		\frac{C_{\eta_0}^{-1}}{\|g_0\|^2}d_G(g,g_0)\leq d_{(\RP^1)^3}(gx,g_0x)\leq\frac{C_{\eta_0}}{\|g_0\|^2}d_G(g,g_0).
	\end{equation*}
\end{prop}

As a corollary of this proposition, we immediately obtain the following, which we will use in the proof of Lemma \ref{L^q_norm_porosity}.

\begin{cor}\label{Lipschitz_continuity_of_the_action}
	There exist constants $0<r_1=r_1(\mu)<1$ and $C_2=C_2(\mu)>1$ such that, for any $g_0\in G$ such that $u_{g_0}^-\notin U_1$, $g\in B^G_{r_1}(g_0)$ and $x\in U$, we have
	\begin{equation*}
		d_{\RP^1}(gx,g_0x)\leq \frac{C_2}{\|g_0\|^2}d_G(g,g_0).
	\end{equation*}
\end{cor}

In the following, we want to get some effective estimate on how the constant $C_{\eta_0}$ becomes large if we make the distances between $x_i$ and $x_j$ smaller. To do this, we notice the following lemma, which is from \cite[Section 2.5]{HS17}.

\begin{lem}\label{bi_Lipschitz_of_the_action_of_G_proof_of_porosity}
	We fix $0<\eta_0<1$. Then, there exists $0<r_0<1$ such that, if $(x_1,x_2,x_3)\in (\RP^1)^3$ is $\eta_0$-separated, then, for any $h\in B^G_{r_0}(1_G)$, we have
	\begin{equation*}
		d_{(\RP^1)^3}(hx,x)=\Theta_{\eta_0}\left(d_G(h,1_G)\right).
	\end{equation*}
\end{lem}

We fix $0<\eta_0<\pi/4$ and $0<r_0<1$ as in Lemma \ref{bi_Lipschitz_of_the_action_of_G_proof_of_porosity} and take arbitrary $0<\varepsilon<\eta_0$. For this $\eta_0$ and $\varepsilon$, we take a sufficiently small constant $0<\rho_\varepsilon\ll_{\eta_0,\varepsilon} 1$ which will be specified later.
Let $(x_1,x_2,x_3)\in(\RP^1)^3$ be $\varepsilon$-separated and assume that any two of $x_1,x_2,x_3$ are at least $\pi/4$ close. For simplicity, we assume that $x_3,x_2,x_1$ lie in this order on $\RP^1\cong\R/\pi\Z$.
Let $g_0\in G$ such that $x_i\notin B_{\eta_0}(u_{g_0}^-)$ for $i=1,2,3$. By Lemma \ref{contraction_bounded_distortion}, for any $g\in B^G_{\rho_\varepsilon}(g_0)$, we have
\begin{equation}\label{contraction_effective_bi_Lipschitz_proof_of_porosity}
	d_{(\RP^1)^3}(gx,g_0x)=\Theta_{\eta_0}(\|g_0\|^{-2})d_{(\RP^1)^3}(g_0^{-1}gx,x).
\end{equation}
We write $h=g_0^{-1}g\in B^{G}_{\rho_\varepsilon}(1_G)$.

We take $k\in\SO(2)$ so that $kx_2=[1:0]\in\RP^1$. Then, we can take
\begin{equation*}
	a=
	\begin{pmatrix}
		e^{-t}&\\&e^t
	\end{pmatrix},
	\quad t>0
\end{equation*}
so that $\min\{d_{\RP^1}([1:0], akx_1), d_{\RP^1}([1:0], akx_3)\}=\eta_0$\footnote{Here, we assume $\min\{d_{\RP^1}([1:0],kx_1), d_{\RP^1}([1:0],kx_3)\}<\eta_0$. If not, we can directly apply Proposition \ref{bi_Lipschitz_of_the_action_HS17} and we have nothing to prove.}.
We assume $\min\{d_{\RP^1}([1:0], akx_1), d_{\RP^1}([1:0], akx_3)\}=d_{\RP^1}([1:0], akx_1)=\eta_0$ (see Figure 1).
From this and the same calculation as in Section \ref{subsection_G_action_on_RP^1}\footnote{We replace $\angle_{v^+}\circ g\circ\gamma_{u^+}$ in Section \ref{subsection_G_action_on_RP^1} with $\angle_{v^-}\circ g\circ\gamma_{u^-}$.},
it follows that
\begin{equation*}
	\arctan\left(e^{2t}\tan\varepsilon\right)\leq\arctan\left(e^{2t}\tan d_{\RP^1}(kx_1,[1:0])\right)=\eta_0,
\end{equation*}
and hence
\begin{equation}\label{how_large_diagonal_matrix_a_proof_of_porosity}
	e^{2t}\leq\frac{\tan\eta_0}{\tan\varepsilon}=O_{\eta_0}(\varepsilon^{-1}).
\end{equation}
\begin{figure}[h]\label{Figure_effective_bi_Lipschitz}
	\centering
	\includegraphics[width=120mm]{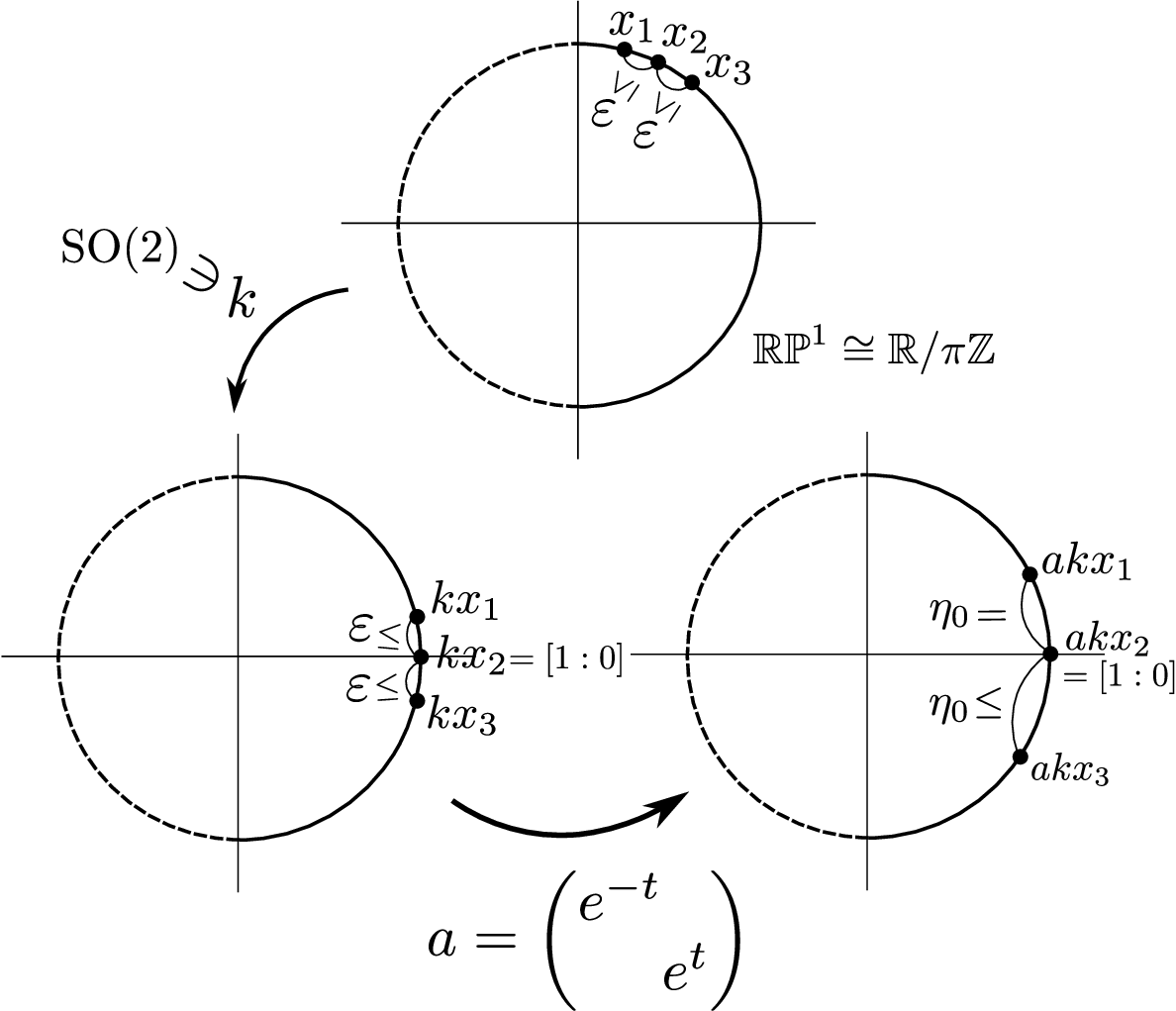}
	\caption{$x_1,x_2,x_3$ mapped by $k$ and $a$}
\end{figure}

We consider two points $akx, akhx\in(\RP^1)^3$. For each $i=1,2,3$, we have $kx_i\in B_{\pi/4}([1:0])$ and $khx_i=khk^{-1}kx_i\in B_{\pi/3}([1:0])$, since $k\in\SO(2)$, $h\in B^G_{\rho_\varepsilon}(1_G)$ and $\rho_\varepsilon\ll1$. So, by the same calculation as in Section \ref{subsection_G_action_on_RP^1} and the mean value theorem, there is $\theta_i\in\R$, $|\theta_i|<\pi/3$ such that\footnote{More precisely, we have to use (\ref{how_large_diagonal_matrix_a_proof_of_porosity}) and $d_{\RP^1}(x_i,hx_i)\leq O(1)d_G(h,1_G)\leq O(1)\rho_\varepsilon=O(c_0\varepsilon)$.}
\begin{equation*}
	d_{\RP^1}(akx_i,akhx_i)=\varphi'(\theta_i)d_{\RP^1}(kx_i,khx_i)=\varphi'(\theta_i)d_{\RP^1}(x_i, hx_i),
\end{equation*}
where
\begin{equation*}
	\varphi(\theta)=\arctan\left(e^{2t}\tan\theta\right).
\end{equation*}
Here, we have
$\varphi'(\theta)=1/(e^{-2t}\cos^2\theta+e^{2t}\sin^2\theta)$, and hence
\begin{equation*}
	e^{-2t}\leq \varphi'(\theta_i)\leq e^{2t}.
\end{equation*}
From these two relations and (\ref{contraction_effective_bi_Lipschitz_proof_of_porosity}), we have
\begin{equation}\label{action_of_the_diagonal_matrix_a_proof_of_porosity}
	\Omega_{\eta_0}\left(\|g_0\|^{-2}e^{-2t}\right)d_{(\RP^1)^3}(akx,akhx)\leq
	d_{(\RP^1)^3}(gx,g_0x)
	\leq O_{\eta_0}\left(\|g_0\|^{-2}e^{2t}\right)d_{(\RP^1)^3}(akx,akhx).
\end{equation}

We can see that
\begin{equation}\label{akx_eta_0_separated_effective_bi_Lipshitz}
	d_{(\RP^1)^3}(akx,akhx)=d_{(\RP^1)^3}(akx,akhk^{-1}a^{-1}\cdot akx),\quad akx\in (\RP^1)^3 \text{\ is $\eta_0$-separated.}
\end{equation}
Furthermore, by (\ref{how_large_diagonal_matrix_a_proof_of_porosity}) and $h\in B^G_{\rho_\varepsilon}(1_G)$, we have
\begin{equation}\label{norm_of_akhk^-1a^-1_bi_effective_bi_Lipschitz}
	\|akhk^{-1}a^{-1}-1_G\|\leq\|a\|\|a^{-1}\|\|h-1_G\|=e^{2t}\|h-1_G\|\leq O_{\eta_0}(\varepsilon^{-1}d_G(h,1_G))\leq O_{\eta_0}(\varepsilon^{-1}\rho_\varepsilon),
\end{equation}
where we have used the fact that the norm metric on $G$ and the metric $d_G$ by the left-invariant Riemannian metric of $G$ is bi-Lipschitz equivalent on a small neighborhood of $G$ (see \cite[Section 2.3]{HS17}). Hence, if we take a sufficiently small constant $c_0>0$ determined by $\eta_0$ and define $\rho_\varepsilon=c_0\varepsilon$, we have
\begin{equation*}
	\|akhk^{-1}a^{-1}-1_G\|\leq c_0^{1/2},
\end{equation*}
and hence, by the bi-Lipschitz equivalence of the norm metric and $d_G$ again,
\begin{equation*}
	d_G(akhk^{-1}a^{-1},1_G)<r_0.
\end{equation*}
By this and (\ref{akx_eta_0_separated_effective_bi_Lipshitz}), we can apply Lemma \ref{bi_Lipschitz_of_the_action_of_G_proof_of_porosity} and obtain that
\begin{equation}\label{bi_Lipschitz_at_akx_effective_bi_Lipschitz}
	d_{(\RP^1)^3}(akhx,akx)=\Theta_{\eta_0}\left(d_G(akhk^{-1}a^{-1},1_G)\right).
\end{equation}
Here, by (\ref{norm_of_akhk^-1a^-1_bi_effective_bi_Lipschitz}) and the bi-Lipschitz equivalence of the norm metric and $d_G$, we have
\begin{equation*}
	d_G(akhk^{-1}a^{-1},1_G)\leq O(e^{2t}d_G(h,1_G))=O(e^{2t}d_G(g,g_0))
\end{equation*}
(we recall that $h=g_0^{-1}g$). By the similar calculation, we also have
\begin{equation*}
	d_G(akhk^{-1}a^{-1},1_G)\geq \Omega\left(\|a\|^{-1}\|a^{-1}\|^{-1}d_G(h,1_G)\right)=\Omega\left(e^{-2t}d_G(g,g_0)\right).
\end{equation*}
From these two equations and (\ref{bi_Lipschitz_at_akx_effective_bi_Lipschitz}), we obtain that
\begin{equation}\label{dist_akhx_akx_effective_bi_Lipschitz}
	\Omega_{\eta_0}\left(e^{-2t}d_G(g,g_0)\right)\leq d_{(\RP^1)^3}(akhx,akx)\leq O_{\eta_0}(e^{2t}d_G(g,g_0)).
\end{equation}

By (\ref{action_of_the_diagonal_matrix_a_proof_of_porosity}), (\ref{dist_akhx_akx_effective_bi_Lipschitz}) and (\ref{how_large_diagonal_matrix_a_proof_of_porosity}), we obtain that
\begin{align*}
	\Omega_{\eta_0}\left(\|g_0\|^{-2}\varepsilon^2\right)d_G(g,g_0)&\leq \Omega_{\eta_0}\left(\|g_0\|^{-2}e^{-4t}\right)d_G(g,g_0)\\&\leq d_{(\RP^1)^3}(gx,g_0x)\\
	&\leq O_{\eta_0}\left(\|g_0\|^{-2}e^{4t}\right)d_G(g,g_0)\leq O_{\eta_0}\left(\|g_0\|^{-2}\varepsilon^{-2}\right)d_G(g,g_0).
\end{align*}
Finally, we obtain the following.
\begin{prop}\label{bi_Lipschitz_of_the_action_at_varepsilon_separated_points}
	We fix $0<\eta_0<\pi/4$. Then, there exist $0<c_0<1$ and $C_0>1$ such that the following holds. Let $0<\varepsilon<\eta_0$ and $x=(x_1,x_2,x_3)\in(\RP^1)^3$ be $\varepsilon$-separated such that any two of $x_1,x_2,x_3$ are at least $\pi/4$ close. Furthermore, let $g_0\in G$ such that $x_i\notin B_{\eta_0}(u^-_{g_0})$ for $i=1,2,3$. Then, for any $g\in B^G_{c_0\varepsilon}(g_0)$, we have
	\begin{equation*}
		\frac{{C_0}^{-1}\varepsilon^2}{\|g_0\|^2}d_G(g,g_0)\leq d_{(\RP^1)^3}(gx,g_0x)\leq\frac{C_0\varepsilon^{-2}}{\|g_0\|^2} d_G(g,g_0).
	\end{equation*}
\end{prop}

\subsection{Multifractal properties of stationary measures}\label{subsection_multifractal_structure}

Recall that we are fixing a non-empty finite family $\A\subset G$ which is uniformly hyperbolic and the stationary measure $\nu$ of a probability measure $\mu$ on $G$ such that $\supp\ \mu=\A$.
In this section, we see multifractal properties of $\nu$ regarding its $L^q$ norm. 
Almost all of the contents in this section are from \cite[Section 4.3]{Shm19}. The proofs are also the same, but we include them for completeness.

As stated in \cite[Section 4.3]{Shm19}, there is the heuristic idea for the $L^q$ spectrum $\tau(q)$, $\alpha=\tau'(q)$ and its Legendre transform $\tau^*(\alpha)$ when $\tau'(q)$ exists for $q>1$. That is, we have in this case that $\tau^*(\alpha)=\alpha q-\tau(q)$ and, for sufficiently large $m$,
\begin{equation*}
	\|\nu^{(m)}\|_q^q=\sum_{I\in\D_m}\nu(I)^q\sim 2^{-\tau(q)m}=2^{\tau^*(\alpha)m}2^{-\alpha mq},
\end{equation*}
and we think that most of the contribution to $\|\nu^{(m)}\|_q^q$ comes from $2^{\tau^*(\alpha)m}$ elements of $\D_m$ each of which has $\nu$-mass about $2^{-\alpha m}$. The following lemmas illustrate this idea in some sense.

\begin{lem}[{\cite[Lemma 4.10]{Shm19}}]\label{close_to_typical_tau^*(alpha)}
	Let $q>1$ and $\alpha^-=\tau'^{,-}(q)$. For $0<\varepsilon<1$, if $0<\delta\ll_{\varepsilon,q,\tau} 1$ is sufficiently small and $m\in\N$, $m\gg_{\varepsilon,q,\tau,\delta}1$ is sufficiently large, the following holds. If $\D'\subset\D_m$ satisfies
	\begin{enumerate}
		\renewcommand{\labelenumi}{(\arabic{enumi})}
		\item there is $\alpha'\geq0$ such that $2^{-\alpha'm}\leq\nu(I)\leq 2\cdot2^{-\alpha'm}$ for every $I\in\D'$,
		\item $\sum_{I\in\D'}\nu(I)^q\geq 2^{-(\tau(q)+\delta)m}$,
	\end{enumerate}
	then
	\begin{equation*}
		|\D'|\leq 2^{(\tau^*(\alpha^-)+\varepsilon)m}.
	\end{equation*}
\end{lem}

\begin{lem}[{\cite[Lemma 4.12]{Shm19}}]\label{less_than_typical_tau(q)}
	Let $q>1$ and $\alpha^+=\tau'^{,+}(q)$. For $0<\kappa<1$, there exists $\varepsilon=\varepsilon(\kappa,q,\tau)>0$ such that, if $m\in\N$, $m\gg_{\kappa,q,\tau,\varepsilon}1$ is sufficiently large, then the following holds.
	If $\D'\subset\D_m$ satisfies $|\D'|\leq 2^{(\tau^*(\alpha^+)-\kappa)m}$, then
	\begin{equation*}
		\sum_{I\in\D'}\nu(I)^q\leq2^{-(\tau(q)+\varepsilon)m}.
	\end{equation*}
\end{lem}

Assume that $\tau(q)$ is differentiable at $q>1$ and let $\alpha=\tau'(q)$.
Lemma \ref{close_to_typical_tau^*(alpha)} says that, for sufficiently large $m$, if we take a subset $\D'\subset\D_m$ such that $\nu(I)\ (I\in\D')$ are “uniform” and the restricted $L^q$ norm $\sum_{I\in\D'}\nu(I)^q$ is close to $2^{-\tau(q)m}$, then $|\D'|$ is indeed close to $2^{\tau^*(\alpha)m}$.\footnote{Actually, we can see that $\alpha'$ in the condition (1) on $\D'$ is close to $\alpha$.}
Lemma \ref{less_than_typical_tau(q)} says that, for sufficiently large $m$, if we restrict the $2^{-m}$ $L^q$ norm to $\D'\subset\D_m$ such that $|\D'|$ is strictly less than $2^{\tau^*(\alpha)m}$, then $\sum_{I\in\D'}\nu(I)^q$ is strictly less than $2^{-\tau(q)m}$.

As a corollary of Lemma \ref{close_to_typical_tau^*(alpha)}, we obtain that $\tau^*(\alpha)\geq0$ for $\alpha=\tau'(q)$, which is already mentioned in Section \ref{section_L^q_norm_flattening} and also stated in \cite[Proposition 3.5]{LN99}.

\begin{cor}\label{tau^*(alpha)_is_nonnegative}
	Let $q>1$ and $\alpha^-=\tau'^{,-}(q)$. Then, we have
	\begin{equation*}
		\tau^*(\alpha^-)\geq0.
	\end{equation*}
\end{cor}

\begin{proof}
	If $\tau^*(\alpha^-)<0$, we can take $0<\varepsilon<1$ and $0<\delta<1$ so that $\tau^*(\alpha^-)+\varepsilon<0$ and $\delta>0$ is sufficiently small in terms of this $\varepsilon$, $q$ and $\tau$ and satisfies Lemma \ref{close_to_typical_tau^*(alpha)}. For some sufficiently large $m\in\N$, if we can construct a non-empty subset $\D'\subset\D_m$ satisfying (1) and (2) of Lemma \ref{close_to_typical_tau^*(alpha)} for our $\delta$, then we obtain that $1\leq|\D'|\leq 2^{(\tau^*(\alpha^-)+\varepsilon)m}<1$ and a contradiction. Hence, what we have to do is to construct $\D'$. The following pigeonholing argument is from the proof of \cite[lemma 3.3]{Shm19}.
	
	By the definition of $\tau(q)$, for some sufficiently large $m\in\N$, we have
	\begin{equation}\label{tau^*(alpha)_nonnegative_L^q_spectrum_converge}
		\sum_{I\in\D_m}\nu(I)^q\geq 2^{-(\tau(q)+\delta/2)m}.
	\end{equation}
	For this $m$ and each integer $j$, we define
	\begin{equation*}
		\A_j=\left\{I\in\D_m\left|\ 2^{-j-1}\|\nu^{(m)}\|_q^{q/(q-1)}<\nu(I)\leq2^{-j}\|\nu^{(m)}\|_q^{q/(q-1)}\right.\right\}.
	\end{equation*}
	We notice that $\{I\in\D_m\left|\ \nu(I)>0\right.\}=\bigsqcup_{j\in\Z}\A_j$ and $\A_j=\emptyset$ if $j\leq -m-1$, since $\|\nu^{(m)}\|_q\geq 2^{-(q-1)m/q}$. Furthermore, we have
	\begin{equation*}
		\sum_{I\in\bigsqcup_{j\geq(q-1)^{-1}}\A_j}\nu(I)^q\leq\left(2^{-(q-1)^{-1}}\|\nu^{(m)}\|_q^{q/(q-1)}\right)^{q-1}\sum_{I\in\bigsqcup_{j\geq(q-1)^{-1}}\A_j}\nu(I)\leq\frac{\|\nu^{(m)}\|_q^q}{2}.
	\end{equation*}
	From this and (\ref{tau^*(alpha)_nonnegative_L^q_spectrum_converge}), we have
	\begin{equation*}
		\sum_{-m\leq j<(q-1)^{-1}}\sum_{I\in\A_j}\nu(I)^q\geq \frac{1}{2}\sum_{I\in\D_m}\nu(I)^q\geq 2^{-1}2^{-(\tau(q)+\delta/2)m}.
	\end{equation*}
	Hence, by pigeonholing, there is $-m\leq j<(q-1)^{-1}$ such that
	\begin{equation*}
		\sum_{I\in\A_j}\nu(I)^q\geq (2(m+\lfloor(q-1)^{-1}\rfloor+1))^{-1}2^{-(\tau(q)+\delta/2)m}=2^{-(\tau(q)+\delta/2+\log(2(m+\lfloor(q-1)^{-1}\rfloor+1))/m)m}.
	\end{equation*}
	If $m$ is sufficiently large, then $2^{-(\tau(q)+\delta/2+\log(2(m+\lfloor(q-1)^{-1}\rfloor+1))/m)m}>2^{-(\tau(q)+\delta)m}$. We can see that $\D'=\A_j$ satisfies the condition (1) and (2) of Lemma \ref{close_to_typical_tau^*(alpha)}. So we complete the proof.
\end{proof}

We prove Lemmas \ref{close_to_typical_tau^*(alpha)} and \ref{less_than_typical_tau(q)}. We notice that we do not use special properties of the stationary measure $\nu$ for these lemmas.

\begin{proof}[Proof of Lemma \ref{close_to_typical_tau^*(alpha)}]
	Let $0<\varepsilon<1$ and we take a small constant $\eta=\eta(\varepsilon, q, \tau)>0$ specified later.
	Then, since $\alpha^-=\tau'^{,-}(q)$, if we take $0<\delta\ll_{\varepsilon,\eta, q, \tau}1$ sufficiently small and write $q_1=q-\delta$, we have
	\begin{equation}\label{estimate_left_derivative_realize_heuristic}
		\left|\frac{\tau(q)-\tau(q_1)}{\delta}-\alpha^-\right|<\eta\iff \delta\alpha^--\delta\eta<\tau(q)-\tau(q_1)<\delta\alpha^-+\delta\eta
	\end{equation}
	Here, by the definition of $\tau(q_1)=\lim_{m\to\infty}\left(-1/m\cdot \log\sum_{I\in\D_m}\nu(I)^{q_1}\right)$ and $\tau(q)$, if $m\in\N$, $m\gg_{\eta,q,\tau,\delta}1$ is sufficiently large, then we have
	\begin{equation}\label{tau(q_1)_delta^2_realize_heuristic}
		\sum_{I\in\D_m}\nu(I)^{q_1}<2^{-(\tau(q_1)-\delta^2)m}
	\end{equation}
	and
	\begin{equation}\label{tau(q)_delta_realize_heuristic}
		\sum_{I\in\D_m}\nu(I)^q<2^{-(\tau(q)-\delta)m}.
	\end{equation}
	For $\delta^2$ and this $m$, we assume that we have $\D'\subset\D_m$ satisfying the assumptions (1) and (2) (for $\delta^2$) in the statement.
	
	Since we assume that $2\cdot 2^{-\alpha'm}\geq\nu(I)\geq 2^{-\alpha'm}$ for each $I\in\D'$, we can see from (\ref{tau(q_1)_delta^2_realize_heuristic}) that
	\begin{equation}\label{estimate_D'_q_1_realize_heuristic}
		2^{-\alpha'mq_1}|\D'|\leq\sum_{I\in\D'}\nu(I)^{q_1}\leq \sum_{I\in\D_m}\nu(I)^{q_1}< 2^{-(\tau(q_1)-\delta^2)m}.
	\end{equation}
	Since we also assume that $\sum_{I\in\D'}\nu(I)^q\geq 2^{-(\tau(q)+\delta^2)m}$ (we notice that we have assumed (2) for $\delta^2$), we also have
	\begin{equation}\label{estimate_D'_q_realize_heuristic}
		2^{-(\tau(q)+\delta^2)m}\leq\sum_{I\in\D'}\nu(I)^q\leq 2^q2^{-\alpha'mq}|\D'|.
	\end{equation}
	By (\ref{estimate_D'_q_1_realize_heuristic}), (\ref{estimate_left_derivative_realize_heuristic}) and (\ref{estimate_D'_q_realize_heuristic}), we have
	\begin{equation*}
		2^{-\alpha'mq_1}|\D'|<2^{-(\tau(q_1)-\delta^2)m}<2^{-(\tau(q)-\delta\alpha^--\delta\eta-\delta^2)m}\leq
		2^{(\delta\alpha^-+\delta\eta+2\delta^2)m}2^q2^{-\alpha'mq}|\D'|.
	\end{equation*}
	By canceling $|\D'|$ and recalling $q-q_1=\delta$, we have
	\begin{equation*}
		2^{\delta\alpha'm}<2^{(\delta\alpha^-+\delta\eta+2\delta^2)m}2^q,
	\end{equation*}
	and hence,
	\begin{equation}\label{alpha'_close_alpha^-_realize_heuristic}
		\alpha'-\alpha^-<\eta+2\delta+q/(\delta m)<\eta+3\delta,
	\end{equation}
	where we use that $m\gg_{\eta,q,\tau,\delta}1$, so we can assume that $q/(\delta m)<\delta$.
	
	By (\ref{tau(q)_delta_realize_heuristic}) and our assumption that $\nu(I)\geq 2^{-\alpha'm}$ for each $I\in\D'$ again, we have
	\begin{equation*}
		2^{-\alpha'mq}|\D'|\leq\sum_{I\in\D'}\nu(I)^q\leq\sum_{I\in\D_m}\nu(I)^q<2^{-(\tau(q)-\delta)m}.
	\end{equation*}
	From this and (\ref{alpha'_close_alpha^-_realize_heuristic}), we can see that
	\begin{equation*}
		|\D'|\leq2^{\alpha'mq}2^{-(\tau(q)-\delta)m}<2^{(\alpha^-q-\tau(q)+q\eta+3q\delta+\delta)m}
	\end{equation*}
	Here, we can assume for $\eta=\eta(\varepsilon,q,\tau)$ and $0<\delta\ll_{\varepsilon,\eta,q,\tau}1$ that $q\eta+3q\delta+\delta<\varepsilon$. By this and (\ref{Legendre_transform_at_left_right_derivative}): $\alpha^-q-\tau(q)=\tau^*(\alpha^-)$, we obtain
	\begin{equation*}
		|\D'|<2^{(\tau^*(\alpha^-)+\varepsilon)m}
	\end{equation*}
	and complete the proof.
\end{proof}

To prove Lemma \ref{less_than_typical_tau(q)}, we need the following lemma, which can also be seen as an illustration of the heuristic above.

\begin{lem}[{\cite[Lemma 4.11]{Shm19}}]\label{sum_of_mass_more_than_typical_2^{-alpham}}
	Let $q>1$ and $\alpha^+=\tau'^{,+}(q)$. For $0<\sigma<1$, there exists $\varepsilon=\varepsilon(\sigma, q, \tau)>0$ such that, if $m\in\N$, $m\gg_{\sigma,q,\tau,\varepsilon}1$ is sufficiently large, then we have
	\begin{equation*}
		\sum_{I\in\D_m,\ \nu(I)\geq 2^{-(\alpha^+-\sigma)m}}\nu(I)^q\leq 2^{-(\tau(q)+\varepsilon)m}.
	\end{equation*}
\end{lem}

\begin{proof}\footnote{The proof here is different from that of \cite[Lemma 4.11]{Shm19}. It is actually the combination of the ideas of the proof of Lemma \ref{close_to_typical_tau^*(alpha)} and of pigeonholing used in the proof of Corollary \ref{tau^*(alpha)_is_nonnegative}.}
	Let $0<\sigma<1$ and we take a small constant $\eta=\eta(\sigma)>0$ specified later.
	Then, since $\alpha^+=\tau'^{,+}(q)$, if we take $0<\delta\ll_{\eta,q,\tau}1$ small and write $q_1=q+\delta$, we have
	\begin{equation}\label{tau_right_derivative_realize_heuristic}
		\left|\frac{\tau(q_1)-\tau(q)}{\delta}-\alpha^+\right|<\eta\iff\delta\alpha^+-\delta\eta<\tau(q_1)-\tau(q)<\delta\alpha^++\delta\eta.
	\end{equation}
	
	Let $m\in\N,\ m\gg_{\sigma,\eta,\delta,q,\tau}1$ be sufficiently large. We prove Lemma \ref{sum_of_mass_more_than_typical_2^{-alpham}} for $\varepsilon=\varepsilon(\sigma,q,\tau)=\delta^2$ by contradiction. We assume that
	\begin{equation*}
		\sum_{I\in\D_m,\ \nu(I)\geq 2^{-(\alpha^+-\sigma)m}}\nu(I)^q>2^{-(\tau(q)+\delta^2)m}.
	\end{equation*}
	As we did in the proof of Corollary \ref{tau^*(alpha)_is_nonnegative} (but in a different manner), for $j\in\N$, we define
	\begin{equation*}
		\A'_j=\left\{I\in\D_m\left|\ 2^{-j}<\nu(I)\leq2^{-j+1}\right.\right\}.
	\end{equation*}
	Then, by the assumption, we have
	\begin{equation*}
		\sum_{1\leq j\leq (\alpha^+-\sigma)m+1}\sum_{I\in\A'_j}\nu(I)^q\geq\sum_{I\in\D_m,\ \nu(I)\geq 2^{-(\alpha^+-\sigma)m}}\nu(I)^q>2^{-(\tau(q)+\delta^2)m}.
	\end{equation*}
	Hence, by pigeonholing, there is $1\leq j\leq (\alpha^+-\sigma)m+1$ such that
	\begin{align}\label{contradiction_L^q_strictly_larger_than_2^-alpham_realize_heuristic}
		\sum_{I\in\A'_j}\nu(I)^q\geq ((\alpha^+-\sigma)m+1)^{-1}2^{-(\tau(q)+\delta^2)m}&=2^{-(\tau(q)+\delta^2+\log((\alpha^+-\sigma)m+1)/m)m}\nonumber\\
		&\geq 2^{-(\tau(q)+2\delta^2)m},
	\end{align}
	where we have used $m\gg_{\sigma,\delta,q,\tau}1$ in the last inequality. If we write $\alpha'=j/m$, then we have $\A'_j=\left\{I\in\D_m\left|\ 2^{-\alpha'm}<\nu(I)\leq 2\cdot2^{-\alpha'm}\right.\right\}$ and
	\begin{equation}\label{estimate_alpha'_key_to_contradiction_realize_heuristic}
		0<\alpha'\leq \alpha^+-\sigma+\frac{1}{m}.
	\end{equation}
	
	By the definition of the $L^q$ spectrum $\tau(q_1)$ for $q=q_1$, if $m\gg_{\delta,q,\tau}1$ is sufficiently large, we have
	\begin{equation*}
		\sum_{I\in\D_m}\nu(I)^{q_1}\leq 2^{-(\tau(q_1)-\delta^2)m}.
	\end{equation*}
	Hence, we have
	\begin{equation}\label{estimate_L^q_1_A'_j_realize_heuristic}
		2^{-\alpha'mq_1}|\A'_j|\leq\sum_{I\in\A'_j}\nu(I)^{q_1}\leq\sum_{I\in\D_m}\nu(I)^{q_1}\leq 2^{-(\tau(q_1)-\delta^2)m}.
	\end{equation}
	Furthermore, we can see from (\ref{contradiction_L^q_strictly_larger_than_2^-alpham_realize_heuristic}) that
	\begin{equation}\label{estimate_from_assump_L^q_A'_j_realize_heuristic}
		2^{-(\tau(q)+2\delta^2)m}\leq\sum_{I\in\A'_j}\nu(I)^q\leq 2^q2^{-\alpha'mq}|\A'_j|.
	\end{equation}
	By (\ref{estimate_L^q_1_A'_j_realize_heuristic}), (\ref{tau_right_derivative_realize_heuristic}) and (\ref{estimate_from_assump_L^q_A'_j_realize_heuristic}), we have
	\begin{equation*}
		2^{-\alpha'mq_1}|\A'_j|\leq 2^{-(\tau(q_1)-\delta^2)m}\leq 2^{-(\tau(q)+\delta\alpha^+-\delta\eta-\delta^2)m}\leq 2^{-(\delta\alpha^+-\delta\eta-3\delta^2)m}2^q2^{-\alpha'mq}|\A'_j|.
	\end{equation*}
	By canceling $|\A'_j|$ and recalling that $q_1-q=\delta$, we have
	\begin{equation*}
		2^{-q}2^{(\delta\alpha^+-\delta\eta-3\delta^2)m}\leq 2^{\delta\alpha'm},
	\end{equation*}
	and hence
	\begin{equation}\label{alpha'_not_so_smaller_than_alpha^+_realize_heuristic}
		\alpha'-\alpha^+\geq-\eta-3\delta-\frac{q}{\delta m}.
	\end{equation}
	However, if we take $\eta=\eta(\sigma)$ and $0<\delta\ll_{\eta,q,\tau}1$ so that $\eta+3\delta<\sigma/2$ and $m\gg_{\sigma,\delta,q}1$ sufficiently large, (\ref{alpha'_not_so_smaller_than_alpha^+_realize_heuristic}) contradicts (\ref{estimate_alpha'_key_to_contradiction_realize_heuristic}). Hence, we get the contradiction and
	\begin{equation*}
		\sum_{I\in\D_m,\ \nu(I)\geq 2^{-(\alpha^+-\sigma)m}}\nu(I)^q\leq2^{-(\tau(q)+\delta^2)m},
	\end{equation*}
	and hence complete the proof.
\end{proof}

By using Lemma \ref{sum_of_mass_more_than_typical_2^{-alpham}}, we prove Lemma \ref{less_than_typical_tau(q)}.

\begin{proof}[Proof of Lemma \ref{less_than_typical_tau(q)}]
	Let $0<\kappa<1$. We take a constant $\sigma=\sigma(\kappa,q,\tau)\in(0,1)$ specified later. For this $\sigma$, we take a constant $\varepsilon=\varepsilon(\sigma,q,\tau)>0$ and sufficiently large $m\in\N$, $m\gg_{\sigma,\kappa,q,\tau,\varepsilon}1$ satisfying Lemma \ref{sum_of_mass_more_than_typical_2^{-alpham}}. Let $\D'\subset\D_m$ such that $|\D'|\leq 2^{(\tau^*(\alpha^+)-\kappa)m}$. Using Lemma \ref{sum_of_mass_more_than_typical_2^{-alpham}}, we have
	\begin{align*}
		\sum_{I\in\D'}\nu(I)^q&\leq\sum_{I\in\D',\ \nu(I)<2^{-(\alpha^+-\sigma)m}}\nu(I)^q+\sum_{I\in\D_m,\ \nu(I)\geq 2^{-(\alpha^+-\sigma)m}}\nu(I)^q\\
		&\leq |\D'|2^{-(\alpha^+-\sigma)qm}+2^{-(\tau(q)+\varepsilon)m}\\
		&\leq 2^{-(\alpha^+q-\tau^*(\alpha^+)-\sigma q+\kappa)m}+2^{-(\tau(q)+\varepsilon)m}\\
		&=2^{-(\tau(q)-\sigma q+\kappa)m}+2^{-(\tau(q)+\varepsilon)m}.
	\end{align*}
	In the last equation, we have used (\ref{Legendre_transform_at_left_right_derivative}): $\tau^*(\alpha^+)=\alpha^+q-\tau(q)$. Here, we take $\sigma=\kappa/(2q)$, then we have
	\begin{equation*}
		\sum_{I\in\D'}\nu(I)^q\leq2^{-(\tau(q)+\kappa/2)m}+2^{-(\tau(q)+\varepsilon)m}\leq2^{-(\tau(q)+\min\{\kappa/2,\varepsilon\}-1/m)m}.
	\end{equation*}
	Since $m\gg_{\kappa,\varepsilon}1$, we can assume that $1/m<\min\{\kappa/2,\varepsilon\}/2$. Therefore, we finally obtain
	\begin{equation*}
		\sum_{I\in\D'}\nu(I)^q\leq2^{-(\tau(q)+\min\{\kappa/2,\varepsilon\}/2)m}
	\end{equation*}
	and complete the proof.
\end{proof}

We notice that, by replacing $\alpha^+$ in Lemma \ref{sum_of_mass_more_than_typical_2^{-alpham}} with $\alpha^-$ and doing the similar argument as the proof of Lemma \ref{sum_of_mass_more_than_typical_2^{-alpham}} (and using Proposition \ref{existence_of_limit_of_L^q_dim}), we can also obtain the following similar statement, but we require this only for Proposition \ref{justification_heuristic_singularity} (iii).

\begin{lem}\label{L^q_norm_strictly_smaller_math_than_2^-alpham}
	Let $q>1$ and $\alpha^-=\tau'^{,-}(q)$. For $0<\sigma<1$, there exists $\varepsilon=\varepsilon(\sigma,q,\tau)>0$ such that, if $m\in\N,\ m\gg_{\sigma,q,\tau,\varepsilon}1$ is sufficiently large, then we have
	\begin{equation*}
		\sum_{I\in\D_m,\ \nu(I)\leq2^{-(\alpha^-+\sigma)m}}\nu(I)^q\leq 2^{-(\tau(q)+\varepsilon)m}.
	\end{equation*}
\end{lem}

The following lemma, which may be interpreted as the “local $L^q$ norm lemma”, plays an essential role in the proof of \cite[Theorem 5.1]{Shm19} and also in the proof of Lemma \ref{L^q_norm_porosity} and Theorem \ref{L^q_norm_flattening_theorem}. For $m\in\N$ and a subset $E\subset\RP^1$, we write
\begin{equation*}
	\D_m(E)=\left\{I\in\D_m\left|\ I\cap E\neq\emptyset\right.\right\}.
\end{equation*}

\begin{lem}[Analogy of {\cite[Proposition 4.13]{Shm19}}]\label{local_L^q_norm_lemma}
	Let $q>1$.
	\begin{enumerate}
		\renewcommand{\labelenumi}{(\roman{enumi})}
		\item Let $\alpha^+=\tau'^{,+}(q)$. For $0<\kappa<1$, there exists $\eta=\eta(\kappa,q,\tau,\mu)>0$ such that, if $m\in\N$, $m\gg_{\eta,\kappa,q,\tau,\mu}1$ is sufficiently large, the following holds. Let $s\in\N$,
		$I\in\D_s$ and $\D'\subset\D_{s+m}(I)$ be such that $|\D'|\leq 2^{(\tau^*(\alpha^+)-\kappa)m}$. Then we have
		\begin{equation*}
			\sum_{J\in\D'}\nu(J)^q\leq 2^{-(\tau(q)+\eta)m}\nu(2I)^q,
		\end{equation*}
		where $2I$ is the $2$ times expansion of $I$ with the same center.
		\item For $0<\delta<1$, if $m\in \N$, $m\gg_{\delta,q,\tau,\mu}1$ is sufficiently large, then, for any $s\in\N$
		and any $I\in\D_s$, we have
		\begin{equation*}
			\sum_{J\in\D_{s+m}(I)}\nu(J)^q\leq 2^{-(\tau(q)-\delta)m}\nu(2I)^q.
		\end{equation*}
		
	\end{enumerate}
\end{lem}

The both statements in Lemma \ref{local_L^q_norm_lemma} are heavily depending on the “self-similar-like” structure of the stationary measure $\nu$, and proved by the similar argument as the proof of Proposition \ref{existence_of_limit_of_L^q_dim}.
Furthermore, Lemma \ref{local_L^q_norm_lemma} (i) is based on Lemma \ref{less_than_typical_tau(q)}. The statement (ii) is shown by the definition of $\tau(q)$, instead of Lemma \ref{less_than_typical_tau(q)}.

\begin{proof}
	We first see the consequence of the “self-similar-like” property of the stationary measure $\nu$. We take $s_0=s_0(\mu)\in\N$ depending only on $\mu$ and specified later. Let $q>1$, $s\in\N$, $I\in\D_s$, $m\in\N$, and $\D'$ be any non-empty subset of $\D_{s+m}(I)$. By (\ref{nu_Omega_m}), we have
	\begin{equation*}
		\nu=\mu_{s+s_0}{\bm .}\nu.
	\end{equation*}
	Hence,
	\begin{align}\label{Omega_s+s_0_nu_local_L^q}
		\sum_{J\in\D'}\nu(J)^q&=\sum_{J\in\D'}\left(\mu_{s+s_0}{\bm .}\nu(J)\right)^q\nonumber\\
		&=\sum_{J\in\D'}\left(\int_G\nu(A^{-1}J)\ d\mu_{s+s_0}(A)\right)^q\nonumber\\
		&=\sum_{J\in\D'}\left(\int_{\left\{A\in G\left|I\cap AK\neq\emptyset\right.\right\}}\nu(A^{-1}J)\ d\mu_{s+s_0}(A)\right)^q,
	\end{align}
	where the last equation follows because, for $A\in G$, $\nu(A^{-1}J)\neq0$ only if $A^{-1}I\cap K\supset A^{-1}J\cap K\neq\emptyset$. We notice that $\supp\ \nu=K$.
	
	By Hölder's inequality, we have for each $J\in\D'$ that
	\begin{align*}
		&\left(\int_{\left\{A\in G\left|I\cap AK\neq\emptyset\right.\right\}}\nu(A^{-1}J)\ d\mu_{s+s_0}(A)\right)^q\\
		\leq&\ \mu_{s+s_0}\left(\left\{A\in G\left|\ I\cap AK\neq\emptyset\right.\right\}\right)^{q-1}\int_{\left\{A\in G\left|I\cap AK\neq\emptyset\right.\right\}}\nu(A^{-1}J)^q\ d\mu_{s+s_0}(A).
	\end{align*}
	By this and (\ref{Omega_s+s_0_nu_local_L^q}), it follows that
	\begin{align*}
		&\sum_{J\in\D'}\nu(J)^q\nonumber\\
		\leq&\ \sum_{J\in\D'}\mu_{s+s_0}\left(\left\{A\in G\left|\ I\cap AK\neq\emptyset\right.\right\}\right)^{q-1}\int_{\left\{A\in G\left|I\cap AK\neq\emptyset\right.\right\}}\nu(A^{-1}J)^q\ d\mu_{s+s_0}(A)\nonumber\\
		=&\ \mu_{s+s_0}\left(\left\{A\in G\left|\ I\cap AK\neq\emptyset\right.\right\}\right)^{q-1}\int_{\left\{A\in G\left|I\cap AK\neq\emptyset\right.\right\}}\sum_{J\in\D'}\nu(A^{-1}J)^q\ d\mu_{s+s_0}(A).
	\end{align*}
	
	Here, assume that we can take $M>0$ such that
	\begin{equation}\label{assume_M_local_L^q}
		\sum_{J\in\D'}\nu(A_i^{-1}J)^q\leq M\quad\text{ for every } i\in\Omega_{s+s_0}\text{ such that } I\cap A_iK\neq\emptyset.
	\end{equation}
	Then, by the above inequality, we have
	\begin{equation}\label{M_local_L^q}
		\sum_{J\in\D'}\nu(J)^q\leq M\mu_{s+s_0}\left(\left\{A\in G\left|\ I\cap AK\neq\emptyset\right.\right\}\right)^q.
	\end{equation}
	We recall that we have taken $x_0\in K$ and defined $f_{x_0}: G\ni A\mapsto Ax_0\in \RP^1$. Since $s_0\gg_\mu1$, we can apply (\ref{contraction_2^{-m}}) to $s+s_0$ and obtain that $\diam\ A_iK\leq C_1\pi 2^{-(s+s_0)}<\pi 2^{-s}/4$ for $i\in\Omega_{s+s_0}$\footnote{Since $\supp\ \mu=\A$, that $s_0$ is large in terms of $\mu$ implies that $s_0$ is large in terms of $\A$.}. Hence, if $I\cap A_iK\neq\emptyset$, then $A_ix_0\in(3/2)I$. So we have from (\ref{M_local_L^q}) that
	\begin{equation}\label{M_3/2I_local_L^q}
		\sum_{J\in\D'}\nu(J)^q\leq M\mu_{s+s_0}\left(f_{x_0}^{-1}\left(\frac{3}{2}I\right)\right)^q=M\left(\mu_{s+s_0}{\bm .}\delta_{x_0}\left(\frac{3}{2}I\right)\right)^q.
	\end{equation}
	Here, we also recall that, for the coding map $\pi: \I^\N\ni i\mapsto \lim_{n\to\infty}A_{i_1}\cdots A_{i_n}x_0\in K$ and $\pi_{s+s_0}: \I^\N\ni i\mapsto A_{i_1}\cdots A_{i_{n(s+s_0,i)}}x_0\in K$, $\nu=\pi P$ and $\mu_{s+s_0}{\bm .}\delta_{x_0}=\pi_{s+s_0}P$. Furthermore, since $s_0\gg_\mu1$, we have by (\ref{pi_pi_m}) that
	\begin{equation*}
		d_{\RP^1}(\pi(i),\pi_{s+s_0}(i))\leq C_1\pi2^{-(s+s_0)}<\frac{\pi 2^{-s}}{4},\quad i\in I^\N.
	\end{equation*}
	Therefore, we have $\pi_{s+s_0}^{-1}((3/2)I)\subset\pi^{-1}(2I)$, and hence
	\begin{equation*}
		\mu_{s+s_0}{\bm .}\delta_{x_0}\left(\frac{3}{2}I\right)\leq \nu(2I).
	\end{equation*}
	By this and (\ref{M_3/2I_local_L^q}), we finally obtain that
	\begin{equation}\label{M_nu_2I_local_L^q}
		\sum_{J\in\D'}\nu(J)^q\leq M\nu(2I)^q.
	\end{equation}
	
	Next, we show the statement (i) using (\ref{assume_M_local_L^q}) and (\ref{M_nu_2I_local_L^q}). Let $0<\kappa<1$ and we take $\varepsilon=\varepsilon(\kappa/2, q, \tau)>0$ so that Lemma \ref{less_than_typical_tau(q)} holds for $\kappa/2$ and $\varepsilon$. We write $\alpha^+=\tau'^{,+}(q)$. Assume that $m\gg_{\varepsilon,\kappa,q,\tau,\nu}1$ is sufficiently large and $\D'\subset\D_{s+m}(I)$ such that $|\D'|\leq 2^{(\tau^*(\alpha^+)-\kappa)m}$.
	
	We take each $i\in\Omega_{s+s_0}$ such that $I\cap A_iK\neq\emptyset$. We define
	\begin{equation*}
		\mathcal{C}'_{A_i}=\{J'\in\D_m\left|\ J'\cap A_i^{-1}J\cap K\neq\emptyset\text{ for some } J\in\D'\right.\}.
	\end{equation*}
	For each $J\in\D'\subset\D_{s+m}(I)$ and $x,y\in A_i^{-1}J\cap K$, we have by (\ref{contraction_2^{-m}}) for $s+s_0\gg_\mu 1$ that
	\begin{equation*}
		C_1^{-1}2^{-(s+s_0)}d_{\RP^1}(x,y)\leq d_{\RP^1}(A_ix,A_iy)\leq \pi 2^{-(s+m)}.
	\end{equation*}
	Hence,
	\begin{equation*}
		\diam\ A_i^{-1}J\cap K\leq C_1\pi2^{s_0}2^{-m}.
	\end{equation*}
	From this, we can see that the number of $J'\in\D_m$ which intersect $A_i^{-1}J\cap K$ is $\leq O_{\mu,s_0}(1)$. Therefore, we have
	\begin{equation}\label{less_than_tau(alpha+)_local_L^q}
		|\mathcal{C}'_{A_i}|\leq O_{\mu,s_0}(1)|\D'|\leq O_\mu(1)2^{(\tau^*(\alpha^+)-\kappa)m}\leq 2^{(\tau^*(\alpha^+)-\kappa/2)m},
	\end{equation}
	where, since $m\gg_{\kappa,\mu} 1$, we can assume that $O_\mu(1)\leq 2^{\kappa/2\cdot m}$.
	
	We estimate $\sum_{J\in\D'}\nu(A_i^{-1}J)^q$. For each $J\in\D'$, we have seen that $A_i^{-1}J\cap K$ is covered by at most $O_{\mu,s_0}(1)$ elements of $\mathcal{C}'_{A_i}$. Here, we notice that $A_i^{-1}\D_{s+m}$ is a partition of $\RP^1$ by intervals, and hence, for each $A_i^{-1}J\ (J\in\D_{s+m})$ but finite ones which contain some endpoints of $U$ (at most $O_\mu(1)$\footnote{We notice that we have fixed $U$ for $\A$ at (\ref{domains_the_Möbius_IFS_acts_on_preliminaries}), and hence $\mu$ with $\supp\ \mu=\A$.}), we have $A_i^{-1}J\subset U$ or $A_i^{-1}J\cap U=\emptyset$. For each $J\in\D'$ such that $A_i^{-1}J\subset U$, if we take $x,y\in A_i^{-1}J$ such that $A_ix, A_iy$ are close to the two endpoints of $J$, we have by (\ref{contraction_2^{-m}}) for $s+s_0$ that
	\begin{equation*}
		2^{-(s+m)}\leq d_{\RP^1}(A_ix,A_iy)\leq C_12^{-(s+s_0)} d_{\RP^1}(x,y),
	\end{equation*}
	and hence
	\begin{equation*}
		d_{\RP^1}(x,y)\geq C_1^{-1}2^{-m}.
	\end{equation*}
	Therefore, each $A_i^{-1}J\ (J\in\D')$ such that $A_i^{-1}J\subset U$ contains an open interval of length $\Omega_\mu(2^{-m})$. So it follows that, for each $J'\in\mathcal{C}'_{A_i}$, the number of $A_i^{-1}J\cap K\ (J\in\D')$ which intersect $J'$ is at most $O_\mu(1)$. Hence, by Lemma \ref{L^q_norms_of_two_partitions}, we obtain
	\begin{equation}\label{norm_of_Ai^-1D'_local_L^q}
		\sum_{J\in\D'}\nu(A_i^{-1}J)^q=\sum_{J\in\D'}\nu(A_i^{-1}J\cap K)^q\leq O_{\mu,q}(1)\sum_{J'\in\mathcal{C}'_{A_i}}\nu(J')^q.
	\end{equation}
	
	By (\ref{less_than_tau(alpha+)_local_L^q}), we can apply Lemma \ref{less_than_typical_tau(q)} for $\kappa/2$ to $\mathcal{C}'_{A_i}$. Then, we have
	\begin{equation*}
		\sum_{J'\in\mathcal{C}'_{A_i}}\nu(J')^q\leq 2^{-(\tau(q)+\varepsilon)m}.
	\end{equation*}
	By this and (\ref{norm_of_Ai^-1D'_local_L^q}), we obtain that
	\begin{equation*}
		\sum_{J\in\D'}\nu(A_i^{-1}J)^q\leq O_{\mu,q}(1)2^{-(\tau(q)+\varepsilon)m}\leq 2^{-(\tau(q)+\varepsilon/2)m},
	\end{equation*}
	where, since $m\gg_{\varepsilon,q,\mu}1$, we can assume that $O_{\mu,q}(1)\leq 2^{(\varepsilon/2)m}$. Hence, we showed (\ref{assume_M_local_L^q}) for $M=2^{-(\tau(q)+\varepsilon/2)m}$. So, by (\ref{M_nu_2I_local_L^q}), we finally obtain
	\begin{equation*}
		\sum_{J\in\D'}\nu(J)^q\leq 2^{-(\tau(q)+\varepsilon/2)m}\nu(2I)^q,
	\end{equation*}
	which is the desired consequence for $\eta(\kappa,q,\tau,\mu)=\varepsilon/2$.
	
	Next, we show the statement (ii) using (\ref{assume_M_local_L^q}) and (\ref{M_nu_2I_local_L^q}) for $\D'=\D_{s+m}(I)$. Let $0<\delta<1$. From the same argument as above, we can see that, for each $i\in\Omega_{s+s_0}$ such that $I\cap A_iK\neq\emptyset$,
	\begin{equation*}
		\sum_{J\in\D_{s+m}(I)}\nu(A_i^{-1}J)^q\leq O_{\mu,q}(1)\sum_{J'\in\D_m}\nu(J')^q.
	\end{equation*}
	By the definition of $\tau(q)$, if $m\gg_{\delta,q,\tau}1$ is sufficiently large, then
	\begin{equation*}
		\sum_{J'\in\D_m}\nu(J')^q<2^{-(\tau(q)-\delta/2)m}.
	\end{equation*}
	By this and the above inequality, we have
	\begin{equation*}
		\sum_{J\in\D_{s+m}(I)}\nu(A_i^{-1}J)^q\leq O_{\mu,q}(1)2^{-(\tau(q)-\delta/2)m}\leq 2^{-(\tau(q)-\delta)m},
	\end{equation*}
	where, since $m\gg_{\delta,q,\mu}1$, we can assume that $O_{\mu,q}(1)\leq 2^{(\delta/2)m}$. This is (\ref{assume_M_local_L^q}) for $M=2^{-(\tau(q)-\delta)m}$. Hence, by (\ref{M_nu_2I_local_L^q}), we finally have
	\begin{equation*}
		\sum_{J\in\D_{s+m}(I)}\nu(J)^q\leq2^{-(\tau(q)-\delta)m}\nu(2I)^q,
	\end{equation*}
	and complete the proof.
\end{proof}

In addition to the above lemmas, we also need the following local version of Lemma \ref{close_to_typical_tau^*(alpha)}.
We recall that we have defined $\widehat{\nu_I}=\nu|_I/\nu(2I)=\nu(I)/\nu(2I)\cdot\nu|_I$ for $I\in\D_s$ with $\nu(I)>0$.

\begin{lem}\label{close_to_tau^*(alpha)_local_version}
Let $q>1$ and $\alpha^-=\tau'^{,-}(q)$. For $0<\varepsilon<1$, if $0<\delta\ll_{\varepsilon,q,\tau} 1$ is sufficiently small and $m\in\N$, $m\gg_{\varepsilon,q,\tau,\mu,\delta}1$ is sufficiently large, the following holds.
Let $s\in\N$, $I\in\D_s$ and $\D'\subset\D_{s+m}(I)$ satisfy $\nu(I)>0$ and
\begin{enumerate}
	\renewcommand{\labelenumi}{(\arabic{enumi})}
	\item there is $\alpha'\geq0$ such that $2^{-\alpha'm}\leq\widehat{\nu_I}(J)\leq 2\cdot2^{-\alpha'm}$ for every $J\in\D'$,
	\item $\sum_{J\in\D'}\widehat{\nu_I}(J)^q\geq 2^{-(\tau(q)+\delta)m}$,
\end{enumerate}
then
\begin{equation*}
	|\D'|\leq 2^{(\tau^*(\alpha^-)+\varepsilon)m}.
\end{equation*}
\end{lem}

\begin{proof}
We can prove this lemma by the same way as the proof of Lemma \ref{close_to_typical_tau^*(alpha)}, by just replacing (\ref{tau(q_1)_delta^2_realize_heuristic}) and (\ref{tau(q)_delta_realize_heuristic}) with
\begin{equation*}
\sum_{J\in\D_{s+m}(I)}\widehat{\nu_I}(J)^{q_1}\leq2^{-(\tau(q_1)-\delta^2)m}
\end{equation*}
and
\begin{equation*}
\sum_{J\in\D_{s+m}(I)}\widehat{\nu_I}(J)^q\leq2^{-(\tau(q)-\delta)m}
\end{equation*}
for $q_1=q-\delta$ and $m\gg_{\delta,q,\tau,\mu}1$, which follow from Lemma \ref{local_L^q_norm_lemma} (ii).
\end{proof}

\section{Counterexamples to Problem \ref{main_problem_L^q_dim_of_Furstenberg_measures}}\label{section_counterexamples}

In this section, we discuss counterexamples to Problem \ref{main_problem_L^q_dim_of_Furstenberg_measures}. We first show the first half of Theorem \ref{counterexample_to_natural_extension}, that is, a uniformly hyperbolic and strongly Diophantine family $\A\subset G$ such that two distinct elements of $\A$ share a common fixed point in the attractor is a counterexample to Problem \ref{main_problem_L^q_dim_of_Furstenberg_measures}.
Then, we see in Section \ref{subsection_counterexamples_specific_examples} that there indeed exists such $\A$.
According to the main Theorem \ref{the_main_theorem_L^q_dim_Mobius_IFS}, counterexamples to Problem \ref{main_problem_L^q_dim_of_Furstenberg_measures} must be the case (II) of Theorem \ref{the_main_theorem_L^q_dim_Mobius_IFS}, and hence we show Proposition \ref{justification_heuristic_singularity} about the “singular” properties of such measures in Section \ref{subsection_singularity}.

\subsection{Proof of the first half of Theorem \ref{counterexample_to_natural_extension}}\label{subsection_counterexamples_proof}

Here, we begin the proof of the first half of Theorem \ref{counterexample_to_natural_extension}.
The proof is not difficult.

\begin{proof}[Proof of the first half of Theorem \ref{counterexample_to_natural_extension}]
	We take a non-empty finite family $\A=\{A_i\}_{i\in\I}$ of elements of $G$ which is uniformly hyperbolic and strongly Diophantine and write $K$ for its attractor.
	For $i_0,j_0\in\I, i_0\neq j_0$, we assume that there is $x_0\in K$ such that $A_{i_0}x_0=A_{j_0}x_0=x_0$. We also take $0<p_0<1/2$ sufficiently close to $1/2$ in terms of $\A$ and a non-degenerate probability vector $p=(p_i)_{i\in\I}$ such that $p_{i_0}=p_{j_0}=p_0$.
	Then, let $\nu$ be the stationary measure of $\mu=\sum_{i\in\I}p_i\delta_{A_i}$.
	
	Since $\A$ is uniformly hyperbolic, we assume (\ref{def_of_uniform_hyperbolicity_of_A_preliminaries}) for $\A$ and some $c=c(\A)>0$ and $r=r(\A)>1$. Let $C=C(\A)>0$ be a constant determined only by $\A$ specified later and, for sufficiently large $m\in\N$, $n=n(\A,m)$ be the smallest integer such that
	\begin{equation*}
		r^{2n}\geq C2^m.
	\end{equation*}
	By applying $\nu=\mu{\bm .}\nu$ $n$ times, we have
	\begin{align}\label{munu_n_times_proof_of_counterexample}
		\nu(B_{\pi2^{-m}}(x_0))&=\sum_{i\in\I^n}p_iA_i\nu(B_{\pi2^{-m}}(x_0))\nonumber\\
		&\geq\sum_{i\in\{i_0,j_0\}^n}p_i\nu(A_i^{-1}B_{\pi2^{-m}}(x_0)).
	\end{align}
	Here, by Corollary \ref{contraction_on_U_by_A_preliminaries}, $A_{i_0}x_0=A_{j_0}x_0=x_0$, (\ref{def_of_uniform_hyperbolicity_of_A_preliminaries}) and the definition of $n$, we have for $i\in\{i_0,j_0\}^n$ and $x\in K$ that
	\begin{equation*}
		d_{\RP^1}(A_ix,x_0)=d_{\RP^1}(A_ix,A_ix_0)\leq \frac{C_1\pi}{\|A_i\|^2}
		\leq\frac{C_1\pi}{c^2r^{2n}}\leq\frac{C_1\pi}{c^2C}\cdot2^{-m}<\pi 2^{-m},
	\end{equation*}
	where we take the constant $C=C(\A)>0$ so that $C>C_1/c^2$. Hence, we have
	\begin{equation*}
		K\subset A_i^{-1}B_{\pi2^{-m}}(x_0).
	\end{equation*}
	From this, $p_{i_0}=p_{j_0}=p_0$ and (\ref{munu_n_times_proof_of_counterexample}), we obtain that
	\begin{equation}\label{bound_overlapped_2^-m_ball_proof_of_cpunterexample}
		\nu(B_{\pi2^{-m}}(x_0))\geq(2p_0)^n.
	\end{equation}
	
	Let $q>1$. Then, for sufficiently large $m\in\N$, we have by (\ref{bound_overlapped_2^-m_ball_proof_of_cpunterexample}) that
	\begin{equation*}
		\|\nu^{(m)}\|_q^q=\sum_{I\in\D_m}\nu(I)^q\geq 3^{-(q-1)}\nu(B_{\pi2^{-m}}(x_0))^q\geq 3^{-(q-1)}(2p_0)^{qn}.
	\end{equation*}
	Hence, we have
	\begin{equation}\label{upper_bound_L^q_spectrum_proof_of_counteexample}
		\tau(q)=\lim_{m\to\infty}\left(-\frac{1}{m}\log\|\nu^{(m)}\|_q^q\right)\leq\lim_{m\to\infty}\left(-\frac{1}{m}\log\left(3^{-(q-1)}(2p_0)^{qn}\right)\right)
		=-\frac{\log(2p_0)}{2\log r}\cdot q.
	\end{equation}
	Here, we have used $r^{2n}=\Theta_\A(2^m)$. Since $0<p_0<1/2$ is sufficiently close to $1/2$, we have
	\begin{equation}\label{smaller_than_q-1_proof_of_counterexample}
		-\frac{\log(2p_0)}{2\log r}\cdot q<q-1
	\end{equation}
	if $q$ is sufficiently large. We take a constant $L=L(\A)>0$ so that $\|A_i\|\leq 2^L$ for any $i\in\I$. Then, we have
	\begin{align*}
		\Psi_q\left(-\frac{\log(2p_0)}{2\log r}\cdot q\right)&=\lim_{n\to\infty}\frac{1}{n}\log\sum_{i\in\I^n}p_i^q\|A_i\|^{2(-\log(2p_0)/(2\log r)\cdot q)}\\
		&\leq \lim_{n\to\infty}\frac{1}{n}\log\sum_{i\in\I^n}p_i^q2^{-nL\log(2p_0)/\log r\cdot q}\\
		&=\log\sum_{i\in\I}p_i^q-\frac{L}{\log r}\cdot q\log (2p_0)\\
		&\leq\log|\I|+q\log\left(\frac{\max_{i\in\I}p_i}{(2p_0)^{L/\log r}}\right).
	\end{align*}
	Since $p_0=p_{i_0}=p_{j_0}$ is sufficiently close to $1/2$, we have $\max_{i\in\I}p_i=p_0<1/2<(2p_0)^{L/\log r}$, and hence the right-hand side of the above inequality is $<0$ if $q$ is sufficiently large. From this and the fact that $\Phi_q(s)\ (s>0)$ is non-decreasing, we obtain that
	\begin{equation}\label{smaller_than_widetildetau_q_proof_of_counterexample}
		-\frac{\log(2p_0)}{2\log r}\cdot q<\widetilde{\tau}(q)
	\end{equation}
	for sufficiently large $q$. By (\ref{upper_bound_L^q_spectrum_proof_of_counteexample}), (\ref{smaller_than_q-1_proof_of_counterexample}) and (\ref{smaller_than_widetildetau_q_proof_of_counterexample}), we have
	\begin{equation*}
		\tau(q)<\min\left\{\widetilde{\tau}(q), q-1\right\}
	\end{equation*}
	for sufficiently large $q$, and, by dividing the both sides by $q-1$, complete the proof.
\end{proof}

\subsection{Specific counterexamples}\label{subsection_counterexamples_specific_examples}

In this section, we see the second half of Theorem \ref{counterexample_to_natural_extension}, that is, we give specific examples satisfying the conditions of Theorem \ref{counterexample_to_natural_extension}. We notice that these examples were considered first in \cite{Sol24}, and Theorem \ref{freeness_of_Solomyak_example} below gives an affirmative answer to Solomyak's question \cite[Question 2]{Sol24}.

Let
\begin{equation*}
	A=
	\begin{pmatrix}
		\frac{1}{2}&0\\2&2
	\end{pmatrix},\quad
	B=\begin{pmatrix}
		\frac{1}{2}&0\\0&2
	\end{pmatrix}\in\SL_2(\Q)
\end{equation*}
and, for a parameter $t>0$,
\begin{equation*}
	C_t=
	\begin{pmatrix}
		\frac{1}{2}&t\\0&2
	\end{pmatrix}\in G=\SL_2(\R).
\end{equation*}
We notice that, via the identification $F^{-1}:\RP^1\stackrel{\sim}{\longrightarrow}\R\sqcup\{\infty\}$ defined in Section \ref{section_main_problem} by
\begin{equation*}
	F^{-1}([x:y])=\frac{x}{y},\quad [x:y]\in\RP^1\setminus\{[1:0]\},\quad F^{-1}([1:0])=\infty,
\end{equation*}
$A,B,C_t$ act on $\R\sqcup\{\infty\}$ as Möbius transformations
\begin{equation*}
	A(x)=\frac{x}{4x+4},\quad B(x)=\frac{x}{4},\quad C_t(x)=\frac{x}{4}+\frac{t}{2}
\end{equation*}
respectively.

The point $0\in\R$ is a common fixed point for $A$ and $B$, and $2t/3\in\R$ is a fixed point for $C_t$. If we define $I_t=[0,2t/3]\subset\R$, we can see that
\begin{equation}\label{invariant_interval_for_ABCt}
	A(I_t)=\left[0,\frac{t}{4t+6}\right],\quad B(I_t)=\left[0,\frac{t}{6}\right],\quad C_t(I_t)=\left[\frac{t}{2},\frac{2t}{3}\right]\subsetneq I_t
\end{equation}
and the actions of $A$, $B$ and $C_t$ are contractive on $I_t$. Hence, $\A_t=\{A,B,C_t\}$ is an Möbius IFS on $I_t$ as a family of Möbius transformations, and a uniformly hyperbolic family of elements of $G$ (Proposition \ref{proposition_Möbius_uniformly_hyperbolic_SL_2(R)}).
Let $K_t\subset I_t$ be the attractor of $\A_t$. Since $0\in I_t$ is a fixed point for $A$ and $B$ but not for $C_t$, we have that $0\in K_t$ and $K_t$ is not a singleton.

From these things, we can say that $\A_t$ is an example for Theorem \ref{counterexample_to_natural_extension} if $\A_t$ is strongly Diophantine. It is not hard to see that, if we can take an algebraic $t>0$ so that the semigroup $\{A,B,C_t\}^+$ generated by $\A_t=\{A,B,C_t\}$ is free, then $\A_t$ is strongly Diophantine (see \cite[Lemma 6.1]{HS17}). Whether this is possible is Solomyak's question \cite[Question 2]{Sol24}.
Here, we give an affirmative answer to this question.

\begin{thm}\label{freeness_of_Solomyak_example}
	For $t=9n\ (n\in\N)$, the semigroup $\{A,B,C_t\}^+$ is free.
\end{thm}

\begin{proof}
	Actually, the proof is a simple modification of the arguments in \cite{Sol24}.
	
	Let $t>0$ and assume that the semigroup $\{A,B,C_t\}^+$ is not free. Then, there are two distinct words $X,Y$ consisting of $A, B, C_t$ such that $X=Y$ as matrix products.
	We write $|X|,|Y|$ for the lengths of the words, and we can assume that $|X|\leq |Y|$. Let $Y|_{|X|}$ be the word obtained by restricting $Y$ to its first (from the left) $|X|$ letters. If $|X|<|Y|$ and $X=Y|_{|X|}$ as words, then, by multiplying $X=Y$ by $X^{-1}=Y|_{|X|}^{-1}$ from the left, we obtain a non-empty word $Z$ consisting of $A,B,C_t$ such that $Z=1_G$ as a matrix product.
	However, by (\ref{invariant_interval_for_ABCt}), we have $I_t=1_G(I_t)=Z(I_t)\subsetneq I_t$ as a Möbius transformation and this is a contradiction. Hence, we have $|X|=|Y|$ or $X\neq Y|_{|X|}$ as words. From this, we can see that two words $XY$ and $YX$ are distinct, $|XY|=|YX|=|X|+|Y|$ and $XY=YX$ as matrix products. In the following, by replacing $X$ and $Y$ with $XY$ and $YX$, we assume that $|X|=|Y|$ and write $m=|X|=|Y|$.
	
	By canceling the same letters from $X$ and $Y$ from the left, we can assume that the left-end letters of $X$ and $Y$ are different.
	If either of the left-end letter of $X$ or $Y$ is $C_t$ (if $X$ for example), we have by (\ref{invariant_interval_for_ABCt}) that $X(I_t)\subset[t/2,2t/3]$ and $Y(I_t)\subset[0,t/6]$ as Möbius transformations, and these are disjoint. But this contradicts to $X=Y$ as matrix products. Hence, both of the left-end letters of $X$ and $Y$ are $A$ or $B$, and they are distinct.
	
	Let
	\begin{equation*}
		R=\begin{pmatrix}
			4&0\\\frac{4}{3}&1
		\end{pmatrix}.
	\end{equation*}
	Since $X=Y$ as matrix products and $|X|=|Y|=m$, by taking the inverse of $X=Y$ and the conjugate by $R$ and multiplying $2^m$, we obtain two distinct word $V$ and $W$ consisting of $2RA^{-1}R^{-1}, 2RB^{-1}R^{-1},2RC_t^{-1}R^{-1}$ such that $V=W$ as matrix products and both of the right-end letters of the word $V$ and $W$ are $2RA^{-1}R^{-1}$ and $2RB^{-1}R^{-1}$, and they are distinct. Here, we notice that
	\begin{equation*}
		2RA^{-1}R^{-1}=
		\begin{pmatrix}
			4&0\\0&1
		\end{pmatrix},\quad
		2RB^{-1}R^{-1}=
		\begin{pmatrix}
			4&0\\1&1
		\end{pmatrix},\quad
		2RC_t^{-1}R^{-1}=
		\begin{pmatrix}
			\frac{8}{3}t+4&-8t\\\frac{8}{9}t+1&-\frac{8}{3}t+1
		\end{pmatrix}.
	\end{equation*}
	If we put $t=9n$ for $n\in\N$, we have
	\begin{equation*}
		2RC_t^{-1}R^{-1}=
		\begin{pmatrix}
			24n+4&-72n\\8n+1&-24n+1
		\end{pmatrix}
	\end{equation*}
	and $2RA^{-1}R^{-1},2RB^{-1}R^{-1}, 2RC_t^{-1}R^{-1}\in M_2(\Z)$.
	
	For $a\in\Z$, we write $\overline{a}$ for the image of $a$ in $\Z/4\Z$, and define a map $\phi:M_2(\Z)\to M_2(\Z/4\Z)$ by
	\begin{equation*}
		\phi\left(
		\begin{pmatrix}
			a&b\\c&d
		\end{pmatrix}
		\right)
		=\begin{pmatrix}
			\overline{a}&\overline{b}\\\overline{c}&\overline{d}
		\end{pmatrix}.
	\end{equation*}
	We notice that $\phi$ is a homomorphism between (non-commutative) rings and
	\begin{equation}\label{key_point}
		\phi(2RA^{-1}R^{-1})=
		\begin{pmatrix}
			\overline{0}&\overline{0}\\\overline{0}&\overline{1}
		\end{pmatrix},\quad
		\phi(2RB^{-1}R^{-1})=\phi(2RC_t^{-1}R^{-1})=
		\begin{pmatrix}
			\overline{0}&\overline{0}\\\overline{1}&\overline{1}
		\end{pmatrix}.
	\end{equation}
	Since $V=W$ as matrix products in $M_2(\Z)$, we have $\phi(V)=\phi(W)$ in $M_2(\Z/4\Z)$. We can assume that the right-end letter of $V$ and $W$ are $2RA^{-1}R^{-1}$ and $2RB^{-1}R^{-1}$, respectively. Then, we have from (\ref{key_point}) that
	\begin{equation*}
		\phi(V)=
		\begin{pmatrix}
			\overline{a}&\overline{0}\\\overline{b}&\overline{1}
		\end{pmatrix}
		\begin{pmatrix}
			\overline{0}&\overline{0}\\\overline{0}&\overline{1}
		\end{pmatrix}
		=\begin{pmatrix}
			\overline{0}&\overline{0}\\\overline{0}&\overline{1}
		\end{pmatrix}
		\quad (\overline{a},\overline{b}\in\Z/4\Z)
	\end{equation*}
	and
	\begin{equation*}
		\phi(W)=
		\begin{pmatrix}
			\overline{a}&\overline{0}\\\overline{b}&\overline{1}
		\end{pmatrix}
		\begin{pmatrix}
			\overline{0}&\overline{0}\\\overline{1}&\overline{1}
		\end{pmatrix}=
		\begin{pmatrix}
			\overline{0}&\overline{0}\\\overline{1}&\overline{1}
		\end{pmatrix}
		\quad (\overline{a},\overline{b}\in\Z/4\Z),
	\end{equation*}
	which is a contradiction. Hence, we obtain that $\{A,B,C_t\}^+$ is free for $t=9n$ and complete the proof.
\end{proof}

\subsection{“Singular” properties of counterexamples}\label{subsection_singularity}

As we mentioned above, counterexamples to Problem \ref{main_problem_L^q_dim_of_Furstenberg_measures} are the case (II) of Theorem \ref{the_main_theorem_L^q_dim_Mobius_IFS}.
In this section, we show Proposition \ref{justification_heuristic_singularity} which describes the “singular” properties of measures of the case (II).

We take a non-empty finite family $\A=\{A_i\}_{i\in\I}$ of elements of $G$ which is uniformly hyperbolic and whose attractor $K$ is not a singleton
(actually, we will not use the strongly Diophantine property here) and a probability measure $\mu$ on $G$ such that $\supp\ \mu=\A$.
We write $\nu$ for the stationary measure of $\mu$ and $\tau(q)\ (q>1)$ for its $L^q$ spectrum, and assume that the case (II) holds for $\nu$, that is, there exist $q_0>1$ and $0<\alpha<1$ such that
\begin{equation}\label{assumption_Legendre_transform_is_zero_counterexample}
\tau(q)=
\begin{cases}
	\min\left\{\widetilde{\tau}(q),q-1\right\}&\text{if }1<q<q_0,\\
	\alpha q&\text{if }q\geq q_0.
\end{cases}
\end{equation}

\begin{proof}[Proof of Proposition \ref{justification_heuristic_singularity}]
	We see (i) and (ii). Let $\alpha'>0$ be such that
	\begin{equation*}
		\alpha'\geq\liminf_{r\searrow0}\frac{\log\nu(B_r(x))}{\log r}=\liminf_{m\to\infty}\left(-\frac{\log\nu(B_{\pi2^{-m}}(x))}{m}\right)
	\end{equation*}
	for some $x\in K$ (the right-hand equation can be easily checked). If we take arbitrarily small $\varepsilon>0$, then we have $2^{-(\alpha'+\varepsilon)m}\leq \nu(B_{\pi 2^{-m}}(x))$ for infinitely many $m\in\N$. Hence, we have
	\begin{equation*}
		2^{-(\alpha'+\varepsilon)q_0m}\leq \nu(B_{\pi2^{-m}}(x))^{q_0}\leq 3^{q_0-1}\sum_{I\in\D_m}\nu(I)^{q_0}
	\end{equation*}
	for $q_0$ and infinitely many $m\in\N$. On the other hand, by the assumption (\ref{assumption_Legendre_transform_is_zero_counterexample}), we have $\tau(q_0)=\alpha q_0$, and hence
	\begin{equation*}
		\sum_{I\in\D_m}\nu(I)^{q_0}\leq 2^{-(\alpha q_0-\varepsilon)m}
	\end{equation*}
	for sufficiently large $m\in\N$. By combining these two inequalities and taking $m$ arbitrarily large and $\varepsilon\to 0$, we have $\alpha\leq \alpha'$. Hence, we obtain
	\begin{equation}\label{alpha0_smaller_than_inf_pointwise_dim_counterexample}
		\alpha\leq\inf\left\{\alpha'>0\left|\ \text{there exists some $x\in K$ such that }\liminf_{r\searrow 0}\right.\frac{\log\nu(B_r(x))}{\log r}\leq\alpha'\right\}.
	\end{equation}
	
By the assumption (\ref{assumption_Legendre_transform_is_zero_counterexample}), $\tau(q)=\alpha q$ for any $q>q_0$. Then, by \cite[Theorem 1.1]{Fen07} (more precisely we use Theorem 1.2 and Proposition 2.1 in \cite{Fen07} to ensure $E_\alpha\neq\emptyset$), $E_\alpha$ is non-empty and
	\begin{equation*}
		\dim_H E_\alpha=\tau^*(\alpha)=\alpha q-\tau(q)=0.
	\end{equation*}
	Since $E_\alpha\neq\emptyset$, we can replace the inequality in (\ref{alpha0_smaller_than_inf_pointwise_dim_counterexample}) with the equation and $\inf$ with $\min$, and obtain (i).
	
	To show that $E_\alpha$ is dense in $K$, it is sufficient to show that $E_\alpha$ is $\A$-invariant (e.g., see (\ref{pi_pi_m})). Let $x\in E_\alpha$ and $i\in\I$. Then, by $\nu=\mu{\bm .}\nu$, we have
	\begin{equation*}
		\nu(B_r(A_ix))\geq p_i\nu(A_i^{-1}B_r(A_ix))\geq p_i\nu(B_{c'r}(x))
	\end{equation*}
	for any $r>0$, where $c'=c'(\A)>0$ is a constant. Hence, we have
	\begin{equation*}
		\limsup_{r\searrow0}\frac{\log\nu(B_r(A_ix))}{\log r}\leq \limsup_{r\searrow0}\frac{\log\nu(B_{c'r}(x))}{\log r}=\alpha.
	\end{equation*}
	On the other hand, by the same argument as in the first paragraph, we have
	\begin{equation*}
		\liminf_{r\searrow0}\frac{\log\nu(B_r(A_ix))}{\log r}=\liminf_{m\to\infty}\left(-\frac{1}{m}\log\nu(B_{\pi2^{-m}}(A_ix))\right)\geq\alpha.
	\end{equation*}
	From these inequalities, we have $A_ix\in E_\alpha$, and hence $E_\alpha\neq\emptyset$ is $\A$-invariant. So we complete the proof of (ii).
	
	In the end, we see (iii). We take $q\geq q_0$ and small $\varepsilon>0$. Then, for sufficiently large $m\in\N$, we have $\sum_{I\in\D_m}\nu(I)^q\leq 2^{-(\tau(q)-\varepsilon)m}=2^{-(\alpha q-\varepsilon)m}$. The first statements of (iii) immediately follow from this. The second statement follows from Lemma \ref{L^q_norm_strictly_smaller_math_than_2^-alpham} and Proposition \ref{existence_of_limit_of_L^q_dim}.
\end{proof}

\section{Proof of the $L^q$ norm porosity lemma}\label{section_proof_of_L^q_norm_porosity}

Throughout the following sections, we prove the $L^q$ norm porosity Lemma \ref{L^q_norm_porosity}.
We take a non-empty finite family $\A$ of elements of $G$ which is uniformly hyperbolic. We also take a probability measure $\mu$ on $G$ such that $\supp\ \mu=\A$ and write $\nu$ for the stationary measure of $\mu$. Furthermore, we take open subsets $U\subset U_1\subset U_0\subsetneq\RP^1$ as (\ref{domains_the_Möbius_IFS_acts_on_preliminaries}). We write $\tau(q)$ for the $L^q$ spectrum of $\nu$. Since all of these things are determined by $\mu$, we write $\mu$ for dependence on some of these things all together.

We take $q>1$. In the beginning, we do not use the assumption that $\tau(q)$ is differentiable at $q$ and $\tau^*(\alpha)>0$ for $\alpha=\tau'(q)$.

\subsection{Uniformizing the measure on $G$}\label{subsection_uniformizing_theta}

As the first step of the proof, we uniformize the measure $\theta$ on $G$ in Lemma \ref{L^q_norm_porosity}.
For $m\in\N$ and a subset $E\subset G$, we define
\begin{equation*}
	\D^G_m(E)=\left\{\xi\in\D^G_m\left|\ \xi\cap E\neq\emptyset\right.\right\}.
\end{equation*}
Let $D, l\in\N$ and $(R_i)_{i\in[l]}\ ([l]=\{0,1,\dots,l-1\})$ be a sequence such that $R_0\in\N$, $R_i\in\{1,M,\dots, M^D\}\ (i>0)$. We say that a $\D_{lD}^G$-measurable set $E$ is {\it $(D,l,(R_i)_{i\in[l]})$-uniform} if, for any $i\in[l]$ and $\xi\in\D_{iD}^G$ such that $\xi\cap E\neq\emptyset$, we have
\begin{equation*}
	\left|\D^G_{(i+1)D}(\xi\cap E)\right|=R_i.
\end{equation*}
(Here, we consider $\D_0^G$ as $\{G\}$.)
In this section, we prove the following lemma.

\begin{lem}\label{tree_structure_supp_theta}
	Let $\sigma,\varepsilon>0$ be constants such that $0<\varepsilon<\sigma/100<\sigma<1$ and $C, L>1$. Then, there exists a sufficiently large constant $D=D(M,\mu,q,\sigma,\varepsilon, C)\in\N$ such that the following holds for $l\in\N$, $l\gg_{M,\mu,q,\sigma,\varepsilon,C,L}1$.
	We write $m=lD$.
	Let $\theta$ be a Borel probability measure on $G$ and $r\in\N$ such that $\diam\ \supp\ \theta \leq L$ and $C^{-1}2^r\leq \|g\|^2\leq C2^r$ and $u^-_g\notin U_1$ for every $g\in\supp\ \theta$. Assume that
	\begin{enumerate}
		\renewcommand{\labelenumi}{(\roman{enumi})}
		\item $\|\theta^{(m)}\|_q^q\leq 2^{-\sigma m}$,
		\item $2^{-\varepsilon m}\|\nu^{(m)}\|_q^q\leq \|(\theta{\bm .}\nu)^{(m+r)}\|_q^q$.
	\end{enumerate}
	Then, there exists a $\D_m^G$-measurable set $\widetilde{E}$ such that, if $\widetilde{\theta}$ be the normalized restriction of $\theta$ to $\widetilde{E}$, we have
	\begin{enumerate}
		\renewcommand{\labelenumi}{(\Roman{enumi})}
		\item $\|\widetilde{\theta}^{(m)}\|_q^q\leq 2^{-4\sigma/5\cdot m}$,
		\item $2^{-13\varepsilon/10\cdot m}\|\nu^{(m)}\|_q^q\leq\|(\widetilde{\theta}{\bm .}\nu)^{(m+r)}\|_q^q$,
		\item there exists $a>0$ such that, for every $\xi\in\D_m^G$ with $\widetilde{\theta}(\xi)>0$,
		\begin{equation*}
			\frac{a}{2}\leq\widetilde{\theta}(\xi)\leq a,
		\end{equation*}
		\item there exists a sequence $(R_i)_{i\in[l]}$ with $R_0=O_{M,L,D}(1)$, $R_i\in\{1,M,\dots, M^D\}\ (i>0)$ such that $\widetilde{E}$ is $(D,l,(R_i)_{i\in[l]})$-uniform.
	\end{enumerate}
\end{lem}

\begin{rem} In the proof of Lemma \ref{tree_structure_supp_theta}, $D=D(M,\mu,q,\sigma,\varepsilon,C)$ can be taken as the sufficiently large constant in terms only of $M,q,\varepsilon$. However, throughout Section \ref{section_proof_of_L^q_norm_porosity}, $D$ also must be sufficiently large in terms of $M,\mu,q,\sigma,\delta,\varepsilon,C$. See
	(\ref{diam_xi_proof_of_linearization}), (\ref{L^q_norm_widehatnuI_proof_of_linearization}), (\ref{upper_bound_logYi_proof_of_linearization}), (\ref{diam_of_product_of_two_balls_proof_of_porpsity}), (\ref{widetildeDxi'_L^q_norm_proof_of_porosity}),
	(\ref{each_interval_in_D'_sD+2_proof_of_porosity}),
	(\ref{D'_(s+1)D(I')_strictly_small_L^q_norm_proof_of_porosity}),
	(\ref{estimate_L_s+1_without_condition_proof_of_porosity}),
	(\ref{contradiction_inequality_delta^4/3_proof_of_porosity}).
\end{rem}

The idea of the proof of Lemma \ref{tree_structure_supp_theta} is (1) reducing $\theta$ to make the measure “uniform”, and (2) reducing the measure so that its support has a “uniform tree structure”. We begin with the step (1).

\begin{lem}\label{uniforming_theta}
	Let $\sigma,\varepsilon>0$ be constants such that $0<\varepsilon<\sigma/100<\sigma<1$ and $C, L>1$. Then the following holds for $m\in\N$, $m\gg_{M,\mu,q,\sigma,\varepsilon,C,L}1$. Let $\theta$ be a Borel probability measure on $G$ and $r\in\N$ such that $\diam\ \supp\ \theta \leq L$ and $C^{-1}2^r\leq \|g\|^2\leq C2^r$ and $u^-_g\notin U_1$ for every $g\in\supp\ \theta$. Assume that
	\begin{enumerate}
		\renewcommand{\labelenumi}{(\roman{enumi})}
		\item $\|\theta^{(m)}\|_q^q\leq 2^{-\sigma m}$,
		\item $2^{-\varepsilon m}\|\nu^{(m)}\|_q^q\leq \|(\theta{\bm .}\nu)^{(m+r)}\|_q^q$.
	\end{enumerate}
	Then, there exists a Borel probability measure $\theta'$ on $G$ which is a normalized restriction of $\theta$ to a $\D^G_m$-measurable set such that
	\begin{enumerate}
		\renewcommand{\labelenumi}{(\arabic{enumi})}
		\item $\|{\theta'}^{(m)}\|_q^q\leq 2^{-9\sigma/10\cdot m}$,
		\item $2^{-6\varepsilon/5\cdot m}\|\nu^{(m)}\|_q^q\leq \|(\theta'{\bm .}\nu)^{(m+r)}\|_q^q$,
		\item there exists $a'>0$ such that, for every $\xi\in\D_m^G$ with $\theta'(\xi)>0$,
		\begin{equation*}
			\frac{a'}{2}\leq\theta'(\xi)\leq a'.
		\end{equation*}
	\end{enumerate}
\end{lem}

\begin{proof}
	The idea of the proof is from \cite[Lemma 3.3]{Shm19}. We first notice that, since $\diam\ \supp\ \theta\leq L$ and each element of $\D_1^G$ contains a ball of radius $1/2M$ and has the diameter $\leq 2M$, there is $M'=M'(L,M)>1$ such that $\supp\ \theta$ intersects at most $M'$ elements of $\D_1^G$.
	For each integer $j$, we define
	\begin{equation*}
		\A_j=\left\{\xi\in\D_m^G\left|\ 2^{-j-1}M'^{-1}M^{-(m-1)}<\theta(\xi)\leq2^{-j}M'^{-1}M^{-(m-1)}\right.\right\},\quad A_j=\bigsqcup_{\xi\in\A_j}\xi.
	\end{equation*}
	We have that each $A_j$ is $\D^G_m$-measurable, $\bigsqcup_{j\in\Z}A_j=\bigsqcup_{\xi\in\D_m^G, \theta(\xi)>0}\xi$ and $A_j=\emptyset$ if $j<-(m-1)\log M-\log M'-1$. We write $\delta=\delta(q,\varepsilon)=3\varepsilon/q$ and
	\begin{equation*}
		E=\bigsqcup_{j=-\infty}^{\lceil\delta m\rceil-1}A_j,\quad F=\bigsqcup_{j=\lceil\delta m\rceil}^\infty A_j.
	\end{equation*}
	Then $\theta=\theta|_E+\theta|_F$, and hence, by Minkowski's inequality,
	\begin{equation}\label{minkowski_1_uniforming_theta}
		\|(\theta{\bm .}\nu)^{(m+r)}\|_q=\|(\theta|_E{\bm .}\nu+\theta|_F{\bm .}\nu)^{(m+r)}\|_q\leq\|(\theta|_E{\bm .}\nu)^{(m+r)}\|_q+\|(\theta|_F{\bm .}\nu)^{(m+r)}\|_q.
	\end{equation}
	
	We estimate $\|(\theta|_F{\bm .}\nu)^{(m+r)}\|_q$. Since $\supp\ \theta$ intersects at most $M'M^{m-1}$ elements of $\D_m^G$, we have
	\begin{equation}\label{mass_F_uniforming_theta}
		\theta|_F(G)\leq M'M^{m-1}\cdot 2^{-\lceil\delta m\rceil}M'^{-1}M^{-(m-1)}\leq2^{-\delta m}.
	\end{equation}
	By Hölder's inequality,
	\begin{align*}
		\|(\theta|_F{\bm .}\nu)^{(m+r)}\|_q^q&=\sum_{I\in\D_{m+r}}\theta|_F{\bm .}\nu(I)^q\\
		&=\sum_{I\in\D_{m+r}}\left(\int_G\nu(g^{-1}I)\ d\theta|_F(g)\right)^q\\
		&\leq\sum_{I\in\D_{m+r}}\theta|_F(G)^{q-1}\int_G\nu(g^{-1}I)^q\ d\theta|_F(g)\\
		&=\theta|_F(G)^{q-1}\int_G\sum_{I\in\D_{m+r}}\nu(g^{-1}I)^q\ d\theta|_F(g).
	\end{align*}
	For each $g\in\supp\ \theta$, we have $C^{-1}2^r\leq \|g\|^2\leq C2^r$ and $u^-_g\notin U_1$. Hence, by Corollary \ref{contraction_on_U_by_A_preliminaries}, we have $\diam\ g^{-1}I\cap K\leq C_1C\pi2^{-m}$ for each $I\in\D_{m+r}$, where $C_1>1$ is a global constant depending only on $\mu$ and $\overline{U}\subset U_1$, so $g^{-1}I\cap K$ intersects at most $O_{\mu,C}(1)$ elements of $\D_m$.
	Furthermore, by Corollary \ref{contraction_on_U_by_A_preliminaries} again, each $g^{-1}I\ (I\in\D_{m+r})$ included in $U$ contains an open interval of length $C_1^{-1}C^{-1}2^{-m}$, so each $J\in\D_m$ intersects at most $O_{\mu,C}(1)$ elements of $g^{-1}I\cap K\ (I\in\D_{m+r})$. Hence, by Lemma \ref{L^q_norms_of_two_partitions}, we have
	\begin{equation*}
		\sum_{I\in\D_{m+r}}\nu(g^{-1}I)^q\leq O_{\mu,q,C}(1)\|\nu^{(m)}\|_q^q.
	\end{equation*}
	Therefore, we obtain that
	\begin{align}\label{L^q_norm_thetaFnu_uniforming_theta}
		\|(\theta|_F{\bm .}\nu)^{(m+r)}\|_q^q&\leq\theta|_F(G)^{q-1}\int_G\sum_{I\in\D_{m+r}}\nu(g^{-1}I)^q\ d\theta|_F(g)\nonumber\\
		&\leq O_{\mu,q,C}(1)\theta|_F(G)^q\|\nu^{(m)}\|_q^q\nonumber\\
		&\leq O_{\mu,q,C}(1)2^{-q\delta m}\|\nu^{(m)}\|_q^q\nonumber\\
		&\leq 2^{-q\delta/2\cdot m}\|\nu^{(m)}\|_q^q,
	\end{align}
	where, in the last two inequalities, we have used (\ref{mass_F_uniforming_theta}) and $m\gg_{\mu,q,\varepsilon,C}1$.
	
	From our assumption (ii), (\ref{minkowski_1_uniforming_theta}) and (\ref{L^q_norm_thetaFnu_uniforming_theta}), we have
	\begin{align}\label{L^q_norm_thetaEnu_uniforming_theta}
		\|(\theta|_E{\bm .}\nu)^{(m+r)}\|_q&\geq \|(\theta{\bm .}\nu)^{(m+r)}\|_q-\|(\theta|_F{\bm .}\nu)^{(m+r)}\|_q\nonumber\\
		&\geq 2^{-\varepsilon/q\cdot m}\|\nu^{(m)}\|_q-2^{-\delta/2\cdot m}\|\nu^{(m)}\|_q\nonumber\\
		&=(1-2^{-\varepsilon/(2q)\cdot m})2^{-\varepsilon/q\cdot m}\|\nu^{(m)}\|_q\nonumber\\
		&\geq2^{-1-\varepsilon/q\cdot m}\|\nu^{(m)}\|_q\nonumber\\
		&\geq2^{-11\varepsilon/(10q)\cdot m}\|\nu^{(m)}\|_q,
	\end{align}
	where, in the last two inequalities, we have used $m\gg_{q,\varepsilon}1$. Furthermore, by the definition of $E$ and Minkowski's inequality, we have
	\begin{equation*}
		\|(\theta|_E{\bm .}\nu)^{(m+r)}\|_q=\left\|\left(\left(\sum_{j\leq\lceil\delta m\rceil-1}\theta|_{A_j}\right){\bm .}\nu\right)^{(m+r)}\right\|_q\leq\sum_{j=\lceil-(m-1)\log M-\log M'-1\rceil}^{\lceil\delta m\rceil-1}\|(\theta|_{A_j}{\bm .}\nu)^{(m+r)}\|_q.
	\end{equation*}
	Hence, by pigeonholing and (\ref{L^q_norm_thetaEnu_uniforming_theta}), there exists an integer $\lceil-(m-1)\log M-\log M'-1\rceil\leq j\leq\lceil\delta m\rceil-1$ such that
	\begin{align}\label{L^q_norm_thetaAjnu_uniforming_theta}
		\|(\theta|_{A_j}{\bm .}\nu)^{(m+r)}\|_q&\geq\frac{1}{\lceil\delta m\rceil-\lceil-(m-1)\log M-\log M'-1\rceil}\cdot
		2^{-11\varepsilon/(10q)\cdot m}\|\nu^{(m)}\|_q\nonumber\\
		&=2^{-(\log(\lceil\delta m\rceil-\lceil-(m-1)\log M-\log M'-1\rceil)/m+11\varepsilon/(10q))m}\|\nu^{(m)}\|_q\nonumber\\
		&\geq2^{-6\varepsilon/(5q)\cdot m}\|\nu^{(m)}\|_q,
	\end{align}
	where, in the last inequality, we have used $m\gg_{M,q,\varepsilon,L}1$. Hence, we obtain
	\begin{equation}\label{conclusion_II_uniforming_theta}
		\|(\theta_{A_j}{\bm .}\nu)^{(m+r)}\|_q^q=\theta(A_j)^{-q}\|(\theta|_{A_j}{\bm .}\nu)^{(m+r)}\|_q^q
		\geq\|(\theta|_{A_j}{\bm .}\nu)^{(m+r)}\|_q^q\geq2^{-6\varepsilon/5\cdot m}\|\nu^{(m)}\|_q^q.
	\end{equation}
	
	From (\ref{L^q_norm_thetaAjnu_uniforming_theta}) and the same calculation as (\ref{L^q_norm_thetaFnu_uniforming_theta}), we have
	\begin{align*}
		2^{-6\varepsilon/5\cdot m}\|\nu^{(m)}\|_q^q&\leq\|(\theta|_{A_j}{\bf .}\nu)^{(m+r)}\|_q^q\\
		&\leq O_{\mu,q,C}(1)\theta(A_j)^q\|\nu^{(m)}\|_q^q,
	\end{align*}
	and hence, by using $m\gg_{\mu,q,\varepsilon,C}1$,
	\begin{equation*}
		\theta(A_j)^q\geq 2^{-13\varepsilon/10\cdot m}.
	\end{equation*}
	By this and our assumption (i), we obtain
	\begin{equation*}
		2^{-\sigma m}\geq\|\theta^{(m)}\|_q^q\geq\|(\theta|_{A_j})^{(m)}\|_q^q=\theta(A_j)^q\|(\theta_{A_j})^{(m)}\|_q^q\geq2^{-13\varepsilon/10\cdot m}\|(\theta_{A_j})^{(m)}\|_q^q,
	\end{equation*}
	and hence, by using $0<\varepsilon<\sigma/100$,
	\begin{equation}\label{conclusion_I_uniforming_theta}
		\|(\theta_{A_j})^{(m)}\|_q^q\leq 2^{-9\sigma/10\cdot m}.
	\end{equation}
	If we define the probability measure $\theta'\ll\theta$ by $\theta'=\theta_{A_j}$, then the conclusion (1) and (2) hold from (\ref{conclusion_I_uniforming_theta}) and (\ref{conclusion_II_uniforming_theta}). The conclusion (3) is clear from the definition of $A_j$, so we complete the proof.
\end{proof}

We will reduce $\theta'$ so that the support has the “uniform tree structure” and prove Lemma \ref{tree_structure_supp_theta}. The idea of the proof is from \cite[Lemma 3.6]{Shm19}, which is originated at least from \cite{Bou10}.

\begin{proof}[Proof of Lemma \ref{tree_structure_supp_theta}]
	We take a large constant $D\in\N$ in terms of $M,q,\varepsilon$ which will be specified later.
	Let $l\in\N$ be sufficiently large and $m=lD$, and $\theta$ be a Borel probability measure on $G$ satisfying the assumption for $m$. If $l\gg_{M,\mu,q,\sigma,\varepsilon,C,L}1$, then so is $m=lD$ and we can apply Lemma \ref{uniforming_theta} to $\theta$ and obtain the normalized restriction $\theta'$ of $\theta$ to some $\D^G_m$-measurable set. We define the $\D_{lD}^G$-measurable set $E'$ by
	\begin{equation*}
		E'=\bigsqcup_{\xi\in\D_{lD}^G,\theta'(\xi)>0}\xi.
	\end{equation*}
	Since $\diam\ \supp\ \theta'\leq\diam\ \supp\ \theta\leq L$, $E'$ is a finite union.
	
	We inductively reduce $\theta'$ from the bottom scale $\D_{lD}^G$, using the “dyadic pigeonhole principle” at each scale. First, for each $j\in\{0,1,\dots, D-1\}$, we define $\A^{(l-1,j)}\subset\D_{(l-1)D}^G$ by
	\begin{align*}
		&\A^{(l-1,j)}=\left\{\zeta\in\D_{(l-1)D}^G\left|\ M^j\leq |\D_{lD}^G(\zeta\cap E')|<M^{j+1}\right.\right\}\quad(j<D-1),\\
		&\A^{(l-1,D-1)}=\left\{\zeta\in\D_{(l-1)D}^G\left|\ M^{D-1}\leq |\D_{lD}^G(\zeta\cap E')|\leq M^D\right.\right\}\quad(j=D-1).
	\end{align*}
	Since each $\zeta\in\D_{(l-1)D}^G$ contains at most $M^D$ elements of $\D_{lD}^G$, we have $\bigsqcup_{0\leq j\leq D-1}\A^{(l-1, j)}=\{\zeta\in\D_{(l-1)D}^G\left|\ \zeta\cap E'\neq\emptyset\right.\}$. We also define $\D_{(l-1)D}^G$-measurable sets $A^{(l-1,j)}=\bigsqcup_{\zeta\in\A^{(l-1,j)}}\zeta\ (0\leq j\leq D-1)$, then
	\begin{equation*}
		\theta'=\sum_{0\leq j\leq D-1}\theta'|_{A^{(l-1,j)}}.
	\end{equation*}
	By using this and Minkowski's inequality, we obtain from the property (2) of $\theta'$ that
	\begin{align*}
		2^{-6\varepsilon/(5q)\cdot m}\|\nu^{(m)}\|_q\leq \|(\theta'{\bm .}\nu)^{(m+r)}\|_q&=\left\|\left(\sum_{0\leq j\leq D-1}\theta'|_{A^{(l-1,j)}}{\bm .\nu}\right)^{(m+r)}\right\|_q\\
		&\leq\sum_{0\leq j\leq D-1}\|(\theta'|_{A^{(l-1,j)}}{\bm .}\nu)^{(m+r)}\|_q.
	\end{align*}
	Hence, by pigeonholing, there is $j$ such that
	\begin{equation}\label{dyadic_pigeonholing_tree_structure_supp_theta}
		\|(\theta'|_{A^{(l-1,j)}}{\bm .}\nu)^{(m+r)}\|_q\geq D^{-1}2^{-6\varepsilon/(5q)\cdot m}\|\nu^{(m)}\|_q.
	\end{equation}
	
	Here, $A^{(l-1,j)}=\bigsqcup_{\zeta\in\A^{(l-1,j)}}\zeta$ and each $\zeta\in\A^{(l-1,j)}$ satisfies $M^j\leq |\D_{lD}^G(\zeta\cap E')|\leq M^{j+1}$. We take $\B^{(l,j)}\subset\left\{\xi\in\D_{lD}^G\left|\ \xi\subset A^{(l-1,j)}\cap E'\right.\right\}$ so that $\left|\left\{\xi\in\B^{(l,j)}\left|\ \xi\subset\zeta\right.\right\}\right|\leq M^j$ for each $\zeta\in\A^{(l-1,j)}$ and
	\begin{align*}
		&\left\|\left(\theta'|_{\bigsqcup_{\xi\in\B^{(l,j)}}\xi}{\bm .}\nu\right)^{(m+r)}\right\|_q
		=\max\left\{\left\|\left(\theta'|_{\bigsqcup_{\xi\in\B}\xi}{\bm .}\nu\right)^{(m+r)}\right\|_q\right.\\
		&\qquad\qquad\left.\Bigg|\ \B\subset\left\{\xi\in\D_{lD}^G\left|\ \xi\subset A^{(l-1,j)}\cap E'\right.\right\}
		,\ \left|\left\{\xi\in\B\left|\ \xi\subset\zeta\right.\right\}\right|\leq M^j
		\text{ for each }\zeta\in\A^{(l-1,j)}
		\right\}.
	\end{align*}
	Then, by the maximality, we can have $\left|\left\{\xi\in\B^{(l,j)}\left|\ \xi\subset\zeta\right.\right\}\right|=M^j$ for each $\zeta\in\A^{(l-1,j)}$. Furthermore, we can divide
	$\left\{\xi\in\D_{lD}^G\left|\ \xi\subset A^{(l-1,j)}\cap E'\right.\right\}\setminus\B^{(l,j)}$ as
	\begin{equation*}
		\left\{\xi\in\D_{lD}^G\left|\ \xi\subset A^{(l-1,j)}\cap E'\right.\right\}\setminus\B^{(l,j)}=\bigsqcup_{k=1}^{M-1}\B_k
	\end{equation*}
	and each $\B_k$ satisfies
	\begin{equation*}
		\left|\left\{\xi\in\B_k\left|\ \xi\subset\zeta\right.\right\}\right|\leq M^j
		\text{ for each }\zeta\in\A^{(l-1,j)}.
	\end{equation*}
	By Minkowski's inequality and the maximality of $\B^{(l,j)}$, we obtain
	\begin{align}\label{uniforming_A^j-1_tree_structure_supp_theta}
		\|(\theta'|_{A^{(l-1,j)}}{\bm .}\nu)^{(m+r)}\|_q&=\left\|\left(\left(\theta'|_{\bigsqcup_{\xi\in\B^{(l,j)}}\xi}+\sum_{k=1}^{M-1}\theta'|_{\bigsqcup_{\xi\in\B_k}\xi}\right){\bm .}\nu\right)^{(m+r)}\right\|_q\nonumber\\
		&\leq\left\|\left(\theta'|_{\bigsqcup_{\xi\in\B^{(l,j)}}\xi}{\bm .}\nu\right)^{(m+r)}\right\|_q+\sum_{k=1}^{M-1}\left\|\left(\theta'|_{\bigsqcup_{\xi\in\B_k}\xi}{\bm .}\nu\right)^{(m+r)}\right\|_q\nonumber\\
		&\leq M\left\|\left(\theta'|_{\bigsqcup_{\xi\in\B^{(l,j)}}\xi}{\bm .}\nu\right)^{(m+r)}\right\|_q.
	\end{align}
	Here, we write $\A^{(l-1)}=\A^{(l-1,j)}$, $A^{(l-1)}=A^{(l-1,j)}$, $\B^{(l)}=\B^{(l,j)}$, $R_{l-1}=M^j$ and
	\begin{equation*}
		B^{(l)}=\bigsqcup_{\xi\in\B^{(l)}}\xi.
	\end{equation*}
	Then, $\left|\left\{\xi\in\B^{(l)}\left|\ \xi\subset\zeta\right.\right\}\right|=R_{l-1}$ for each $\zeta\in\A^{(l-1)}$ and, by (\ref{uniforming_A^j-1_tree_structure_supp_theta}) and (\ref{dyadic_pigeonholing_tree_structure_supp_theta}), we have
	\begin{equation}\label{L^q_norm_thetaBlnu_tree_structure_supp_theta}
		\left\|\left(\theta'|_{B^{(l)}}{\bm .}\nu\right)^{(m+r)}\right\|_q
		\geq (MD)^{-1}2^{-6\varepsilon/(5q)\cdot m}\|\nu^{(m)}\|_q.
	\end{equation}
	
	Next, we consider the scale $\D_{(l-1)D}^G$ and reduce $\theta'|_{B^{(l)}}$. We notice that $\{\zeta\in\D_{(l-1)D}^G\left|\ \theta'|_{B^{(l)}}(\zeta)>0\right.\}=\A^{(l-1)}$. For each $j\in\{0,1,\dots, D-1\}$, we define $\A^{(l-2,j)}\subset\D_{(l-2)D}^G$ by
	\begin{align*}
		&\A^{(l-2,j)}=\left\{\eta\in\D_{(l-2)D}^G\left|\ M^j\leq |\D_{(l-1)D}^G(\eta\cap A^{(l-1)})|<M^{j+1}\right.\right\}\quad(j<D-1),\\
		&\A^{(l-2,D-1)}=\left\{\eta\in\D_{(l-2)D}^G\left|\ M^{D-1}\leq |\D_{(l-1)D}^G(\eta\cap A^{(l-1)})|\leq M^D\right.\right\}\quad(j=D-1).
	\end{align*}
	Then, if we write $A^{(l-2,j)}=\bigsqcup_{\eta\in\A^{(l-2,j)}}\eta$, we have
	\begin{equation*}
		\theta'|_{B^{(l)}}=\sum_{0\leq j\leq D-1}\theta'|_{B^{(l)}\cap A^{(l-2,j)}}.
	\end{equation*}
	Hence, by (\ref{L^q_norm_thetaBlnu_tree_structure_supp_theta}), Minkowski's inequality and pigeonholing, there is $0\leq j\leq D-1$ such that
	\begin{equation*}
		\left\|\left(\theta'|_{B^{(l)}\cap A^{(l-2,j)}}{\bm .}\nu\right)^{(m+r)}\right\|_q
		\geq D^{-1}(MD)^{-1}2^{-6\varepsilon/(5q)\cdot m}\|\nu^{(m)}\|_q.
	\end{equation*}
	As the same way as above, we reduce $A^{(l-1)}$ so that the reduced set $B^{(l-1,j)}$: $\D_{(l-1)D}^G$-measurable satisfies that, for each $\eta\in\A^{(l-2,j)}$, $\eta\cap B^{(l-1,j)}\subset A^{(l-1)}$ and the number of $\D_{(l-1)D}^G$-atoms contained in $\eta\cap B^{(l-1,j)}$ is exactly $M^j$, and
	\begin{equation*}
		\left\|\left(\theta'|_{B^{(l)}\cap B^{(l-1,j)}}{\bm .}\nu\right)^{(m+r)}\right\|_q
		\geq (MD)^{-2}2^{-6\varepsilon/(5q)\cdot m}\|\nu^{(m)}\|_q.
	\end{equation*}
	Hence, by putting
	$\A^{(l-2)}=\A^{(l-2,j)}$, $A^{(l-2)}=A^{(l-2,j)}$, $R_{l-2}=M^j$ and $B^{(l-2)}=B^{(l-2,j)}$,
	we have the desired reduction at the scale $\D_{(l-1)D}^G$.
	
	We continue this process to reach the top scale $\D_D^G$. We notice that $\D_0^G=\{G\}$ and, since $\diam\ \supp\ \theta'\leq L$, the number of $\D_D^G$-atoms intersecting $\supp\ \theta'$ is $O_{M, L, D}(1)$. As the result, we obtain a sequence $(R_i)_{i\in[l]}$ with $R_i\in\{1,M,\dots,M^{D-1}\}\ (i>0)$, $R_0=O_{M,L,D}(1)$ and a $(D,l,(R_i)_{i\in[l]})$-uniform set $\widetilde{E}$ such that
	\begin{equation*}
		\left\|\left(\theta'|_{\widetilde{E}}{\bm .}\nu\right)^{(m+r)}\right\|_q
		\geq (MD)^{-(l-1)}2^{-6\varepsilon/(5q)\cdot m}\|\nu^{(m)}\|_q\geq2^{-(\log(MD)/D+6\varepsilon/(5q))m}\|\nu^{(m)}\|_q.
	\end{equation*}
	Here, if $D=D(M,q,\varepsilon)$ is sufficiently large, then $\log(MD)/D<\varepsilon/(10q)$, and hence
	\begin{equation}\label{L^q_norm_theta'tildeEnu_tree_structure_supp_theta}
		\left\|\left(\theta'|_{\widetilde{E}}{\bm .}\nu\right)^{(m+r)}\right\|_q\geq 2^{-13\varepsilon/(10q)\cdot m}\|\nu^{(m)}\|_q.
	\end{equation}
	Therefore, we obtain
	\begin{equation}\label{conclusion_II_tree_structure_supp_theta}
		\left\|\left(\theta'_{\widetilde{E}}{\bm .}\nu\right)^{(m+r)}\right\|_q^q=\theta'(\widetilde{E})^{-q}\left\|\left(\theta'|_{\widetilde{E}}{\bm .}\nu\right)^{(m+r)}\right\|_q^q\geq 2^{-13\varepsilon/10\cdot m}\|\nu^{(m)}\|_q^q.
	\end{equation}
	
	By (\ref{L^q_norm_theta'tildeEnu_tree_structure_supp_theta}) and the same calculation as (\ref{L^q_norm_thetaFnu_uniforming_theta}), we have
	\begin{align*}
		2^{-13\varepsilon/10\cdot m}\|\nu^{(m)}\|_q^q&\leq\left\|\left(\theta'|_{\widetilde{E}}{\bm .}\nu\right)^{(m+r)}\right\|_q^q\\
		&\leq\sum_{I\in\D_{m+r}}\theta'|_{\widetilde{E}}(G)^{q-1}\int_G\nu(g^{-1}I)^q\ d\theta'|_{\widetilde{E}}(g)\\
		&\leq O_{\mu,q,C}(1)\theta'(\widetilde{E})^q\|\nu^{(m)}\|_q^q.
	\end{align*}
	Here, in the last inequality, we notice that $\theta'\ll\theta$, so $\theta'$ has the property that $C^{-1}2^r\leq \|g\|^2\leq C2^r$ and $u^-_g\notin U_1$ for every $g\in\supp\ \theta'$. By using $m=lD$, $l\gg_{\mu,q,\varepsilon,C} 1$, we obtain
	\begin{equation*}
		\theta'(\widetilde{E})^q\geq 2^{-7\varepsilon/5\cdot m}.
	\end{equation*}
	From this and the property (1) of $\theta'$, it follows that
	\begin{equation*}
		2^{-9\sigma/10\cdot m}\geq \|\theta'^{(m)}\|_q^q\geq\|(\theta'|_{\widetilde{E}})^{(m)}\|_q^q=\theta'(\widetilde{E})^q\|(\theta'_{\widetilde{E}})^{(m)}\|_q^q\geq2^{-7\varepsilon/5\cdot m}\|(\theta'_{\widetilde{E}})^{(m)}\|_q^q,
	\end{equation*}
	and hence, by $0<\varepsilon<\sigma/100$,
	\begin{equation}\label{conclusion_I_tree_structure_supp_theta}
		\|(\theta'_{\widetilde{E}})^{(m)}\|_q^q\leq 2^{-4\sigma/5\cdot m}.
	\end{equation}
	
	Finally, we define $\widetilde{\theta}=\theta'_{\widetilde{E}}$. Then, the conclusion (I) and (II) follow from (\ref{conclusion_I_tree_structure_supp_theta}) and (\ref{conclusion_II_tree_structure_supp_theta}). Since $\widetilde{\theta}$ is the normalized restriction of $\theta'$ on the $\D_m^G$-measurable set $\widetilde{E}$, the conclusion (III) follows from the property (3) of $\theta'$. We have already seen that $\widetilde{E}$ is a $(D,l,(R_i)_{i\in [l]})$-uniform, so the conclusion (IV) holds. So we complete the proof.
\end{proof}

\subsection{Finding a nice small scale}\label{subsection_finding_nice_small_scale}

From here, we consider the situation in Lemma \ref{L^q_norm_porosity} except for the assumption (\ref{condition_tau^*(apha)_is_positive}). Actually, we do not use this assumption until Lemma \ref{key_2^{-iD}_components_are_blanching_proof_of_porosity}.
Let $0<\sigma<1$ be a constant, $0<\varepsilon<\delta<\sigma, \delta\ll_{M,\mu,q,\sigma} 1, \varepsilon\ll_{M,\mu,q,\sigma,\delta}1$ be sufficiently small constants and  $C,L>1$.
We can obtain the constant $D=D(M,\mu,q,\sigma,\varepsilon, C)$ from Lemma \ref{tree_structure_supp_theta} for this $0<\varepsilon\ll\sigma$ and $C,L>1$ so that $D$ is also large in terms of $\delta$.
We take sufficiently large $l\in\N$, $l\gg_{M,\mu,q,\sigma,\delta,\varepsilon,D, C,L}1$ and write $m=lD$ and $n=\lfloor\delta l\rfloor$. We also take a Borel probability measure $\theta$ on $G$ satisfying the assumptions of Lemma \ref{L^q_norm_porosity} for this $m$.

By applying Lemma \ref{tree_structure_supp_theta} to $\theta$, we can take a normalized restriction $\widetilde{\theta}$ of $\theta$ to some $\D^G_m$-measurable set satisfying the properties of Lemma \ref{tree_structure_supp_theta}. We see how this $\widetilde{\theta}$ works.
From the properties (III) and (IV) of $\widetilde{\theta}$ and $\widetilde{E}$ and $\widetilde{\theta}(\widetilde{E})=1$, we have
\begin{equation*}
	1\leq\prod_{j\in[l]}R_j\cdot a\leq 2.
\end{equation*}
Similarly,
\begin{equation*}
	\|\widetilde{\theta}^{(m)}\|_q^q\leq\prod_{j\in[l]}R_j\cdot a^q\leq 2^q\|\widetilde{\theta}^{(m)}\|_q^q.
\end{equation*}
From these inequalities, we have
\begin{equation}\label{L^q_norm_widetildetheta_how_widetildetheta_works}
	\|\widetilde{\theta}^{(m)}\|_q^q=\Theta_q(1)\cdot\left(\prod_{j\in[l]}R_j\right)^{-(q-1)}.
\end{equation}
Let $i\in\N$ such that $1\leq i\leq l-n$ and $\xi$ be any element of $\D^G_{iD}$ such that $\widetilde{\theta}(\xi)>0$. Then the component measure $\widetilde{\theta}_\xi$ is the probability measure on $\widetilde{E}\cap \xi$ with an almost uniform mass about $\widetilde{\theta}(\xi)^{-1}a$ on each $\D^G_m$-atom in $\widetilde{E}\cap\xi$.
From the properties (III) and (IV) again, we have
\begin{equation*}
	\prod_{j=i}^{l-1}R_j\cdot\frac{a}{2}\leq\widetilde{\theta}(\xi)\leq \prod_{j=i}^{l-1}R_j\cdot a,
\end{equation*}
$\widetilde{E}\cap\xi$ intersects exactly $\prod_{j=i}^{i+n-1} R_j$ $\D_{(i+n)D}^G$-atoms, and each $\zeta$ of such $\D_{(i+n)D}^G$-atoms has
\begin{equation*}
	\prod_{j=i+n}^{l-1}R_j\cdot\frac{a}{2\widetilde{\theta}(\xi)}\leq\widetilde{\theta}_\xi(\zeta)\leq\prod_{j=i+n}^{l-1}R_j\cdot\frac{a}{\widetilde{\theta}(\xi)}.
\end{equation*}
Hence, we obtain that
\begin{align}\label{L^q_norm_component_theta_how_widetildetheta_works}
	\|\widetilde{\theta}_\xi^{((i+n)D)}\|_q^q&=\Theta_q(1)\cdot\prod_{j=i}^{i+n-1}R_j\cdot\left(\prod_{j=i+n}^{l-1}R_j\cdot\frac{a}{\widetilde{\theta}(\xi)}\right)^q\nonumber\\
	&=\Theta_q(1)\cdot\prod_{j=i}^{i+n-1}R_j\cdot\left(\frac{1}{\prod_{j=i}^{i+n-1}R_j}\right)^q\nonumber\\
	&=\Theta_q(1)\cdot\left(\prod_{j=i}^{i+n-1}R_j\right)^{-(q-1)}.
\end{align}

For $i\in\N$ with $1\leq i\leq l-n$, we define
\begin{equation*}
	T_n(i)=\left(\prod_{j=i}^{i+n-1}R_j\right)^{-(q-1)}.
\end{equation*}
Then, by (\ref{L^q_norm_component_theta_how_widetildetheta_works}), $T_n(i)=\Theta_q(1)\cdot \|\widetilde{\theta}_\xi^{((i+n)D)}\|_q^q$ for any $\xi\in\D_{iD}^G$ with $\widetilde{\theta}(\xi)>0$. Furthermore, we have
\begin{align*}
	\frac{1}{n}\sum_{i=1}^{l-n}\log T_n(i)&=-\frac{q-1}{n}\sum_{i=1}^{l-n}\sum_{j=i}^{i+n-1}\log R_j\\
	&=-\frac{q-1}{n}\sum_{\substack{1\leq i\leq l-n,\ 1\leq j\leq l-1,\\j-n+1\leq i\leq j}}\log R_j\\
	&=-\frac{q-1}{n}\left(\sum_{j=n}^{l-n}n\log R_j+\sum_{j=1}^{n-1}j\log R_j+\sum_{j=l-n+1}^{l-1}(l-j)\log R_j\right)\\
	&=-(q-1)\sum_{j=1}^{l-1}\log R_j+O(qDn\log M),
\end{align*}
where, in the last equation, we have used $R_j\in\{1,M,\dots, M^D\}$ for $j\geq1$.
From this, $R_0=O_{M,L,D}(1)$, $n=\lfloor\delta l\rfloor$ and (\ref{L^q_norm_widetildetheta_how_widetildetheta_works}), it follows that
\begin{align*}
	\frac{1}{n}\sum_{i=1}^{l-n}\log T_n(i)&=-(q-1)\sum_{j=0}^{l-1}\log R_j+O(\delta q\log M\cdot lD)+O_{M,q,L,D}(1)\\
	&=\log\|\widetilde{\theta}^{(m)}\|_q^q+O(\delta q\log M)m+O_{M,q,L,D}(1).
\end{align*}
Here, we notice $m=lD$ in the last equation.
Hence, by the property (I) of $\widetilde{\theta}$ and $m\gg_{M,q,\sigma,L,D}1$, if we take $\delta\ll\sigma/(q\log M)$, we have
\begin{equation*}
	\frac{1}{n}\sum_{i=1}^{l-n}\log T_n(i)\leq -\frac{7\sigma}{10}\cdot m,
\end{equation*}
so
\begin{equation}\label{average_logTni_how_widetildetheta_works}
	\frac{1}{l}\sum_{i=1}^{l-n}\frac{1}{nD}\log T_n(i)\leq-\frac{7\sigma}{10}.
\end{equation}
On the other hand, since $R_j\in\{1,M,\dots, M^D\}$ for $j\geq1$, we have for each $i$ that
\begin{equation*}
	\frac{1}{nD}\log T_n(i)=-\frac{q-1}{nD}\sum_{j=i}^{i+n-1}\log R_j\geq-(q-1)\log M.
\end{equation*}
Hence, if we define
\begin{equation*}
	S=\left\{1\leq i\leq l-n\left|\ \frac{1}{nD}\log T_n(i)\leq -\frac{3\sigma}{5}\right.\right\},
\end{equation*}
then we have
\begin{align*}
	\frac{1}{l}\sum_{i=1}^{l-n}\frac{1}{nD}\log T_n(i)&=\frac{1}{l}\sum_{i\in S}\frac{1}{nD}\log T_n(i)+\frac{1}{l}\sum_{\substack{1\leq i\leq l-n,\\i\notin S}}\frac{1}{nD}\log T_n(i)\\
	&\geq-\frac{|S|}{l}\cdot(q-1)\log M-\frac{l-n-|S|}{l}\frac{3\sigma}{5}\\
	&\geq-\frac{|S|}{l}\cdot(q-1)\log M-\frac{3\sigma}{5}.
\end{align*}
So, by this and (\ref{average_logTni_how_widetildetheta_works}), we obtain\footnote{One might think that, if $q-1\ll\sigma$, (\ref{logTni_dense_how_widetildetheta_works}) is a contradiction. However, this never happen, because it holds that $2^{-\sigma m}\geq\|\theta^{(m)}\|_q^q\geq (M'M^{m-1})^{-(q-1)}$.}
\begin{equation}\label{logTni_dense_how_widetildetheta_works}
	|S|\geq \frac{\sigma}{10(q-1)\log M}\cdot l.
\end{equation}
Here, we recall that, by (\ref{L^q_norm_component_theta_how_widetildetheta_works}), $T_n(i)=\Theta_q(1)\cdot \|\widetilde{\theta}_\xi^{((i+n)D)}\|_q^q$ for any $\xi\in\D_{iD}^G$ with $\widetilde{\theta}(\xi)>0$. Since $nD\geq n=\lfloor\delta l\rfloor\gg_{\sigma,\delta,q}1$ is sufficiently large, we have
\begin{equation*}\label{logTni_L^q_component_widetildetheta_how_widetildetheta_works}
	\frac{1}{nD}\log T_n(i)\geq \frac{1}{nD}\log\|\widetilde{\theta}_\xi^{((i+n)D)}\|_q^q-\frac{\sigma}{10}
\end{equation*}
for any $\xi\in\D_{iD}^G$ with $\widetilde{\theta}(\xi)>0$. By this and (\ref{logTni_dense_how_widetildetheta_works}), we obtain
\begin{align}\label{widetildetheta_any_component_has_flat_L^q_norm}
	&\left|\left\{1\leq i\leq l-n\left|\ \|\widetilde{\theta}_\xi^{((i+n)D)}\|_q^q\leq 2^{-\sigma/2\cdot nD}\text{ for any }\xi\in\D_{iD}^G\text{ with }\widetilde{\theta}(\xi)>0\right.\right\}\right|\nonumber\\
	\geq&\  \frac{\sigma}{10(q-1)\log M}\cdot l.
\end{align}
This is the key property of $\widetilde{\theta}$ which we will use the following argument.

For the proof of Lemma \ref{L^q_norm_porosity}, we need the following lemma.

\begin{lem}\label{Lemma_many_components_which_increase_L^q_norm_under_comvolutions}
	In the setting above, there exist $i\in\N$ with $n<\sigma/(30(q-1)\log M)\cdot l\leq i\leq l-n$ and $\xi\in\D^G_{iD}$ with $\widetilde{\theta}(\xi)>0$ such that
	\begin{equation*}
		\|\widetilde{\theta}_\xi^{((i+n)D)}\|_q^q\leq 2^{-\sigma/2\cdot nD},
	\end{equation*}
	and, if we define
	\begin{equation*}
		\D_\xi=\left\{I\in\D_{iD}\left|\ \nu(I)>0, \|(\widetilde{\theta}_\xi{\bm .}\widehat{\nu_I})^{((i+n)D+r)}\|_q^q\geq 2^{-(\tau(q)+\sqrt{\delta}/2)nD} \right.\right\},
	\end{equation*}
	then we have
	\begin{equation*}
		\sum_{I\in\D_\xi}\nu(2I)^q\geq2^{-(\tau(q)+2\delta^{4/3})iD}.
	\end{equation*}
\end{lem}

\begin{rem}
	\begin{enumerate}
		\renewcommand{\labelenumi}{\rm{(\arabic{enumi})}}
		\item Those $i$ and $\xi$ in Lemma \ref{Lemma_many_components_which_increase_L^q_norm_under_comvolutions} will turn out to be those in Lemma \ref{L^q_norm_porosity}, and $\rho_\xi$ in Lemma \ref{L^q_norm_porosity} will be $\widetilde{\theta}_\xi$.
		\item The term $\delta^{4/3}$ in the last statement will be important for the proof of Lemma \ref{L^q_norm_porosity} (see (\ref{contradiction_inequality_delta^4/3_proof_of_porosity})).
	\end{enumerate}
\end{rem}

In the remaining Section \ref{subsection_finding_nice_small_scale}, we prove Lemma \ref{Lemma_many_components_which_increase_L^q_norm_under_comvolutions}.

We take $n'\in\N$ so that $n'=n=\lfloor \delta l\rfloor$ or $n'=1$, and proceed with the argument simultaneously.
For each $i=1,\dots,l-n'$, we have
\begin{align*}
	\frac{1}{\|(\widetilde{\theta}{\bm .}\nu)^{(iD+r)}\|_q^q}\sum_{J\in\D_{iD+r}}
	\|((\widetilde{\theta}{\bm .}\nu)|_J)^{((i+n')D+r)}\|_q^q
	&=\frac{1}{\|(\widetilde{\theta}{\bm .}\nu)^{(iD+r)}\|_q^q}\sum_{J\in\D_{iD+r}}
	\sum_{H\in\D_{(i+n')D+r}(J)}\widetilde{\theta}{\bm .}\nu(H)^q\\
	&=\frac{\|(\widetilde{\theta}{\bm .}\nu)^{((i+n')D+r)}\|_q^q}{\|(\widetilde{\theta}{\bm .}\nu)^{(iD+r)}\|_q^q},
\end{align*}
and hence
\begin{equation*}
	\log\left(\frac{1}{\|(\widetilde{\theta}{\bm .}\nu)^{(iD+r)}\|_q^q}\sum_{J\in\D_{iD+r}}
	\|((\widetilde{\theta}{\bm .}\nu)|_J)^{((i+n')D+r)}\|_q^q
	\right)
	=\log\|(\widetilde{\theta}{\bm .}\nu)^{((i+n')D+r)}\|_q^q-\log\|(\widetilde{\theta}{\bm .}\nu)^{(iD+r)}\|_q^q.
\end{equation*}
By taking the sum over $1\leq i\leq l-n'$, we obtain
\begin{align*}
	&\sum_{i=1}^{l-n'}\log\left(\frac{1}{\|(\widetilde{\theta}{\bm .}\nu)^{(iD+r)}\|_q^q}\sum_{J\in\D_{iD+r}}
	\|((\widetilde{\theta}{\bm .}\nu)|_J)^{((i+n')D+r)}\|_q^q\right)\\
	=\ &\sum_{i=1}^{l-n'}\left(\log\|(\widetilde{\theta}{\bm .}\nu)^{((i+n')D+r)}\|_q^q-\log\|(\widetilde{\theta}{\bm .}\nu)^{(iD+r)}\|_q^q\right)\\
	=\ &\sum_{i=l-n'+1}^l\log\|(\widetilde{\theta}{\bm .}\nu)^{(iD+r)}\|_q^q-\sum_{i=1}^{n'}\log\|(\widetilde{\theta}{\bm .}\nu)^{(iD+r)}\|_q^q\\
	\geq\ &n'\log\|(\widetilde{\theta}{\bm .}\nu)^{(m+r)}\|_q^q,
\end{align*}
where, in the last inequality, we notice that  $\|(\widetilde{\theta}{\bm .}\nu)^{(iD+r)}\|_q^q\leq1$, $\|(\widetilde{\theta}{\bm .}\nu)^{((i+1)D+r)}\|_q^q\leq\|(\widetilde{\theta}{\bm .}\nu)^{(iD+r)}\|_q^q$ for each $i$ and $m=lD$.
By the property (I) of $\widetilde{\theta}$, we have $\log\|(\widetilde{\theta}{\bm .}\nu)^{(m+r)}\|_q^q\geq -13\varepsilon/10\cdot m+\log\|\nu^{(m)}\|_q^q$. Furthermore, by the definition of $\tau(q)$, we have $\log\|\nu^{(m)}\|_q^q\geq-(\tau(q)+\varepsilon/10)m$ for sufficiently large $m=lD$, $l\gg_{\mu,q,\varepsilon}1$. Hence, we have
\begin{equation}\label{bound_from_below_tau_proof_of_linearization}
	\sum_{i=1}^{l-n'}\log\left(\frac{1}{\|(\widetilde{\theta}{\bm .}\nu)^{(iD+r)}\|_q^q}\sum_{J\in\D_{iD+r}}
	\|((\widetilde{\theta}{\bm .}\nu)|_J)^{((i+n')D+r)}\|_q^q\right)\geq -\left(\tau(q)+\frac{7\varepsilon}{5}\right)n'm.
\end{equation}

We take arbitrary $1\leq i\leq l-n'$ and $J\in\D_{iD+r}$ with $\widetilde{\theta}{\bm .}\nu(J)>0$. 
If we define $f:G\times \RP^1\to\RP^1$ by $f(g,x)=gx$, by the bilinearity of convolution, we have
\begin{equation*}
	(\widetilde{\theta}{\bm .}\nu)|_J=\sum_{\xi\in\D_{iD}^G,\ I\in\D_{iD}}(\widetilde{\theta}|_\xi{\bm .}\nu|_I)|_J=\sum_{(\xi,I)\in\Xi^{(i)}_J,\ \widetilde{\theta}(\xi), \nu(I)>0}\widetilde{\theta}(\xi)\nu(2I)(\widetilde{\theta}_\xi{\bm .}\widehat{\nu_I})|_J,
\end{equation*}
where
\begin{equation*}
	\Xi^{(i)}_J=\left\{(\xi,I)\in\D_{iD}^G\times\D_{iD}\left|\ (\xi\cap\supp\ \widetilde{\theta})\times(I\cap K)\cap f^{-1}J\neq\emptyset\right.\right\}.
\end{equation*}
Hence, if we write
\begin{equation*}
	W_J=\sum_{(\xi,I)\in\Xi^{(i)}_J,\ \widetilde{\theta}(\xi),\nu(I)>0}\widetilde{\theta}(\xi)\nu(2I),
\end{equation*}
we have
\begin{align*}
	\|((\widetilde{\theta}{\bm .}\nu)|_J)^{((i+n')D+r)}\|_q^q&=\sum_{H\in\D_{(i+n')D+r}(J)}(\widetilde{\theta}{\bm .}\nu)|_J(H)^q\\
	&=\sum_{H\in\D_{(i+n')D+r}(J)}\left(\sum_{(\xi,I)\in\Xi^{(i)}_J,\ \widetilde{\theta}(\xi),\nu(I)>0}\widetilde{\theta}(\xi)\nu(2I)(\widetilde{\theta}_\xi{\bm .}\widehat{\nu_I})(H)\right)^q\\
	&=\sum_{H\in\D_{(i+n')D+r}(J)}\left(W_J\right)^q\left(\sum_{(\xi,I)\in\Xi^{(i)}_J,\ \widetilde{\theta}(\xi),\nu(I)>0}\frac{\widetilde{\theta}(\xi)\nu(2I)}{W_J}\cdot(\widetilde{\theta}_\xi{\bm .}\widehat{\nu_I})(H)\right)^q\\
	&\leq\left(W_J\right)^q\sum_{H\in\D_{(i+n')D+r}(J)}
	\sum_{(\xi,I)\in\Xi^{(i)}_J,\ \widetilde{\theta}(\xi),\nu(I)>0}\frac{\widetilde{\theta}(\xi)\nu(2I)}{W_J}\cdot(\widetilde{\theta}_\xi{\bm .}\widehat{\nu_I})(H)^q\\
	&\leq\left(W_J\right)^q\sum_{(\xi,I)\in\Xi^{(i)}_J,\  \widetilde{\theta}(\xi),\nu(I)>0}\frac{\widetilde{\theta}(\xi)\nu(2I)}{W_J}\cdot\|(\widetilde{\theta}_\xi{\bm .}\widehat{\nu_I})^{((i+n')D+r)}\|_q^q,
\end{align*}
where, in the second to the last inequality, we have used the fact that $\left(\widetilde{\theta}(\xi)\nu(2I)/W_J\right)_{(\xi,I)\in\Xi^{(i)}_J,\ \widetilde\theta(\xi),\nu(I)>0}$ is a probability vector and $t\mapsto t^q$ is convex. From this inequality, we obtain for $J\in\D_{iD+r}$ with $\widetilde{\theta}{\bm .}\nu(J)>0$ that
\begin{equation}\label{bound_P_L^q_component_widetildethetanu_from_above_proof_of_linearization}
	\frac{\|((\widetilde{\theta}{\bm .}\nu)|_J)^{((i+n')D+r)}\|_q^q}{\|(\widetilde{\theta}{\bm .}\nu)^{(iD+r)}\|_q^q}\leq \frac{(W_J)^q}{\sum_{J'\in\D_{iD+r}}\widetilde{\theta}{\bm .}\nu(J')^q}\sum_{(\xi,I)\in\Xi^{(i)}_J,\ \widetilde{\theta}(\xi),\nu(I)>0}\frac{\widetilde{\theta}(\xi)\nu(2I)}{W_J}\cdot\|(\widetilde{\theta}_\xi{\bm .}\widehat{\nu_I})^{((i+n')D+r)}\|_q^q.
\end{equation}

Next, we consider $(\sum_{J\in\D_{iD+r}}(W_J)^q)/(\sum_{J'\in\D_{iD+r}}\widetilde{\theta}{\bm .}\nu(J')^q)$ for each $1\leq i\leq l-n'$.
Let $J\in\D_{iD+r}$. For each $I\in\D_{iD}$, we write $I^{(+)}$ and $I^{(-)}$ for the elements of $\D_{iD}$ which are to the right and the left of $I$ (with respect to the canonical direction of $\RP^1\cong\R/\pi\Z$), respectively. Since $2I\subset I^{(-)}\sqcup I\sqcup I^{(+)}$, we have
\begin{equation}\label{WJ_I-_I_I+_proof_of_linearization}
	W_J=\sum_{(\xi,I)\in\Xi^{(i)}_J,\ \widetilde{\theta}(\xi),\nu(I)>0}\widetilde{\theta}(\xi)\nu(2I)\leq\sum_{(\xi,I)\in\Xi_J^{(i)}}\widetilde{\theta}(\xi)\nu(I^{(-)})+\sum_{(\xi,I)\in\Xi_J^{(i)}}\widetilde{\theta}(\xi)\nu(I)+\sum_{(\xi,I)\in\Xi_J^{(i)}}\widetilde{\theta}(\xi)\nu(I^{(+)}).
\end{equation}
Let $(\xi,I)\in\Xi^{(i)}_J$. Then, there exist $g\in\xi\cap\supp\ \widetilde{\theta}$ and $x\in I\cap K$ such that $f(g,x)=gx\in J$. If $\nu(I^{(+)})>0$, then we can take $x'\in I^{(+)}\cap K$. Since $g\in\supp\ \widetilde{\theta}\subset\supp\ \theta$, we have $\|g\|^2\geq C^{-1}2^r$ and $u_g^-\notin U_1$, and we can apply Corollary \ref{contraction_on_U_by_A_preliminaries} to $g$ and $x,x'\in K$. Hence, we have
\begin{equation}\label{dist_xiI_xiI+_proof_of_linearization}
	d_{\RP^1}(gx,gx')\leq \frac{C_1}{\|g\|^2} d_{\RP^1}(x,x')\leq 2C_1C\pi2^{-(iD+r)}.
\end{equation}
Here, for $J\in\D_{iD+r}$ and $k\in\Z$, we write $J^{(k)}$ for the element of $\D_{iD+r}$ which is the $k$-th from $J$ (with respect to the canonical direction of $\RP^1\cong\R/\pi\Z$). Then, (\ref{dist_xiI_xiI+_proof_of_linearization}) and $gx\in J$ tell us that there is $k=O_{\mu, C}(1)$ such that $gx'\in J^{(k)}$, so we have $(\xi\cap\supp\ \widetilde{\theta})\times (I^{(+)}\cap K)\cap f^{-1}J^{(k)}\neq\emptyset$, that is, $(\xi, I^{(+)})\in\Xi^{(i)}_{J^{(k)}}$.
The same argument holds for $I^{(-)}$.
Hence, by (\ref{WJ_I-_I_I+_proof_of_linearization}), we have
\begin{equation*}
	W_J\leq 3\sum_{k=-O_{\mu,C}(1)}^{O_{\mu,C}(1)}\sum_{(\xi, I)\in\Xi^{(i)}_{J^{(k)}}}\widetilde{\theta}(\xi)\nu(I)=3\sum_{k=-O_{\mu,C}(1)}^{O_{\mu,C}(1)}\widetilde{\theta}\times\nu\Bigg(\bigsqcup_{(\xi, I)\in\Xi^{(i)}_{J^{(k)}}}(\xi\cap\supp\ \widetilde{\theta})\times (I\cap K)\Bigg).
\end{equation*}
By taking $(\cdot)^q$ and taking the sum over $J\in\D_{iD+r}$, we have
\begin{align}\label{first_estimate_sum_WJ_proof_of_linearization}
	\sum_{J\in\D_{iD+r}}(W_J)^q&\leq\sum_{J\in \D_{iD+r}}O_{\mu,q,C}(1)\sum_{k=-O_{\mu,C}(1)}^{O_{\mu,C}(1)}\widetilde{\theta}\times\nu\Bigg(\bigsqcup_{(\xi, I)\in\Xi^{(i)}_{J^{(k)}}}(\xi\cap\supp\ \widetilde{\theta})\times (I\cap K)\Bigg)^q\nonumber\\
	&= O_{\mu,q,C}(1)\sum_{J\in\D_{iD+r}}\widetilde{\theta}\times\nu\Bigg(\bigsqcup_{(\xi, I)\in\Xi^{(i)}_J}(\xi\cap\supp\ \widetilde{\theta})\times (I\cap K)\Bigg)^q.
\end{align}

Here, we take $J\in\D_{iD+r}$ and see the mass above. Let $(\xi,I)\in\Xi^{(i)}_J$. We take $(g,x)\in(\xi\cap\supp\ \widetilde{\theta})\times (I\cap K)\cap f^{-1}J$ and arbitrary $(g',x')\in(\xi\cap\supp\ \widetilde{\theta})\times (I\cap K)$. Since
\begin{equation}\label{diam_xi_proof_of_linearization}
	d_G(g,g')\leq\diam\ \xi\leq M2^{-iD}\leq M2^{-D}<r_1=r_1(\mu),
\end{equation}
we can apply Corollary \ref{Lipschitz_continuity_of_the_action} (we notice that $D=D(M,\mu,q,\sigma,\varepsilon,C)$ is a large constant).
We also apply Corollary \ref{contraction_on_U_by_A_preliminaries}. Then, we have
\begin{align}\label{diam_of_xiI_proof_of_porosity}
	d_{\RP^1}(gx,g'x')\leq d_{\RP^1}(gx,g'x)+d_{\RP^1}(g'x,g'x')&\leq\frac{C_2}{\|g\|^2} d_G(g,g')+\frac{C_1}{\|g'\|^2}d_{\RP^1}(x,x')\nonumber\\
	&\leq C_2CM 2^{-(iD+r)}+C_1C\pi 2^{-(iD+r)}.
\end{align}
This inequality and $gx\in J$ tell us that there is $k=O_{M,\mu,C}(1)$ such that $g'x'\in J^{(k)}$, so $(g',x')\in f^{-1}J^{(k)}$. Hence, we obtain
\begin{equation*}
	\bigsqcup_{(\xi, I)\in\Xi^{(i)}_J}(\xi\cap\supp\ \widetilde{\theta})\times (I\cap K)\subset\bigsqcup_{k=-O_{M,\mu,C}(1)}^{O_{M,\mu,C}(1)}f^{-1}J^{(k)}.
\end{equation*}
Since $\widetilde{\theta}{\bm .}\nu=f(\widetilde{\theta}\times\nu)$, we have from this inclusion and (\ref{first_estimate_sum_WJ_proof_of_linearization}) that
\begin{align*}
	\sum_{J\in\D_{iD+r}}(W_J)^q&\leq O_{\mu,q,C}(1)\sum_{J\in\D_{iD+r}}\Bigg(\sum_{k=-O_{M,\mu,C}(1)}^{O_{M,\mu,C}(1)}\widetilde{\theta}\times\nu(f^{-1}J^{(k)})\Bigg)^q\\
	&\leq O_{M,\mu,q,C}(1)\sum_{J\in\D_{iD+r}}\sum_{k=-O_{M,\mu,C}(1)}^{O_{M,\mu,C}(1)}\widetilde{\theta}\times\nu(f^{-1}J^{(k)})^q\\
	&=O_{M,\mu,q,C}(1)\sum_{J\in\D_{iD+r}}\widetilde{\theta}{\bm .}\nu(J)^q.
\end{align*}
Hence, we obtain
\begin{equation}\label{ratio_sum_WJ^q_sum_widetildethetanu^q_bounded_proof_of_linearization}
	\sum_{J\in\D_{iD+r}}\frac{(W_J)^q}{\sum_{J'\in\D_{iD+r}}\widetilde{\theta}{\bm .}\nu(J')^q}\leq O_{M,\mu,q,C}(1).
\end{equation}

For each $1\leq i\leq l-n'$ and $J\in\D_{iD+r}$ with $\widetilde{\theta}{\bm .}\nu(J)>0$, we define
\begin{equation*}
	X^{(i)}_{n',J}=\sum_{(\xi,I)\in\Xi^{(i)}_J,\ \widetilde{\theta}(\xi),\nu(I)>0}\frac{\widetilde{\theta}(\xi)\nu(2I)}{W_J}\cdot\|(\widetilde{\theta}_\xi{\bm .}\widehat{\nu_I})^{((i+n')D+r)}\|_q^q,
\end{equation*}
and, for each $1\leq i\leq l-n'$,
\begin{equation*}
	Y^{(i)}_{n'}=\sum_{J\in\D_{iD+r},\ \widetilde{\theta}{\bm .}\nu(J)>0}\frac{(W_J)^q}{\sum_{J'\in\D_{iD+r}}\widetilde{\theta}{\bm.}\nu(J')^q}\cdot X^{(i)}_{n',J}.
\end{equation*}
Then, by (\ref{bound_from_below_tau_proof_of_linearization}) and (\ref{bound_P_L^q_component_widetildethetanu_from_above_proof_of_linearization}), we have
\begin{equation}\label{average_Yi_tau_proof_of_linearization}
	\frac{1}{l}\sum_{i=1}^{l-n'}\frac{1}{n'D}\log Y^{(i)}_{n'}\geq -\tau(q)-\frac{7\varepsilon}{5}.
\end{equation}
Here, we bound $X^{(i)}_{n',J}$ and $Y^{(i)}_{n'}$ from above. For each $(\xi,I)\in\Xi^{(i)}_J$ with $\widetilde{\theta}(\xi),\nu(I)>0$, we have
\begin{align*}
	\|(\widetilde{\theta}_\xi{\bm.}\widehat{\nu_I})^{((i+n')D+r)}\|_q^q&=\sum_{H\in\D_{(i+n')D+r}}\widetilde{\theta}_\xi{\bm .}\widehat{\nu_I}(H)^q\\
	&=\sum_{H\in\D_{(i+n')D+r}}\left(\int_\xi\widehat{\nu_I}(g^{-1}H)\ d\widetilde{\theta}_\xi(g)\right)^q\\
	&\leq\sum_{H\in\D_{(i+n')D+r}}\int_\xi\widehat{\nu_I}(g^{-1}H)^q\ d\widetilde{\theta}_\xi(g)\\
	&\leq O_{\mu,q,C}(1)\sum_{E\in\D_{(i+n')D}(I)}\widehat{\nu_I}(E)^q,
\end{align*}
where, in the last inequality, we have done the same calculation as (\ref{L^q_norm_thetaFnu_uniforming_theta}).
Since $n'D\geq D$ and $D\gg_{M,\mu,q,\sigma,\varepsilon}1$ is a large constant, we can apply Lemma \ref{local_L^q_norm_lemma} (ii) to $\varepsilon/2$, $n'D$ and $I\in\D_{iD}$. Hence, we obtain
\begin{equation}\label{L^q_norm_widehatnuI_proof_of_linearization}
	\sum_{E\in\D_{(i+n')D}(I)}\widehat{\nu_I}(E)^q=\frac{1}{\nu(2I)^q}\sum_{E\in\D_{(i+n')D}(I)}\nu(E)^q\leq 2^{-(\tau(q)-\varepsilon/2)n'D}.
\end{equation}
By these two inequalities, we obtain
\begin{equation}\label{L^q_norm_component_convolution_proof_of_porosity}
	\|(\widetilde{\theta}_\xi{\bm.}\widehat{\nu_I})^{((i+n')D+r)}\|_q^q\leq O_{\mu,q,C}(1)2^{-(\tau(q)-\varepsilon/2)n'D}.
\end{equation}
Since $\left(\widetilde{\theta}(\xi)\nu(2I)/W_J\right)_{(\xi,I)\in\Xi^{(i)}_J,\ \widetilde\theta(\xi),\nu(I)>0}$ is a probability vector, it follows from (\ref{L^q_norm_component_convolution_proof_of_porosity}) that
\begin{equation*}
	X_{n',J}^{(i)}\leq O_{\mu,q,C}(1)2^{-(\tau(q)-\varepsilon/2)n'D}.
\end{equation*}
By this and (\ref{ratio_sum_WJ^q_sum_widetildethetanu^q_bounded_proof_of_linearization}), we have
\begin{equation*}
	Y_{n'}^{(i)}\leq O_{M,\mu,q,C}(1)2^{-(\tau(q)-\varepsilon/2)n'D}.
\end{equation*}
Since $D\gg_{M,\mu,q,\sigma,\varepsilon,C}1$ is a large constant, we obtain from this that
\begin{equation}\label{upper_bound_logYi_proof_of_linearization}
	\frac{1}{n'D}\log Y_{n'}^{(i)}\leq \frac{1}{n'D}O_{M,\mu,q,C}(1)-\tau(q)+\frac{\varepsilon}{2}\leq-\tau(q)+\varepsilon.
\end{equation}

Here, we take $n'=n=\lfloor\delta l\rfloor$ and define
\begin{equation*}
	V_n=\left\{1\leq i\leq l-n\left|\ \frac{1}{nD}\log Y_n^{(i)}\geq -\tau(q)-\frac{\sqrt{\delta}}{4}\right.\right\}.
\end{equation*}
Then, by (\ref{average_Yi_tau_proof_of_linearization}) and (\ref{upper_bound_logYi_proof_of_linearization}), we have
\begin{align*}
	-\tau(q)-\frac{7\varepsilon}{5}&\leq\frac{1}{l}\sum_{i=1}^{l-n}\frac{1}{nD}\log Y_n^{(i)}\\
	&=\frac{1}{l}\sum_{i\in V_n}\frac{1}{nD}\log Y_n^{(i)}+\frac{1}{l}\sum_{\substack{i=1\\i\notin V_n}}^{l-n}\frac{1}{nD}\log Y_n^{(i)}\\
	&\leq\frac{|V_n|}{l}\left(-\tau(q)+\varepsilon\right)+\frac{l-n-|V_n|}{l}\left(-\tau(q)-\frac{\sqrt{\delta}}{4}\right)\\
	&=\frac{l-n}{l}(-\tau(q))-\frac{l-n}{l}\frac{\sqrt{\delta}}{4}+\frac{|V_n|}{l}\left(\varepsilon+\frac{\sqrt{\delta}}{4}\right),
\end{align*}
and hence\footnote{If $\sqrt{\delta}$ in the definition of $V_n$ is $\sqrt{\varepsilon}$, we can't obtain a meaningful estimate here, because we take $n$ so that $n/l\simeq \delta$.}
\begin{equation*}
	\frac{|V_n|}{l}\left(\varepsilon+\frac{\sqrt{\delta}}{4}\right)\geq\frac{l-n}{l}\frac{\sqrt{\delta}}{4}-\frac{n}{l}\tau(q)-\frac{7\varepsilon}{5}\geq\frac{1-\delta}{4}\sqrt{\delta}-\tau(q)\delta-\frac{7\varepsilon}{5},
\end{equation*} 
where, in the last inequality, we have used $n=\lfloor\delta l\rfloor$. So, we have
\begin{equation*}
	\frac{|V_n|}{l}\geq\frac{(1-\delta)/4-\tau(q)\sqrt{\delta}-7\varepsilon/(5\sqrt{\delta})}{1/4+\varepsilon/\sqrt{\delta}}.
\end{equation*}
From this and $\varepsilon<\delta\ll_{M,\mu,q,\sigma}1$, we obtain that
\begin{equation}\label{large_Yi_dense_proof_of_linearization}
	\frac{|V_n|}{l}>1-\frac{\sigma}{20(q-1)\log M}.
\end{equation}

By (\ref{widetildetheta_any_component_has_flat_L^q_norm}) and (\ref{large_Yi_dense_proof_of_linearization}), we have
\begin{align*}
	&\left|\left\{i\in V_n\left|\ \|\widetilde{\theta}_\xi^{((i+n)D)}\|_q^q\leq 2^{-\sigma/2\cdot nD}\text{ for any }\xi\in\D_{iD}^G\text{ with }\widetilde{\theta}(\xi)>0\right. \right\}\right|\nonumber\\
	\geq&\  \frac{\sigma}{20(q-1)\log M}\cdot l.
\end{align*}
Hence, there exists $i\geq\sigma/(30(q-1)\log M)\cdot l$ such that $i\in V_n$ and 
\begin{equation}\label{L^q_widetildetheta_xi_for_every_xi_proof_of_linearization}
	\|\widetilde{\theta}_\xi^{((i+n)D)}\|_q^q\leq 2^{-\sigma/2\cdot nD}
\end{equation}
for any $\xi\in\D_{iD}^G$ with $\widetilde{\theta}(\xi)>0$.
Since $\delta\ll_{M,q,\sigma}1$, we have
\begin{equation}\label{i_larger_than_i0_proof_of_linearization}
	i\geq\frac{\sigma}{30(q-1)\log M}\cdot l>\lfloor\delta l\rfloor=n.
\end{equation}
Furthermore, since $i\in V_n$, we have
\begin{align}\label{lower_bound_Y^i_proof_of_porosity}
	Y_n^{(i)}&=\sum_{J\in\D_{iD+r},\ \widetilde{\theta}{\bm .}\nu(J)>0}\frac{(W_J)^q}{\sum_{J'\in\D_{iD+r}}\widetilde{\theta}{\bm.}\nu(J')^q}\cdot X^{(i)}_{n,J}\nonumber\\
	&=\sum_{J\in\D_{iD+r},\ \widetilde{\theta}{\bm .}\nu(J)>0}\frac{(W_J)^q}{\sum_{J'\in\D_{iD+r}}\widetilde{\theta}{\bm.}\nu(J')^q}
	\sum_{(\xi,I)\in\Xi^{(i)}_J,\ \widetilde{\theta}(\xi),\nu(I)>0}\frac{\widetilde{\theta}(\xi)\nu(2I)}{W_J}\cdot\|(\widetilde{\theta}_\xi{\bm .}\widehat{\nu_I})^{((i+n)D+r)}\|_q^q\nonumber\\
	&\geq 2^{-(\tau(q)+\sqrt{\delta}/4)nD}.
\end{align}

For each $J\in\D_{iD+r}$ with $\widetilde{\theta}{\bm .}\nu(J)>0$, we define
\begin{equation*}
	Z_J^{(i)}=\left\{(\xi,I)\in\Xi_J^{(i)}\left|\ \widetilde{\theta}(\xi),\nu(I)>0,\|(\widetilde{\theta}_\xi{\bm .}\widehat{\nu_I})^{((i+n)D+r)}\|_q^q\geq 2^{-(\tau(q)+\sqrt{\delta}/2)nD} \right.\right\}.
\end{equation*}
Then, by (\ref{lower_bound_Y^i_proof_of_porosity}), we have
\begin{align*}
	&2^{-(\tau(q)+\sqrt{\delta}/4)nD}\\
	\leq\ &\sum_{J\in\D_{iD+r},\ \widetilde{\theta}{\bm .}\nu(J)>0}\frac{(W_J)^q}{\sum_{J'\in\D_{iD+r}}\widetilde{\theta}{\bm.}\nu(J')^q}
	\sum_{(\xi,I)\in\Xi^{(i)}_J,\ \widetilde{\theta}(\xi),\nu(I)>0}\frac{\widetilde{\theta}(\xi)\nu(2I)}{W_J}\cdot\|(\widetilde{\theta}_\xi{\bm .}\widehat{\nu_I})^{((i+n)D+r)}\|_q^q\\
	=\ &\sum_{J\in\D_{iD+r},\ \widetilde{\theta}{\bm .}\nu(J)>0}\frac{(W_J)^q}{\sum_{J'\in\D_{iD+r}}\widetilde{\theta}{\bm.}\nu(J')^q}
	\sum_{(\xi,I)\in Z^{(i)}_J}\frac{\widetilde{\theta}(\xi)\nu(2I)}{W_J}\cdot\|(\widetilde{\theta}_\xi{\bm .}\widehat{\nu_I})^{((i+n)D+r)}\|_q^q\\
	&+\sum_{J\in\D_{iD+r},\ \widetilde{\theta}{\bm .}\nu(J)>0}\frac{(W_J)^q}{\sum_{J'\in\D_{iD+r}}\widetilde{\theta}{\bm.}\nu(J')^q}
	\sum_{(\xi,I)\in\Xi^{(i)}_J\setminus Z_J^{(i)},\ \widetilde{\theta}(\xi),\nu(I)>0}\frac{\widetilde{\theta}(\xi)\nu(2I)}{W_J}\cdot\|(\widetilde{\theta}_\xi{\bm .}\widehat{\nu_I})^{((i+n)D+r)}\|_q^q\\
	\leq\ &O_{\mu,q,C}(1)\sum_{J\in\D_{iD+r},\ \widetilde{\theta}{\bm .}\nu(J)>0}\frac{(W_J)^q}{\sum_{J'\in\D_{iD+r}}\widetilde{\theta}{\bm.}\nu(J')^q}
	\sum_{(\xi,I)\in Z^{(i)}_J}\frac{\widetilde{\theta}(\xi)\nu(2I)}{W_J}\cdot2^{-(\tau(q)-\varepsilon/2)nD}\\
	&+O_{M,\mu,q,C}(1)2^{-(\tau(q)+\sqrt{\delta}/2)nD},
\end{align*}
where, in the last inequality, we have used (\ref{L^q_norm_component_convolution_proof_of_porosity}) for $n'=n$, the fact that $\left(\widetilde{\theta}(\xi)\nu(2I)/W_J\right)_{(\xi,I)\in\Xi^{(i)}_J,\ \widetilde\theta(\xi),\nu(I)>0}$ is a probability vector and (\ref{ratio_sum_WJ^q_sum_widetildethetanu^q_bounded_proof_of_linearization}). Since $nD\geq n=\lfloor\delta l\rfloor\gg_{M,\mu,q,\delta,C}1$ is sufficiently large, we can see from this inequality that
\begin{equation*}
	\sum_{J\in\D_{iD+r},\ \widetilde{\theta}{\bm .}\nu(J)>0}\frac{(W_J)^q}{\sum_{J'\in\D_{iD+r}}\widetilde{\theta}{\bm.}\nu(J')^q}
	\sum_{(\xi,I)\in Z^{(i)}_J}\frac{\widetilde{\theta}(\xi)\nu(2I)}{W_J}\cdot2^{-(\tau(q)-\varepsilon/2)nD}
	\geq \Omega_{\mu,q,C}(1)2^{-(\tau(q)+\sqrt{\delta}/4)nD},
\end{equation*}
and hence
\begin{equation}\label{mass_of_Z^i_proof_of_porosity}
	\sum_{J\in\D_{iD+r},\ \widetilde{\theta}{\bm .}\nu(J)>0}\frac{(W_J)^q}{\sum_{J'\in\D_{iD+r}}\widetilde{\theta}{\bm.}\nu(J')^q}
	\sum_{(\xi,I)\in Z^{(i)}_J}\frac{\widetilde{\theta}(\xi)\nu(2I)}{W_J}\geq 2^{-\sqrt{\delta}/2\cdot nD}
\end{equation}
(we notice that $\varepsilon\ll\delta$). 

Here, we estimate the left-hand side. We can see that
\begin{align}\label{estimate_mass_of_Z^i_proof_of_porosity}
	&\sum_{J\in\D_{iD+r},\ \widetilde{\theta}{\bm .}\nu(J)>0}\frac{(W_J)^q}{\sum_{J'\in\D_{iD+r}}\widetilde{\theta}{\bm.}\nu(J')^q}
	\sum_{(\xi,I)\in Z^{(i)}_J}\frac{\widetilde{\theta}(\xi)\nu(2I)}{W_J}\nonumber\\
	=\ &\frac{1}{\sum_{J'\in\D_{iD+r}}\widetilde{\theta}{\bm .}\nu(J')^q}\sum_{J\in\D_{iD+r},\  \widetilde{\theta}{\bm .}\nu(J)>0}(W_J)^{q-1}\sum_{(\xi, I)\in Z^{(i)}_J}\widetilde{\theta}(\xi)\nu(2I)\nonumber\\
	\leq\ &\frac{1}{\sum_{J'\in\D_{iD+r}}\widetilde{\theta}{\bm .}\nu(J')^q}\left(\sum_{J\in\D_{iD+r},\  \widetilde{\theta}{\bm .}\nu(J)>0}(W_J)^q\right)^{(q-1)/q}\left(\sum_{J\in\D_{iD+r},\  \widetilde{\theta}{\bm .}\nu(J)>0}\left(\sum_{(\xi, I)\in Z^{(i)}_J}\widetilde{\theta}(\xi)\nu(2I)\right)^q\right)^{1/q}\nonumber\\
	\leq\ &O_{M,\mu,q,C}(1)\left(\frac{1}{\sum_{J'\in\D_{iD+r}}\widetilde{\theta}{\bm .}\nu(J')^q}\sum_{J\in\D_{iD+r},\  \widetilde{\theta}{\bm .}\nu(J)>0}\left(\sum_{(\xi, I)\in Z^{(i)}_J}\widetilde{\theta}(\xi)\nu(2I)\right)^q\right)^{1/q},
\end{align}
where we have used Hölder's inequality and (\ref{ratio_sum_WJ^q_sum_widetildethetanu^q_bounded_proof_of_linearization}). Furthermore, by the convexity of $t\mapsto t^q$, we have
\begin{align}\label{estimate2_mass_of_Z^i_proof_of_porosity}
	\sum_{J\in\D_{iD+r},\  \widetilde{\theta}{\bm .}\nu(J)>0}\left(\sum_{(\xi, I)\in Z^{(i)}_J}\widetilde{\theta}(\xi)\nu(2I)\right)^q&=\sum_{J\in\D_{iD+r},\  \widetilde{\theta}{\bm .}\nu(J)>0}\left(\sum_{\xi\in\D^G_{iD},\ \widetilde{\theta}(\xi)>0}\widetilde{\theta}(\xi)\sum_{I\in\D_{iD},\ (\xi, I)\in Z^{(i)}_J}\nu(2I)\right)^q\nonumber\\
	&\leq \sum_{J\in\D_{iD+r},\  \widetilde{\theta}{\bm .}\nu(J)>0}\sum_{\xi\in\D^G_{iD},\ \widetilde{\theta}(\xi)>0}\widetilde{\theta}(\xi)\left(\sum_{I\in\D_{iD},\ (\xi, I)\in Z^{(i)}_J}\nu(2I)\right)^q.
\end{align}
Here, we fix $J\in\D_{iD+r}$ with $\widetilde{\theta}{\bm .}\nu(J)>0$ and $\xi\in\D^G_{iD}$ with $\widetilde{\theta}(\xi)>0$, and take $I,I'\in\D_{iD}$ such that $(\xi,I),(\xi,I')\in Z_J^{(i)}\subset\Xi_J^{(i)}$. Let $g,g'\in\xi\cap\supp\ \widetilde{\theta}$, $x\in I\cap K$ and $x'\in I'\cap K$ such that $gx, g'x'\in J$. By Corollary \ref{contraction_on_U_by_A_preliminaries}, Corollary \ref{Lipschitz_continuity_of_the_action} and $g,g'\in\supp\ \widetilde{\theta}\subset\supp\ \theta$, we have
\begin{align*}
	d_{\RP^1}(x,x')&\leq C_1\|g\|^2d_{\RP^1}(gx,gx')\\
	&\leq C_1\|g\|^2\left(d_{\RP^1}(gx,g'x')+d_{\RP^1}(g'x',gx')\right)\\
	&\leq C_1C2^r\cdot\pi2^{-(iD+r)}+C_1\|g\|^2\cdot\frac{C_2}{\|g\|^2}d_G(g',g)\\
	&\leq C_1C\pi2^{-iD}+C_1C_2M2^{-iD}.
\end{align*}
This tells us that the number of $I\in\D_{iD}$ such that $(\xi,I)\in Z_J^{(i)}$ is at most $O_{M,\mu,C}(1)$. Hence, we have
\begin{align*}
	&\sum_{J\in\D_{iD+r},\  \widetilde{\theta}{\bm .}\nu(J)>0}\sum_{\xi\in\D^G_{iD},\ \widetilde{\theta}(\xi)>0}\widetilde{\theta}(\xi)\left(\sum_{I\in\D_{iD},\ (\xi, I)\in Z^{(i)}_J}\nu(2I)\right)^q\\
	\leq\ & O_{M,\mu,q,C}(1)\sum_{J\in\D_{iD+r},\  \widetilde{\theta}{\bm .}\nu(J)>0}\sum_{\xi\in\D^G_{iD},\ \widetilde{\theta}(\xi)>0}\widetilde{\theta}(\xi)\sum_{I\in\D_{iD},\ (\xi, I)\in Z^{(i)}_J}\nu(2I)^q\\
	=\ &O_{M,\mu,q,C}(1)\sum_{\xi\in\D^G_{iD},\ \widetilde{\theta}(\xi)>0}\widetilde{\theta}(\xi)\sum_{J\in\D_{iD+r},\  \widetilde{\theta}{\bm .}\nu(J)>0}\sum_{I\in\D_{iD},\ (\xi, I)\in Z^{(i)}_J}\nu(2I)^q.
\end{align*}
We fix $\xi\in\D^G_{iD}$ with $\widetilde{\theta}(\xi)>0$. Then, for each $I\in\D_{iD}$, we can see from the same inequality as (\ref{diam_of_xiI_proof_of_porosity}) that the number of $J\in\D_{iD+r}$ such that $(\xi, I)\in Z_J^{(i)}\subset\Xi_J^{(i)}$ is at most $O_{M,\mu,C}(1)$. Hence, by the definition of $Z_J^{(i)}$, we have
\begin{align*}
	&\sum_{\xi\in\D^G_{iD},\ \widetilde{\theta}(\xi)>0}\widetilde{\theta}(\xi)\sum_{J\in\D_{iD+r},\  \widetilde{\theta}{\bm .}\nu(J)>0}\sum_{I\in\D_{iD},\ (\xi, I)\in Z^{(i)}_J}\nu(2I)^q\\
	\leq\ &O_{M,\mu,C}(1)\sum_{\xi\in\D^G_{iD},\ \widetilde{\theta}(\xi)>0}\widetilde{\theta}(\xi)\sum_{\substack{I\in\D_{iD},\ \nu(I)>0,\\\|(\widetilde{\theta}_\xi{\bm .}\widehat{\nu_I})^{((i+n)D+r)}\|_q^q\geq 2^{-(\tau(q)+\sqrt{\delta}/2)nD}}}\nu(2I)^q.
\end{align*}
So, from (\ref{estimate2_mass_of_Z^i_proof_of_porosity}) and the above two inequalities, we have
\begin{equation*}
	\sum_{J\in\D_{iD+r},\  \widetilde{\theta}{\bm .}\nu(J)>0}\left(\sum_{(\xi, I)\in Z^{(i)}_J}\widetilde{\theta}(\xi)\nu(2I)\right)^q\leq
	O_{M,\mu,q,C}(1)\sum_{\xi\in\D^G_{iD},\ \widetilde{\theta}(\xi)>0}\widetilde{\theta}(\xi)\sum_{\substack{I\in\D_{iD},\ \nu(I)>0,\\\|(\widetilde{\theta}_\xi{\bm .}\widehat{\nu_I})^{((i+n)D+r)}\|_q^q\geq 2^{-(\tau(q)+\sqrt{\delta}/2)nD}}}\nu(2I)^q.
\end{equation*}
By this inequality, (\ref{estimate_mass_of_Z^i_proof_of_porosity}) and (\ref{mass_of_Z^i_proof_of_porosity}), we obtain that
\begin{align*}
	&2^{-\sqrt{\delta}/2\cdot nD}\\
	\leq\ &O_{M,\mu,q,C}(1)\left(\frac{1}{\sum_{J'\in\D_{iD+r}}\widetilde{\theta}{\bm .}\nu(J')^q}\sum_{\xi\in\D^G_{iD},\ \widetilde{\theta}(\xi)>0}\widetilde{\theta}(\xi)\sum_{\substack{I\in\D_{iD},\ \nu(I)>0,\\\|(\widetilde{\theta}_\xi{\bm .}\widehat{\nu_I})^{((i+n)D+r)}\|_q^q\geq 2^{-(\tau(q)+\sqrt{\delta}/2)nD}}}\nu(2I)^q
	\right)^{1/q},
\end{align*}
and hence,
\begin{align*}
	\sum_{\xi\in\D^G_{iD},\ \widetilde{\theta}(\xi)>0}\widetilde{\theta}(\xi)\sum_{\substack{I\in\D_{iD},\ \nu(I)>0,\\\|(\widetilde{\theta}_\xi{\bm .}\widehat{\nu_I})^{((i+n)D+r)}\|_q^q\geq 2^{-(\tau(q)+\sqrt{\delta}/2)nD}}}\nu(2I)^q&\geq\Omega_{M,\mu,q,C}(1)2^{-q\sqrt{\delta}/2\cdot nD}\|(\widetilde{\theta}{\bm .}\nu)^{(iD+r)}\|_q^q\nonumber\\
	&\geq 2^{-q\sqrt{\delta}nD}\|(\widetilde{\theta}{\bm .}\nu)^{(iD+r)}\|_q^q
\end{align*}
(we notice that $nD\geq n=\lfloor\delta l\rfloor\gg_{M,\mu,q,\delta,C}1$ is large).
So, there exists $\xi\in\D^G_{iD}$ with $\widetilde{\theta}(\xi)>0$ such that, for
$
\D_\xi=\left\{I\in\D_{iD}\left|\ \nu(I)>0, \|(\widetilde{\theta}_\xi{\bm .}\widehat{\nu_I})^{((i+n)D+r)}\|_q^q\geq 2^{-(\tau(q)+\sqrt{\delta}/2)nD} \right.\right\},
$
we have
\begin{equation}\label{L^q_norm_of_key_components_wrt_delta_proof_of_porosity}
	\sum_{I\in\D_\xi}\nu(2I)^q\geq 2^{-q\sqrt{\delta}nD}\|(\widetilde{\theta}{\bm .}\nu)^{(iD+r)}\|_q^q.
\end{equation}
Here, from (\ref{i_larger_than_i0_proof_of_linearization}), we can see that
\begin{equation*}
	n\leq\delta l<\frac{30(q-1)\log M}{\sigma}\cdot \delta i,
\end{equation*}
and by this and (\ref{L^q_norm_of_key_components_wrt_delta_proof_of_porosity}), we obtain that
\begin{equation}\label{L^q_norm_of_key_components_proof_of_porosity}
	\sum_{I\in\D_\xi}\nu(2I)^q>2^{-30q(q-1)\log M/\sigma\cdot \delta^{3/2}iD}\|(\widetilde{\theta}{\bm .}\nu)^{(iD+r)}\|_q^q>2^{-\delta^{4/3}iD}\|(\widetilde{\theta}{\bm .}\nu)^{(iD+r)}\|_q^q
\end{equation}
(we notice that $\delta\ll_{M,q,\sigma}1$).

Next, we bound $\|(\widetilde{\theta}{\bm .}\nu)^{(iD+r)}\|_q^q$ from below. To do this, we return to the argument up to (\ref{upper_bound_logYi_proof_of_linearization}) and take $n'=1$.
By (\ref{bound_from_below_tau_proof_of_linearization}) for $n'=1$, we have
\begin{equation}\label{n'_1_bound_from_below_tau_proof_of_porosity}
	\frac{1}{l}\sum_{j=1}^{l-1}\frac{1}{D}\log\left(\frac{\|(\widetilde{\theta}{\bm .}\nu)^{((j+1)D+r)}\|_q^q}{\|(\widetilde{\theta}{\bm .}\nu)^{(jD+r)}\|_q^q}\right)\geq-\tau(q)-\frac{7\varepsilon}{5}.
\end{equation}
Furthermore, by (\ref{bound_P_L^q_component_widetildethetanu_from_above_proof_of_linearization}) and (\ref{upper_bound_logYi_proof_of_linearization}) for $n'=1$ and the definition of $Y_1^{(j)}$, we have
\begin{equation}\label{upper_bound_ratio_of_L^q_norms_proof_of_porosity}
	\frac{1}{D}\log\left(\frac{\|(\widetilde{\theta}{\bm .}\nu)^{((j+1)D+r)}\|_q^q}{\|(\widetilde{\theta}{\bm .}\nu)^{(jD+r)}\|_q^q}\right)\leq -\tau(q)+\varepsilon
\end{equation}
for $1\leq j\leq l-1$.
We define
\begin{equation*}
	V_1=\left\{1\leq j\leq l-1\left|\ \frac{\|(\widetilde{\theta}{\bm .}\nu)^{((j+1)D+r)}\|_q^q}{\|(\widetilde{\theta}{\bm .}\nu)^{(jD+r)}\|_q^q}\geq 2^{-(\tau(q)+\sqrt{\varepsilon})D}\right.\right\}.
\end{equation*}
Then, we can see from (\ref{n'_1_bound_from_below_tau_proof_of_porosity}) and (\ref{upper_bound_ratio_of_L^q_norms_proof_of_porosity}) that
\begin{align*}
	-\tau(q)-\frac{7\varepsilon}{5}&\leq \frac{1}{l}\sum_{j\in V_1}\frac{1}{D}\log\left(\frac{\|(\widetilde{\theta}{\bm .}\nu)^{((j+1)D+r)}\|_q^q}{\|(\widetilde{\theta}{\bm .}\nu)^{(jD+r)}\|_q^q}\right)
	+\frac{1}{l}\sum_{\substack{j=1\\j\notin V_1}}^{l-1}\frac{1}{D}\log\left(\frac{\|(\widetilde{\theta}{\bm .}\nu)^{((j+1)D+r)}\|_q^q}{\|(\widetilde{\theta}{\bm .}\nu)^{(jD+r)}\|_q^q}\right)\\
	&\leq \frac{|V_1|}{l}(-\tau(q)+\varepsilon)+\frac{l-1-|V_1|}{l}(-\tau(q)-\sqrt{\varepsilon})\\
	&\leq \frac{l-1}{l}(-\tau(q))-\frac{l-1}{l}\sqrt{\varepsilon}+\frac{|V_1|}{l}(\varepsilon+\sqrt{\varepsilon}),
\end{align*}
and hence
\begin{equation}\label{density_of_V_1_proof_of_porosity}
	\frac{|V_1|}{l}\geq\frac{(l-1)/l\cdot\sqrt{\varepsilon}-\tau(q)/l-7\varepsilon/5}{\sqrt{\varepsilon}+\varepsilon}=\frac{(l-1)/l-\tau(q)/(\sqrt{\varepsilon}l)-7\sqrt{\varepsilon}/5}{1+\sqrt{\varepsilon}}\geq\frac{1-2\sqrt{\varepsilon}}{1+\sqrt{\varepsilon}}
\end{equation}
(we notice that $l\gg_{\mu,q,\varepsilon}1$). Moreover, for $1\leq j\leq l-1$, we have
\begin{align*}
	\frac{\|(\widetilde{\theta}{\bm .}\nu)^{((j+1)D+r)}\|_q^q}{\|(\widetilde{\theta}{\bm .}\nu)^{(jD+r)}\|_q^q}&=\frac{1}{\|(\widetilde{\theta}{\bm .}\nu)^{jD+r}\|_q^q}\sum_{J\in\D_{jD+r}}\sum_{H\in\D_{(j+1)D+r}(J)}\widetilde{\theta}{\bm .}\nu(H)^q\nonumber\\
	&\geq \frac{1}{\|(\widetilde{\theta}{\bm .}\nu)^{jD+r}\|_q^q}\sum_{J\in\D_{jD+r}}2^{-(q-1)D}\widetilde{\theta}{\bm .}\nu(J)^q\nonumber\\
	&=2^{-(q-1)D}.
\end{align*}
By this inequality, we have for our $i$ that
\begin{align}\label{estimate_L^q_norm_convolution_iD+r_scale_proof_of_porosity}
	\|(\widetilde{\theta}{\bm .}\nu)^{(iD+r)}\|_q^q&=\prod_{j=1}^{i-1}\frac{\|(\widetilde{\theta}{\bm .}\nu)^{((j+1)D+r)}\|_q^q}{\|(\widetilde{\theta}{\bm .}\nu)^{(jD+r)}\|_q^q}\cdot\|(\widetilde{\theta}{\bm .}\nu)^{(D+r)}\|_q^q\nonumber\\
	&=\prod_{\substack{j=1\\j\in V_1}}^{i-1}\frac{\|(\widetilde{\theta}{\bm .}\nu)^{((j+1)D+r)}\|_q^q}{\|(\widetilde{\theta}{\bm .}\nu)^{(jD+r)}\|_q^q}\cdot
	\prod_{\substack{j=1\\j\notin V_1}}^{i-1}\frac{\|(\widetilde{\theta}{\bm .}\nu)^{((j+1)D+r)}\|_q^q}{\|(\widetilde{\theta}{\bm .}\nu)^{(jD+r)}\|_q^q}\cdot
	\|(\widetilde{\theta}{\bm .}\nu)^{(D+r)}\|_q^q\nonumber\\
	&\geq 2^{-(\tau(q)+\sqrt{\varepsilon})|V_1\cap[1,i-1]|D}2^{-(q-1)|[1,i-1]\setminus V_1|D}\|(\widetilde{\theta}{\bm .}\nu)^{(D+r)}\|_q^q\nonumber\\
	&\geq 2^{-(\tau(q)+\sqrt{\varepsilon})iD}2^{-(q-1)|[1,i-1]\setminus V_1|D}\|(\widetilde{\theta}{\bm .}\nu)^{(D+r)}\|_q^q.
\end{align}
Here, it follows that $|[1,i-1]\setminus V_1|\leq |[1,l-1]\setminus V_1|$ and,
by (\ref{density_of_V_1_proof_of_porosity}) and (\ref{i_larger_than_i0_proof_of_linearization}), we have
\begin{equation*}
	|[1,l-1]\setminus V_1|\leq \left(1-\frac{1-2\sqrt{\varepsilon}}{1+\sqrt{\varepsilon}}\right)l\leq\frac{3\sqrt{\varepsilon}}{1+\sqrt{\varepsilon}}\frac{30(q-1)\log M}{\sigma}\cdot i\leq \frac{\varepsilon^{2/5}i}{(q-1)^2}
\end{equation*}
(we notice that $\varepsilon\ll_{M,q,\sigma}1$). Hence, by these two inequalities, we have
\begin{equation}\label{estimate_[1,i-1]_setminus_V_1_proof_of_porosity}
	|[1,i-1]\setminus V_1|\leq\frac{\varepsilon^{2/5}i}{q-1}.
\end{equation}

Here, we see that
\begin{equation}\label{lower_bound_L^q_norm_of_convolution_indep_of_r_proof_of_porosity}
	\|(\widetilde{\theta}{\bm .}\nu)^{(D+r)}\|_q^q\geq\Omega_{\mu,q,C,L,D}(1)
\end{equation}
(the lower bound is independent of $r$).
We notice that
\begin{equation}\label{L^q_norm_of_convolution_at_scale_D+r_proof_of_porosity}
	\|(\widetilde{\theta}{\bm .}\nu)^{(D+r)}\|_q^q=\sum_{J\in\D_{D+r}}\left(\int_G\nu(g^{-1}J)\ d\widetilde{\theta}(g)\right)^q.
\end{equation}
Since $K=\supp\ \nu$ is covered by at most $2^D$ open balls with radius $\pi2^{-D}$, there is $x_1\in K$ such that
\begin{equation}\label{mass_x_1_ball_proof_of_porosity}
	\nu\left(B_{\pi2^{-D}}(x_1)\right)\geq 2^{-D}.
\end{equation}
Furthermore, since $\diam\ \supp\ \widetilde{\theta}\leq\diam\ \supp\ \theta\leq L$, $\supp\ \widetilde{\theta}$ is covered by at most $O_{L,D}(1)$ open balls with radius $2^{-D}$. Hence, there is $g_1\in\supp\ \widetilde{\theta}$ such that
\begin{equation}\label{mass_g_1_ball_proof_of_porosity}
	\widetilde{\theta}\left(B^G_{2^{-D}}(g_1)\right)\geq \Omega_{L,D}(1).
\end{equation}
Then, for any $x\in B_{\pi2^{-D}}(x_1)\cap K$ and $g\in B^G_{2^{-D}}(g_1)\cap\supp\ \widetilde{\theta}$, we can see from the same argument as (\ref{diam_of_xiI_proof_of_porosity}) that
\begin{align}\label{diam_of_product_of_two_balls_proof_of_porpsity}
	d_{\RP^1}(gx,g_1x_1)\leq d_{\RP^1}(gx, g_1x)+d_{\RP^1}(g_1x,g_1x_1)&\leq\frac{C_2}{\|g_1\|^2}d_G(g,g_1)+\frac{C_1}{\|g_1\|^2}d_{\RP^1}(x,x_1)\nonumber\\
	&\leq O_{\mu,C}(2^{-(D+r)})
\end{align}
(we notice that $D\gg_\mu 1$ is a large constant). Hence, there are at most $O_{\mu,C}(1)$ consecutive interval $J^{(k)}\in\D_{D+r}\ (1\leq k\leq O_{\mu,C}(1))$ such that
\begin{equation*}
	g\left(B_{\pi2^{-D}}(x_1)\cap K\right)\subset\bigsqcup_{k=1}^{O_{\mu,C}(1)}J^{(k)}
\end{equation*}
for any $g\in B^G_{2^{-D}}(g_1)\cap\supp\ \widetilde{\theta}$. From (\ref{L^q_norm_of_convolution_at_scale_D+r_proof_of_porosity}), this inclusion, (\ref{mass_x_1_ball_proof_of_porosity}) and (\ref{mass_g_1_ball_proof_of_porosity}), it follows that
\begin{align*}
	\|(\widetilde{\theta}{\bm .}\nu)^{(D+r)}\|_q^q&=\sum_{J\in\D_{D+r}}\left(\int_G\nu(g^{-1}J)\ d\widetilde{\theta}(g)\right)^q\\
	&\geq\sum_{k=1}^{O_{\mu,C}(1)}\left(\int_{B^G_{2^{-D}}(g_1)}\nu(g^{-1}J^{(k)})\ d\widetilde{\theta}(g)\right)^q\\
	&\geq\Omega_{\mu,q,C}(1)\left(\int_{B^G_{2^{-D}}(g_1)}\sum_{k=1}^{O_{\mu,C}(1)}\nu(g^{-1}J^{(k)})\ d\widetilde{\theta}(g)\right)^q\\
	&\geq\Omega_{\mu,q,C}(1)\left(\int_{B^G_{2^{-D}}(g_1)}\nu\left(B_{\pi 2^{-D}}(x_1)\cap K\right)\ d\widetilde{\theta}(g)\right)^q\\
	&\geq\Omega_{\mu,q,C,L,D}(1).
\end{align*}
So, we obtain (\ref{lower_bound_L^q_norm_of_convolution_indep_of_r_proof_of_porosity}).

By (\ref{estimate_L^q_norm_convolution_iD+r_scale_proof_of_porosity}), (\ref{estimate_[1,i-1]_setminus_V_1_proof_of_porosity}) and (\ref{lower_bound_L^q_norm_of_convolution_indep_of_r_proof_of_porosity}), we have
\begin{equation*}
	\|(\widetilde{\theta}{\bm .}\nu)^{(iD+r)}\|_q^q\geq \Omega_{\mu,q,C,L,D}(1)2^{-(\tau(q)+\sqrt{\varepsilon}+\varepsilon^{2/5})iD}.
\end{equation*}
By (\ref{i_larger_than_i0_proof_of_linearization}), we have $i\geq \sigma/(30(q-1)\log M)\cdot l$ and $l\gg_{\mu,q,\sigma,\varepsilon,C,L,D}1$. Hence, from the above inequality, we obtain that
\begin{equation}\label{conclusion_L^q_norm_of_convolution_iD+r_scale_proof_of_porosity}
	\|(\widetilde{\theta}{\bm .}\nu)^{(iD+r)}\|_q^q\geq2^{-(\tau(q)+2\varepsilon^{2/5})iD}.
\end{equation}
By (\ref{L^q_norm_of_key_components_proof_of_porosity}) and (\ref{conclusion_L^q_norm_of_convolution_iD+r_scale_proof_of_porosity}), we obtain that
\begin{equation}\label{conclusion_L^q_norm_of_key_components_proof_of_porosity}
	\sum_{I\in\D_\xi}\nu(2I)^q>2^{-(\tau(q)+\delta^{4/3}+2\varepsilon^{2/5})iD}>2^{-(\tau(q)+2\delta^{4/3})iD}
\end{equation}
(we notice that $\varepsilon\ll\delta$).
From (\ref{L^q_widetildetheta_xi_for_every_xi_proof_of_linearization}), (\ref{i_larger_than_i0_proof_of_linearization}) and (\ref{conclusion_L^q_norm_of_key_components_proof_of_porosity}), we can see that we have obtained the desired $n<i\leq l-n$ and $\xi\in\D^G_{iD}$ with $\widetilde{\theta}(\xi)>0$, so complete the proof of Lemma \ref{Lemma_many_components_which_increase_L^q_norm_under_comvolutions}.

\subsection{Sufficiently separated nice intervals}\label{subsection_sufficiently_separated_nice_intervals}

Using Lemma \ref{Lemma_many_components_which_increase_L^q_norm_under_comvolutions}, we proceed with the proof of Lemma \ref{L^q_norm_porosity}. We define
\begin{equation*}
	\widetilde{\D_\xi}=\left\{I\in\D_{iD}\left|\ I\in\D_\xi\text{ or }I^{(+)}\in\D_\xi\text{ or }I^{(-)}\in\D_\xi\right.\right\}.
\end{equation*}
Here, we recall that, for $I\in\D_{iD}$, $I^{(+)}$ and $I^{(-)}$ are the two neighboring intervals in $\D_{iD}$ of $I$. Then, by Lemma \ref{Lemma_many_components_which_increase_L^q_norm_under_comvolutions}, we have
\begin{equation*}
	2^{-(\tau(q)+2\delta^{4/3})iD}\leq \sum_{I\in\D_\xi}\nu(2I)^q\leq 3^{q-1}\sum_{I\in\D_\xi}\left(\nu(I^{(-)})^q+\nu(I)^q+\nu(I^{(+)})^q\right)\leq 3^q\sum_{I\in\widetilde{\D_\xi}}\nu(I)^q,
\end{equation*}
and hence
\begin{equation}\label{widetildeDix_L^q_norm_proof_of_porosity}
	\sum_{I\in\widetilde{\D_\xi}}\nu(I)^q\geq 3^{-q}2^{-(\tau(q)+2\delta^{4/3})iD}.
\end{equation}
We reduce $\widetilde{\D_\xi}$ into a “uniform tree”. For $u\in\RP^1$ and $m\in\N$, we define the $2^{-m}$ dyadic partition of $\RP^1$ with the base $u$ by\footnote{We notice that $\angle_u=\angle_{[1:0]}-\angle_{[1:0]}(u): \RP^1\stackrel{\sim}{\longrightarrow}\R/\pi\Z$.}
\begin{equation*}
	\D_{u,m}=\left\{\angle_u^{-1}[\pi k2^{-m},\pi(k+1)2^{-m})\right\}_{k=0}^{2^m-1}=\left\{\angle_{[1:0]}^{-1}[\pi k2^{-m}+\angle_{[1:0]}(u),\pi(k+1)2^{-m}+\angle_{[1:0]}(u))\right\}_{k=0}^{2^m-1}.
\end{equation*}
Note that, if $\angle_{[1:0]}(u)\in\pi 2^{-m}\Z/\pi\Z$, then $\D_{u,m}=\D_m$.
For a subset $E\subset\RP^1$, we write
$\D_{u,m}(E)=\left\{I\in\D_{u,m}\left|\ I\cap E\neq\emptyset\right.\right\}$.
Furthermore, for $I\in\D_{u,m}$, we write $1/2\cdot I$ for the $1/2$ contraction of the interval $I$ with the same center.

\begin{lem}\label{wideltildeDix_into_uniform_tree_proof_of_porosity}
	In the setting above, there exist a subset $\widetilde{\D_\xi}'\subset\widetilde{\D_\xi}$, $u\in\RP^1$ with $\angle_{[1:0]}(u)\in\pi2^{-iD}\Z/\pi\Z$ and a sequence $(R'_s)_{s\in[i]}, R'_s\in\{1,2,\dots, 2^{D}\}$ such that
	\begin{equation*}
		\sum_{I\in\widetilde{\D_\xi}'}\nu(I)^q\geq 2^{-(\tau(q)+3\delta^{4/3})iD}
	\end{equation*}
	and, for any $s\in[i]$ and $J\in\D_{u,sD}$,
	\begin{equation*}
		I\in\widetilde{\D_\xi}', I\subset J\implies I\subset \frac{1}{2}J.
	\end{equation*}
	Furthermore, $\widetilde{\D_\xi}'$ is $(D,i,u,(R'_s)_{s\in[i]})$-uniform, that is, for any $s\in [i]$ and $J\in\D_{u,sD}$ such that $I\subset J$ for some $I\in\widetilde{\D_\xi}'$, we have
	\begin{equation*}
		\left|\left\{J'\in\D_{u,(s+1)D}(J)\left|\ I\subset J'\text{ for some }I\in\widetilde{\D_\xi}'\right.\right\}\right|=R'_s.
	\end{equation*}
\end{lem}

\begin{proof}
	First, we find a “large” subset $\widetilde{\D_\xi}^{(1)}\subset\widetilde{\D_\xi}$ and $u\in\RP^1$ with $\angle_{[1:0]}(u)\in\pi2^{-iD}\Z/\pi\Z$ satisfying the second condition. The idea is from \cite[Lemma 3.8]{Shm19}. Here, we identify $\RP^1$ with $\R/\pi\Z$ by $\angle_{[1:0]}: \RP^1\stackrel{\sim}{\longrightarrow}\R/\pi\Z$.
	
	We notice that, for any $I\in\D_{iD}$, $u_0\in \pi2^{-iD}\Z/\pi\Z$, $s\in[i]$ and $J=J_{I,u_0,s}\in\D_{u_0,sD}$ such that $I\subset J$, there is $u=u_{I,u_0,s}\in\{-\pi2^{-sD-2},0,\pi2^{-sD-2}\}$ such that $I\subset1/2\cdot(J+u)$.
	We start from $s=i-1$. By pigeonholing for (\ref{widetildeDix_L^q_norm_proof_of_porosity}), there exist a subset $\widetilde{\D_\xi}^{(1,i-1)}\subset\widetilde{\D_\xi}$ and $u_{i-1}\in\{-\pi2^{-(i-1)D-2},0,\pi2^{-(i-1)D-2}\}$ such that
	\begin{equation}\label{widetildeDix^1_i-1_proof_of_porosity}
		\sum_{I\in\widetilde{\D_\xi}^{(1,i-1)}}\nu(I)^q\geq 3^{-1}3^{-q}2^{-(\tau(q)+2\delta^{4/3})iD}
	\end{equation}
	and
	\begin{equation*}
		I\subset \frac{1}{2}(J_{I,0,i-1}+u_{i-1})\text{ for any }I\in\widetilde{\D_\xi}^{(1,i-1)}
	\end{equation*}
	
	Next, let $s=i-2$. By pigeonholing for (\ref{widetildeDix^1_i-1_proof_of_porosity}), we can find a subset $\widetilde{\D_\xi}^{(1,i-2)}\subset\widetilde{\D_\xi}^{(1,i-1)}$ and $u_{i-2}\in\{-\pi2^{-(i-2)D-2},0,\pi2^{-(i-2)D-2}\}$ such that
	\begin{equation*}
		\sum_{I\in\widetilde{\D_\xi}^{(1,i-2)}}\nu(I)^q\geq 3^{-2}3^{-q}2^{-(\tau(q)+2\delta^{4/3})iD}
	\end{equation*}
	and
	\begin{equation*}
		I\subset\frac{1}{2}(J_{I,u_{i-1},i-2}+u_{i-2})\text{ for any }I\in\widetilde{\D_\xi}^{(1,i-2)}.
	\end{equation*}
	Here, since $D\geq2$, we have $u_{i-2}\in\{-\pi2^{-(i-2)D-2},0,\pi2^{-(i-2)D-2}\}\subset\pi2^{-(i-1)D}\Z$. Hence, $J_{I,0,i-1}-u_{i-2}\in\D_{(i-1)D}$ and 
	\begin{equation*}
		I\subset \frac{1}{2}(J_{I,0,i-1}+u_{i-1})=\frac{1}{2}(J_{I,0,i-1}-u_{i-2}+(u_{i-2}+u_{i-1}))\text{ for any }I\in\widetilde{\D_\xi}^{(1,i-1)}.
	\end{equation*}
	
	By continuing this process to the top scale $s=0$, we obtain $\widetilde{\D_\xi}^{(1)}=\widetilde{\D_\xi}^{(1,0)}\subset\widetilde{\D_\xi}$ and $u=u_0+\cdots+u_{i-2}+u_{i-1}$ such that
	\begin{equation}\label{widetildeDxi^1_L^q_norm_proof_of_porosity}
		\sum_{I\in\widetilde{\D_\xi}^{(1)}}\nu(I)^q\geq 3^{-(q+i)}2^{-(\tau(q)+2\delta^{4/3})iD}
	\end{equation}
	and, for any $I\in\widetilde{\D_\xi}^{(1)}$ and each $s\in[i]$, there is $J\in\D_{sD}$ such that
	\begin{equation*}
		I\subset\frac{1}{2}(J+u).
	\end{equation*}
	Hence, we have constructed the desired $\widetilde{\D_\xi}^{(1)}\subset\widetilde{\D_\xi}$ and $u\in\pi2^{-iD}\Z/\pi\Z$.
	
	We next reduce $\widetilde{\D_\xi}^{(1)}$ to a large subset with uniform tree structure. We use the idea of the proof of \cite[Lemma 3.6]{Shm19}, which is an inductive application of pigeonholing like the above argument and the proof of Lemma \ref{tree_structure_supp_theta}. In the following, for $s,m\in\N$ with $s<m$, $\D'\subset\D_{u,m}$ and $J\in\D_{u,s}$, we write $\D'(J)=\left\{I\in\D'\left|\ I\subset J\right.\right\}$.
	
	First, we begin from the bottom scale $s=i-1$. For each $j=0,1,\dots,D-1$, we define $\A^{(i-1,j)}\subset\D_{u,(i-1)D}$ by
	\begin{align*}
		&\A^{(i-1,j)}=\left\{J\in\D_{u,(i-1)D}\left|\ 2^j\leq\left|\widetilde{\D_\xi}^{(1)}(J)\right|<2^{j+1}\right.\right\}\quad(j<D-1),\\
		&\A^{(i-1,D-1)}=\left\{J\in\D_{u,(i-1)D}\left|\ 2^{D-1}\leq\left|\widetilde{\D_\xi}^{(1)}(J)\right|\leq2^D\right.\right\}\quad(j=D-1).
	\end{align*}
	Then, since $\left\{J\in\D_{u,(i-1)D}\left|\ I\subset J\ \text{ for some }I\in\widetilde{\D_\xi}^{(1)}\right.\right\}=\bigsqcup_{j=0}^{D-1}\A^{(i-1,j)}$, it follows from (\ref{widetildeDxi^1_L^q_norm_proof_of_porosity}) that
	\begin{equation*}
		\sum_{j=0}^{D-1}\sum_{I\in\bigsqcup_{J\in\A^{(i-1,j)}}\widetilde{\D_\xi}^{(1)}(J)}\nu(I)^q=
		\sum_{I\in\widetilde{\D_\xi}^{(1)}}\nu(I)^q\geq 3^{-(q+i)}2^{-(\tau(q)+2\delta^{4/3})iD}.
	\end{equation*}
	Hence, by pigeonholing, there is $j_{i-1}\in\{0,1,\dots,D-1\}$ such that
	\begin{equation*}
		\sum_{I\in\bigsqcup_{J\in\A^{(i-1,j_{i-1})}}\widetilde{\D_\xi}^{(1)}(J)}\nu(I)^q\geq D^{-1}3^{-(q+i)}2^{-(\tau(q)+2\delta^{4/3})iD}.
	\end{equation*}
	We write $\A^{(i-1)}=\A^{(i-1,j_{i-1})}$ and $R'_{i-1}=2^{j_{i-1}}$.
	For each $J\in\A^{(i-1)}$, since $R'_{i-1}\leq \left|\widetilde{\D_\xi}^{(1)}(J)\right|\leq 2R'_{i-1}$, we can take a subset $\widetilde{\D_\xi}^{(2,i-1)}(J)\subset\widetilde{\D_\xi}^{(1)}(J)$\footnote{Here, we abuse the notation $\widetilde{\D_\xi}^{(2,i-1)}(J)$, because we do not define $\widetilde{\D_\xi}^{(2,i-1)}\subset\widetilde{\D_\xi}^{(1)}$ yet. But we will see that $\widetilde{\D_\xi}^{(2,i-1)}(J)$ and $\widetilde{\D_\xi}^{(2,i-1)}$ which will be defined right below have the reasonable relation.} such that
	\begin{equation*}
		\left|\widetilde{\D_\xi}^{(2,i-1)}(J)\right|=R'_{i-1},\quad\sum_{I\in\widetilde{\D_\xi}^{(2,i-1)}(J)}\nu(I)^q\geq\frac{1}{2}\sum_{I\in\widetilde{\D_\xi}^{(1)}(J)}\nu(I)^q.
	\end{equation*}
	So, if we define
	\begin{equation*}
		\widetilde{\D_\xi}^{(2,i-1)}=\bigsqcup_{J\in\A^{(i-1)}}\widetilde{\D_\xi}^{(2,i-1)}(J)\subset\widetilde{\D_\xi}^{(1)},
	\end{equation*}
	then, by the above two inequalities, we have
	\begin{equation}\label{widetildeDxi_2_i-1_L^q_norm_proof_of_porosity}
		\sum_{I\in\widetilde{\D_\xi}^{(2,i-1)}}\nu(I)^q\geq\frac{1}{2}\sum_{I\in\bigsqcup_{J\in\A^{(i-1)}}\widetilde{\D_\xi}^{(1)}(J)}\nu(I)^q\geq (2D)^{-1}3^{-(q+i)}2^{-(\tau(q)+2\delta^{4/3})iD}.
	\end{equation}
	We end the process at the scale $s=i-1$ here.
	
	Next, let $s=i-2$. For each $j\in\{0,1,\dots,D-1\}$, we define $\A^{(i-2,j)}\subset\D_{u,(i-2)D}$ by
	\begin{align*}
		&\A^{(i-2,j)}=\left\{U\in\D_{u,(i-2)D}\left|\ 2^j\leq\left|\A^{(i-1)}(U)\right|<2^{j+1}\right.\right\}\quad(j<D-1),\\
		&\A^{(i-2,D-1)}=\left\{U\in\D_{u,(i-2)D}\left|\ 2^{D-1}\leq\left|\A^{(i-1)}(U)\right|\leq 2^D\right.\right\}\quad(j=D-1).
	\end{align*}
	Then, since $\bigsqcup_{j=0}^{D-1}\A^{(i-2,j)}=\left\{U\in\D_{u,(i-2)D}\left|\ J\subset U\text{ for some } J\in\A^{(i-1)}\right.\right\}$\\
	$=\left\{U\in\D_{u,(i-2)D}\left|\ I\subset U\text{ for some } I\in\widetilde{\D_\xi}^{(2,i-1)}\right.\right\}$, it follows from (\ref{widetildeDxi_2_i-1_L^q_norm_proof_of_porosity}) that
	\begin{equation*}
		\sum_{j=0}^{D-1}\sum_{J\in\bigsqcup_{U\in\A^{(i-2,j)}}\A^{(i-1)}(U)}\sum_{I\in\widetilde{\D_\xi}^{(2,i-1)}(J)}\nu(I)^q=\sum_{I\in\widetilde{\D_\xi}^{(2,i-1)}}\nu(I)^q\geq (2D)^{-1}3^{(q+i)}2^{-(\tau(q)+2\delta^{4/3})iD}.
	\end{equation*}
	Hence, by pigeonholing, there is $j_{i-2}\in\{0,1,\dots,D-1\}$ such that
	\begin{equation*}
		\sum_{J\in\bigsqcup_{U\in\A^{(i-2,j_{i-2})}}\A^{(i-1)}(U)}\sum_{I\in\widetilde{\D_\xi}^{(2,i-1)}(J)}\nu(I)^q\geq D^{-1}(2D)^{-1}3^{(q+i)}2^{-(\tau(q)+2\delta^{4/3})iD}.
	\end{equation*}
	We write $\A^{(i-2)}=\A^{(i-2,j_{i-2})}$ and $R'_{i-2}=2^{j_{i-2}}$. For each $U\in\A^{(i-2)}$, since $R'_{i-2}\leq \left|\A^{(i-1)}(U)\right|\leq 2R'_{i-2}$, we can take a subset $\B^{(i-1)}(U)\subset\A^{(i-1)}(U)$ such that
	\begin{equation*}
		\left|\B^{(i-1)}(U)\right|=R'_{i-2},\quad
		\sum_{J\in\B^{(i-1)}(U)}\sum_{I\in\widetilde{\D_\xi}^{(2,i-1)}(J)}\nu(I)^q\geq\frac{1}{2}
		\sum_{J\in\A^{(i-1)}(U)}\sum_{I\in\widetilde{\D_\xi}^{(2,i-1)}(J)}\nu(I)^q.
	\end{equation*}
	So, if we define
	\begin{equation*}
		\B^{(i-1)}=\bigsqcup_{U\in\A^{(i-2)}}\B^{(i-1)}(U)\subset\A^{(i-1)},
	\end{equation*}
	then, by the above two inequality, we have
	\begin{align*}
		\sum_{J\in\B^{(i-1)}}\sum_{I\in\widetilde{\D_\xi}^{(2,i-1)}(J)}\nu(I)^q&\geq\frac{1}{2}
		\sum_{J\in\bigsqcup_{U\in\A^{(i-2)}}\A^{(i-1)}(U)}\sum_{I\in\widetilde{\D_\xi}^{(2,i-1)}(J)}\nu(I)^q\\
		&\geq (2D)^{-2}3^{(q+i)}2^{-(\tau(q)+2\delta^{4/3})iD}.
	\end{align*}
	We write
	\begin{equation*}
		\widetilde{\D_\xi}^{(2,i-2)}=\bigsqcup_{J\in\B^{(i-1)}}\widetilde{\D_\xi}^{(2,i-1)}(J)\subset\widetilde{\D_\xi}^{(2,i-1)},
	\end{equation*}
	and end the process at the scale $s=i-2$.
	
	By continuing this process to the top scale $s=0$, for each scale $s$, we obtain $\A^{(s)}\subset\D_{u,sD}$, $\B^{(s+1)}\subset\A^{(s+1)}$ ($\A^{(i)}=\widetilde{\D_\xi}^{(1)},\B^{(i)}=\widetilde{\D_\xi}^{(2,i-1)}\subset\D_{iD}$ for $s=i-1$)
	, $R'_s\in\{1,2,\dots,2^{D-1}\}$ such that, for $J\in\A^{(s)}$, $\left|\B^{(s+1)}(J)\right|=R'_s$ and, if we take
	\begin{equation*}
		\widetilde{\D_\xi}^{(2,s)}=\bigsqcup_{J\in\B^{(s+1)}}\widetilde{\D_\xi}^{(2,s+1)}(J)\subset\widetilde{\D_\xi}^{(2,s+1)}\subset\widetilde{\D_\xi}^{(1)}
	\end{equation*}
	inductively, we have
	\begin{equation*}
		\sum_{I\in\widetilde{\D_\xi}^{(2,s)}}\nu(I)^q\geq (2D)^{s-i}3^{-(q+i)}2^{-(\tau(q)+2\delta^{4/3})iD}.
	\end{equation*}
	We define $\widetilde{\D_\xi}'=\widetilde{\D_\xi}^{(2,0)}$. Then, we can see that $\widetilde{\D_\xi}'$ is $(D,i,u,(R'_s)_{s\in[i]})$-uniform. Since $\widetilde{\D_\xi}'\subset\widetilde{\D_\xi}^{(1)}$, the second condition of Lemma \ref{wideltildeDix_into_uniform_tree_proof_of_porosity} is preserved. Furthermore, we have
	\begin{align*}
		\sum_{I\in\widetilde{\D_\xi}'}\nu(I)^q&\geq(2D)^{-i}3^{-(q+i)}2^{-(\tau(q)+2\delta^{4/3})iD}\\
		&=2^{-(\tau(q)+2\delta^{4/3}+\log(2D)/D+(q+i)/i\cdot\log 3/D)iD}.
	\end{align*}
	Since $D\gg_\delta 1$ is a sufficiently large constant and $i\geq\sigma/(30(q-1)\log M)\cdot l, l\gg_{M,q,\sigma}1$ are sufficiently large, we have
	\begin{equation}\label{widetildeDxi'_L^q_norm_proof_of_porosity}
		\sum_{I\in\widetilde{\D_\xi}'}\nu(I)^q\geq2^{-(\tau(q)+3\delta^{4/3})iD},
	\end{equation}
	and $\widetilde{\D_\xi}'$ satisfies the first condition of Lemma \ref{wideltildeDix_into_uniform_tree_proof_of_porosity}. So we complete the proof.
\end{proof}

By using Lemma \ref{wideltildeDix_into_uniform_tree_proof_of_porosity} and the idea from the proof of \cite[Theorem 5.1]{Shm19}, we prove the following Lemma \ref{key_2^{-iD}_components_are_blanching_proof_of_porosity}. This is the key lemma for the proof of Lemma \ref{L^q_norm_porosity}.
In the following, we take a small constant $0<\sigma'=\sigma'(q,\sigma)<\sigma$ which is determined only by $q$ and $\sigma$ and will be specified later ((\ref{contradiction_upper_bound_for_A_proof_of_discretized_slicing})). We can assume that $\delta\ll_{q,\sigma}1$ is sufficiently small in terms of $\sigma'$.

\begin{lem}\label{key_2^{-iD}_components_are_blanching_proof_of_porosity}
	In the setting above, assume that
	\begin{equation*}
		\tau^*(\alpha)>0\quad\text{for}\quad\alpha=\tau'(q).
	\end{equation*}
	Then, there exist a constant $0<\sigma''=\sigma''(M,\mu,q,\sigma,\sigma')<\sigma'$ determined only by $M,\mu,q,\sigma$ and $\sigma'$
	and a subset $\D_\xi'\subset\D_\xi$ such that
	\begin{equation*}
		|\D_\xi'|\geq 2^{\sigma''nD},
	\end{equation*}
	all elements of $\D_\xi'$ are contained in the same interval of length $\pi/5$, and any two of the elements of $\D_\xi'$ are $\pi2^{-\sigma'nD}$-separated.
\end{lem}

\begin{proof}
	We write
	\begin{equation*}
		\delta'=\sigma'\delta.
	\end{equation*}
	We prove that
	\begin{equation}\label{R's_are_much_blanching_proof_of_porosity}
		\left|\left\{s\in[i]\left|\ R'_s\leq 2^{\tau^*(\alpha)/3\cdot D}\right.\right\}\right|<\frac{\delta'}{10}\cdot i
	\end{equation}
	by contradiction. As remarked above, the idea of the proof is from the proof of \cite[Theorem 5.1]{Shm19}.
	
	We write
	\begin{equation*}
		S'=\left\{s\in[i]\left|\ R'_s\leq 2^{\tau^*(\alpha)/3\cdot D}\right.\right\}
	\end{equation*}
	and assume that
	\begin{equation}\label{assumption_R's_are_not_blanching_proof_of_porosity}
		\left|S'\right|\geq\frac{\delta'}{10}\cdot i.
	\end{equation}
	We first see Lemma \ref{local_L^q_norm_lemma}\footnote{We notice that Lemma \ref{local_L^q_norm_lemma} still holds if we replace the normal dyadic partitions with the dyadic partitions with the base $u$, uniformly in terms of $u$.}.
	By Lemma \ref{local_L^q_norm_lemma} (i) for $\tau^*(\alpha)/2>0$, there exists $\eta=\eta(\tau^*(\alpha)/2,\mu,q)=\eta(\mu,q)>0$ such that, for any $s'\in\N$, $I\in\D_{u,s'}$ and sufficiently large $D'\in\N, D'\gg_{\tau^*(\alpha)/2,\eta,\mu,q}1$ (that is, $D'\gg_{\mu,q}1$), if a subset $\D'\subset\D_{u,s'+D'}(I)$ satisfies $|\D'|\leq 2^{\tau^*(\alpha)/2\cdot D'}$, then we have
	\begin{equation}\label{local_L^q_norm_lemma_strictly_less_ver_proof_of_porosity}
		\sum_{J\in\D'}\nu(J)^q\leq 2^{-(\tau(q)+\eta)D'}\nu(2I)^q.
	\end{equation}
	Furthermore, by Lemma \ref{local_L^q_norm_lemma} (ii) for $\varepsilon$, if $D'\in\N, D'\gg_{\varepsilon,\mu,q}1$ is sufficiently large, then, for any $s'\in\N$ and $I\in\D_{u,s'}$, we have
	\begin{equation}\label{local_L^q_norm_lemma_without_condition_proof_of_porosity}
		\sum_{J\in\D_{u,s'+D'}(I)}\nu(J)^q\leq 2^{-(\tau(q)-\varepsilon)D'}\nu(2I)^q.
	\end{equation}
	
	For each $m'\in\Z, 0\leq m'\leq iD$, we define
	\begin{equation*}
		\D'_{m'}=\left\{I\in\D_{u,m'}\left|\ J\subset I \text{ for some }J\in\widetilde{\D_\xi}'\right.\right\},
	\end{equation*}
	and consider the sequence
	\begin{equation*}
		L_s=-\log\sum_{I\in\D_{sD}'}\nu(I)^q,\quad s=0,1,\dots,i.
	\end{equation*}
	We take $s\in[i]$. Since $\widetilde{\D_\xi}'$ is $(D,i,u,(R'_s)_{s\in[i]})$-uniform, for each $I\in\D_{sD}'$, we have $|\D_{(s+1)D}'(I)|=R'_s$. Hence, for each $I'\in\D'_{sD+2}$ with $I'\subset I$, we have $|\D'_{(s+1)D}(I')|\leq |\D'_{(s+1)D}(I)|=R'_s$. Here, assume that $R'_s\leq2^{\tau^*(\alpha)/3\cdot D}$. Since $D\gg_{\mu,q}1$ is a large constant, for each $I'\in\D'_{sD+2}$, we have
	\begin{equation}\label{each_interval_in_D'_sD+2_proof_of_porosity}
		|\D'_{(s+1)D}(I')|\leq R'_s\leq 2^{\tau^*(\alpha)/3\cdot D}\leq 2^{\tau^*(\alpha)/2\cdot (D-2)}.
	\end{equation}
	Then, we apply (\ref{local_L^q_norm_lemma_strictly_less_ver_proof_of_porosity}) for $s'=sD+2$,
	each $I'\in\D'_{sD+2}$, $D'=D-2$ and $\D'=\D'_{(s+1)D}(I')\subset\D_{u,(s+1)D}(I')$ (we notice that $D\gg_{\mu,q,\varepsilon}1$ is a large constant), and have
	\begin{equation}\label{D'_(s+1)D(I')_strictly_small_L^q_norm_proof_of_porosity}
		\sum_{J\in\D'_{(s+1)D}(I')}\nu(J)^q\leq 2^{-(\tau(q)+\eta)(D-2)}\nu(2I')^q.
	\end{equation}
	By taking the sum for $I'\in\D'_{sD+2}$, we have
	\begin{equation}\label{first_estimate_L_s+1_proof_of_porosity}
		L_{s+1}=-\log\sum_{I'\in\D'_{sD+2}}\sum_{J\in\D'_{(s+1)D}(I')}\nu(J)^q\geq(\tau(q)+\eta)(D-2)-\log\sum_{I'\in\D'_{sD+2}}\nu(2I')^q.
	\end{equation}
	We notice that each $I'\in\D'_{sD+2}$ is contained in some $I\in\D'_{sD}$ and, by the second condition of Lemma \ref{wideltildeDix_into_uniform_tree_proof_of_porosity}, in the four subintervals of length $\pi2^{-(sD+2)}$ of $I$, the only possibility for $I'$ is either of the two intervals in the center of $I$. In particular, $2I'\subset I$.
	From this, we can see that
	\begin{equation*}
		\sum_{I'\in\D'_{sD+2}}\nu(2I')^q\leq 2\sum_{I\in\D'_{sD}}\nu(I)^q.
	\end{equation*}
	From this and (\ref{first_estimate_L_s+1_proof_of_porosity}), we obtain that
	\begin{equation}\label{main_estimate_L_s+1_proof_of_porosity}
		L_{s+1}\geq L_s+(\tau(q)+\eta)(D-2)-1.
	\end{equation}
	If $R'_s\leq2^{\tau^*(\alpha)/3\cdot D}$ does not hold, we apply (\ref{local_L^q_norm_lemma_without_condition_proof_of_porosity}) instead of (\ref{local_L^q_norm_lemma_strictly_less_ver_proof_of_porosity}) in the above argument. Then, we obtain that
	\begin{equation}\label{estimate_L_s+1_without_condition_proof_of_porosity}
		L_{s+1}\geq L_s+(\tau(q)-\varepsilon)(D-2)-1.
	\end{equation}
	
	By telescoping, $L_0=0$ (by definition), (\ref{main_estimate_L_s+1_proof_of_porosity}) and (\ref{estimate_L_s+1_without_condition_proof_of_porosity}), we have
	\begin{align*}
		L_i&=\sum_{s=0}^{i-1}\left(L_{s+1}-L_s\right)+L_0\\
		&=\sum_{s\in S'}\left(L_{s+1}-L_s\right)+\sum_{\substack{s=0\\s\notin S'}}^{i-1}\left(L_{s+1}-L_s\right)\\
		&\geq|S'|(\tau(q)+\eta)(D-2)+(i-|S'|)(\tau(q)-\varepsilon)(D-2)-i\\
		&=\tau(q)i(D-2)+\eta|S'|(D-2)-\varepsilon(i-|S'|)(D-2)-i\\
		&\geq\tau(q)i(D-2)+\eta|S'|(D-2)-\varepsilon i(D-2)-i.
	\end{align*}
	From this and the assumption (\ref{assumption_R's_are_not_blanching_proof_of_porosity}), it follows that
	\begin{equation}\label{lower_bound_of_L_i_proof_of_porosity}
		L_i\geq\tau(q)(D-2)i+\frac{1}{10}\cdot\eta\delta'i(D-2)-\varepsilon i(D-2)-i.
	\end{equation}
	On the other hand, by the definition, we have $\D'_{iD}=\widetilde{\D_\xi}'$. Hence, by Lemma \ref{wideltildeDix_into_uniform_tree_proof_of_porosity}, we have
	\begin{equation}\label{upper_bound_of_L_i_proof_of_porosity}
		L_i=-\log\sum_{I\in\widetilde{\D_\xi}'}\nu(I)^q\leq (\tau(q)+3\delta^{4/3})iD.
	\end{equation}
	By (\ref{lower_bound_of_L_i_proof_of_porosity}) and (\ref{upper_bound_of_L_i_proof_of_porosity}), we have
	\begin{equation*}
		-\frac{2\tau(q)}{D}+\frac{1}{10}\cdot\eta\delta'\left(1-\frac{2}{D}\right)-\varepsilon\left(1-\frac{2}{D}\right)-\frac{1}{D} \leq 3\delta^{4/3}.
	\end{equation*}
	Since $D\gg_{\mu,q,\delta}1$ is a large constant and $\varepsilon\ll\delta$, from this and the definition of $\delta'$ we obtain that
	\begin{equation}\label{contradiction_inequality_delta^4/3_proof_of_porosity}
		\eta\delta'\leq 40\delta^{4/3}
		\iff\eta\sigma'\leq 40\delta^{1/3}.
	\end{equation}
	However, $\eta=\eta(\mu,q)>0$ is a constant determined only by $\mu$ and $q$ (see (\ref{local_L^q_norm_lemma_strictly_less_ver_proof_of_porosity})). So, by $\delta\ll_{\mu,q,\sigma'}1$, (\ref{contradiction_inequality_delta^4/3_proof_of_porosity}) is a contradiction. Therefore, we complete the proof of (\ref{R's_are_much_blanching_proof_of_porosity}).
	
	From (\ref{R's_are_much_blanching_proof_of_porosity}), we have
	$\left|\left\{\left.0\leq s\leq\delta'i/2\right| R'_s\leq 2^{\tau^*(\alpha)/3\cdot D}\right\}\right|
	\leq\left|\left\{\left.s\in[i]\right| R'_s\leq 2^{\tau^*(\alpha)/3\cdot D}\right\}\right|\leq\delta'i/10.
	$
	By this inequality, $\delta'=\sigma'\delta$, $i\geq \sigma/(30(q-1)\log M)\cdot l$ and $n=\lfloor\delta l\rfloor$, we have
	\begin{align*}\label{number_large_R'_s_proof_of_porosity}
		\left|\left\{\left. 0\leq s\leq\frac{\delta'i}{2}\right|\ R'_s>2^{\tau^*(\alpha)/3\cdot D}\right\}\right|
		&\geq \frac{\delta'i}{2}-\left|\left\{\left.0\leq s\leq \frac{\delta'i}{2}\right| R'_s\leq 2^{\tau^*(\alpha)/3\cdot D}\right\}\right|\nonumber\\
		&\geq\frac{\delta'i}{3}\nonumber\\
		&\geq\frac{\sigma\sigma'}{100(q-1)\log M}\cdot \delta l\nonumber\\
		&\geq\frac{\sigma\sigma'}{100(q-1)\log M}\cdot n.
	\end{align*}
	Since $(R'_s)_{s\in[i]}$ are the branching numbers of $\widetilde{\D_\xi}'$, this inequality tells us that there are at least $2^{\tau^*(\alpha)\sigma\sigma'/(300(q-1)\log M)\cdot nD}$ elements in $\D_{(\lfloor\delta'i/2\rfloor+1)D}'=\left\{I\in\D_{u,(\lfloor\delta'i/2\rfloor+1)D}\left|\ J\subset I\text{ for some }J\in\widetilde{\D_\xi}'\right.\right\}$.
	Hence, by pigeonholing, there are $8^{-1}2^{\tau^*(\alpha)\sigma\sigma'/(300(q-1)\log M)\cdot nD}$ elements of $\D_{(\lfloor\delta'i/2\rfloor+1)D}'$ contained in an interval $T$ of length $\pi/8$. By this, we have $16^{-1}2^{\tau^*(\alpha)\sigma\sigma'/(300(q-1)\log M)\cdot nD}$ elements of $\D_{(\lfloor\delta'i/2\rfloor+1)D}'$ contained in $T$ any two of which are $\pi 2^{-(\lfloor\delta'i/2\rfloor+1)D}$-separated.
	Therefore, we obtain $16^{-1}2^{\tau^*(\alpha)\sigma\sigma'/(300(q-1)\log M)\cdot nD}$ elements of $\widetilde{\D_\xi}'\subset\widetilde{\D_\xi}$ contained in $T$ any two of which are $\pi2^{-(\lfloor\delta'i/2\rfloor+1)D}$-separated. We write $\widetilde{\D_\xi}''$ for the set of such elements of $\widetilde{\D_\xi}$.
	
	By the definition of $\widetilde{\D_\xi}$, for each $J\in\widetilde{\D_\xi}''$, we can take $I\in\D_\xi$ so that $I$ is either $J$ or one of the two neighboring $\pi2^{-iD}$ interval of $J$. Let $I, I'\in\D_\xi$ be taken as above for distinct $J,J'\in\widetilde{\D_\xi}''$ respectively. Then, since $J$ and $J'$ are
	$\pi2^{-(\lfloor\delta'i/2\rfloor+1)D}$-separated, we have
	\begin{equation}\label{separation_of_D'xi_proof_of_porosity}
		d_{\RP^1}(I,I')\geq d_{\RP^1}(J,J')-2\cdot\pi 2^{-iD}\geq2^{-1}\cdot\pi 2^{-(\lfloor\delta'i/2\rfloor+1)D}\geq \pi 2^{-3\sigma'\delta iD/4}\geq \pi 2^{-\sigma'nD}
	\end{equation}
	(we notice that $\delta'=\sigma'\delta$, $n=\lfloor\delta l\rfloor$ and $l>i\geq \sigma/(30(q-1)\log M)\cdot l\gg_{\sigma',\delta}1$). Let $\D_\xi'$ be the set of $I\in\D_\xi$ taken as above for each $J\in\widetilde{\D_\xi}''$. Then, since each $J\in\widetilde{\D_\xi}''$ is contained in the interval $T$ of length $\pi/8$, all elements of $\D_\xi'$ are contained in the same interval of length $\pi/5$. By the above inequality, any two of the elements of $\D_\xi'$ are $\pi2^{-\sigma'nD}$ -separated. Furthermore, we have
	\begin{equation*}
		|\D_\xi'|=\left|\widetilde{\D_\xi}''\right|\geq \frac{1}{16}\cdot 2^{\tau^*(\alpha)\sigma\sigma'/(300(q-1)\log M)\cdot nD}\geq 2^{\tau^*(\alpha)\sigma\sigma'/(400(q-1)\log M)\cdot nD}
	\end{equation*}
	(we notice that $n=\lfloor\delta l\rfloor\gg_{M,\mu,q,\sigma,\sigma',\delta}1$). Hence, by putting $0<\sigma''=\sigma''(M,\mu,q,\sigma,\sigma')=\min\{\sigma'/2,\tau^*(\alpha)\sigma\sigma'/(400(q-1)\log M)\}<\sigma'$, we complete the proof.
\end{proof}

\subsection{The discretized slicing lemma}\label{subsection_discretized_slicing_lemma}

Under the assumption $\tau^*(\alpha)>0$ for $\alpha=\tau'(q)$, we prove Lemma \ref{L^q_norm_porosity}.
We write
\begin{equation*}
	A=\left|\left\{\zeta\in\D^G_{(i+n)D}(\xi)\left|\ \widetilde{\theta}_\xi(\zeta)>0\right.\right\}\right|.
\end{equation*}
We recall that, for $x\in\RP^1$, we have defined the map $f_x: G\rightarrow\RP^1$ by
\begin{equation*}
	f_x(g)=gx,\quad g\in G.
\end{equation*}
Here, we state the crucial lemma for the proof of Lemma \ref{L^q_norm_porosity}.

\begin{lem}[The discretized slicing lemma]\label{discretized_slicing}
	
	In the setting above, there exist a constant $\kappa=\kappa(M,\mu,q,\sigma)>0$, determined only by $M,\mu,q$ and $\sigma$, and $I_0\in\D_\xi'$ such that, for $x_0\in I_0\cap K$ and every $H\in\D_{(i+n)D+r}$, we have
	\begin{equation*}
		\left|\left\{\zeta\in\D^G_{(i+n)D}(\xi)\left|\ \widetilde{\theta}_\xi(\zeta\cap f_{x_0}^{-1}H)>0\right.\right\}\right|\leq 2^{-\kappa nD}A.
	\end{equation*}
\end{lem}

\begin{rem}
	In \cite{Bou10} is the well-known discretized projection theorem (\cite[Theorem 5]{Bou10}). In our case, the assumption is much stronger than that of \cite[Theorem 5]{Bou10}. Roughly speaking, our assumption corresponds to the case  in \cite[Theorem 5]{Bou10} that the non-concentration property (0.14) for $\mu_1$ is replaced with some much stronger condition depending on the “dimension” on $E$
	(on the  other hand, there is no non-concentration property like (0.16) for $E$ in our assumption).
\end{rem}

Before we prove Lemma \ref{discretized_slicing}, we see how Lemma \ref{L^q_norm_porosity} follows from Lemma \ref{discretized_slicing}.

\begin{proof}[Proof of Lemma \ref{L^q_norm_porosity}]
	We assume Lemma \ref{discretized_slicing} and take $I_0\in\D_\xi'\subset\D_\xi$ as in Lemma \ref{discretized_slicing} and $x_0\in I_0\cap K$. We take $\rho_\xi=\widetilde{\theta}_\xi$. 
	
	By the definition of $\D_\xi$, we have
	$\nu(I_0)>0$ and
	\begin{equation*}\label{I_0_in_D_xi_proof_of_porosity}
		2^{-(\tau(q)+\sqrt{\delta}/2)nD}\leq\|(\widetilde{\theta}_\xi{\bm .}\widehat{\nu_{I_0}})^{((i+n)D+r)}\|_q^q,
	\end{equation*}
	which is the condition (II) of Lemma \ref{L^q_norm_porosity} for $\rho_\xi=\widetilde{\theta}_\xi$.
	Next, we see the condition (I) of Lemma \ref{L^q_norm_porosity}. By Lemma \ref{discretized_slicing}, we have
	\begin{align}\label{every_slice_is_small_proof_of_porosity}
		&\sum_{H\in\D_{(i+n)D+r}}\left|\left\{\zeta\in\D^G_{(i+n)D}(\xi)\left|\ \widetilde{\theta}_\xi(\zeta\cap f_{x_0}^{-1}H)>0\right.\right\}\right|^q\nonumber\\
		\leq\ & 2^{-(q-1)\kappa nD}A^{q-1}
		\sum_{H\in\D_{(i+n)D+r}}\left|\left\{\zeta\in\D^G_{(i+n)D}(\xi)\left|\ \widetilde{\theta}_\xi(\zeta\cap f_{x_0}^{-1}H)>0\right.\right\}\right|\nonumber\\
		=\ &2^{-(q-1)\kappa nD}A^{q-1}\sum_{H\in\D_{(i+n)D+r}}\sum_{\zeta\in\D^G_{(i+n)D}(\xi),\ \widetilde{\theta}_\xi(\zeta)>0}\mathbbm{1}_{\{(H,\zeta)\left|\ \widetilde{\theta}_\xi(\zeta\cap f_{x_0}^{-1}H)>0\right.\}}(H,\zeta).
	\end{align}
	Here, for a given $\zeta\in\D^G_{(i+n)D}(\xi)$ with $\widetilde{\theta}_\xi(\zeta)>0$, we take $H,H'\in\D_{(i+n)D+r}$ satisfying $\widetilde{\theta}_\xi(\zeta\cap f_{x_0}^{-1}H),\widetilde{\theta}_\xi(\zeta\cap f_{x_0}^{-1}H')>0$. Then, for $g\in\supp\ \widetilde{\theta}_\xi\cap\zeta\cap f_{x_0}^{-1}H$ and $g'\in\supp\ \widetilde{\theta}_\xi\cap\zeta\cap f_{x_0}^{-1}H'$, we have $gx_0\in H$, $g'x_0\in H'$ and, by Corollary \ref{Lipschitz_continuity_of_the_action}, we have
	\begin{equation*}
		d_{\RP^1}(gx_0,g'x_0)\leq C_2C2^{-r}d_G(g,g')\leq MC_2C2^{-(i+n)D-r}.
	\end{equation*}
	This tells us that the number of $H\in\D_{(i+n)D+r}$ such that $\widetilde{\theta}_\xi(\zeta\cap f_{x_0}^{-1}H)>0$ is at most $O_{M,\mu,C}(1)$. Hence, by (\ref{every_slice_is_small_proof_of_porosity}), we have
	\begin{equation*}
		\sum_{H\in\D_{(i+n)D+r}}\left|\left\{\zeta\in\D^G_{(i+n)D}(\xi)\left|\ \widetilde{\theta}_\xi(\zeta\cap f_{x_0}^{-1}H)>0\right.\right\}\right|^q\leq O_{M,\mu,C}(1)2^{-(q-1)\kappa nD}A^q,
	\end{equation*}
	so
	\begin{equation}\label{L^q_distribution_of_slices_proof_of_porosity}
		\sum_{H\in\D_{(i+n)D+r}}\left(A^{-1}\left|\left\{\zeta\in\D^G_{(i+n)D}(\xi)\left|\ \widetilde{\theta}_\xi(\zeta\cap f_{x_0}^{-1}H)>0\right.\right\}\right|\right)^q\leq O_{M,\mu,C}(1)2^{-(q-1)\kappa nD}.
	\end{equation}
	
	We recall that, by Lemma \ref{tree_structure_supp_theta}, $\widetilde{\theta}$ has almost uniform mass on each $\D^G_m$-atom with positive mass and $\supp\ \widetilde{\theta}$ has a uniform tree structure. Hence, we can see that, for our $\xi\in\D^G_{iD}$ with $\widetilde{\theta}(\xi)>0$, there is $b>0$ such that $b/2\leq\widetilde{\theta}_\xi(\zeta)\leq b$ for any $\zeta\in\D^G_{(i+n)D}(\xi)$ with $\widetilde{\theta}_\xi(\zeta)>0$. It follows from this that
	\begin{equation*}
		\widetilde{\theta}_\xi(\zeta)\leq b\leq\frac{2}{A}\quad \text{ for any }\ \zeta\in\D^G_{(i+n)D}(\xi)\ \text{ with }\ \widetilde{\theta}_\xi(\zeta)>0.
	\end{equation*}
	By this and (\ref{L^q_distribution_of_slices_proof_of_porosity}), we obtain that
	\begin{align*}
		O_{M,\mu,C}(1)2^{-(q-1)\kappa nD}&\geq \sum_{H\in \D_{(i+n)D+r}}\left(\frac{1}{2}\sum_{\zeta\in\D^G_{(i+n)D}(\xi),\ \widetilde{\theta}_\xi(\zeta\cap f_{x_0}^{-1}H)>0}\widetilde{\theta}_\xi(\zeta)\right)^q\\
		&\geq\frac{1}{2^q}\sum_{H\in\D_{(i+n)D+r}}\widetilde{\theta}_\xi(f_{x_0}^{-1}H)^q,
	\end{align*}
	and hence,
	\begin{equation}\label{flat_L^q_norm_for_widetildetheta_xi_delta_x_0_proof_of_porosity}
		\sum_{H\in\D_{(i+n)D+r}}\widetilde{\theta}_\xi(f_{x_0}^{-1}H)^q\leq O_{M,\mu,q,C}(1)2^{-(q-1)\kappa nD}\leq 2^{-(q-1)\kappa/2\cdot nD}.
	\end{equation}
	Here, we have used that $\kappa>0$ depends only on $M,\mu,q,\sigma$ and $nD\geq n=\lfloor\delta l\rfloor\gg_{M,\mu,q,\sigma,\delta,C}1$ is sufficiently large.
	If we re-write $\kappa$ for $(q-1)\kappa/2>0$ (still depending only on $M,\mu,q,\sigma$), (\ref{flat_L^q_norm_for_widetildetheta_xi_delta_x_0_proof_of_porosity}) is the condition (I) of Lemma \ref{L^q_norm_porosity} for $\rho_\xi=\widetilde{\theta}_\xi$. Therefore, we complete the proof of Lemma \ref{L^q_norm_porosity}.
\end{proof}

In the following, we prove Lemma \ref{discretized_slicing}. For $x\in\RP^1$ and $H\in\D_{(I+n)D+r}$, we define
\begin{equation*}
	P_{x,H}=\left\{\zeta\in\D^G_{(i+n)D}(\xi)\left|\ \widetilde{\theta}_\xi(\zeta\cap f_x^{-1}H)>0\right.\right\}.
\end{equation*}
To prove Lemma \ref{discretized_slicing}, we need some preparations.

\begin{lem}\label{three_Ps_intersect_at_most_small_set_proof_of_discretized_slicing}
	Let $x_1, x_2, x_3\in K$ and assume that any two of them are $\pi2^{-\sigma'nD}$-separated and $\pi/4$ close. Then, for any $H_1, H_2, H_3\in\D_{(i+n)D+r}$, we have
	\begin{equation*}
		\left|P_{x_1,H_1}\cap P_{x_2,H_2}\cap P_{x_3,H_3}\right|\leq 2^{7\sigma'nD}.
	\end{equation*}
\end{lem}

\begin{proof}
	Let $\zeta,\zeta'\in P_{x_1,H_1}\cap P_{x_2,H_2}\cap P_{x_3,H_3}$. Then, for $i=1,2,3$, we can take $g_i\in\zeta\cap\supp\ \widetilde{\theta}_\xi\cap f_{x_i}^{-1}H_i$ and $g_i'\in\zeta'\cap\supp\ \widetilde{\theta}_\xi\cap f_{x_i}^{-1}H_i$. By Corollary \ref{Lipschitz_continuity_of_the_action}, we have for $i=2,3$ that
	\begin{equation*}
		d_{\RP^1}(g_1x_i,g_ix_i)\leq C_2C2^{-r}d_G(g_1,g_i)\leq MC_2C2^{-(i+n)D-r}.
	\end{equation*}
	Since $g_ix_i\in H_i$, we can see from the above inequality that $g_1x_i$ is in a $O_{M,\mu,C}(2^{-(i+n)D-r})$ neighborhood of $H_i$. The same statement holds for $g'_1x_i$. Hence, we have
	\begin{equation}\label{g1x_g'1x_close_proof_of_discretized_slicing_lemma}
		d_{(\RP^1)^3}(g_1x,g'_1x')\leq O_{M,\mu,C}(2^{-(i+n)D-r}).
	\end{equation}
	
	By the assumption, $x=(x_1,x_2,x_3)\in (\RP^1)^3$ is $\pi2^{-\sigma'nD}$-separated and any two of $x_1,x_2,x_3$ are $\pi/4$ close. Furthermore, since $g_1\in\supp\ \widetilde{\theta}_\xi$ and $x_i\in K$ for $i=1,2,3$, if we take $\eta_0=\eta_0(\mu)>0$ so that $d_{\RP^1}(K,\RP^1\setminus U_1)>\eta_0$ for $K\subset U_1$, we have $x_i\notin B_{\eta_0}(u^-_{g_1})$ for $i=1,2,3$. Let $c_0=c_0(\eta_0)=c_0(\mu)>0$ be the constant obtained from Proposition \ref{bi_Lipschitz_of_the_action_at_varepsilon_separated_points} for this $\eta_0=\eta_0(\mu)$. Since $i\geq \sigma/(30(q-1)\log M)\cdot l\gg_{M,\mu} n=\lfloor\delta l\rfloor>\sigma' n$, we can assume $\diam\ \xi\leq M2^{-iD}<c_0\cdot\pi 2^{-\sigma'nD}$.
	Then, we can apply Proposition \ref{bi_Lipschitz_of_the_action_at_varepsilon_separated_points} to $g_1,g'_1\in\xi$ and $x=(x_1,x_2,x_3)$ and obtain from (\ref{g1x_g'1x_close_proof_of_discretized_slicing_lemma}) that
	\begin{align*}\label{zeta_zeta'_are_close_proof_of_discretized_slicing}
		d_G(g_1,g'_1)\leq O_\mu(1)(\pi 2^{-\sigma'nD})^{-2}\|g_1\|^2d_{(\RP^1)^3}(g_1x,g_1'x)&\leq O_{M,\mu,C}(1)2^{2\sigma'nD}2^r2^{-(i+n)D-r}\nonumber\\
		&=O_{M,\mu,C}(1)2^{2\sigma'nD}2^{-(i+n)D}.
	\end{align*}
	This inequality tells us that all elements of $P_{x_1,H_1}\cap P_{x_2, H_2}\cap P_{x_3,H_3}$ are contained in the same ball in $G$ of radius $O_{M,\mu,C}(1)2^{2\sigma'nD}2^{-(i+n)D}$.
	By using this, the property of the dyadic-like partition $\D^G_{(i+n)D}$ and the fact that $m_G(B^G_r(g))=\Theta(r^3)$ for any $g\in G$ and sufficiently small $r>0$ (where $m_G$ is the Haar measure on $G$ with $m_G(B^G_1(1_G))=1$), we can see that
	\begin{equation*}
		\left|P_{x_1,H_1}\cap P_{x_2,H_2}\cap P_{x_3,H_3}\right|\leq O_{M,\mu,C}(1)(2^{2\sigma'nD})^3\leq2^{7\sigma'nD}
	\end{equation*}
	(we notice that $nD\geq n=\lfloor\delta l\rfloor\gg_{M,\mu,\sigma',\delta,C}1$ is sufficiently large). Hence, we complete the proof.
\end{proof}

From Lemma \ref{three_Ps_intersect_at_most_small_set_proof_of_discretized_slicing}, we lead the following lemma, which is the key for the proof of Lemma \ref{discretized_slicing}.

\begin{lem}\label{cardinality_of_union_of_Ps_proof_of_discretized_slicing}
	Let $k\in\N,k\geq 2$ and $x_1,\dots,x_k\in K$ such that any two of them are $\pi2^{-\sigma'nD}$-separated and $\pi/4$ close\footnote{Hence, it is meaningful only if $k\leq O(2^{\sigma'nD})$.}. Furthermore, let $H_1,\dots,H_k\in\D_{(i+n)D+r}$. Then, we have
	\begin{equation*}
		\left|P_{x_1,H_1}\cup\cdots \cup P_{x_k,H_k}\right|\geq\sum_{j=1}^k\left|P_{x_j,H_j}\right|-\left|\bigcup_{1\leq j<j'\leq k}P_{x_j,H_j}\cap P_{x_{j'},H_{j'}}\right|-\binom{k}{3}\cdot 2^{7\sigma'nD}.
	\end{equation*}
\end{lem}

\begin{proof}
	We show the following stronger statement by induction on $k\geq 2$: for any subsets $P'_1\subset P_{x_1,H_1},\dots, P'_k\subset P_{x_k,H_k}$, we have
	\begin{equation}\label{cardinality_of_union_of_P's_proof_of_discretized_slicing}
		\left|P'_1\cup\cdots \cup P'_k\right|\geq\sum_{j=1}^k\left|P'_j\right|-\left|\bigcup_{1\leq j<j'\leq k}P'_j\cap P'_{j'}\right|-\binom{k}{3}\cdot 2^{7\sigma'nD}.
	\end{equation}
	We can easily see that (\ref{cardinality_of_union_of_P's_proof_of_discretized_slicing}) holds for $k=2$ (more precisely, the equality holds). We assume that $k\geq 2$ and (\ref{cardinality_of_union_of_P's_proof_of_discretized_slicing}) hols for this $k$.
	
	Let
	$P'_1\subset P_{x_1,H_1},\dots, P'_k\subset P_{x_k,H_k}, P'_{k+1}\subset P_{x_{k+1},H_{k+1}}$ be subsets. Then, we have
	\begin{align}\label{card_P'1_P'k+1_proof_pf_discretized_slicing}
		\left|P'_1\cup\cdots\cup P'_k\cup P'_{k+1}\right|&=\left|P'_{k+1}\right|+\left|P'_1\cup\cdots\cup P'_k\setminus P'_{k+1}\right|\nonumber\\
		&=\left|P'_{k+1}\right|+\left|\left(P'_1\setminus P'_{k+1}\right)\cup\cdots\cup \left(P'_k\setminus P'_{k+1}\right)\right|.
	\end{align}
	Here, we apply the assumption of induction to $P'_1\setminus P'_{k+1}\subset P_{x_1,H_1},\dots, P'_k\setminus P'_{k+1}\subset P_{x_k,H_k}$. Then, we have
	\begin{align}\label{card_union_P'j_setminus_P'k+1_proof_of_discretized_slicing}
		&\left|\left(P'_1\setminus P'_{k+1}\right)\cup\cdots\cup \left(P'_k\setminus P'_{k+1}\right)\right|\nonumber\\
		\geq\ &
		\sum_{j=1}^k\left|P'_j\setminus P'_{k+1}\right|-\left|\bigcup_{1\leq j<j'\leq k}\left(P'_j\setminus P'_{k+1}\right)\cap\left(P'_{j'}\setminus P'_{k+1}\right)\right|-\binom{k}{3}\cdot2^{7\sigma'nD}\nonumber\\
		=\ &\sum_{j=1}^k\left(\left|P'_j\right|-\left|P'_j\cap P'_{k+1}\right|\right)-\left|\left(\bigcup_{1\leq j<j'\leq k}P'_j\cap P'_{j'}\right)\setminus P'_{k+1}\right|-\binom{k}{3}\cdot2^{7\sigma'nD}\nonumber\\
		=\ &\sum_{j=1}^k\left|P'_j\right|-\left\{\sum_{j=1}^k\left|P'_j\cap P'_{k+1}\right|+\left|\left(\bigcup_{1\leq j<j'\leq k}P'_j\cap P'_{j'}\right)\setminus P'_{k+1}\right|\right\}-\binom{k}{3}\cdot2^{7\sigma'nD}.
	\end{align}
	By using Lemma \ref{three_Ps_intersect_at_most_small_set_proof_of_discretized_slicing} repeatedly, we have
	\begin{align*}
		&\sum_{j=1}^k\left|P'_j\cap P'_{k+1}\right|\\
		=\ &\left|\left(P'_1\cup P'_2\right)\cap P'_{k+1}\right|+\left|\left(P'_1\cap P'_2\right)\cap P'_{k+1}\right|+\sum_{j=3}^k\left|P'_j\cap P'_{k+1}\right|\\
		\leq\ &\left|\left(P'_1\cup P'_2\right)\cap P'_{k+1}\right|+\left|P_{x_1,H_1}\cap P_{x_2,H_2}\cap P_{x_{k+1},H_{k+1}}\right|+\sum_{j=3}^k\left|P'_j\cap P'_{k+1}\right|\\
		\leq\ &\left|\left(P'_1\cup P'_2\right)\cap P'_{k+1}\right|+\sum_{j=3}^k\left|P'_j\cap P'_{k+1}\right|+2^{7\sigma'nD}\\
		=\ &\left|\left(P'_1\cup P'_2\cup P'_3\right)\cap P'_{k+1}\right|+\left|(P'_1\cup P'_2)\cap P'_{k+1}\cap P'_3\right|+\sum_{j=4}^k\left|P'_j\cap P'_{k+1}\right|+2^{7\sigma'nD}\\
		\leq\ &\left|\left(P'_1\cup P'_2\cup P'_3\right)\cap P'_{k+1}\right|+\left|P_{x_1,H_1}\cap P_{x_{k+1},H_{k+1}}\cap P_{x_3,H_3}\right|+\left|P_{x_2,H_2}\cap P_{x_{k+1},H_{k+1}}\cap P_{x_3,H_3}\right|\\
		&+\sum_{j=4}^k\left|P'_j\cap P'_{k+1}\right|+2^{7\sigma'nD}\\
		\leq\ &\left|\left(P'_1\cup P'_2\cup P'_3\right)\cap P'_{k+1}\right|+\sum_{j=4}^k\left|P'_j\cap P'_{k+1}\right|+(1+2)\cdot2^{7\sigma'nD}\\
		&\vdots\\
		\leq\ &\left|\left(P'_1\cup\cdots\cup P'_k\right)\cap P'_{k+1}\right|+\sum_{j=1}^{k-1}j\cdot 2^{7\sigma'nD}\\
		=\ &\left|\left(P'_1\cup\cdots\cup P'_k\right)\cap P'_{k+1}\right|+\binom{k}{2}\cdot 2^{7\sigma'nD}.
	\end{align*}
	From this inequality, (\ref{card_P'1_P'k+1_proof_pf_discretized_slicing}) and (\ref{card_union_P'j_setminus_P'k+1_proof_of_discretized_slicing}), it follows that
	\begin{align*}
		&\left|P'_1\cup\cdots\cup P'_k\cup P'_{k+1}\right|\\
		=\ &\left|P'_{k+1}\right|+\left|\left(P'_1\setminus P'_{k+1}\right)\cup\cdots\cup \left(P'_k\setminus P'_{k+1}\right)\right|\\
		\geq\ &\sum_{j=1}^{k+1}\left|P'_j\right|-\left\{\sum_{j=1}^k\left|P'_j\cap P'_{k+1}\right|+\left|\left(\bigcup_{1\leq j<j'\leq k}P'_j\cap P'_{j'}\right)\setminus P'_{k+1}\right|\right\}-\binom{k}{3}\cdot2^{7\sigma'nD}\\
		\geq\ &\sum_{j=1}^{k+1}\left|P'_j\right|-\left\{\left|\left(P'_1\cup\cdots\cup P'_k\right)\cap P'_{k+1}\right|
		+\left|\left(\bigcup_{1\leq j<j'\leq k}P'_j\cap P'_{j'}\right)\setminus P'_{k+1}\right|\right\}-\left(\binom{k}{2}+\binom{k}{3}\right)\cdot2^{7\sigma'nD}\\
		=\ &\sum_{j=1}^{k+1}\left|P'_j\right|-\left|\left(\left(P'_1\cup\cdots\cup P'_k\right)\cap P'_{k+1}\right)\sqcup\left(\left(\bigcup_{1\leq j<j'\leq k}P'_j\cap P'_{j'}\right)\setminus P'_{k+1}\right)\right|-\binom{k+1}{3}\cdot 2^{7\sigma'nD}\\
		\geq\ &\sum_{j=1}^{k+1}\left|P'_j\right|-\left|\bigcup_{1\leq j<j'\leq k+1}P'_j\cap P'_{j'}\right|-\binom{k+1}{3}\cdot 2^{7\sigma'nD}
	\end{align*}
	Hence, (\ref{cardinality_of_union_of_P's_proof_of_discretized_slicing}) holds for $k+1$, and we complete the induction. Therefore, we have shown that (\ref{cardinality_of_union_of_P's_proof_of_discretized_slicing}) holds for any $k$.
\end{proof}

Now, we prove Lemma \ref{discretized_slicing} using Lemma \ref{three_Ps_intersect_at_most_small_set_proof_of_discretized_slicing} and Lemma \ref{cardinality_of_union_of_Ps_proof_of_discretized_slicing}.

\begin{proof}[Proof of Lemma \ref{discretized_slicing}]
	First, we notice the following. By Lemma \ref{Lemma_many_components_which_increase_L^q_norm_under_comvolutions}, we have $\|\widetilde{\theta}_\xi^{((i+n)D)}\|_q^q=\sum_{\zeta\in\D^G_{(i+n)D}(\xi)}\widetilde{\theta}_\xi(\zeta)^q\leq 2^{-\sigma/2\cdot nD}$. From this and Hölder's inequality, we obtain that
	\begin{align*}
		1=\sum_{\zeta\in\D^G_{(i+n)D}(\xi),\ \widetilde{\theta}_\xi(\zeta)>0}\widetilde{\theta}_\xi(\zeta)
		&\leq\left(\sum_{\zeta\in\D^G_{(i+n)D}(\xi),\ \widetilde{\theta}_\xi(\zeta)>0}\widetilde{\theta}_\xi(\zeta)^q\right)^{1/q}\left(\sum_{\zeta\in\D^G_{(i+n)D}(\xi),\ \widetilde{\theta}_\xi(\zeta)>0}1^{q/(q-1)}\right)^{(q-1)/q}\\
		&\leq 2^{-\sigma/(2q)\cdot nD}A^{(q-1)/q},
	\end{align*}
	and hence
	\begin{equation}\label{A_has_large_dim_proof_of_discretized_slicing}
		A\geq2^{\sigma/(2(q-1))\cdot nD}.
	\end{equation}
	
	We prove Lemma \ref{discretized_slicing} for $\kappa=\kappa(M,\mu,q,\sigma)=\sigma''/2>0$ (we notice that $\sigma''=\sigma''(M,\mu,q,\sigma,\sigma')$ and $\sigma'$ will be determined only by $q$ and $\sigma$). We will do this by contradiction. Assume that, for each $I\in\D'_\xi$, there are $x_I\in I\cap K$ and $H_I\in\D_{(i+n)D+r}$, we have
	\begin{equation}\label{contradiction_assumption_proof_of_discretixed_slicing}
		\left|P_{x_I,H_I}\right|=\left|\left\{\zeta\in\D^G_{(i+n)D}(\xi)\left|\ \widetilde{\theta}_\xi(\zeta\cap f_{x_I}^{-1}H_I)>0\right.\right\}\right|>2^{-\sigma''/2\cdot nD}A.
	\end{equation}
	By Lemma \ref{key_2^{-iD}_components_are_blanching_proof_of_porosity}, we have $|\D_\xi'|\geq 2^{\sigma''nD}$. Hence, if we write $k=4\lceil2^{\sigma''/2\cdot nD}\rceil<2^{\sigma''nD}$ (we notice that $nD\geq n=\lfloor\delta l\rfloor\gg_{M,\mu,q,\sigma,\delta}$ is large), we can take distinct $k$ elements $I_1,\dots,I_k$ from $\D'_\xi$. For each $j=1,\dots,k$ we write $x_j=x_{I_j}$ and $H_j=H_{I_j}$.
	
	Since, by Lemma \ref{key_2^{-iD}_components_are_blanching_proof_of_porosity}, any two of the elements of $\D'_\xi$ are $\pi2^{-\sigma'nD}$-separated and $\pi/5$ close and $x_j\in I_j\cap K$, any two of $x_1,\dots,x_k\in K$ are $\pi2^{-\sigma'nD}$-separated and $\pi/4$ close.
	Here, we estimate $\left|P_{x_1,H_1}\cup\cdots\cup P_{x_k,H_k}\right|$ from below.
	
	First, we assume that
	\begin{equation}\label{case_union_of_intersections_of_Ps_is_large_proof_of_discretized_slicing}
		\left|\bigcup_{1\leq j<j'\leq k-1}P_{x_j,H_j}\cap P_{x_{j'},H_{j'}}\right|\geq\frac{k-2}{2}\cdot 2^{-\sigma''/2\cdot nD}A.
	\end{equation}
	We have
	\begin{align}\label{union_of_intersections_of_Ps_is_large_proof_of_discretized_slicing}
		&\left|P_{x_1,H_1}\cup\cdots\cup P_{x_k,H_k}\right|\nonumber\\
		\geq\ &\left|P_{x_k,H_k}\cup\left(\bigcup_{1\leq j<j'\leq k-1}P_{x_j,H_j}\cap P_{x_{j'},H_{j'}}\right)\right|\nonumber\\
		=\ &\left|P_{x_k,H_k}\right|+\left|\bigcup_{1\leq j<j'\leq k-1}P_{x_j,H_j}\cap P_{x_{j'},H_{j'}}\right|
		-\left|P_{x_k,H_k}\cap\left(\bigcup_{1\leq j<j'\leq k-1}P_{x_j,H_j}\cap P_{x_{j'},H_{j'}}\right)\right|\nonumber\\
		\geq\ &\left|P_{x_k,H_k}\right|+\left|\bigcup_{1\leq j<j'\leq k-1}P_{x_j,H_j}\cap P_{x_{j'},H_{j'}}\right|-\sum_{1\leq j<j'\leq k-1}\left|P_{x_k,H_k}\cap P_{x_j,H_j}\cap P_{x_{j'},H_{j'}}\right|
	\end{align}
	By Lemma \ref{three_Ps_intersect_at_most_small_set_proof_of_discretized_slicing}, wa have for $1\leq j<j'\leq k-1$ that
	\begin{equation*}
		\left|P_{x_k,H_k}\cap P_{x_j,H_j}\cap P_{x_{j'},H_{j'}}\right|\leq 2^{7\sigma'nD}.
	\end{equation*}
	From this, (\ref{contradiction_assumption_proof_of_discretixed_slicing}), (\ref{case_union_of_intersections_of_Ps_is_large_proof_of_discretized_slicing}) and (\ref{union_of_intersections_of_Ps_is_large_proof_of_discretized_slicing}), we obtain that
	\begin{align*}\label{conclusion_union_intersections_of_Ps_is_large_proof_of_discretized_slicing}
		\left|P_{x_1,H_1}\cup\cdots\cup P_{x_k,H_k}\right|&\geq 2^{-\sigma''/2\cdot nD}A+\frac{k-2}{2}\cdot 2^{-\sigma''/2\cdot nD}A-\binom{k-1}{2}\cdot 2^{7\sigma'nD}\nonumber\\
		&=\frac{k}{2}\cdot 2^{-\sigma''/2\cdot nD}A-\binom{k-1}{2}\cdot 2^{7\sigma'nD}.
	\end{align*}
	If
	\begin{equation*}\label{case_union_intersections_of_Ps_is_small_proof_of_discretized_slicing}
		\left|\bigcup_{1\leq j<j'\leq k-1}P_{x_j,H_j}\cap P_{x_{j'},H_{j'}}\right|<\frac{k-2}{2}\cdot 2^{-\sigma''/2\cdot nD}A,
	\end{equation*}
	by Lemma \ref{cardinality_of_union_of_Ps_proof_of_discretized_slicing} and (\ref{contradiction_assumption_proof_of_discretixed_slicing}), we have
	\begin{align*}\label{conclusion_union_intersections_of_Ps_is_small_proof_of_discretized_slicing}
		\left|P_{x_1,H_1}\cup\cdots\cup P_{x_k,H_k}\right|&\geq\left|P_{x_1,H_1}\cup\cdots\cup P_{x_{k-1},H_{k-1}}\right|\nonumber\\
		&\geq\sum_{j=1}^{k-1}\left|P_{x_j,H_j}\right|-\left|\bigcup_{1\leq j<j'\leq k-1}P_{x_j,H_j}\cap P_{x_{j'},H_{j'}}\right|-\binom{k-1}{3}\cdot 2^{7\sigma'nD}\nonumber\\
		&\geq (k-1)2^{-\sigma''/2\cdot nD}A-\frac{k-2}{2}\cdot 2^{-\sigma''/2\cdot nD}A-\binom{k-1}{3}\cdot 2^{7\sigma'nD}\nonumber\\
		&\geq \frac{k}{2}\cdot 2^{-\sigma''/2\cdot nD}A-\binom{k-1}{3}\cdot 2^{7\sigma'nD}.
	\end{align*}
	
	Hence, in either of the two cases, we have
	\begin{equation*}
		\left|P_{x_1,H_1}\cup\cdots\cup P_{x_k,H_k}\right|\geq\frac{k}{2}\cdot 2^{-\sigma''/2\cdot nD}A-k^32^{7\sigma'nD}.
	\end{equation*}
	Since $\left\{\zeta\in\D^G_{(i+n)D}(\xi)\left|\ \widetilde{\theta}_\xi(\zeta)>0\right.\right\}\supset P_{x_1,H_1}\cup\cdots\cup P_{x_k,H_k}$ and $k=4\lceil2^{\sigma''/2\cdot nD}\rceil$, it follows that
	\begin{equation*}
		A\geq\left|P_{x_1,H_1}\cup\cdots\cup P_{x_k,H_k}\right|\geq 2A-100\cdot2^{(3\sigma''/2+7\sigma')nD},
	\end{equation*}
	and hence
	\begin{equation}\label{contradiction_upper_bound_for_A_proof_of_discretized_slicing}
		A\leq 100\cdot2^{(3\sigma''/2+7\sigma')nD}.
	\end{equation}
	However, we have by (\ref{A_has_large_dim_proof_of_discretized_slicing}) that $2^{\sigma/(2(q-1))\cdot nD}\leq A$. Hence, if we take $\sigma'=\sigma'(q,\sigma)$ sufficiently small in terms of $q$ and $\sigma$, (\ref{contradiction_upper_bound_for_A_proof_of_discretized_slicing}) is a contradiction (we notice that $\sigma''<\sigma'$). So, we complete the proof of Lemma \ref{discretized_slicing}.
\end{proof}

\section{Proof of the $L^q$ norm flattening theorem}\label{section_proof_of_L^q_norm_flattening_theorem}

In this section, we prove the $L^q$ norm flattening Theorem \ref{L^q_norm_flattening_theorem} using the $L^q$ norm porosity Lemma \ref{L^q_norm_porosity}. As stated in Section \ref{section_L^q_norm_porosity}, the strategy for the proof is by contradiction, doing linearization by Lemma \ref{L^q_norm_porosity} and then applying the similar argument as the proof of Theorem \ref{L^q_norm_flattening_theorem_for_self_similar_measure} in \cite{Shm19}.

\subsection{Linear approximation of the $G$-action on $\RP^1$}\label{subsection_linear_approximation}

We first show the local linear approximation of the $G$-action on $\RP^1$ which is fundamental to the linearization argument. All of the contents in this section are from \cite[Section 5.4]{HS17}, but we give the proofs for completeness.
We recall that, for $u\in\RP^1$, $\RP^1$ is identified with $\R/\pi\Z$ by $\angle_u:\RP^1\stackrel{\sim}{\longrightarrow}\R/\pi\Z$: the map associating each $x\in\RP^1$ with the angle mod $\pi$ from $u$ to $x$ in the counterclockwise orientation. In the following argument, for $g\in G$, we write $\widehat{g}:\R/\pi\Z\to\R/\pi\Z$ for
\begin{equation*}
\widehat{g}=\angle_{v_g^+}\circ g\circ \angle_{u_g^+}^{-1}.
\end{equation*}

\begin{lem}[{\cite[Lemma 5.8]{HS17}}]\label{approximation_near_identity_proof_of_flattening}
Let $0<r<1$ be sufficiently small and $x_0\in\RP^1$. Then, for $h\in B^G_r(1_G)$ and $x\in B_r(x_0)$, we have
\begin{equation*}
\angle_v(hx)-\angle_v(hx_0)=\angle_u(x)-\angle_u(x_0)+O(r^2),
\end{equation*}
where $u,v\in\RP^1$ are arbitrary and $O(r^2)$ is independent of $u,v$ in particular.
\end{lem}

We notice that, since $\angle_u=\angle_{[1:0]}-\angle_{[1:0]}(u)$ for $u\in\RP^1$, $\angle_v(hx)-\angle_v(hx_0)$ and $\angle_u(x)-\angle_u(x)\in\R/\pi\Z$ are independent of choices of $u,v\in\RP^1$.

\begin{proof}
It is sufficient to show the statement for $u=u^+_h$ and $v=v^+_h$. By Taylor's theorem, we have
\begin{align}\label{Taylor_expansion_for_h_proof_of_flattening}
\angle_{v^+_h}(hx)-\angle_{v^+_h}(hx_0)&=\widehat{h}\left(\angle_{u^+_h}(x)\right)-\widehat{h}\left(\angle_{u^+_h}(x_0)\right)\nonumber\\
&=\widehat{h}'\left(\angle_{u^+_h}(x_0)\right)\left(\angle_{u^+_h}(x)-\angle_{u^+_h}(x_0)\right)+\frac{\widehat{h}''(\theta)}{2}\left(\angle_{u^+_h}(x)-\angle_{u^+_h}(x_0)\right)^2
\end{align}
for some $\theta\in\R/\pi\Z$. We notice that, in the right-most side, we locally identify $\R/\pi\Z$ with $\R$ around $0$.

Here, by (\ref{derivative_action_of_G_preliminaries}), we have
\begin{equation*}
\widehat{h}'(\theta)=\frac{1}{\|h\|^2\cos^2\theta+\|h\|^{-2}\sin^2\theta},\quad\widehat{h}''(\theta)=\frac{2\cos\theta\sin\theta(\|h\|^2-\|h\|^{-2})}{\left(\|h\|^2\cos^2\theta+\|h\|^{-2}\sin^2\theta\right)^2},
\end{equation*}
and hence $\|h\|^{-2}\leq\widehat{h}'(\theta)\leq\|h\|^2$ and $|\widehat{h}''(\theta)|\leq O(1)$ for any $h\in B^G_r(1_G)\ (0<r<1)$ and $\theta\in\R/\pi\Z$. Furthermore, by the bi-Lipschitz equivalence of the norm metric and $d_G$ on a small neighborhood of $1_G$ on $G$ (see \cite[Section 2.3]{HS17}), we have $\|h\|\leq 1+\|h-1_G\|\leq1+O(r)$, and hence
\begin{equation*}
\widehat{h}'(\theta)=1+O(r)
\end{equation*}
for any $\theta\in\R/\pi\Z$. From (\ref{Taylor_expansion_for_h_proof_of_flattening}), these facts and $x\in B_r(x_0)$, we have
\begin{align*}
\angle_{v^+_h}(hx)-\angle_{v^+_h}(hx_0)&=\left(1+O(r)\right)\left(\angle_{u^+_h}(x)-\angle_{u^+_h}(x_0)\right)+O(1)\left(\angle_{u^+_h}(x)-\angle_{u^+_h}(x_0)\right)^2\\
&=\angle_{u^+_h}(x)-\angle_{u^+_h}(x_0)+O(r^2),
\end{align*}
and complete the proof.
\end{proof}

\begin{lem}[{\cite[Corollary 5.9]{HS17}}]\label{approximation_near_1G_and_x0_proof_of_flattening}
Let $0<r<1$ be sufficiently small and $x_0\in\RP^1$. Then, for $h\in B^G_r(1_G)$ and $x\in B_r(x_0)$, we have
\begin{equation*}
d_{\RP^1}(hx,x_0)=O(r).
\end{equation*}
\end{lem}

\begin{proof}
We have
\begin{align}\label{intermediate_anglex_proof_of_flattening}
\angle_{v_h^+}(hx)-\angle_{v_h^+}(x_0)&=\angle_{v_h^+}(hx)-\angle_{v_h^+}(x)+\angle_{v_h^+}(x)-\angle_{v_h^+}(x_0)\nonumber\\
&=\angle_{v_h^+}(hx)-\angle_{v_h^+}(x)+O(r).
\end{align}
By the mean value theorem, we have\footnote{Here, the product $\widehat{h}'(\theta)\angle_{u_h^+}(x)$ of the scalar $\widehat{h}'(\theta)\in\R$ and the element $\angle_{u_h^+}(x)\in\R/\pi\Z$ is defined as stated just below Proposition \ref{linear_approximation_of_G_action_proof_flattening}.}
\begin{equation*}
\angle_{v_h^+}(hx)=\widehat{h}\left(\angle_{u_h^+}(x)\right)=\widehat{h}\left(\angle_{u_h^+}(x)\right)-\widehat{h}(0)=\widehat{h}'(\theta)\angle_{u_h^+}(x)
\end{equation*}
for some $\theta\in\R/\pi\Z$. As we have seen in the proof of Lemma \ref{approximation_near_identity_proof_of_flattening}, it holds that $\widehat{h}'(\theta)=1+O(r)$, and hence
\begin{equation}\label{h_close_identity_up_to_r_proof_of_flattening}
\angle_{v_h^+}(hx)=\angle_{u_h^+}(x)+O(r).
\end{equation}

By $h\in B_r(1_G)$, $\|h\|=1+O(r)$ and the definition of $u_h^+,v_h^+\in\R^2$, we have
$(1+O(r))v_h^+=\|h\|v_h^+=hu_h^+=u_h^++(h-1_G)u_h^+$, and hence $\|v_h^+-u_h^+\|\leq O(r)$. Therefore, as elements of $\RP^1$, we have $d_{\RP^1}(v_h^+,u_h^+)\leq O(r)$ and
\begin{equation}\label{uh+_vh+_are_close_proof_of_flattening}
\angle_{u_h^+}(x)=\angle_{v_h^+}(x)+O(r).
\end{equation}
By (\ref{intermediate_anglex_proof_of_flattening}), (\ref{h_close_identity_up_to_r_proof_of_flattening}) and (\ref{uh+_vh+_are_close_proof_of_flattening}), we obtain that
\begin{equation*}
\angle_{v_h^+}(hx)-\angle_{v_h^+}(x_0)=O(r)
\end{equation*}
and complete the proof.
\end{proof}

By using Lemmas \ref{approximation_near_identity_proof_of_flattening} and \ref{approximation_near_1G_and_x0_proof_of_flattening}, we show the following local linear approximation of the $G$-action.

\begin{prop}[{\cite[Proposition 5.10]{HS17}}]\label{linear_approximation_of_G_action_proof_flattening}
Let $0<\eta<1$ and $0<r\ll_\eta 1$, and let $g_0\in G$ and $x_0\in\RP^1$. Assume that $x_0\notin B_\eta(u_{g_0}^-)$. Then, for any $g\in B^G_r(g_0)$ and $x\in B_r(x_0)$, we have
\begin{equation*}
\angle_v(gx)=\angle_v(gx_0)+\widehat{g_0}'\left(\angle_{u_{g_0}^+}(x_0)\right)\left(\angle_u(x)-\angle_u(x_0)\right)+O_\eta\left(\frac{r^2}{\|g_0\|^2}\right),
\end{equation*}
where $u,v\in\RP^1$ are arbitrary and $O_\eta(r^2/\|g_0\|^2)$ is independent of $u,v$ in particular.
\end{prop}

In the above, the product $\widehat{g_0}'(\angle_{u_{g_0}^+}(x_0))(\angle_u(x)-\angle_u(x_0))$ of the scalar $\widehat{g_0}'(\angle_{u_{g_0}^+}(x_0))\in\R$ and the element $\angle_u(x)-\angle_u(x_0)\in\R/\pi\Z$ close to $0$ is defined as follows. We write $\Pi:\R\to\R/\pi\Z$ for the canonical map, that is, $\Pi(x)=x\mod\pi$ for $x\in\R$. Then, for $a\in\R$ and $\overline{x}\in\R/\pi\Z$, we define $a\overline{x}=\Pi(ax)$, where $x\in\R$ is the unique element of $\R$ such that $-\pi/2<x\leq \pi/2$ and $\Pi(x)=\overline{x}$.

\begin{proof}
We take $g_0\in G$.
By the same reason as above, it is sufficient to show the statement for $u=u^+_{g_0}$ and $v=v^+_{g_0}$. We first notice that, for $0<\eta<1$ sufficiently small and $x_0,y\in\RP^1$ such that $x_0\notin B_\eta(u_{g_0}^-)$ and $d_{\RP^1}(y,x_0)<\eta/2$, we have by Taylor's theorem that
\begin{equation*}
\widehat{g_0}\left(\angle_{u_{g_0}^+}(y)\right)-\widehat{g_0}\left(\angle_{u_{g_0}^+}(x_0)\right)=\widehat{g_0}'\left(\angle_{u_{g_0}^+}(x_0)\right)\left(\angle_{u_{g_0}^+}(y)-\angle_{u_{g_0}^+}(x_0)\right)+\frac{\widehat{g_0}''(\theta)}{2}\left(\angle_{u_{g_0}^+}(y)-\angle_{u_{g_0}^+}(x_0)\right)^2
\end{equation*}
for some $\theta\in\R/\pi\Z$ such that $|\theta|\leq \pi/2-\eta/2$. Since
\begin{equation*}
\widehat{g_0}''(\theta)=\frac{2\cos\theta\sin\theta(\|g_0\|^2-\|g_0\|^{-2})}{\left(\|g_0\|^2\cos^2\theta+\|g_0\|^{-2}\sin^2\theta\right)^2}\leq O_\eta\left(\frac{1}{\|g_0\|^2}\right),
\end{equation*}
we obtain
\begin{equation}\label{second_order_approximation_of_g_proof_of_flattening}
\widehat{g_0}\left(\angle_{u_{g_0}^+}(y)\right)-\widehat{g_0}\left(\angle_{u_{g_0}^+}(x_0)\right)=\widehat{g_0}'\left(\angle_{u_{g_0}^+}(x_0)\right)\left(\angle_{u_{g_0}^+}(y)-\angle_{u_{g_0}^+}(x_0)\right)+O_\eta\left(\frac{d_{\RP^1}(y,x_0)^2}{\|g_0\|^2}\right).
\end{equation}

Let $0<r<1$ be sufficienlty small in terms of $\eta$, $g\in B^G_r(g_0)$, $x_0\in\RP^1$ such that $x_0\notin B_\eta(u_{g_0}^-)$ and $x\in B_r(x_0)$.
We write $h=g_0^{-1}g\in B^G_r(1_G)$. Then, from $0<r\ll_\eta 1$, we can see that $hx_0,hx\in B_{\eta/2}(x_0)$. By applying (\ref{second_order_approximation_of_g_proof_of_flattening}) to $y=hx,hx_0$, we have
\begin{align*}
&\angle_{v_{g_0}^+}(gx)-\angle_{v_{g_0}^+}(gx_0)\\
=\ &\angle_{v_{g_0}^+}(g_0hx)-\angle_{v_{g_0}^+}(g_0hx_0)\\
=\ & \widehat{g_0}\left(\angle_{u_{g_0}^+}(hx)\right)-\widehat{g_0}\left(\angle_{u_{g_0}^+}(hx_0)\right)\\
=\ &\left(\widehat{g_0}\left(\angle_{u_{g_0}^+}(hx)\right)-\widehat{g_0}\left(\angle_{u_{g_0}^+}(x_0)\right)\right)-\left(\widehat{g_0}\left(\angle_{u_{g_0}^+}(hx_0)\right)-\widehat{g_0}\left(\angle_{u_{g_0}^+}(x_0)\right)\right)\\
=\ & \widehat{g_0}'\left(\angle_{u_{g_0}^+}(x_0)\right)\left(\angle_{u_{g_0}^+}(hx)-\angle_{u_{g_0}^+}(x_0)\right)-\widehat{g_0}'\left(\angle_{u_{g_0}^+}(x_0)\right)\left(\angle_{u_{g_0}^+}(hx_0)-\angle_{u_{g_0}^+}(x_0)\right)\\
&+O_\eta\left(\frac{d_{\RP^1}(hx,x_0)^2+d_{\RP^1}(hx_0,x_0)^2}{\|g_0\|^2}\right)\\
=\ & \widehat{g_0}'\left(\angle_{u_{g_0}^+}(x_0)\right)\left(\angle_{u_{g_0}^+}(hx)-\angle_{u_{g_0}^+}(hx_0)\right)+O_\eta\left(\frac{d_{\RP^1}(hx,x_0)^2+d_{\RP^1}(hx_0,x_0)^2}{\|g_0\|^2}\right),
\end{align*}
where, in the last equation, we have used $|\angle_{u_{g_0}^+}(hx)-\angle_{u_{g_0}^+}(x_0)|,|\angle_{u_{g_0}^+}(hx_0)-\angle_{u_{g_0}^+}(x_0)|<\eta/2$ and $|\angle_{u_{g_0}^+}(hx)-\angle_{u_{g_0}^+}(hx_0)|<\eta\ll 1$.
By Lemma \ref{approximation_near_1G_and_x0_proof_of_flattening}, we have
\begin{equation*}
\angle_{u_{g_0}^+}(hx)-\angle_{u_{g_0}^+}(hx_0)=\angle_{u_{g_0}^+}(x)-\angle_{u_{g_0}}(x_0)+O(r^2).
\end{equation*}
By (\ref{derivative_action_of_G_preliminaries}) and $x_0\notin B_{\eta}(u_{g_0}^-)$, we also have
\begin{equation*}
\widehat{g_0}'\left(\angle_{u_{g_0}^+}(x_0)\right)=O_{\eta}\left(\frac{1}{\|g_0\|^2}\right).
\end{equation*}
Furthermore, by Lemma \ref{approximation_near_1G_and_x0_proof_of_flattening}, we have
\begin{equation*}
d_{\RP^1}(hx,x_0)^2+d_{\RP^1}(hx_0,x_0)^2=O(r^2).
\end{equation*}
By combining these equations, we obtain that
\begin{equation*}
\angle_{v_{g_0}^+}(gx)-\angle_{v_{g_0}^+}(gx_0)=\widehat{g_0}'\left(\angle_{u_{g_0}^+}(x_0)\right)\left(\angle_{u_{g_0}^+}(x)-\angle_{u_{g_0}}(x_0)\right)+O_\eta\left(\frac{r^2}{\|g_0\|^2}\right)
\end{equation*}
and complete the proof.
\end{proof}

\subsection{The inverse theorem for $L^q$ norms of linear convolutions}\label{subsection_inverse_theorem}

We present the inverse theorem for $L^q$ norms of linear convolutions. As we noticed in Section \ref{section_L^q_norm_porosity}, this theorem plays an essential role in the proof of Theorem \ref{L^q_dim_of_self_similar_measures} (more precisely, Theorem \ref{L^q_norm_flattening_theorem_for_self_similar_measure}) in \cite{Shm19}, and also in our proof of Theorem \ref{L^q_norm_flattening_theorem}

In general, for two probability measures $\nu$ and $\rho$ on $\R/\pi\Z$, the $L^q$ norm of the linear convolution $\rho*\nu$ (almost) decreases from that of $\nu$, that is, $\|(\rho*\nu)^{(m)}\|_q\leq O(1)\|\nu^{(m)}\|_q$ holds. However, if $\nu$ is the uniform measure on $\R/\pi\Z$ or $\rho$ is a one-point mass, then we have $\|(\nu*\rho)^{(m)}\|_q=\Theta(1)\|\nu^{(m)}\|_q$.
The inverse theorem describes “the inverse” of this implication.
Roughly speaking, the inverse theorem says that, if the $L^q$ norm of $\nu$ does not drop so much under the linear convolution with $\rho$, then we can take two sets $A$ and $B$ which capture large enough parts of $\nu$ and $\rho$ respectively and have a nice “tree structure” such that either $A$ is almost fully branching or $B$ is not branching at all at each scale.

To state the theorem, we prepare some notions.
For $m\in\N$, a $2^{-m}$ measure is a (discrete) probability measure on $\pi2^{-m}\Z/\pi\Z\subset\R/\pi\Z$. For a $2^{-m}$ measure $\nu$ and $q>1$, we write 
\begin{equation*}
\|\nu\|_q=\left(\sum_{x\in\pi2^{-m}\Z/\pi\Z}\nu(x)^q\right)^{1/q}
\end{equation*}
for the $L^q$ norm of $\nu$.

As we did in Section \ref{subsection_sufficiently_separated_nice_intervals} (on $\RP^1$), for $u\in\R/\pi\Z$ and $m\in\N$, we define the $2^{-m}$ dyadic partition of $\R/\pi\Z$ with the base $u$ by
\begin{equation*}
\D_{u,m}=\left\{[\pi k2^{-m}+u,\pi (k+1)2^{-m}+u)\right\}_{k=0}^{2^m-1}
\end{equation*}
and $\D_m=\D_{0,m}$. We notice that, if $u\in\pi 2^{-m}\Z/\pi\Z$, then $\D_{u,m}=\D_m$. For a subset $E\subset\R/\pi\Z$, we write $\D_{u,m}(E)=\left\{I\in\D_{u,m}\left|\ I\cap E\neq\emptyset\right.\right\}$. For $x\in\R/\pi\Z$, we write $\D_{u,m}(x)$ for the unique element of $\D_{u,m}(\{x\})$, that is, the unique element $I$ of $\D_{u,m}$ such that $x\in I$.
Furthermore, for $I\in\D_{u,m}$, we write $1/2\cdot I$ for the $1/2$ contraction of the interval $I$ with the same center.

Let $d\in\N$, $l\in\N$, $u\in\pi2^{-ld}\Z/\pi\Z$ and $(R_s)_{s\in[l]}$ be a sequence of length $l$ such that $R_s\in\{1,2,\dots,2^d\}\ (s\in [l])$. For a subset $A\subset\pi2^{-ld}\Z/\pi\Z$, we say that $A$ is $(d,l,u,(R_s)_{s\in [l]})$-uniform if, for any $s\in [l]$ and $I\in\D_{u,sd}(A)$, we have
\begin{equation*}
	\left|\D_{u,(s+1)d}(A\cap I)\right|=R_s.
\end{equation*}

Here, we state the inverse theorem for $L^q$ norms of linear convolution.

\begin{thm}[The inverse theorem for $L^q$ norms of linear convolution, {\cite[Theorem 2.1]{Shm19}}]\label{the_inverse_theorem_for_L^q_norms}
Let $q>1$, $d_0\in\N$ and $\delta_0>0$. Then, there exist $d\in\N$ with $d>d_0,d\gg_{q,\delta_0}1$ and $0<\delta\ll_{q,\delta_0,d}1$ such that the following holds for sufficiently large $l\in\N, l\gg_{q,\delta_0,d,\delta}1$. Let $m=ld$ and $\nu,\rho$ be $2^{-m}$ measures on $\R/\pi\Z$, and assume that
\begin{equation*}
\|\nu*\rho\|_q\geq 2^{-\delta m}\|\nu\|_q.
\end{equation*}
Then, there exist subsets $A,B\subset \pi2^{-m}\Z/\pi\Z$, $u,v\in\pi2^{-m}\Z/\pi\Z$ and sequences $(R'_s)_{s\in[l]},(R''_s)_{s\in[l]}\in\{1,2,\dots,2^d\}^l$ such that
\begin{itemize}
\item[(A-i)] $\|\nu|_A\|_q\geq 2^{-\delta_0m}\|\nu\|_q$,
\item[(A-ii)] $\nu(y)\leq2\nu(x)$ for any $x,y\in A$,
\item[(A-iii)] $A$ is $(d,l,u,(R'_s)_{s\in [l]})$-uniform,
\item[(A-iv)] for any $x\in A$ and $s\in[l]$, $x\in 1/2\cdot\D_{u,sd}(x)$,
\item[(B-i)] $\rho(B)\geq 2^{-\delta_0m}$,
\item[(B-ii)] $\rho(y)\leq2\rho(x)$ for any $x,y\in B$,
\item[(B-iii)] $B$ is $(d,l,v,(R''_s)_{s\in[l]})$-uniform,
\item[(B-iv)] for any $x\in B$ and $s\in[l]$, $x\in 1/2\cdot\D_{v,sd}(x)$,
\item[(v)] for each $s\in[l]$, either
\begin{equation*}
R''_s=1
\end{equation*}
or
\begin{equation*}
R'_s\geq 2^{(1-\delta_0)d},
\end{equation*}
\item[(vi)] if we write $S=\left\{s\in[l]\left|\ R'_s\geq 2^{(1-\delta_0)d}\right.\right\}$, we have
\begin{equation*}
\log(\|\rho\|_q^{-q/(q-1)})-\delta_0m\leq d|S|\leq \log(\|\nu\|_q^{-q/(q-1)})+\delta_0 m.
\end{equation*}
\end{itemize}
\end{thm}

For more details and the proof, see \cite[Section 2 and 3]{Shm19}. We emphasize that the inverse Theorem \ref{the_inverse_theorem_for_L^q_norms} is shown by using sophisticated results in additive combinatorics, such as the result of Bourgain from \cite{Bou10} and the (asymmetric) Balog-Szemerédi-Gowers theorem.

\subsection{Linearization}\label{subsection_linearization_for_proof_of_L^q_norm_flattening_theorem}

In the whole Sections \ref{subsection_linearization_for_proof_of_L^q_norm_flattening_theorem} and \ref{subsection_Shmerkin_argument}, we prove the $L^q$ norm flattening Theorem \ref{L^q_norm_flattening_theorem}.
Let $\A$ be a non-empty finite family of elements of $G$ which is uniformly hyperbolic and $K$ be its attractor on $\RP^1$. We take a probability measure $\mu$ on $G$ such that $\supp\ \mu=\A$, and write $\nu$ for the stationary measure of $\mu$ on $\RP^1$. For $\A$, we take an open subset $U_1\subsetneq \RP^1$ as in (\ref{domain_U1_the_IFS_acts_on}) (and (\ref{domains_the_Möbius_IFS_acts_on_preliminaries}), equivalently). We write $\tau(q)\ (q>1)$ for the $L^q$ spectrum of $\nu$. Let $q>1$ and assume that $\tau(q)$ is differentiable at $q>1$,
\begin{equation}\label{assumption_tau_smaller_than_q-1_proof_of_flattening}
\tau(q)<q-1
\end{equation}
and, for $\alpha=\tau'(q)$,
\begin{equation}\label{assumption_Legendre_transform_positive_proof_of_flattening}
\tau^*(\alpha)>0.
\end{equation}
Let $C,L>1$ be constants. We take arbitrary small $\sigma>0$. Then, we take a constant $\varepsilon=\varepsilon(M,\mu,q,\sigma)>0$ determined only by $M,\mu,q$ and $\sigma$ and specified later. Actually, we will take a sufficiently small constant $0<\delta=\delta(M,\mu,q,\sigma)<\sigma$ which will be determined only by $M,\mu,q$ and $\sigma$ and specified through the following argument, and take $\varepsilon$ sufficiently small in terms of $M,\mu,q,\sigma$ and $\delta$. To see how small we take $\delta$, see (\ref{how_to_take_delta_L^q_norm_porosity_applicable_proof_of_flattening}), (\ref{delta_is_so_small_for_inverse_theorem_proof_of_flattening}), (\ref{how_to_take_delta0_local_tau*_lemma_applicable_proof_of_flattening}), (\ref{delta0'_sufficiently_small_Mmuqsigma_1_proof_of_flattening}),
(\ref{delta0'_sufficiently_small_Mmuqsigma_2_proof_of_flattening}), (\ref{local_L^q_lemma_for_kappa'_proof_of_flettening}), (\ref{local_L^q_norm_lemma_for_delta0_proof_of_flattening}), (\ref{1/d0_smaller_than_delta0_proof_of_flattening}), (\ref{delta0_much_smaller_than_betagamma_proof_of_flattening}).

We prove Theorem \ref{L^q_norm_flattening_theorem} by contradiction. Let $m\in\N$ be a  sufficiently large number so that $m\gg_{M,\mu,q,\sigma,\varepsilon,C,L}1$. We take a Borel probability measure $\theta$ on $G$ and $r\in\N$ such that $\diam\ \supp\ \theta\leq L$, $C^{-1}2^r\leq\|g\|^2\leq C2^r$ and $u_g^-\notin U_1$ for every $g\in\supp\ \theta$ and
\begin{equation}\label{theta_has_flat_L^q_norm_proof_of_flattening}
\|\theta^{(m)}\|_q^q\leq 2^{-\sigma m}.
\end{equation}
Until the end of Section \ref{subsection_Shmerkin_argument}, we assume that
\begin{equation}\label{contradiction_assumption_proof_of_flattening}
\|(\theta{\bm .}\nu)^{(m+r)}\|_q^q\geq 2^{-(\tau(q)+\varepsilon)m},
\end{equation}
and will derive the contradiction.

We first notice that, by the assumption (\ref{assumption_Legendre_transform_positive_proof_of_flattening}), we can apply the $L^q$ norm porosity Lemma \ref{L^q_norm_porosity} to our $\A$, $\mu$ and $q$.
We take $0<\delta\ll_{M,\mu,q,\sigma}1$ so small that
\begin{equation}\label{how_to_take_delta_L^q_norm_porosity_applicable_proof_of_flattening}
\text{we can apply Lemma \ref{L^q_norm_porosity} to $\sigma/2$ and $\delta$ in place of $\sigma$ and $\delta$, respectively}
\end{equation}
(we will make $\delta$ much smaller only in terms of $M,\mu,q,\sigma$ later). Then, we take a constant $\varepsilon=\varepsilon(M,\mu,q,\sigma/2,\delta)>0$ so small that we can apply Lemma \ref{L^q_norm_porosity} to $\sigma/2$, $\delta$ and $2\varepsilon$ in place of $\sigma,\delta$ and $\varepsilon$, respectively. For these $\sigma/2,\delta,2\varepsilon$ and the constant $C>1$ taken above, we obtain the large constant $D=D(M,\mu,q,\sigma/2,\delta,2\varepsilon,C)\in\N$ from Lemma \ref{L^q_norm_porosity}.
We also obtain a small constant $\kappa=\kappa(M,\mu,q,\sigma/2)>0$ which is determined only by $M,\mu,q,\sigma/2$ from Lemma \ref{L^q_norm_porosity}.

Here, we can assume $m\gg_{M,\mu,q,\sigma/2,\kappa,\delta,2\varepsilon,D,C,L}1$. Let $l=\lceil m/D\rceil$. Then, we have $l\gg_{M,\mu,q,\sigma/2,\kappa,\delta,2\varepsilon,D,C,L}1$.
We also have $lD-D\leq m\leq lD$ and $1-1/l\leq m/(lD)\leq 1$. From these inequalities and (\ref{theta_has_flat_L^q_norm_proof_of_flattening}), we obtain
\begin{equation}\label{flat_L^q_norm_at_lD_scale_proof_of_flattening}
\|\theta^{(lD)}\|_q^q\leq\|\theta^{(m)}\|_q^q\leq 2^{-\sigma m}=2^{-\sigma m/(lD)\cdot lD}\leq 2^{-\sigma/2\cdot lD}.
\end{equation}
Furthermore, by Lemma \ref{L^q_norms_of_two_partitions}, we obtain
\begin{equation*}
\|(\theta{\bm .}\nu)^{(m+r)}\|_q^q\leq 2^{(lD-m)q}\|(\theta{\bm .}\nu)^{(lD+r)}\|_q^q\leq 2^{Dq}\|(\theta{\bm .}\nu)^{(lD+r)}\|_q^q,
\end{equation*}
and hence, by (\ref{contradiction_assumption_proof_of_flattening}), $m\gg_{q,\varepsilon,D}1$ and $lD\gg_{\mu,q,\varepsilon}1$,
\begin{equation}\label{non_decreasing_L^q_norm_under_convolution_at_lD_scale_proof_of_flattening}
\|(\theta{\bm .}\nu)^{(lD+r)}\|_q^q\geq 2^{-Dq}2^{-(\tau(q)+\varepsilon)m}\geq 2^{-(\tau(q)+3\varepsilon/2)m}\geq 2^{-(\tau(q)+3\varepsilon/2)lD}\geq 2^{-2\varepsilon\cdot lD}\|\nu^{(lD)}\|_q^q.
\end{equation}
From $l\gg_{M,\mu,q,\sigma/2,\kappa,\delta,2\varepsilon,D,C,L}1$, (\ref{flat_L^q_norm_at_lD_scale_proof_of_flattening}) and (\ref{non_decreasing_L^q_norm_under_convolution_at_lD_scale_proof_of_flattening}), we can apply Lemma \ref{L^q_norm_porosity} to $\sigma/2, \delta, 2\varepsilon, l$ and $\theta$.
Hence, we obtain the following.
We write $n=\lfloor \delta l\rfloor$. Then, there exist $i\in\N$ with $n<i\leq l-n$, $\xi\in \D^G_{iD}$ with $\theta(\xi)>0$, $I_0\in\D_{iD}$ with $\nu(I_0)>0$, a Borel probability measure $\rho_\xi$ on $\xi$ with $\rho_\xi\ll\theta_\xi$ and $x_0\in I_0\cap K$ such that
\begin{equation}\label{result_of_L^q_norm_porosity_fraction_of_theta_proof_of_flattening}
\|(f_{x_0}\rho_\xi)^{((i+n)D+r)}\|_q^q\leq 2^{-\kappa nD}
\end{equation}
and
\begin{equation}\label{result_of_L^q_norm_porosity_convolution_proof_of_flattening}
2^{-(\tau(q)+\sqrt{\delta}/2)nD}\leq\|(\rho_\xi{\bm .}\widehat{\nu_{I_0}})^{((i+n)D+r)}\|_q^q.
\end{equation}

In the following, we write $\angle=\angle_{[1:0]}:\RP^1\stackrel{\sim}{\longrightarrow}\R/\pi\Z$.
We take $g_0\in \supp\ \rho_\xi$ and fix it. Since $\supp\ \rho_\xi\subset \supp\ \theta$, we have $u_{g_0}^-\notin U_1$, and hence $d_{\RP^1}(x_0,u_{g_0}^-)>\Omega_\mu(1)$ for $x_0\in K$.
We also have $iD>nD>\delta l\gg_{M,\mu}1$. Hence, by Proposition \ref{linear_approximation_of_G_action_proof_flattening}, we have for any $g\in\supp\ \rho_\xi\subset \overline{B}^G_{M2^{-iD}}(g_0)$\footnote{Here, $\overline{B}^G_r(g_0)=\left\{g\in G\left|\ d_G(g,g_0)\leq r\right.\right\}$ is a closed ball in G.} and $x\in I_0\cap K\subset B_{\pi2^{-iD}}(x_0)$ that
\begin{align*}
\angle(gx)&=\angle(gx_0)+\widehat{g_0}'\left(\angle_{u_{g_0}^+}(x_0)\right)(\angle(x)-\angle(x_0))+O_{M,\mu}\left(\frac{2^{-2iD}}{\|g_0\|^2}\right)\nonumber\\
&=\angle(gx_0)+g_0'(x_0)(\angle(x)-\angle(x_0))+O_{M,\mu,C}\left(2^{-2iD-r}\right),
\end{align*}
where we have used that $g_0\in\supp\ \theta$ and $\|g_0\|^2\geq C^{-1}2^r$ and wrote $g_0'(x_0)=\widehat{g_0}'\left(\angle_{u_{g_0}^+}(x_0)\right)$.
If we define $\widetilde{f}:G\times \RP^1\to\R/\pi\Z$ by $\widetilde{f}(g,x)=\angle(gx_0)+g_0'(x_0)(\angle(x)-\angle(x_0))$, the above equation and $n<i$ say that
\begin{equation}\label{linear_approximation_on_small_component_of_theta_nu_proof_of_flattening}
\left|\angle(f(g,x))-\widetilde{f}(g,x)\right|\leq O_{M,\mu,C}\left(2^{-2iD-r}\right)\leq O_{M,\mu,C}\left(2^{-(i+n)D-r}\right)
\end{equation}
for any $g\in\supp\ \rho_\xi$ and $x\in I_0\cap K$. Here, it is clear that
\begin{equation}\label{definition_of_G_convolution_proof_of_flattening}
f(\rho_\xi\times\widehat{\nu_{I_0}})=\rho_\xi{\bm .}\widehat{\nu_{I_0}}.
\end{equation}
For $a\in\R$, we write $S_a:\R/\pi\Z\to\R/\pi\Z$ for the multiplication on $\R/\pi\Z$ by $a$, that is,
\begin{equation*}
S_a\overline{x}=a\overline{x}=\Pi\left(ax\right),\quad \overline{x}\in\R/\pi\Z,
\end{equation*}
where $x\in\R$ is a unique real number such that $\Pi(x)=\overline{x}$ and $-\pi/2<x\leq\pi/2$.
For $u\in\R/\pi\Z$, we also write $T_u:\R/\pi\Z\to\R/\pi\Z$ for the translation on $\R/\pi\Z$ by $u$: $T_u(x)=x+u\ (x\in\R/\pi\Z)$.
Then, we have
\begin{equation}\label{definition_of_linear_convolution_proof_of_flattening}
\widetilde{f}(\rho_\xi\times\widehat{\nu_{I_0}})=\angle(f_{x_0}\rho_\xi)*S_{g_0'(x_0)}\left(T_{-\angle(x_0)}\angle(\widehat{\nu_{I_0}})\right).
\end{equation}
From (\ref{linear_approximation_on_small_component_of_theta_nu_proof_of_flattening}), (\ref{definition_of_G_convolution_proof_of_flattening}), (\ref{definition_of_linear_convolution_proof_of_flattening}) and Lemma \ref{L^q_norms_of_two_partitions}, we can see that\footnote{Of course, we define the $2^{-m}$ $L^q$ norm of a finite Borel measure $\nu$ on $\R/\pi\Z$ by $\|\nu^{(m)}\|_q=\left(\sum_{I\in\D_m}\nu(I)^q\right)^{1/q}$.}
\begin{equation*}
\|(\rho_\xi{\bm .}\widehat{\nu_{I_0}})^{((i+n)D+r)}\|_q^q\leq O_{M,\mu,q,C}(1)
\left\|\left(\angle(f_{x_0}\rho_\xi)*S_{g_0'(x_0)}\left(T_{-\angle(x_0)}\angle(\widehat{\nu_{I_0}})\right)\right)^{((i+n)D+r)}\right\|_q^q.
\end{equation*}
From this estimate and (\ref{result_of_L^q_norm_porosity_convolution_proof_of_flattening}), we obtain that
\begin{align}\label{linear_convolution_on_small_components_L^q_norm_proof_of_porosity}
\left\|\left(\angle(f_{x_0}\rho_\xi)*S_{g_0'(x_0)}\left(T_{-\angle(x_0)}\angle(\widehat{\nu_{I_0}})\right)\right)^{((i+n)D+r)}\right\|_q^q&\geq\Omega_{M,\mu,q,C}(1)2^{-(\tau(q)+\sqrt{\delta}/2)nD}\nonumber\\
&\geq2^{-(\tau(q)+3\sqrt{\delta}/5)nD}
\end{align}
(we notice that $nD>\delta l\gg_{M,\mu,q,\delta,C}1$).

We have $\supp\ T_{-\angle(x_0)}\angle(\widehat{\nu_{I_0}})\subset\overline{B}_{\pi2^{-iD}}(0)$\footnote{We represent open or closed balls in $\R/\pi\Z$ by the same notations as those in $\RP^1$.}.
By $x_0\in K$, $g_0\in\supp\ \theta$, Corollary \ref{Lipschitz_continuity_of_the_action} and (\ref{derivative_action_of_G_preliminaries}), we also have that $\supp\ \angle(f_{x_0}\rho_\xi)\subset B_{O_{M,\mu,C}(2^{-(iD+r)})}(\angle(g_0x_0))$ (we have used that $iD>n\geq\delta l\gg_{M,\mu}1$) and
\begin{equation}\label{contraction_ratio_of_g0_around_x0_proof_of_flattening}
g_0'(x_0)=\Theta_\mu\left(\frac{1}{\|g_0\|^2}\right)=\Theta_{\mu,C}(2^{-r}).
\end{equation}
We can see from these facts that, if we define $S_a^{\angle(g_0x_0)}=T_{\angle(g_0x_0)}\circ S_a\circ T_{-\angle(g_0x_0)}:\R/\pi\Z\to\R/\pi\Z$ for $a>0$, that is, the expansion or contraction by $a$ centering on $\angle(g_0x_0)$, we have
\begin{equation*}
y+S_{g_0'(x_0)}z=S_{g_0'(x_0)}^{\angle(g_0x_0)}\left(S_{g_0'(x_0)^{-1}}^{\angle(g_0x_0)}y+z\right)
\end{equation*}
for any $y\in B_{O_{M,\mu,C}(2^{-(iD+r)})}(\angle(g_0x_0))$ and $z\in\overline{B}_{\pi2^{-iD}}(0)$ (we are using that $iD>n\geq \delta l\gg_{M,\mu,C}1$).
Hence, we have
\begin{equation*}
\angle(f_{x_0}\rho_\xi)*S_{g_0'(x_0)}\left(T_{-\angle(x_0)}\angle(\widehat{\nu_{I_0}})\right)
=S_{g_0'(x_0)}^{\angle(g_0x_0)}\left(S_{g_0'(x_0)^{-1}}^{\angle(g_0x_0)}\angle(f_{x_0}\rho_\xi)*T_{-\angle(x_0)}\angle(\widehat{\nu_{I_0}})\right)
\end{equation*}
and
\begin{equation*}
\supp\ S_{g_0'(x_0)^{-1}}^{\angle(g_0x_0)}\angle(f_{x_0}\rho_\xi)*T_{-\angle(x_0)}\angle(\widehat{\nu_{I_0}})\subset \overline{B}_{O_{M,\mu,C}(2^{-iD})}(\angle(g_0x_0)).
\end{equation*}
From these facts, (\ref{contraction_ratio_of_g0_around_x0_proof_of_flattening}) and Lemma \ref{L^q_norms_of_two_partitions}, we obtain that
\begin{align*}
&\left\|\left(\angle(f_{x_0}\rho_\xi)*S_{g_0'(x_0)}\left(T_{-\angle(x_0)}\angle(\widehat{\nu_{I_0}})\right)\right)^{((i+n)D+r)}\right\|_q^q\\
=\ & \left\|\left(S_{g_0'(x_0)}^{\angle(g_0x_0)}\left(S_{g_0'(x_0)^{-1}}^{\angle(g_0x_0)}\angle(f_{x_0}\rho_\xi)*T_{-\angle(x_0)}\angle(\widehat{\nu_{I_0}})\right)\right)^{((i+n)D+r)}\right\|_q^q\\
\leq\ &O_{\mu,q,C}(1)\left\|\left(S_{g_0'(x_0)^{-1}}^{\angle(g_0x_0)}\angle(f_{x_0}\rho_\xi)*T_{-\angle(x_0)}\angle(\widehat{\nu_{I_0}})\right)^{((i+n)D)}\right\|_q^q.
\end{align*}
Similarly, we obtain that
\begin{equation*}
\left\|\left(S_{g_0'(x_0)^{-1}}^{\angle(g_0x_0)}\angle(f_{x_0}\rho_\xi)\right)^{((i+n)D)}\right\|_q^q\leq O_{\mu,q,C}(1)\|(f_{x_0}\rho_\xi)^{((i+n)D+r)}\|_q^q.
\end{equation*}
In conclusion, if we define a Borel probability measure $\eta$ on $\R/\pi\Z$ by
\begin{equation*}
\eta=S_{g_0'(x_0)^{-1}}^{\angle(g_0x_0)}\angle(f_{x_0}\rho_\xi),
\end{equation*}
then we have
\begin{equation}\label{supp_of_eta_proof_of_flattening}
\supp\ \eta\subset \overline{B}_{O_{M,\mu,C}(2^{-iD})}(\angle(g_0x_0))
\end{equation}
and, by the above estimates, (\ref{result_of_L^q_norm_porosity_fraction_of_theta_proof_of_flattening}) and (\ref{linear_convolution_on_small_components_L^q_norm_proof_of_porosity}),
\begin{equation}\label{eta_L^q_norm_proof_of_flattening}
\|\eta^{((i+n)D)}\|_q^q\leq O_{\mu,q,C}(1)2^{-\kappa nD}\leq 2^{-2\kappa/3\cdot nD}
\end{equation}
(we notice that $nD>\delta l\gg_{\mu,q,\kappa,C}1$) and
\begin{equation*}
\left\|\left(\eta*T_{-\angle(x_0)}\angle(\widehat{\nu_{I_0}})\right)^{((i+n)D)}\right\|_q^q\geq\Omega_{\mu,q,C}(1)2^{-(\tau(q)+3\sqrt{\delta}/5)nD}.
\end{equation*}
Since $\eta*T_{-\angle(x_0)}\angle(\widehat{\nu_{I_0}})=T_{-\angle(x_0)}\left(\eta*\angle(\widehat{\nu_{I_0}})\right)$, the last inequality and Lemma \ref{L^q_norms_of_two_partitions} imply that
\begin{equation}\label{convolution_with_eta_L^q_norm_proof_of_flattening}
\|(\eta*\angle(\widehat{\nu_{I_0}}))^{((i+n)D)}\|_q^q\geq \Omega_{\mu,q,C}(1)2^{-(\tau(q)+3\sqrt{\delta}/5)nD}\geq 2^{-(\tau(q)+7\sqrt{\delta}/10)nD}
\end{equation}
(we notice that $nD>\delta l\gg_{\mu,q,\delta,C}1$).

Finally, we consider the Borel probability measures $S_{2^{iD}}\eta$ and $S_{2^{iD}}\angle(\widehat{\nu_{I_0}})$ on $\R/\pi\Z$. We notice that $S_{2^{iD}}:\R/\pi\Z\to\R/\pi\Z$ is the multiplication by the positive integer, and hence an additive group homomorphism. By (\ref{supp_of_eta_proof_of_flattening}), $S_{2^{iD}}$ has at most $O_{M,\mu,C}(1)$ branches on $\supp\ \eta$. Hence, by (\ref{eta_L^q_norm_proof_of_flattening}) and Lemma \ref{L^q_norms_of_two_partitions}, we have
\begin{equation}\label{flat_L^q_norm_of_Seta_proof_of_flattening}
\|(S_{2^{iD}}\eta)^{(nD)}\|_q^q\leq O_{M,\mu,q,C}(1)\|\eta^{((i+n)D)}\|_q^q\leq O_{M,\mu,q,C}(1)2^{-2\kappa/3\cdot nD}\leq 2^{-\kappa/2\cdot nD}
\end{equation}
(we have used that $nD>\delta l\gg_{M,\mu,q,\kappa,C}1$).
Furthermore, since $S_{2^{iD}}$ is an additive group homomorphism on $\R/\pi\Z$, we have $S_{2^{iD}}\eta*S_{2^{iD}}\angle(\widehat{\nu_{I_0}})=S_{2^{iD}}(\eta*\angle(\widehat{\nu_{I_0}}))$. Hence, we obtain from (\ref{convolution_with_eta_L^q_norm_proof_of_flattening}) that
\begin{align}\label{non_decreasing_L^q_norm_under_convolution_with_Seta_Snu_proof_of_flattening}
\|(S_{2^{iD}}\eta*S_{2^{iD}}\angle(\widehat{\nu_{I_0}}))^{(nD)}\|_q^q&=\|(S_{2^{iD}}(\eta*\angle(\widehat{\nu_{I_0}})))^{(nD)}\|_q^q\nonumber\\
&\geq\|(\eta*\angle(\widehat{\nu_{I_0}}))^{((i+n)D)}\|_q^q\nonumber\\
&\geq 2^{-(\tau(q)+7\sqrt{\delta}/10)nD}.
\end{align}

\subsection{Application of Shmerkin's argument}\label{subsection_Shmerkin_argument}

In this section, we apply Shmerkin's argument in the proof of Theorem \ref{L^q_norm_flattening_theorem_for_self_similar_measure} to our case using (\ref{flat_L^q_norm_of_Seta_proof_of_flattening}), (\ref{non_decreasing_L^q_norm_under_convolution_with_Seta_Snu_proof_of_flattening}), the inverse Theorem \ref{the_inverse_theorem_for_L^q_norms} and multifractal properties of $\nu$, lead a contradiction to (\ref{contradiction_assumption_proof_of_flattening}), and complete the proof of the $L^q$ norm flattening Theorem \ref{L^q_norm_flattening_theorem}.

We first recall that the constant $\kappa=\kappa(M,\mu,q,\sigma/2)>0$ depends only on $M,\mu,q,\sigma/2$, and is independent of $\delta$. Hence, we can take $\delta>0$ so small also in terms of $\kappa$. Here, we take a small constant $0<\delta_0\ll_{M,\mu,q,\sigma,\kappa}1$ determined only by $M,\mu,q,\sigma$ and $\kappa$, and a large constant $d_0\in\N, d_0\gg_{M,\mu,q,\sigma,\kappa,\delta_0}1$ determined only by $M,\mu,q,\sigma,\kappa$ and $\delta_0$, both of which will be specified later.
We take a large constant $d=d(q,\delta_0,d_0)\in\N, d\gg_{q,\delta_0,d_0}1$ so that $d>d_0$ and we can apply the inverse Theorem \ref{the_inverse_theorem_for_L^q_norms} to $q$, $\delta_0$ and $d$. Then, we assume $0<\delta\ll_{M,\mu,q,\sigma,\delta_0,d}1$ is sufficiently small so that $\sqrt{\delta}<\delta_0$ and
\begin{equation}\label{delta_is_so_small_for_inverse_theorem_proof_of_flattening}
\text{we can apply Theorem \ref{the_inverse_theorem_for_L^q_norms} to $q$, $\delta_0$, $d$ and $\sqrt{\delta}/q$ in place of $q$, $\delta$, $d$ and $\delta$, respectively.}
\end{equation}

In the following, for $m\in\N$ and a Borel probability measure $\nu$ on $\R/\pi\Z$, we write $\nu^{(m)}$ for the $2^{-m}$ measure defined by\footnote{This notation does not conflict others, because the $L^q$ norm of the $2^{-m}$ measure $\nu^{(m)}$ is $\left(\sum_{x\in\pi 2^{-m}\Z/\pi\Z}\nu^{(m)}(x)^q\right)^{1/q}=\left(\sum_{I\in\D_m}\nu(I)^q\right)^{1/q}=\|\nu^{(m)}\|_q$}.
\begin{equation*}
\nu^{(m)}(\pi k2^{-m})=\nu\left([\pi k2^{-m},\pi(k+1)2^{-m})\right),\quad k=0,1,\dots,2^m-1.
\end{equation*}

For the above $nD$, let $l'=\lceil nD/d\rceil$. Since $nD>\delta l\gg_{M,\mu,q,\sigma,\kappa,\delta,\varepsilon}1$, we can assume that $nD$ is also sufficiently large in terms of $\delta_0,d$, and hence $l'\gg_{M,\mu,q,\sigma,\kappa,\delta_0,d,\delta,\varepsilon}1$.
We also have $l'd-d\leq nD\leq l'd$ and $1-1/l'\leq nD/(l'd)\leq 1$. Hence, by (\ref{flat_L^q_norm_of_Seta_proof_of_flattening}), we have
\begin{equation}\label{flat_L^q_norm_of_Seta_l'd_scale_proof_of_flattening}
\|(S_{2^{iD}}\eta)^{(l'd)}\|_q^q\leq \|(S_{2^{iD}}\eta)^{(nD)}\|_q^q\leq 2^{-\kappa/2\cdot nD}\leq 2^{-\kappa/3\cdot l'd}.
\end{equation}
Furthermore, by Lemma \ref{L^q_norms_of_two_partitions}, we obtain
\begin{equation*}
\|(S_{2^{iD}}\eta*S_{2^{iD}}\angle(\widehat{\nu_{I_0}}))^{(nD)}\|_q^q\leq 2^{(l'd-nD)q}\|(S_{2^{iD}}\eta*S_{2^{iD}}\angle(\widehat{\nu_{I_0}}))^{(l'd)}\|_q^q\leq
2^{dq}\|(S_{2^{iD}}\eta*S_{2^{iD}}\angle(\widehat{\nu_{I_0}}))^{(l'd)}\|_q^q,
\end{equation*}
and hence, by (\ref{non_decreasing_L^q_norm_under_convolution_with_Seta_Snu_proof_of_flattening}) and $l'\gg_{q,\delta}1$,
\begin{align}\label{L^q_norm_under_convolution_with_Seta_Snu_l'd_scale_proof_of_flattening}
\|(S_{2^{iD}}\eta*S_{2^{iD}}\angle(\widehat{\nu_{I_0}}))^{(l'd)}\|_q^q&\geq 2^{-dq}\|(S_{2^{iD}}\eta*S_{2^{iD}}\angle(\widehat{\nu_{I_0}}))^{(nD)}\|_q^q\nonumber\\
&\geq 2^{-dq}2^{-(\tau(q)+7\sqrt{\delta}/10)nD}\nonumber\\
&\geq 2^{-(\tau(q)+4\sqrt{\delta}/5)l'd}.
\end{align}

Here, by Lemma \ref{local_L^q_norm_lemma} (ii) for $\sqrt{\delta}/10$, $l'd\gg_{q,\mu,\delta}1$ and $I_0\in\D_{iD}$, we have
\begin{equation*}
\sum_{J\in\D_{iD+l'd}(I_0)}\nu(J)^q\leq 2^{-(\tau(q)-\sqrt{\delta}/10)l'd}\nu(2I_0)^q.
\end{equation*}
By this estimate, $\widehat{\nu_{I_0}}=\nu|_{I_0}/\nu(2I_0)$ for $I_0\in\D_{iD}$ and (\ref{L^q_norm_under_convolution_with_Seta_Snu_l'd_scale_proof_of_flattening}), we obtain that
\begin{equation*}
\|(S_{2^{iD}}\eta*S_{2^{iD}}\angle(\widehat{\nu_{I_0}}))^{(l'd)}\|_q^q\geq 2^{-9\sqrt{\delta}/10\cdot l'd}\|\widehat{\nu_{I_0}}^{(iD+l'd)}\|_q^q=2^{-9\sqrt{\delta}/10\cdot l'd}\|(S_{2^{iD}}\angle(\widehat{\nu_{I_0}}))^{(l'd)}\|_q^q,
\end{equation*}
and hence, by multiplying both sides by $\nu(2I_0)^q/\nu(I_0)^q$,
\begin{equation*}
\|(S_{2^{iD}}\eta*S_{2^{iD}}\angle(\nu_{I_0}))^{(l'd)}\|_q^q\geq 2^{-9\sqrt{\delta}/10\cdot l'd}\|(S_{2^{iD}}\angle(\nu_{I_0}))^{(l'd)}\|_q^q,
\end{equation*}
Furthermore, for two $2^{-l'd}$ measures $(S_{2^{iD}}\eta)^{(l'd)}$ and $(S_{2^{iD}}\angle(\nu_{I_0}))^{(l'd)}$, it is easily seen from Lemma \ref{L^q_norms_of_two_partitions} that $\|(S_{2^{iD}}\eta*S_{2^{iD}}\angle(\nu_{I_0}))^{(l'd)}\|_q^q=\Theta_q(1)\|(S_{2^{iD}}\eta)^{(l'd)}*(S_{2^{iD}}\angle(\nu_{I_0}))^{(l'd)}\|_q^q$ (or, see \cite[Lemma 4.3]{Shm19}). By combining this and the above inequality, we obtain that
\begin{equation}\label{non_decreasing_L^q_norm_convolution_with_Seta_and_Snu_l'd_scale_proof_of_flattening}
\|(S_{2^{iD}}\eta)^{(l'd)}*(S_{2^{iD}}\angle(\nu_{I_0}))^{(l'd)}\|_q^q\geq 2^{-\sqrt{\delta}l'd}\|(S_{2^{iD}}\angle(\nu_{I_0}))^{(l'd)}\|_q^q
\end{equation}
(we have used that $l'd\gg_{q,\delta}1$).

By the way of taking $\sqrt{\delta}$ (\ref{delta_is_so_small_for_inverse_theorem_proof_of_flattening}), $l'\gg_{q,\delta_0,d,\delta}1$ and (\ref{non_decreasing_L^q_norm_convolution_with_Seta_and_Snu_l'd_scale_proof_of_flattening}), we can apply the inverse Theorem \ref{the_inverse_theorem_for_L^q_norms} to $q,\delta_0,d,l'$ and two $2^{-l'd}$ measures $(S_{2^{iD}}\eta)^{(l'd)}$ and $(S_{2^{iD}}\angle(\nu_{I_0}))^{(l'd)}$.
Hence, there exist subsets $A,B\subset\pi2^{-l'd}\Z/\pi\Z$, $u,v\in\pi2^{-l'd}\Z/\pi\Z$ and sequences $(R'_s)_{s\in[l']}, (R''_s)_{s\in[l']}\in\{1,2,\dots, 2^d\}^{l'}$ satisfying (A-i)-(vi) of Theorem \ref{the_inverse_theorem_for_L^q_norms} for $(S_{2^{iD}}\angle(\nu_{I_0}))^{(l'd)}$ and $(S_{2^{iD}}\eta)^{(l'd)}$ in place of $\nu$ and $\rho$, respectively.
As Shmerkin did in the proof of Theorem \ref{L^q_norm_flattening_theorem_for_self_similar_measure}, we derive a contradiction by studying the branching numbers $(R'_s)_{s\in[l']}$ of $A$ (in the following argument, $B,v,(R''_s)_{s\in[l']}$ do not appear explicitly).

We first notice that, by the assumption (\ref{assumption_tau_smaller_than_q-1_proof_of_flattening}) and concavity of $\tau(q)$, for $\alpha=\tau'(q)$, we have
\begin{equation*}
\alpha\leq\lim_{q'\searrow1}\frac{\tau(q)-\tau(q')}{q-q'}=\frac{\tau(q)}{q-1}<1,
\end{equation*}
and hence
\begin{equation}\label{tau^*(alpha)_strictly_smaller_than_1_proof_of_flattening}
\tau^*(\alpha)=\inf_{q'\in (1,\infty)}(\alpha q'-\tau(q'))\leq \alpha<1.
\end{equation}
Here, we have used that $\tau(q')\in[0,q'-1]$ for $q'>1$ and hence $\tau(q')\to 0$ as $q'\searrow 1$.

By (A-ii), there is $\alpha'\geq0$ such that
\begin{equation}\label{widehatnu_is_flat_on_A_proof_of_flattening}
2^{-\alpha'l'd}\leq (S_{2^{iD}}\angle(\widehat{\nu_{I_0}}))^{(l'd)}(x)\leq 2\cdot 2^{-\alpha'l'd}\quad\text{for any }x\in A
\end{equation}
(we have multiplied each side by $\nu(I_0)/\nu(2I_0)$). Furthermore, by (A-i) and multiplying both sides by $\nu(I_0)^q/\nu(2I_0)^q$, we have
\begin{equation*}
\|(S_{2^{iD}}\angle(\widehat{\nu_{I_0}}))^{(l'd)}|_A\|_q^q\geq 2^{-q\delta_0l'd}\|(S_{2^{iD}}\angle(\widehat{\nu_{I_0}}))^{(l'd)}\|_q^q.
\end{equation*}
By (\ref{L^q_norm_under_convolution_with_Seta_Snu_l'd_scale_proof_of_flattening}) and the fact that an $L^q$ norm of a measure does not increase under a convolution with a probability measure, we have
\begin{equation}\label{Swidehatnu_L^q_and_2^-tau(q)_proof_of_flattening}
\|(S_{2^{iD}}\angle(\widehat{\nu_{I_0}}))^{(l'd)}\|_q^q\geq\Omega_q(1)\|(S_{2^{iD}}\eta*S_{2^{iD}}\angle(\widehat{\nu_{I_0}}))^{(l'd)}\|_q^q\geq 2^{-(\tau(q)+\sqrt{\delta})l'd}
\end{equation}
(we have used that $l'd\gg_{q,\delta}1$).
From these two inequalities, it follows that
\begin{equation}\label{restricted_L^q_norm_to_A_proof_of_flattening}
\|(S_{2^{iD}}\angle(\widehat{\nu_{I_0}}))^{(l'd)}|_A\|_q^q\geq 2^{-(\tau(q)+q\delta_0+\sqrt{\delta})l'd}\geq 2^{-(\tau(q)+(q+1)\delta_0)l'd}
\end{equation}
(we notice that $\sqrt{\delta}\leq\delta_0$).

Once again, we take a small constant $0<\delta_0'\ll_{M,\mu,q,\sigma,\kappa}1$ determined only by $M,\mu,q,\sigma$ and $\kappa$ and specified later, and we assume $0<\delta_0\ll_{M,\mu,q,\sigma,\kappa}1$ is sufficiently small in terms of $\delta_0'$ so that $\delta_0<\delta_0'$ and
\begin{equation}\label{how_to_take_delta0_local_tau*_lemma_applicable_proof_of_flattening}
\text{we can apply Lemma \ref{close_to_tau^*(alpha)_local_version} to $\nu$, $q$, $\delta_0'$ and $(q+1)\delta_0$ in place of $\varepsilon$ and $\delta$, respectively.}
\end{equation}
By this way of taking $\delta_0$, (\ref{widehatnu_is_flat_on_A_proof_of_flattening}), (\ref{restricted_L^q_norm_to_A_proof_of_flattening}) and $l'd\gg_{\mu,q,\delta_0',\delta_0}1$, we can apply Lemma \ref{close_to_tau^*(alpha)_local_version} to $s=iD$, $I_0\in\D_{iD}$, $m=l'd$ and $\D'=A$, and obtain that\footnote{Here, we abuse the statement of Lemma \ref{close_to_tau^*(alpha)_local_version} a little, but the justification is clear.}
\begin{equation}\label{the_number_of_elements_of_A_proof_of_flattening}
|A|\leq 2^{(\tau^*(\alpha)+\delta'_0)l'd}.
\end{equation}

We write
\begin{equation*}
\kappa'=\kappa'(\mu,q,\kappa)=\kappa'(M,\mu,q,\sigma)=\frac{(1-\tau^*(\alpha))\kappa}{12(q-1)}.
\end{equation*}
Then, by (\ref{tau^*(alpha)_strictly_smaller_than_1_proof_of_flattening}), we have $\kappa'>0$. Furthermore, by (\ref{assumption_tau_smaller_than_q-1_proof_of_flattening}) and (\ref{tau^*(alpha)_strictly_smaller_than_1_proof_of_flattening}), we can assume for $0<\delta'_0\ll_{M,\mu,q,\sigma,\kappa}1$ that
\begin{equation}\label{delta0'_sufficiently_small_Mmuqsigma_1_proof_of_flattening}
1-\frac{\tau(q)+\delta_0'}{q-1}-\delta_0'>\frac{1}{2}\left(1-\frac{\tau(q)}{q-1}\right)>0
\end{equation}
and
\begin{equation}\label{delta0'_sufficiently_small_Mmuqsigma_2_proof_of_flattening}
2\delta_0'<1-\tau^*(\alpha),\quad
\delta_0'<\frac{\kappa}{3(q-1)},\quad
\left(1-\tau^*(\alpha)-2\delta_0'\right)\left(\frac{\kappa}{3(q-1)}-\delta_0'\right)-\delta_0'>2\kappa'.
\end{equation}
Here, we show the following key claim for completing the proof of Theorem \ref{L^q_norm_flattening_theorem}.

\begin{claim}
We have $\tau^*(\alpha)\geq2\kappa'$.\footnote{This is not mysterious. In fact, we can recall that $\kappa$ was taken as $(q-1)\sigma''/4\ll\tau^*(\alpha)$ in Section \ref{subsection_discretized_slicing_lemma}.}
Furthermore, if we write
\begin{equation*}
S_1=\left\{s\in[l']\left|\ R_s'\leq 2^{(\tau^*(\alpha)-\kappa')d}\right.\right\}
\end{equation*}
and\footnote{It may seem that, since $\tau^*(\alpha)-\kappa'>0$ has been shown above, taking the absolute value of the denominator of $\gamma$ is unnecessary. However, in our proof of $\tau^*(\alpha)-\kappa'>0$, we use $A$ and its branching numbers $(R'_s)_{s\in[l']}$, which are defined only after we determine $\sqrt{\delta}<\delta_0$ completely, and we actually have to take $\delta_0$ sufficiently small in terms of $\gamma>0$ in the following argument. Hence, to ensure the positivity of $\gamma=\gamma(M,\mu,q,\sigma)$ without $\delta_0$, we take the absolute value of the denominator.}
\begin{equation*}
\gamma=\gamma(\mu,q,\kappa')=\gamma(M,\mu,q,\sigma)=\frac{\kappa'\left(1-\tau(q)/(q-1)\right)}{2|\tau^*(\alpha)-\kappa'|}>0,
\end{equation*}
then we have
\begin{equation*}
|S_1|\geq \gamma l'.
\end{equation*}
\end{claim}

The heuristic idea for Claim is the following. We can see from Theorem \ref{the_inverse_theorem_for_L^q_norms} and (\ref{flat_L^q_norm_of_Seta_l'd_scale_proof_of_flattening}) that $A$ almost fully branches at an appropriately large proportion of scales. On the other hand, by (\ref{the_number_of_elements_of_A_proof_of_flattening}) and $\tau^*(\alpha)<1$ ((\ref{tau^*(alpha)_strictly_smaller_than_1_proof_of_flattening})), we have $|A|\lesssim 2^{\tau^*(\alpha)l'd}$ and this is essentially smaller than the entire $2^{l'd}$.
Hence, there must be an appropriately large proportion of scales at which $A$ has branches essentially smaller that $2^{\tau^*(\alpha)d}$.

\begin{proof}
We first consider
\begin{equation*}
S=\left\{s\in[l']\left|\ R'_s\geq 2^{(1-\delta_0)d}\right.\right\}\quad\text{and}\quad S'=[l']\setminus S.
\end{equation*}
Since $A$ is $(d,l',u,(R'_s)_{s\in[l']})$-uniform, we have
\begin{equation*}
|A|=\prod_{s\in[l']}R'_s\geq 2^{(1-\delta_0)|S|d}\prod_{s\in S'}R'_s.
\end{equation*}
From this and (\ref{the_number_of_elements_of_A_proof_of_flattening}), we obtain that
\begin{equation}\label{product_of_branching_numbers_at_S'_proof_of_flattening}
\prod_{s\in S'}R'_s\leq 2^{-(1-\delta_0)|S|d}2^{(\tau^*(\alpha)+\delta_0')l'd}=2^{-(1-\tau^*(\alpha)-2\delta_0')|S|d}2^{(\tau^*(\alpha)+\delta_0')|S'|d}
\end{equation}
(we have used that $\delta_0<\delta_0'$). Here, it follows from (vi) of Theorem \ref{the_inverse_theorem_for_L^q_norms} and (\ref{flat_L^q_norm_of_Seta_l'd_scale_proof_of_flattening}) that
\begin{equation}\label{many_scales_at_which_A_almost_fully_branches_proof_of_flattening}
|S|d\geq\log\left(\|(S_{2^{iD}}\eta)^{(l'd)}\|_q^{-q/(q-1)}\right)-\delta_0l'd\geq \frac{\kappa}{3(q-1)}\cdot l'd-\delta_0l'd.
\end{equation}
In addition, by (\ref{delta0'_sufficiently_small_Mmuqsigma_2_proof_of_flattening}), we have $1-\tau^*(\alpha)-2\delta_0'>0$. Hence, by combining (\ref{product_of_branching_numbers_at_S'_proof_of_flattening}) and (\ref{many_scales_at_which_A_almost_fully_branches_proof_of_flattening}) and using (\ref{delta0'_sufficiently_small_Mmuqsigma_2_proof_of_flattening}), we have
\begin{align}\label{estimate_blanching_numbers_at_S'_proof_of_flattening}
\log\prod_{s\in S'}R'_s&\leq-(1-\tau^*(\alpha)-2\delta_0')\left(\frac{\kappa}{3(q-1)}-\delta_0\right)l'd+(\tau^*(\alpha)+\delta_0')|S'|d\nonumber\\
&\leq-(1-\tau^*(\alpha)-2\delta_0')\left(\frac{\kappa}{3(q-1)}-\delta_0'\right)|S'|d+(\tau^*(\alpha)+\delta_0')|S'|d\nonumber\\
&\leq (\tau^*(\alpha)-2\kappa')|S'|d
\end{align}
(we notice that $\delta_0<\delta_0'$ and $|S'|\leq l'$).

Furthermore, by (\ref{Swidehatnu_L^q_and_2^-tau(q)_proof_of_flattening}), $\nu_{I_0}\geq\widehat{\nu_{I_0}}=\nu(I_0)/\nu(2I_0)\cdot\nu_{I_0}$ and $\sqrt{\delta}<\delta_0<\delta_0'$, it holds that
\begin{equation*}
	\|(S_{2^{iD}}\angle(\nu_{I_0}))^{(l'd)}\|_q^q\geq\|(S_{2^{iD}}\angle(\widehat{\nu_{I_0}}))^{(l'd)}\|_q^q\geq2^{-(\tau(q)+\sqrt{\delta})l'd}\geq 2^{-(\tau(q)+\delta_0')l'd}.
\end{equation*}
By this inequality and (vi) of Theorem \ref{the_inverse_theorem_for_L^q_norms}, we have
\begin{equation*}
	d|S|\leq\log\left(\|(S_{2^{iD}}\angle(\nu_{I_0}))^{(l'd)}\|_q^{-q/(q-1)}\right)+\delta_0l'd
	\leq\frac{\tau(q)+\delta_0'}{q-1}\cdot l'd+\delta_0'l'd,
\end{equation*}
and hence, by (\ref{delta0'_sufficiently_small_Mmuqsigma_1_proof_of_flattening}),
\begin{equation}\label{lower_bound_S'_proof_of_flattening}
	|S'|=l'-|S|\geq\left(1-\frac{\tau(q)+\delta_0'}{q-1}-\delta_0'\right)l'>\frac{1}{2}\left(1-\frac{\tau(q)}{q-1}\right)l'.
\end{equation}
In particular, $S'$ is non-empty.
Hence, we obtain from (\ref{estimate_blanching_numbers_at_S'_proof_of_flattening}) that $\tau^*(\alpha)-2\kappa'\geq0$.

We estimate $|S_1|$. By the definition of $S_1$ and (\ref{estimate_blanching_numbers_at_S'_proof_of_flattening}), we have\footnote{This argument holds even if $S'\setminus S_1=\emptyset$.}
\begin{align*}
|S'|-|S_1|\leq |S'\setminus S_1|&\leq \frac{1}{(\tau^*(\alpha)-\kappa')d}\sum_{s\in S'\setminus S_1}\log R'_s\\
&\leq\frac{1}{(\tau^*(\alpha)-\kappa')d}\log\prod_{s\in S'}R'_s\\
&\leq\frac{\tau^*(\alpha)-2\kappa'}{\tau^*(\alpha)-\kappa'}\cdot |S'|,
\end{align*}
and hence
\begin{equation*}
|S_1|\geq \frac{\kappa'}{\tau^*(\alpha)-\kappa'}\cdot |S'|.
\end{equation*}
From this estimate and (\ref{lower_bound_S'_proof_of_flattening}), we obtain that
\begin{equation*}
|S_1|\geq\frac{\kappa'(1-\tau(q)/(q-1))}{2(\tau^*(\alpha)-\kappa')}\cdot l'
\end{equation*}
and complete the proof.
\end{proof}

Finally, we derive a contradiction from Claim, and complete the proof of Theorem \ref{L^q_norm_flattening_theorem}.
We use Lemma \ref{local_L^q_norm_lemma}\footnote{We notice that Lemma \ref{local_L^q_norm_lemma} still holds if we replace $\D_s$ with $\D_{u,s}$ with the base $u$, uniformly in terms of $u$.}.
Here, we recall that $d_0\in\N$ is a sufficiently large constant determined only by $M,\mu,q,\sigma,\kappa$ and $\delta_0$.
By Lemma \ref{local_L^q_norm_lemma} (i) for $\kappa'/2=\kappa'(\mu,q,\kappa)/2>0$ defined above,
there exists a constant $\beta=\beta(\mu,q,\kappa'/2)>0$ depending only on $\mu,q,\kappa'/2$ such that,
if we assume that $d_0\gg_{\mu,q,\kappa'/2,\beta}1$ (that is, $d_0\gg_{\mu,q,\kappa}1$) is sufficiently large, the following holds.
For any $s'\in\N$, $w\in\RP^1$, $I\in\D_{w,s'}$ and $d'\in\N$ with $d'\geq d_0-2$, if a subset $\D'\subset\D_{w,s'+d'}(I)$ satisfies $|\D'|\leq 2^{(\tau^*(\alpha)-\kappa'/2)d'}$, we have
\begin{equation}\label{local_L^q_lemma_for_kappa'_proof_of_flettening}
\sum_{J\in\D'}\nu(J)^q\leq 2^{-(\tau(q)+\beta)d'}\nu(2I)^q.
\end{equation}
Furthermore, by Lemma \ref{local_L^q_norm_lemma} (ii) for $\delta_0>0$, if we assume that $d_0\gg_{\mu,q,\delta_0}1$ is sufficiently large, we obtain for any $s'\in\N$, $w\in\RP^1$, $I\in\D_{w,s'}$ and $d'\in\N$ with $d'\geq d_0-2$ that
\begin{equation}\label{local_L^q_norm_lemma_for_delta0_proof_of_flattening}
\sum_{J\in\D_{w,s'+d'}(I)}\nu(J)^q\leq 2^{-(\tau(q)-\delta_0)d'}\nu(2I)^q.
\end{equation}
We also assume that\footnote{The reason for taking the absolute values is the same as stated in the footnote in Claim.}
\begin{equation}\label{1/d0_smaller_than_delta0_proof_of_flattening}
\frac{1}{d_0}<\min\left\{\frac{1}{10}, \delta_0,\frac{\kappa'}{3}\right\},\quad \frac{|\tau^*(\alpha)-2\kappa'/3|}{|\tau^*(\alpha)-2\kappa'/3|+\kappa'/6}<1-\frac{2}{d_0}.
\end{equation}

In the following argument, we identify $\RP^1$ with $\R/\pi\Z$ by $\angle$ and often omit it.
Let
\begin{equation*}
A'=S_{2^{iD}}^{-1}A\cap I_0\subset\pi2^{-(iD+l'd)}\Z/\pi\Z\cap I_0
\end{equation*}
and, for $u\in\pi2^{-l'd}\Z/\pi\Z$, $u'$ be the unique element of $\pi2^{-(iD+l'd)}\Z/\pi\Z\cap I_0$ such that
\begin{equation*}
S_{2^{iD}}u'=u.
\end{equation*}
We consider the sequence
\begin{equation*}
L'_s=-\log \sum_{I\in\D_{u',iD+sd}(A')}\nu(I)^q,\quad s=0,1,\dots,l'.
\end{equation*}
It can be seen from the definitions that
\begin{equation}\label{def_of_L'l'_proof_of_flattening}
L'_{l'}=-\log\sum_{x\in A'}\nu^{(iD+l'd)}(x)^q=-\log \|(S_{2^{iD}}\angle(\nu|_{I_0}))^{(l'd)}|_A\|_q^q.
\end{equation}

We take $s\in[l']$. For each $I\in\D_{u',iD+sd}(A')$, we notice that $S_{2^{iD}}I\in\D_{u,sd}(A)$ and each $J\in\D_{u',iD+(s+1)d}(I\cap A')$ is mapped by $S_{2^{iD}}$ into $\D_{u,(s+1)d}(S_{2^{iD}}I\cap A)$. Furthermore, for two distinct elements $J=[\pi k2^{-(iD+(s+1)d)}+u',\pi(k+1)2^{-(iD+(s+1)d)}+u'),J'=[\pi k'2^{-(iD+(s+1)d)}+u',\pi(k'+1)2^{-(iD+(s+1)d)}+u')$ ($k,k'\in\Z$) of $\D_{u',iD+(s+1)d}(I\cap A')$, if $S_{2^{iD}}J=S_{2^{iD}}J'$, then $\pi k2^{-(iD+(s+1)d)}=\pi k'2^{-(iD+(s+1)d)}+\pi p2^{-iD}$ for some $p\in\Z$. From this fact and $J\cap I_0, J'\cap I_0\neq\emptyset$, it follows that $J$ and $J'$ must be two extreme elements of $\D_{u',iD+(s+1)d}(I_0)$ containing one of the endpoints of $I_0$. Hence, we have
\begin{equation*}
\left|\D_{u',iD+(s+1)d}(I\cap A')\right|\leq \left|\D_{u,(s+1)d}(S_{2^{iD}}I\cap A)\right|+1.
\end{equation*}
Since $S_{2^{iD}}I\in\D_{u,sd}(A)$ and $A$ is $(d,l',u,(R'_s)_{s\in[l']})$-uniform, we have  $\left|\D_{u,(s+1)d}(S_{2^{iD}}I\cap A)\right|=R'_s$, and hence
\begin{equation}\label{branching_numbers_of_A'_proof_of_flattening}
\left|\D_{u',iD+(s+1)d}(I\cap A')\right|\leq R'_s+1.
\end{equation}

Here, assume that $s\in S_1$. Then, by (\ref{branching_numbers_of_A'_proof_of_flattening}), (\ref{1/d0_smaller_than_delta0_proof_of_flattening})  and $d>d_0$, we have
\begin{equation*}
\left|\D_{u',iD+(s+1)d}(I\cap A')\right|\leq 2^{(\tau^*(\alpha)-\kappa')d}+1\leq 2^{(\tau^*(\alpha)-2\kappa'/3)d}
\end{equation*}
for any $I\in\D_{u',iD+sd}(A')$, and hence
\begin{equation*}
\left|\D_{u',iD+(s+1)d}(I'\cap A')\right|\leq 2^{(\tau^*(\alpha)-2\kappa'/3)d}\leq 2^{(\tau^*(\alpha)-\kappa'/2)(d-2)}
\end{equation*}
for any $I'\in\D_{u',iD+sd+2}(A')$. Then, we can apply (\ref{local_L^q_lemma_for_kappa'_proof_of_flettening}) for $s'=iD+sd+2$, $w=u'$, each $I'\in\D_{u',iD+sd+2}(A')$, $d-2>d_0-2$ and $\D_{u',iD+(s+1)d}(I'\cap A')\subset\D_{u',iD+(s+1)d}(I')$
and have
\begin{equation*}
\sum_{J\in\D_{u',iD+(s+1)d}(I'\cap A')}\nu(J)^q\leq 2^{-(\tau(q)+\beta)(d-2)}\nu(2I')^q.
\end{equation*}
By taking the sum for $I'\in\D_{u',iD+sd+2}(A')$, we have
\begin{align}\label{L's+1_lower_bound_proof_of_flattening}
L'_{s+1}&=-\log\sum_{I'\in\D_{u',iD+sd+2}(A')}\sum_{J\in\D_{u',iD+(s+1)d}(I'\cap A')}\nu(J)^q\nonumber\\
&\geq (\tau(q)+\beta)(d-2)-\log\sum_{I'\in\D_{u',iD+sd+2}(A')}\nu(2I')^q.
\end{align}

Let $s>0$.
For each $I'\in\D_{u',iD+sd+2}(A')$, we take a unique $I\in\D_{u',iD+sd}(A')$ such that $I'\subset I$.
Then, we have $S_{2^{iD}}I'\in\D_{u,sd+2}(A)$, $S_{2^{iD}}I\in\D_{u,sd}(A)$ and $S_{2^{iD}}I'\subset S_{2^{iD}}I$. However, by (A-iv) of Theorem \ref{the_inverse_theorem_for_L^q_norms}, in the four subintervals of length $\pi2^{-(sd+2)}$ of $S_{2^{iD}}I$, the only possibility of $S_{2^{iD}}I'$ is either of the two intervals in the center. Hence, the same thing holds for $I'\subset I$, and this fact implies that
\begin{equation*}
\sum_{I'\in\D_{u',iD+sd+2}(A')}\nu(2I')^q\leq 2\sum_{I\in\D_{u',iD+sd}(A')}\nu(I)^q.
\end{equation*}
From this inequality and (\ref{L's+1_lower_bound_proof_of_flattening}), we obtain that
\begin{equation}\label{L's+1_L's_estimate_s_in_S1_proof_of_flattening}
L'_{s+1}\geq L'_s+(\tau(q)+\beta)(d-2)-1
\end{equation}
for $s\in\left\{1,\dots,l'-1\right\}\cap S_1$.
When $s=0$, each $2I'$ for $I'\in\D_{u',iD+2}(A')$ is contained in $2I_0$, and the number of $I'\in\D_{u',iD+2}(A')$ is at most $5$. Hence, by (\ref{L's+1_lower_bound_proof_of_flattening}), we have
\begin{equation*}
L'_1\geq -\log\nu(2I_0)^q+(\tau(q)+\beta)(d-2)-\log5.
\end{equation*}

If $s\notin S_1$, we apply (\ref{local_L^q_norm_lemma_for_delta0_proof_of_flattening}) instead of (\ref{local_L^q_lemma_for_kappa'_proof_of_flettening}) in the above argument, and obtain that
\begin{equation}\label{L's+1_L's_estimate_s_notin_S1_proof_of_flattening}
L'_{s+1}\geq L'_s+(\tau(q)-\delta_0)(d-2)-1
\end{equation}
for $s\in\{1,\dots,l'-1\}\setminus S_1$. For $s=0$, we have
\begin{equation}\label{L'1_nu(2I0)_estimate_proof_of_flattening}
L'_1\geq
\begin{cases}
-\log\nu(2I_0)^q+(\tau(q)+\beta)(d-2)-\log 5&0\in S_1,\\
-\log\nu(2I_0)^q+(\tau(q)-\delta_0)(d-2)-\log 5&0\notin S_1.
\end{cases}
\end{equation}

By (\ref{def_of_L'l'_proof_of_flattening}), (\ref{L's+1_L's_estimate_s_in_S1_proof_of_flattening}), (\ref{L's+1_L's_estimate_s_notin_S1_proof_of_flattening}), (\ref{L'1_nu(2I0)_estimate_proof_of_flattening}) and telescoping, we have
\begin{align*}
&-\log \|(S_{2^{iD}}\angle(\nu|_{I_0}))^{(l'd)}|_A\|_q^q=L'_{l'}\\
=\ &\sum_{s=1}^{l'-1}(L'_{s+1}-L'_s)+L'_1\\
=\ &\sum_{\substack{s=1,\dots,l'-1,\\s\in S_1}}(L'_{s+1}-L'_s)+\sum_{\substack{s=1,\dots,l'-1,\\s\notin S_1}}(L'_{s+1}-L'_s)+L'_1\\
\geq\ &|S_1|(\tau(q)+\beta)(d-2)+(l'-|S_1|)(\tau(q)-\delta_0)(d-2)-l'\log 5-\log\nu(2I_0)^q\\
=\ &\tau(q)l'(d-2)+\beta|S_1|(d-2)-\delta_0(l'-|S_1|)(d-2)-l'\log 5-\log\nu(2I_0)^q.
\end{align*}
From this estimate, $\widehat{\nu_{I_0}}=\nu|_{I_0}/\nu(2I_0)$ and Claim, we obtain that
\begin{equation*}
-\log \|(S_{2^{iD}}\angle(\widehat{\nu_{I_0}}))^{(l'd)}|_A\|_q^q\geq \tau(q)l'(d-2)+\beta\gamma l'(d-2)-\delta_0l'd-l'\log 5.
\end{equation*}
On the other hand, by (\ref{restricted_L^q_norm_to_A_proof_of_flattening}), we have
\begin{equation*}
-\log \|(S_{2^{iD}}\angle(\widehat{\nu_{I_0}}))^{(l'd)}|_A\|_q^q\leq (\tau(q)+(q+1)\delta_0)l'd.
\end{equation*}
Hence, by combining these two inequalities and (\ref{1/d0_smaller_than_delta0_proof_of_flattening}), we obtain that
\begin{equation}\label{resulting_contradiction_proof_of_flattening}
\frac{\beta\gamma}{2}\leq\beta\gamma\left(1-\frac{2}{d}\right)\leq \frac{2\tau(q)}{d}+(q+2)\delta_0+\frac{\log 5}{d}\leq (2\tau(q)+q+2+\log 5)\delta_0
\end{equation}
However, we recall that $\gamma=\gamma(M,\mu,q,\sigma)>0$ and $\beta=\beta(\mu,q,\kappa'/2)=\beta(M,\mu,q,\sigma)>0$ are determined only on $M,\mu,q,\sigma$. Hence, we can assume that $0<\delta_0\ll_{M,\mu,q,\sigma}1$ satisfies
\begin{equation}\label{delta0_much_smaller_than_betagamma_proof_of_flattening}
\delta_0<\frac{\beta\gamma}{3(2\tau(q)+q+2+\log 5)}.
\end{equation}
Then, (\ref{resulting_contradiction_proof_of_flattening}) contradicts (\ref{delta0_much_smaller_than_betagamma_proof_of_flattening}). Therefore, we derived a contradiction from the assumption (\ref{contradiction_assumption_proof_of_flattening}), and hence completed the proof of Theorem \ref{L^q_norm_flattening_theorem}.

\section{Proof of the main theorem}\label{section_proof_of_the_main_theorem}

In this final section, using the $L^q$ norm flattening Theorem \ref{L^q_norm_flattening_theorem}, we prove our main Theorem \ref{the_main_theorem_L^q_dim_Mobius_IFS}.

\subsection{The $L^q$ norm of $\mu_m$ at a finer scale}\label{subsection_L^q_norm_of_mum_finer scale}

In this section, we show the following proposition corresponding to \cite[Proposition 5.2]{Shm19}.
Here is the point at which we use the $L^q$ form flattening Theorem \ref{L^q_norm_flattening_theorem}, and this proposition leads the essential part of the main Theorem \ref{the_main_theorem_L^q_dim_Mobius_IFS} immediately.

\begin{prop}\label{the_L^q_morm_of_mum_at_a_finer_scale_proof_of_the_main_theorem}
Let $\A=\{A_i\}_{i\in\I}$ be a non-empty finite family of elements of $G$ which is uniformly hyperbolic. We take a probability measure $\mu=\sum_{i\in\I}p_i\delta_{A_i}$ on $G$ such that $\supp\ \mu=\A$ and write $\nu$ for the stationary measure of $\mu$ on $\RP^1$. We also write $\tau(q)\ (q>1)$ for the $L^q$ spectrum of $\nu$.
Let $q>1$ be such that $\tau(q)$ is differentiable at $q$ and assume that
\begin{equation*}
\tau(q)<q-1
\end{equation*}
and, for $\alpha=\tau'(q)$,
\begin{equation*}
\tau^*(\alpha)>0.
\end{equation*}
Then, for any $R\in\N$, we have
\begin{equation*}
\lim_{m\to\infty}\left(-\frac{1}{m}\log\|\mu_m^{(Rm)}\|_q^q\right)=\tau(q).
\end{equation*}
\end{prop}

In the statement above, we recall that, for $m\in\N$, we have defined
\begin{equation*}
\Omega_m=\left\{i=(i_1,\dots,i_n)\in \I^*\left|\ \|A_{i_1}\|^2<2^m,\dots,\|A_{i_1}\cdots A_{i_{n-1}}\|^2<2^m,\|A_{i_1}\cdots A_{i_n}\|^2\geq 2^m\right.\right\}
\end{equation*}
and
\begin{equation*}
\mu_m=\sum_{i\in\Omega_m}p_i\delta_{A_i}
\end{equation*}
in Section \ref{subsection_stationary_measures}. It is easily seen that there is $C=C(\A)>1$ such that
\begin{equation}\label{norm_Omegam_proof_of_the_main_theorem}
2^m\leq\|A_i\|^2\leq C2^m
\end{equation}
for any $i\in\Omega_m$ (we are using the finiteness of $\A$).
We also notice that the $L^q$ norm $\|\mu_m^{(Rm)}\|_q$ in the statement is that with respect to the $2^{-Rm}$ dyadic-like partition $\D^G_{Rm}$ of $G$.

Let $K\subset\RP^1$ be the attractor of $\A$ and we fix one point $x_0\in K$. We recall that we have defined a map $f_{x_0}:G\ni g\mapsto gx_0\in\RP^1$. To prove Proposition \ref{the_L^q_morm_of_mum_at_a_finer_scale_proof_of_the_main_theorem}, we need the following lemma.

\begin{lem}\label{bounded_supp_of_restrictions_of_mum_proof_of_the_main_theorem}
In the setting of Proposition \ref{the_L^q_morm_of_mum_at_a_finer_scale_proof_of_the_main_theorem}, there exists a constant $L=L(\A)>1$ determined only by $\A$ such that, for any $m\in\N$ and $I\in\D_m$ with $\mu_m(f_{x_0}^{-1}I)>0$, we have
\begin{equation*}
\diam\ \supp\ (\mu_m)_{f_{x_0}^{-1}I}\leq L.
\end{equation*}
\end{lem}

\begin{proof}
It is sufficient to show that there is a constant $L=L(\A)>1$ satisfying the following. For any sufficiently large $m\in\N, m\gg_\A1$ in terms of $\A$ and any $I\in\D_m$, if we define
\begin{equation*}
E_I=\left\{A_i\left|\ i\in\Omega_m, A_ix_0\in I\right.\right\},
\end{equation*}
then we have
\begin{equation}\label{bounded_diam_of_EI_proof_of_the_main_theorem}
\diam\ E_I\leq L.
\end{equation}

We fix $g_0\in E_I$ and take open sets $U\subset U_1\subsetneq\RP^1$ as in (\ref{domains_the_Möbius_IFS_acts_on_preliminaries}) for $\A$. We have $x_0\in K\subset U$ and, by Lemma \ref{center_of_expansion_notin_U}, if $m\gg_{\A}1$ is sufficiently large, $u_g^-\notin U_1$ for any $g\in E_I$. These facts, (\ref{norm_Omegam_proof_of_the_main_theorem}) and (\ref{derivative_action_of_G_preliminaries}) tell us that
\begin{equation*}
d_{\RP^1}(gx_0,v_g^+)\leq O_\A(2^{-m})
\end{equation*}
for any $g\in E_I$. Since $gx_0, g_0x_0\in I$ and $\diam\ I=\pi 2^{-m}$, the above inequalities for $g\in E_I$ and $g_0$ imply
\begin{equation}\label{vg+_vg0+_2^-m_close_proof_or_the_main_theorem}
d_{\RP^1}(v_g^+,v_{g_0}^+)\leq O_\A(2^{-m})
\end{equation}
for any $g\in E_I$.

Let $g\in E_I$ and
we identify $u_g^+,u_g^-,u_{g_0}^+,u_{g_0}^-,v_g^+,v_g^-,v_{g_0}^+,v_{g_0}^-\in\RP^1$ with the unit vectors in $\R^2$ with each direction. As we have seen in Section \ref{subsection_G_action_on_RP^1}, each $\pm$ pair forms an orthonormal basis of $\R^2$ and
\begin{equation}\label{contraction_expasion_g_g_0_proof_of_the_main_theorem}
	gu_g^+=\|g\|v_g^+,\quad gu_g^-=\|g\|^{-1}v_g^-,\quad g_0u_{g_0}^+=\|g_0\|v_{g_0}^+,\quad g_0u_{g_0}^-=\|g_0\|^{-1}v_{g_0}^-.
\end{equation}
We write
\begin{equation*}
	v_g^+=a^+v_{g_0}^++b^+v_{g_0}^-,\quad v_g^-=a^-v_{g_0}^++b^-v_{g_0}^-,\quad (a^+)^2+(b^+)^2=(a^-)^2+(b^-)^2=1.
\end{equation*}
Then, by the definition of the metric on $\RP^1$ and (\ref{vg+_vg0+_2^-m_close_proof_or_the_main_theorem}), we have\footnote{For $u,v\in\R^2$, we write $\langle u,v\rangle$ for the Euclidean inner product on $\R^2$}
\begin{equation}\label{estimate_a_+_proof_of_the_main_theorem}
	|a^+|=\left|\langle v_g^+,v_{g_0}^+\rangle\right|=|\cos( d_{\RP^1}(v_g^+,v_{g_0}^+))|=1-O_\A(2^{-2m}),
\end{equation}
and hence
\begin{equation}\label{estimate_b_+_proof_of_the_main_theorem}
	|b^+|=\sqrt{1-|a^+|^2}=O_\A(2^{-m}).
\end{equation}
Since $v_g^+$ and $v_g^-$, $v_{g_0}^+$ and $v_{g_0}^-$ are orthogonal, we have by (\ref{vg+_vg0+_2^-m_close_proof_or_the_main_theorem}) that $d_{\RP^1}(v_g^-,v_{g_0}^-)=d_{\RP^1}(v_g^+,v_{g_0}^+)\leq O_\A(2^{-m})$. Hence, from the same argument, we obtain that
\begin{equation}\label{estimate_a^-_b^-_proof_of_the_main_theorem}
	|b^-|=1-O_\A(2^{-2m}),\quad |a^-|=O_\A(2^{-m}).
\end{equation}

Here, we estimate $\|g_0^{-1}g\|$. Let $u$ be an arbitrary element in $\R^2$ such that $\|u\|=1$. By using the orthonormal basis $\{u_g^+,u_g^-\}$, we write $u=su_g^++tu_g^-, s^2+t^2=1$. From (\ref{contraction_expasion_g_g_0_proof_of_the_main_theorem}), it follows that
\begin{equation*}
	gu=sgu_g^++tgu_g^-=\|g\|sv_g^++\|g\|^{-1}tv_g^-=(\|g\|a^+s+\|g\|^{-1}a^-t)v_{g_0}^++(\|g\|b^+s+\|g\|^{-1}b^-t)v_{g_0}^-,
\end{equation*}
and hence
\begin{align*}
	g_0^{-1}gu=&\ (\|g\|a^+s+\|g\|^{-1}a^-t)g_0^{-1}v_{g_0}^++(\|g\|b^+s+\|g\|^{-1}b^-t)g_0^{-1}v_{g_0}^-\\
	=&\ (\|g\|\|g_0\|^{-1}a^+s+\|g\|^{-1}\|g_0\|^{-1}a^-t)u_{g_0}^++(\|g\|\|g_0\|b^+s+\|g\|^{-1}\|g_0\|b^-t)u_{g_0}^-.
\end{align*}
So, by using (\ref{norm_Omegam_proof_of_the_main_theorem}), (\ref{estimate_a_+_proof_of_the_main_theorem}), (\ref{estimate_b_+_proof_of_the_main_theorem}) and (\ref{estimate_a^-_b^-_proof_of_the_main_theorem}), we have
\begin{align*}
	\|g_0^{-1}gu\|=&\ O\left(|\|g\|\|g_0\|^{-1}a^+s+\|g\|^{-1}\|g_0\|^{-1}a^-t|+|\|g\|\|g_0\|b^+s+\|g\|^{-1}\|g_0\|b^-t|\right)\\
	\leq&\ O_\A\left(|s|+|t|+2^m2^{-m}|s|+|t|\right)\\
	\leq&\ O_\A(1).
\end{align*}
Hence, we obtain that
\begin{equation*}
	\|g_0^{-1}g\|\leq O_\A(1).
\end{equation*}
for any $g\in E_I$. This fact tells us that $E_I\subset g_0\left\{h\in G\left|\ \|h\|\leq O_\A(1)\right.\right\}$ and, since the metric $d_G$ on $G$ is left-invariant, $\diam\ E_I\leq\diam\ \left\{h\in G\left|\ \|h\|\leq O_\A(1)\right.\right\}\leq O_\A(1)$. This is (\ref{bounded_diam_of_EI_proof_of_the_main_theorem}) and we complete the proof.
\end{proof}

Here, we begin the proof of Proposition \ref{the_L^q_morm_of_mum_at_a_finer_scale_proof_of_the_main_theorem}.

\begin{proof}[Proof of Proposition \ref{the_L^q_morm_of_mum_at_a_finer_scale_proof_of_the_main_theorem}]
Let $R\in\N$. We first see $\liminf_{m\to\infty}\left(-m^{-1}\log\|\mu_m^{(Rm)}\|_q^q\right)\geq \tau(q)$, which is the trivial part.
For sufficiently large $m\in\N$, we can see from Lemma \ref{center_of_expansion_notin_U} and Corollary \ref{Lipschitz_continuity_of_the_action} that, for each $\xi\in\D^G_{Rm}$ (or any subset $\xi\subset G$ such that $\diam\ \xi$ is sufficiently small in terms of $\A$), there is $I\in\D_m$ such that $\supp\ \mu_m\cap \xi\subset f_{x_0}^{-1}(I^{(-)}\sqcup I\sqcup I^{(+)})$, where $I^{(+)}$ and $I^{(-)}$ are the two neighboring intervals in $\D_m$ of $I$. Hence, we have
\begin{align*}
\sum_{I\in\D_m}\mu_m(f_{x_0}^{-1}I)^q&\geq 3^{-q}\sum_{I\in\D_m}\mu_m\left(f_{x_0}^{-1}(I^{(-)}\sqcup I\sqcup I^{(+)})\right)^q\\
&\geq 3^{-q}\sum_{I\in\D_m}\left(\sum_{\xi\in\D^G_{Rm},\ \supp\ \mu_m\cap\xi\subset f_{x_0}^{-1}(I^{(-)}\sqcup I\sqcup I^{(+)})}\mu_m(\xi)\right)^q\\
&\geq 3^{-q}\sum_{I\in\D_m}\sum_{\xi\in\D^G_{Rm},\ \supp\ \mu_m\cap\xi\subset f_{x_0}^{-1}(I^{(-)}\sqcup I\sqcup I^{(+)})}\mu_m(\xi)^q\\
&\geq 3^{-q}\|\mu_m^{(Rm)}\|_q^q.
\end{align*}
Furthermore, for the coding map $\pi:\I^\N\to K$, the $2^m$ stopping coding map $\pi_m:\I^\N\to K$ and the Bernoulli measure $P$ on $\I^\N$ associated to $(p_i)_{i\in\I}$, we have $\nu=\pi P$ and $f_{x_0}\mu_m=\pi_m P$. From these facts, (\ref{pi_pi_m}) and Lemma \ref{L^q_norms_of_two_partitions}, it follows that
\begin{equation}\label{fx0mum_nu_2^-m_L^q_norm_comparing_proof_of_the_main_theorem}
\sum_{I\in\D_m}\mu_m(f_{x_0}^{-1}I)^q=\sum_{I\in\D_m}P(\pi_m^{-1}I)^q\leq O_{\A,q}(1)\sum_{I\in\D_m}P(\pi^{-1}I)^q=O_{\A,q}(1)\|\nu^{(m)}\|_q^q.
\end{equation}
By combining these two inequalities, we obtain $\|\mu_m^{(Rm)}\|_q^q\leq O_{\A,q}(1)\|\nu^{(m)}\|_q^q$, and hence
\begin{equation}\label{L^q_spectrum_mum_finer_scale_trivial_bound_proof_of_the_main_theorem}
\liminf_{m\to\infty}\left(-\frac{1}{m}\log\|\mu_m^{(Rm)}\|_q^q\right)\geq\lim_{m\to\infty}\left(-\frac{1}{m}\log\|\nu^{(m)}\|_q^q\right)=\tau(q).
\end{equation}

We show $\limsup_{m\to\infty}\left(-m^{-1}\log\|\mu_m^{(Rm)}\|_q^q\right)\leq \tau(q)$. We take arbitrarily small $\sigma>0$.
First, we recall (\ref{estimate_1_existence_limit_Lq_norm}) in the proof of Proposition \ref{existence_of_limit_of_L^q_dim}. For sufficiently large $m\in\N$, we apply (\ref{estimate_1_existence_limit_Lq_norm}) to $m$ and $Rm$ in place of $m$ and $n$, respectively, and obtain that
\begin{align}\label{self_conformality_of_nu_at_a_finer_scale_proof_of_the_main_theorem}
\|\nu^{((R+1)m)}\|_q^q&\leq O_{\A,q}(1)\sum_{I\in\D_m}\sum_{J\in\D_{(R+1)m}}\left(\int_{f_{x_0}^{-1}I}\nu(A^{-1}J)\ d\mu_m(A)\right)^q\nonumber\\
&=O_{\A,q}(1)\sum_{I\in\D_m}\|(\mu_m|_{f_{x_0}^{-1}I}{\bm .}\nu)^{((R+1)m)}\|_q^q\nonumber\\
&=O_{\A,q}(1)\sum_{I\in\D_m,\  \mu_m(f_{x_0}^{-1}I)>0}\mu_m(f_{x_0}^{-1}I)^q\|((\mu_m)_{f_{x_0}^{-1}I}{\bm .}\nu)^{((R+1)m)}\|_q^q.
\end{align}

Here, we define
\begin{equation*}
\D'=\left\{I\in\D_m\left|\ \mu_m(f_{x_0}^{-1}I)>0, \|((\mu_m)_{f_{x_0}^{-1}I})^{(Rm)}\|_q^q\leq 2^{-\sigma m}\right.\right\}.
\end{equation*}
We notice that, by the assumption that $\tau(q)$ is differentiable at $q$, $\tau(q)<q-1$ and $\tau^*(\alpha)>0$, we can apply the $L^q$ norm flattening Theorem \ref{L^q_norm_flattening_theorem} to our $\mu$ and $q$. We take a constant $\varepsilon=\varepsilon(M,\mu,q,\sigma/R)>0$ obtained from Theorem \ref{L^q_norm_flattening_theorem} for $\mu$, $q$ and $\sigma/R>0$. Let $C=C(\A), L=L(\A)>1$ be the constants in (\ref{norm_Omegam_proof_of_the_main_theorem}) and Lemma \ref{bounded_supp_of_restrictions_of_mum_proof_of_the_main_theorem}, respectively, and assume that $m\gg_{M,\mu,q,\sigma/R,\varepsilon,C,L}1$, that is, $m$ is sufficiently large in terms only of $M,\mu,q$ and $\sigma/R$.
For any $I\in\D_m$ with $\mu_m(f_{x_0}^{-1}I)>0$, by Lemma \ref{bounded_supp_of_restrictions_of_mum_proof_of_the_main_theorem}, (\ref{norm_Omegam_proof_of_the_main_theorem}) and Lemma \ref{center_of_expansion_notin_U}, we have $\diam\ \supp\ (\mu_m)_{f_{x_0}^{-1}I}\leq L$, $2^m\leq \|g\|^2\leq C2^m$ and $u_g^-\notin U_1$ for every $g\in\supp\ (\mu_m)_{f_{x_0}^{-1}I}$.
Hence, for each $I\in\D'$, we can apply Theorem \ref{L^q_norm_flattening_theorem} to $\theta=(\mu_m)_{f_{x_0}^{-1}I}$, $r=m$ and $Rm$ in place of $\theta$, $r$ and $m$, respectively.
Therefore, we obtain that
\begin{equation}\label{application_of_the_L^q_norm_flattening_proof_of_the_main_theorem}
\|((\mu_m)_{f_{x_0}^{-1}I}{\bm .}\nu)^{((R+1)m)}\|_q^q\leq 2^{-(\tau(q)+\varepsilon)Rm}
\end{equation}
for each $I\in\D'$.

By (\ref{self_conformality_of_nu_at_a_finer_scale_proof_of_the_main_theorem}), (\ref{application_of_the_L^q_norm_flattening_proof_of_the_main_theorem}) and (\ref{fx0mum_nu_2^-m_L^q_norm_comparing_proof_of_the_main_theorem}), we have
\begin{align}\label{flattened_L^q_norms_application_proof_of_the_main_theorem}
&\|\nu^{((R+1)m)}\|_q^q\nonumber\\
\leq\ &
O_{\A,q}(1)\sum_{I\in\D_m,\  \mu_m(f_{x_0}^{-1}I)>0}\mu_m(f_{x_0}^{-1}I)^q\|((\mu_m)_{f_{x_0}^{-1}I}{\bm .}\nu)^{((R+1)m)}\|_q^q\nonumber\\
\leq\ &O_{\A,q}(1)\sum_{I\in\D'}\mu_m(f_{x_0}^{-1}I)^q\|((\mu_m)_{f_{x_0}^{-1}I}{\bm .}\nu)^{((R+1)m)}\|_q^q\nonumber\\
&+O_{\A,q}(1)\sum_{I\in\D_m\setminus\D',\  \mu_m(f_{x_0}^{-1}I)>0}\mu_m(f_{x_0}^{-1}I)^q\|((\mu_m)_{f_{x_0}^{-1}I}{\bm .}\nu)^{((R+1)m)}\|_q^q\nonumber\\
\leq\ &O_{\A,q}(1)2^{-(\tau(q)+\varepsilon)Rm}\sum_{I\in\D'}\mu_m(f_{x_0}^{-1}I)^q
+O_{\A,q}(1)\|\nu^{(Rm)}\|_q^q\sum_{I\in\D_m\setminus\D',\  \mu_m(f_{x_0}^{-1}I)>0}\mu_m(f_{x_0}^{-1}I)^q\nonumber\\
\leq\ &O_{\A,q}(1)2^{-(\tau(q)+\varepsilon)Rm}\sum_{I\in\D_m}\mu_m(f_{x_0}^{-1}I)^q
+O_{\A,q}(1)\|\nu^{(Rm)}\|_q^q\sum_{I\in\D_m\setminus\D',\  \mu_m(f_{x_0}^{-1}I)>0}\mu_m(f_{x_0}^{-1}I)^q\nonumber\\
\leq\ &O_{\A,q}(1)2^{-(\tau(q)+\varepsilon)Rm}\|\nu^{(m)}\|_q^q+O_{\A,q}(1)\|\nu^{(Rm)}\|_q^q\sum_{I\in\D_m\setminus\D',\  \mu_m(f_{x_0}^{-1}I)>0}\mu_m(f_{x_0}^{-1}I)^q.
\end{align}
In the third inequality, we also used the fact that
\begin{align*}
\|((\mu_m)_{f_{x_0}^{-1}I}{\bm .}\nu)^{((R+1)m)}\|_q^q&=\sum_{J\in\D_{(R+1)m}}\left(\int_G\nu(g^{-1}J)\ d(\mu_m)_{f_{x_0}^{-1}I}(g)\right)^q\\
&\leq \int_G\sum_{J\in\D_{(R+1)m}}\nu(g^{-1}J)^q\ d(\mu_m)_{f_{x_0}^{-1}I}(g)\\
&\leq O_{\A,q}(1)\|\nu^{(Rm)}\|_q^q,
\end{align*}
which follows from the argument we have repeatedly used (e.g., in the proof of Proposition \ref{existence_of_limit_of_L^q_dim} or Lemma \ref{local_L^q_norm_lemma}).
Furthermore, by Proposition \ref{existence_of_limit_of_L^q_dim}, we have
\begin{equation*}
2^{-(\tau(q)+\varepsilon/10)m'}<\|\nu^{(m')}\|_q^q<2^{-(\tau(q)-\varepsilon/10)m'}
\end{equation*}
for sufficiently large $m'\in\N, m'\gg_{\mu,q,\varepsilon}1$. By applying this estimate to $m'=m, Rm, (R+1)m$ ($m\gg_{\mu,q,\varepsilon}1$) and (\ref{flattened_L^q_norms_application_proof_of_the_main_theorem}), we obtain that
\begin{align*}
&2^{-(\tau(q)+\varepsilon/10)(R+1)m}\\
\leq\ & O_{\A,q}(1)2^{-(\tau(q)+\varepsilon)Rm}2^{-(\tau(q)-\varepsilon/10)m}+O_{\A,q}(1)2^{-(\tau(q)-\varepsilon/10)Rm}\sum_{I\in\D_m\setminus\D',\  \mu_m(f_{x_0}^{-1}I)>0}\mu_m(f_{x_0}^{-1}I)^q\\
\leq\ &O_{\A,q}(1)2^{-(\tau(q)+R\varepsilon/(R+1)-\varepsilon/(10(R+1)))(R+1)m}+2^{-(\tau(q)-\varepsilon/5)Rm}\sum_{I\in\D_m\setminus\D',\  \mu_m(f_{x_0}^{-1}I)>0}\mu_m(f_{x_0}^{-1}I)^q\\
\leq\ &2^{-(\tau(q)+\varepsilon/4)(R+1)m}+2^{-(\tau(q)-\varepsilon/5)Rm}\sum_{I\in\D_m\setminus\D',\  \mu_m(f_{x_0}^{-1}I)>0}\mu_m(f_{x_0}^{-1}I)^q
\end{align*}
(we have used $m\gg_{\A,q,\varepsilon}1$ to estimate $O_{\A,q}(1)$). Since $m\gg_\varepsilon1$, we have $2^{-(\tau(q)+\varepsilon/4)(R+1)m}<2^{-1}2^{-(\tau(q)+\varepsilon/10)(R+1)m}$. By combining this and the above inequality, we obtain that
\begin{align}\label{the_L^q_norm_of_fx0mun_restricted_to_complement_D'_proof_of_the_main_theorem}
\sum_{I\in\D_m\setminus\D',\  \mu_m(f_{x_0}^{-1}I)>0}\mu_m(f_{x_0}^{-1}I)^q&\geq 2^{(\tau(q)-\varepsilon/5)Rm}2^{-1}2^{-(\tau(q)+\varepsilon/10)(R+1)m}\nonumber\\
&= 2^{-1}2^{-(\tau(q)+(3R+1)\varepsilon/10)m}\nonumber\\
&>2^{-(\tau(q)+\sigma)m}
\end{align}
(we can assume $\varepsilon<10\sigma/(3R+1)$ and $m\gg_{R,\sigma,\varepsilon}1$).

By the definition of $\D'$ and (\ref{the_L^q_norm_of_fx0mun_restricted_to_complement_D'_proof_of_the_main_theorem}), we have
\begin{align*}
\|\mu_m^{(Rm)}\|_q^q&=\sum_{\xi\in\D^G_{Rm}}\left(\sum_{I\in\D_m}\mu_m|_{f_{x_0}^{-1}I}(\xi)\right)^q\\
&\geq \sum_{\xi\in\D^G_{Rm}}\sum_{I\in\D_m}\mu_m|_{f_{x_0}^{-1}I}(\xi)^q\\
&=\sum_{I\in\D_m,\ \mu_m(f_{x_0}^{-1}I)>0}\mu_m(f_{x_0}^{-1}I)^q\|((\mu_m)_{f_{x_0}^{-1}I})^{(Rm)}\|_q^q\\
&\geq\sum_{I\in\D_m\setminus\D',\ \mu_m(f_{x_0}^{-1}I)>0}\mu_m(f_{x_0}^{-1}I)^q\|((\mu_m)_{f_{x_0}^{-1}I})^{(Rm)}\|_q^q\\
&\geq 2^{-\sigma m}\sum_{I\in\D_m\setminus\D',\ \mu_m(f_{x_0}^{-1}I)>0}\mu_m(f_{x_0}^{-1}I)^q\\
&>2^{-(\tau(q)+2\sigma)m}
\end{align*}
for sufficiently large $m\gg_{M,\mu,q,\sigma,R,\varepsilon}1$. Hence, we have $\limsup_{m\to\infty}(-m^{-1}\log\|\mu_m^{(Rm)}\|_q^q)\leq \tau(q)+2\sigma$. Since $\sigma>0$ is arbitrarily small, we obtain
\begin{equation}\label{L^q_spectrum_mum_finer_scale_nontrivial_bound_proof_of_the_main_theorem}
\limsup_{m\to\infty}\left(-\frac{1}{m}\log\|\mu_m^{(Rm)}\|_q^q\right)\leq \tau(q).
\end{equation}
By (\ref{L^q_spectrum_mum_finer_scale_trivial_bound_proof_of_the_main_theorem}) and (\ref{L^q_spectrum_mum_finer_scale_nontrivial_bound_proof_of_the_main_theorem}), we complete the proof.
\end{proof}

\subsection{The essential part of the main theorem}\label{subsection_essential_part_main_theorem}

In this section, we prove the essential part of the main Theorem \ref{the_main_theorem_L^q_dim_Mobius_IFS}, that is, for a stationary measure as in Theorem \ref{the_main_theorem_L^q_dim_Mobius_IFS}, the case (I) or (II) occurs (“only one of them” part will be discussed in the later Section \ref{subsection_incompatibility}). We take a non-empty finite family $\A=\{A_i\}_{i\in\I}\ (|\I|\geq2)$ of elements of $G$ which is uniformly hyperbolic and whose attractor $K$ is not a singleton (when we further assume the strongly Diophantine condition, we notice that). We also take a probability measure $\mu=\sum_{i\in\I}p_i\delta_{A_i}$ on $G$ such that $\supp\ \mu=\A$ and write $\nu$ for the stationary measure of $\mu$ on $\RP^1$. We write $\tau(q)\ (q>1)$ for the $L^q$ spectrum of $\nu$. Throughout the remaining sections, we fix them.

First, we show some preliminary lemmas. We recall that, for $q>1$, we have defined the pressure function
\begin{equation*}
\Psi_q(s)=\lim_{n\to\infty}\frac{1}{n}\log \sum_{i\in\I^n}p_i^q\|A_i\|^{2s},\quad s>0
\end{equation*}
and its zero $\widetilde{\tau}(q)$ in Section \ref{section_main_problem}. The first lemma gives another description of $\widetilde{\tau}(q)$.

\begin{lem}\label{another_description_of_widetildetau_proof_of_the_main_theorem}
In the setting above, for $q>1$, we have\footnote{Actually, we can show that the limit of the left-hand side exists. However, it is not necessary for us and we do not prove it.}
\begin{equation*}
\liminf_{m\to\infty}\left(-\frac{1}{m}\log\sum_{i\in\Omega_m}p_i^q\right)=\widetilde{\tau}(q).
\end{equation*}
\end{lem}

To prove Lemma \ref{another_description_of_widetildetau_proof_of_the_main_theorem}, we need the following fact on the uniformly hyperbolic family $\A$. This seems non-trivial directly from the definition of the uniform hyperbolicity (Definition \ref{definition_uniformly_hyperbolic}), but is easily understood from the viewpoint of its action on $\RP^1$ and Proposition \ref{proposition_Möbius_uniformly_hyperbolic_SL_2(R)}.

\begin{lem}\label{norm_product_of_two_matrices_proof_of_the_main_theorem}
In the setting above, there exists a constant $\rho=\rho(\A)>0$ determined only by $\A$ such that, for any $i,i'\in\I^*$, we have
\begin{equation*}
\|A_{i'}A_i\|\geq\rho\|A_{i'}\|\|A_i\|.
\end{equation*}
\end{lem}

\begin{proof}
Let $i,i'\in \I^*$. Suppose that $i\in \I^n,i'\in \I^{n'}$ and $n,n'\gg_\A1$ are sufficiently large. In terms of the $G$-action on $\RP^1$, we have
\begin{align*}
	&\angle_{v_{A_{i'}A_i}^+}\circ (A_{i'}A_i)\circ\angle_{u_{A_{i'}A_i}^+}^{-1}\\
	=&\left(\angle_{v_{A_{i'}A_i}^+}\circ\angle_{v_{A_{i'}}^+}^{-1}\right)\circ\left(\angle_{v_{A_{i'}}^+}\circ A_{i'}\circ\angle_{u_{A_{i'}}^+}^{-1}\right)\circ
	\left(\angle_{u_{A_{i'}}^+}\circ\angle_{v_{A_i}^+}^{-1}\right)\circ\left(\angle_{v_{A_i}^+}\circ A_i\circ\angle_{u_{A_i}^+}^{-1}\right)\circ
	\left(\angle_{u_{A_i}^+}\circ\angle_{u_{A_{i'}A_i}^+}^{-1}\right).
\end{align*}
Since we have (\ref{derivative_action_of_G_preliminaries}) and $\angle_{v_{A_{i'}A_i}^+}\circ\angle_{v_{A_{i'}}^+}^{-1}, \angle_{u_{A_{i'}}^+}\circ\angle_{v_{A_i}^+}^{-1}$ and $\angle_{u_{A_i}^+}\circ\angle_{u_{A_{i'}A_i}^+}^{-1}$ are translations on $\R/\pi\Z$ by some elements of $\R/\pi\Z$, by taking the derivative of both sides, we obtain for $\theta\in\R/\pi\Z$ that
\begin{equation*}
	\frac{1}{(\lambda_{A_{i'}A_i}^+)^2\cos^2\theta+(\lambda_{A_{i'}A_i}^+)^{-2}\sin^2\theta}
	=\frac{1}{(\lambda_{A_{i'}}^+)^2\cos^2\theta''+(\lambda_{A_{i'}}^+)^{-2}\sin^2\theta''}\cdot
	\frac{1}{(\lambda_{A_i}^+)^2\cos^2\theta'+(\lambda_{A_i}^+)^{-2}\sin^2\theta'},
\end{equation*}
where $\theta'=\angle_{u_{A_i}^+}(\angle_{u_{A_{i'}A_i}^+}^{-1}(\theta))$ and $\theta''=\angle_{u_{A_{i'}}^+}(A_i(\angle_{u_{A_{i'}A_i}^+}^{-1}(\theta)))$.

We take open sets $U\subset U_1\subsetneq\RP^1$ as in (\ref{domains_the_Möbius_IFS_acts_on_preliminaries}) for $\A$, and $\theta\in \angle_{u_{A_{i'}A_i}^+}(U)$. Then, by $A\overline{U}\subset U$ for any $A\in\A$, we have $\angle_{u_{A_{i'}A_i}^+}^{-1}(\theta)\in U$ and $A_i(\angle_{u_{A_{i'}A_i}^+}^{-1}(\theta))\in U$. Furthermore, since $n,n'\gg_\A1$, we obtain from Lemma \ref{center_of_expansion_notin_U} that $u_{A_i}^-, u_{A_{i'}}^-\notin U_1$. Hence, we can see that there exists a constant $\varepsilon>0$ determined from $\overline{U}\subset U_1$ such that $\theta',\theta''\in \R/\pi\Z\setminus(\pi/2-\varepsilon,\pi/2+\varepsilon)$. From this and the above equation, we obtain that
\begin{align*}
	\frac{1}{(\lambda_{A_{i'}A_i}^+)^2}\leq&\ \frac{1}{(\lambda_{A_{i'}A_i}^+)^2\cos^2\theta+(\lambda_{A_{i'}A_i}^+)^{-2}\sin^2\theta}\\
	=&\ \frac{1}{(\lambda_{A_{i'}}^+)^2\cos^2\theta''+(\lambda_{A_{i'}}^+)^{-2}\sin^2\theta''}\cdot
	\frac{1}{(\lambda_{A_i}^+)^2\cos^2\theta'+(\lambda_{A_i}^+)^{-2}\sin^2\theta'}\\
	\leq&\ \frac{1}{(\lambda_{A_{i'}}^+)^2\cos^2(\pi/2-\varepsilon)}\cdot\frac{1}{(\lambda_{A_i}^+)^2\cos^2(\pi/2-\varepsilon)}.
\end{align*}
Since $\lambda_A^+=\|A\|$ for $A\in G$, we have $\|A_{i'}A_i\|\geq \cos^2(\pi/2-\varepsilon)\|A_{i'}\|\|A_i\|$.

If $n$ or $n'$ are not sufficiently large, we can use $\|A_{i'}\|\|A_i\|=\|A_{i'}A_iA_i^{-1}\|\|A_i\|\leq\|A_i^{-1}\|\|A_i\|\|A_{i'}A_i\|$ and $\bigcup_{n\leq n_0}\I^n$ is finite. Hence, we complete the proof.
\end{proof}

Here, we begin the proof of Lemma \ref{another_description_of_widetildetau_proof_of_the_main_theorem}.

\begin{proof}
We write $\underline{s}=\liminf_{m\to\infty}\left(-m^{-1}\log\sum_{i\in\Omega_m}p_i^q\right)$.
We first show $\underline{s}\geq\widetilde{\tau}(q)$.
We take arbitrary $0<s<\widetilde{\tau}(q)$. Then, we have $\Psi_q(s)<\Psi_q(\widetilde{\tau}(q))=0$. Hence, if we take $\varepsilon>0$ sufficiently small, we have
\begin{equation*}
	\frac{1}{n}\log\sum_{i\in \I^n}p_i^q\|A_i\|^{2s}<-\varepsilon
\end{equation*}
for sufficiently large $n\in\N$ (say $n\geq n_0=n_0(\mu,q,s,\varepsilon)$). By taking $2^x$ and the sum for $n\geq n_0$, we have
\begin{equation*}
	\sum_{i\in \I^*,\ |i|\geq n_0}p_i^q\|A_i\|^{2s}<\sum_{n=n_0}^\infty 2^{-\varepsilon n}=\frac{2^{-\varepsilon n_0}}{1-2^{-\varepsilon}}.
\end{equation*}

Since $\A=\{A_i\}_{i\in \I}$ is finite, we can see that, if $m$ is sufficiently large, then $\Omega_m\subset \bigcup_{n\geq n_0}\I^n$. Hence, by the above inequality, we have
\begin{equation*}
	\sum_{i\in\Omega_m}p_i^q\|A_i\|^{2s}\leq\sum_{i\in \I^*,\ |i|\geq n_0}p_i^q\|A_i\|^{2s}<\frac{2^{-\varepsilon n_0}}{1-2^{-\varepsilon}}.
\end{equation*}
We have $\|A_i\|^2\geq 2^m$ for $i\in\Omega_m$ and $s>0$. So it follows that
\begin{equation*}
	2^{sm}\sum_{i\in\Omega_m}p_i^q<\frac{2^{-\varepsilon n_0}}{1-2^{-\varepsilon}}.
\end{equation*}
By taking the $\log$ and $\liminf_{m\to\infty}$, we obtain that
\begin{equation*}
	\underline{s}=\liminf_{m\to\infty}\left(-\frac{1}{m}\log\sum_{i\in\Omega_m}p_i^q\right)\geq\liminf_{m\to\infty}\left(-\frac{1}{m}\log\left(2^{-sm}\cdot \frac{2^{-\varepsilon n_0}}{1-2^{-\varepsilon}}\right)\right)=s.
\end{equation*}
Since $0<s<\widetilde{\tau}(q)$ is arbitrary, we have $\underline{s}\geq\widetilde{\tau}(q)$.

Next, we see $\widetilde{\tau}(q)\geq \underline{s}$. The strategy is the same as above. We first notice that, since we have already shown that $\underline{s}\geq \widetilde{\tau}(q)>0$, we have $\underline{s}>0$. We take arbitrary $0<s<\underline{s}$. If we take $\varepsilon>0$ sufficiently small, then $0<s+\varepsilon<\underline{s}$, and hence
\begin{equation*}
	-\frac{1}{m}\log\sum_{i\in\Omega_m}p_i^q>s+\varepsilon\iff\sum_{i\in\Omega_m}p_i^q2^{sm}<2^{-\varepsilon m}
\end{equation*}
for sufficiently large $m\in \N$ (say $m\geq m_0=m_0(\mu,q,s,\varepsilon)$). If we write $C=\max_{i\in \I}\|A_i\|^2\geq 1$, then, for each $i\in\Omega_m$, we have $\|A_i\|^2<C2^m$. From this, the above inequality and $s>0$, we have
\begin{equation*}
	\sum_{i\in\Omega_m}p_i^q\|A_i\|^{2s}<C^s2^{-\varepsilon m}
\end{equation*}
for $m\geq m_0$. By taking the sum for $m\geq m_0$, we obtain
\begin{equation}\label{key_bound_zero_of_pressure_function_proof_of_the_main_theorem}
	\sum_{m=m_0}^\infty\sum_{i\in\Omega_m}p_i^q\|A_i\|^{2s}<C^{s}\sum_{m=m_0}^\infty 2^{-\varepsilon m}=\frac{C^s2^{-\varepsilon m_0}}{1-2^{-\varepsilon}}.
\end{equation}

We fix $i_0\in \I$. For $i=(i_1,\dots,i_n)\in \I^*$, let $m_i\in\N$ be the smallest integer such that $\|A_{i_1}\|^2<2^m,\dots, \|A_{i_1}\cdots A_{i_n}\|^2<2^m$ and $k_i\in\N$ be the smallest integer such that $\|A_iA_{i_0}^k\|^2\geq 2^{m_i}$.
By (\ref{def_of_uniform_hyperbolicity_of_A_preliminaries}), if $n\in\N$ is sufficiently large and $i\in \I^n$, then $m_i\geq m_0$.
For $i=(i_1,\dots,i_n)\in \I^n$, let $n'\in\{1,\dots,n\}$ be such that $\|A_{i_1}\cdots A_{i_{n'}}\|^2\geq2^{m_i-1}$ (by the minimality of $m_i$). Then, for $k\in \N$, we have by Lemma \ref{norm_product_of_two_matrices_proof_of_the_main_theorem} and (\ref{def_of_uniform_hyperbolicity_of_A_preliminaries}) that
\begin{align*}
	\|A_iA_{i_0}^k\|^2\geq&\ \rho^2\|A_i\|^2\|A_{i_0}^k\|^2\\
	\geq&\ \rho^4\|A_{i_1}\cdots A_{i_{n'}}\|^2\|A_{i_{n'+1}}\cdots A_{i_n}\|^2\|A_{i_0}^k\|^2\\
	\geq&\ c^2\rho^42^{m_i-1}r^{2k}.
\end{align*}
So we have
\begin{equation}\label{upper_bound_ki_zero_of_pressure_function_proof_of_the_main_theorem}
	k_i\leq\left\lceil\frac{\log(2c^{-2}\rho^{-4})}{2\log r}\right\rceil.
\end{equation}
Let $n\in \N$ be sufficiently large. Then, for $i\in \I^n$, we have
\begin{equation*}
	ii_0^{k_i}=(i,\underbrace{i_0,\dots ,i_0}_{k_i})\in\Omega_{m_i},\quad m_i\geq m_0.
\end{equation*}
Hence, by (\ref{key_bound_zero_of_pressure_function_proof_of_the_main_theorem}), we have
\begin{equation*}
	\sum_{i\in \I^n}p_{ii_0^{k_i}}^q\|A_{ii_0^{k_i}}\|^{2s}\leq\sum_{m=m_0}^\infty\sum_{i\in\Omega_m}p_i^q\|A_i\|^{2s}<\frac{C^s2^{-\varepsilon m_0}}{1-2^{-\varepsilon}}.
\end{equation*}
Furthermore, by $s>0$, Lemma \ref{norm_product_of_two_matrices_proof_of_the_main_theorem} and (\ref{upper_bound_ki_zero_of_pressure_function_proof_of_the_main_theorem}), we have for $i\in \I^n$ that
\begin{equation*}
	p_{ii_0^{k_i}}^q\|A_{ii_0^{k_i}}\|^{2s}\geq p_{i_0}^{qk_i}p_i^q\cdot\rho^{2s}\|A_i\|^{2s}\|A_{i_0}^{k_i}\|^{2s}\geq \rho^{2s}p_{i_0}^{q\lceil\log(2c^{-2}\rho^{-4})/(2\log r)\rceil}p_i^q\|A_i\|^{2s}.
\end{equation*}
Therefore, we obtain that
\begin{equation*}
	\sum_{i\in \I^n}p_i^q\|A_i\|^{2s}<\rho^{-2s}p_{i_0}^{-q\lceil\log(2c^{-2}\rho^{-4})/(2\log r)\rceil}\cdot
	\frac{C^s2^{-\varepsilon m_0}}{1-2^{-\varepsilon}}.
\end{equation*}
The right-hand side is a constant independent of $n$, so we have
\begin{equation*}
	\Psi_q(s)=\lim_{n\to\infty}\frac{1}{n}\log\sum_{i\in I^n}p_i^q\|A_i\|^{2s_1'}\leq 0=\Psi_q(\widetilde{\tau}(q)).
\end{equation*}
Since $\Psi_q$ is non-decreasing, we have $s\leq \widetilde{\tau}(q)$. Since $0<s<\underline{s}$ is arbitrary, we finally obtain $\underline{s}\leq \widetilde{\tau}(q)$, and complete the proof.
\end{proof}

The second preliminary lemma gives us a further consequence of the strongly Diophantine condition.

\begin{lem}\label{strongly_Diophantine_Omegam_proof_of_the_main_theorem}
In addition to the setting above, we assume that $\A$ is strongly Diophantine.
Then, there exists a constant $a'=a'(\A)>0$ such that, for any $m\in\N$, we have
\begin{equation*}
i,i'\in\Omega_m, i\neq i'\implies d_G(A_i,A_{i'})\geq a'^m.
\end{equation*}
\end{lem}

\begin{proof}
We recall (Definition \ref{dfn_strongly_Diophantine_condition}) that $\A=\{A_i\}_{i\in\I}$ is strongly Diophantine means that there is a constant $a>0$ such that, for any $n\in\N$,
\begin{equation*}
i,i'\in\I^n, i\neq i'\implies d_G(A_i,A_{i'})\geq a^n.
\end{equation*}
Here, by the left-invariance of the metric $d_G$, we have $d_G(A_i,A_{i'})=d_G(A_{i'}^{-1}A_i,1_G)\geq a^n$, and, by the bi-Lipschitz equivalence of $d_G$ and the norm metric on a small neighborhood of $1_G$ (see \cite[Section 2.3]{HS17}),
\begin{equation*}
\|A_{i'}^{-1}A_i-1_G\|\geq \eta a^n
\end{equation*}
for a constant $\eta=\eta(G)>0$. We also have
\begin{equation*}
\|A_{i'}^{-1}A_i-1_G\|=\|A_{i'}^{-1}(A_i-A_{i'})\|\leq \|A_{i'}^{-1}\|\|A_i-A_{i'}\|\leq R^n\|A_i-A_{i'}\|,
\end{equation*}
where $R=R(\A)=\max_{i\in\I}\{\|A_i\|, \|A_i^{-1}\|\}\geq1$. By combining these two inequalities, we can see that there is a constant $a'>0$ such that, for any $n\in\N$
\begin{equation}\label{another_def_of_strongly_Diophantine_proof_of_the_main_theorem}
i,i'\in\I^n,i\neq i'\implies \|A_i-A_{i'}\|\geq a'^n.
\end{equation}

For $m\in\N$, we see the length of any $i=(i_1,\dots,i_n)\in\Omega_m$. By (\ref{def_of_uniform_hyperbolicity_of_A_preliminaries}), we have
\begin{equation*}
2^m>\|A_{i_1}\cdots A_{i_{n-1}}\|^2\geq c^2r^{2(n-1)},\quad\text{and hence }n<\frac{m-2\log c}{2\log r}+1.
\end{equation*}
Therefore, there is a constant $b=b(\A)>1$ such that, for any $m\in\N$ and $i\in\Omega_m$,
\begin{equation}\label{length_of_elements_of_Omegam_proof_of_the_main_theorem}
|i|\leq bm.
\end{equation}

We take $m\in\N$ and $i=(i_1,\dots,i_n),i'=(i'_1,\dots,i'_{n'})\in\Omega_m$ such that $i\neq i'$. We assume $n\leq n'$. Then, we have
\begin{equation*}
i\neq (i'_1,\dots,i'_n)\in\I^n.
\end{equation*}
In fact, if $i= (i'_1,\dots,i'_n)\in\I^n$, we have by $i\neq i'$ that $n<n'$ and, by $i\in\Omega_m$, $\|A_i\|^2=\|A_{i'_1}\cdots A_{i'_n}\|^2\geq 2^m$, but this contradicts $i'=(i'_1,\dots,i'_n,i'_{n+1},\dots, i'_{n'})\in\Omega_m$. From this fact, we can see that $ii'\neq i'i\in\I^{n+n'}$. Hence, by applying (\ref{another_def_of_strongly_Diophantine_proof_of_the_main_theorem}) to $ii'$ and $i'i$, we have
\begin{equation}\label{ii'_i'i_distinct_Diohantine_proof_of_the_main_theorem}
\|A_iA_{i'}-A_{i'}A_i\|\geq a'^{n+n'}.
\end{equation}

Let $B=A_{i'}-A_i$. Then, we have
\begin{align*}
\|A_iA_{i'}-A_{i'}A_i\|&=\|A_i(B+A_i)-(B+A_i)A_i\|\\
&=\|A_iB-BA_i\|\\
&\leq \|A_iB\|+\|BA_i\|\\
&\leq 2R^n\|B\|.
\end{align*}
From this inequality and (\ref{ii'_i'i_distinct_Diohantine_proof_of_the_main_theorem}), we have
\begin{equation*}
\|B\|=\|A_{i'}-A_i\|\geq \frac{R^{-n}a'^{n+n'}}{2}.
\end{equation*}
Since $\|A_{i'}-A_i\|=\|A_{i'}(1_G-A_{i'}^{-1}A_i)\|\leq R^{n'}\|1_G-A_{i'}^{-1}A_i\|$, we obtain from the above inequality that
\begin{equation*}
\|1_G-A_{i'}^{-1}A_i\|\geq\frac{R^{-(n+n')}a'^{n+n'}}{2}.
\end{equation*}
From this inequality, the bi-Lipschitz equivalence of $d_G$ and the norm metric on a neighborhood of $1_G$ and (\ref{length_of_elements_of_Omegam_proof_of_the_main_theorem}), we obtain that
\begin{equation*}
d_G(A_i,A_{i'})=d_G(A_{i'}^{-1}A_i,1_G)\geq\frac{\eta R^{-(n+n')}a'^{n+n'}}{2}\geq \frac{\eta(R^{-1}a')^{2bm}}{2}
\end{equation*}
(where $0<\eta=\eta(G)<1$ and we assume $a'<1$). Hence, by re-writing $a'=a'(\A)>0$ for a constant $\eta(R^{-1}a')^{2b}/2$, we complete the proof.
\end{proof}

As the third preliminary lemma, we notice the fact that, once we get $\tau^*(\alpha_0)=0$ for some $q_0>1$ and $\alpha_0=\tau'(q_0)$, $\tau(q)$ coincides the line $\alpha_0q$ through the origin for $q\geq q_0$.

\begin{lem}\label{tau(q)_coincides_the_line_through_the_origin_proof_of_the_main_theorem}
In the setting above,
let $q_0>1$ and assume that $\tau(q)$ is differentiable at $q_0$ and, if we write $\alpha_0=\tau'(q_0)$, we have
\begin{equation*}
\tau^*(\alpha_0)=0.
\end{equation*}
Then, we have
\begin{equation*}
\tau(q)=\alpha_0 q\quad\text{for any }q\geq q_0.
\end{equation*}
\end{lem}

\begin{proof}
Let $q_1>q_0$ be such that $\tau(q)$ is differentiable at $q_1$ and $\alpha_1=\tau'(q_1)$. Then, by (\ref{Legendre_transform_at_derivative}), we have $\tau(q_1)=\alpha_1q_1-\tau^*(\alpha_1)$, and hence the line $t=\alpha_1q-\tau^*(\alpha_1)$ in the $(q,t)$-plane is the tangent to $t=\tau(q)$ at $q=q_1$. By the concavity of $\tau(q)$, the graph of $\tau(q)$ is below this tangent. Hence, by the assumption of $\tau^*(\alpha_0)=\alpha_0q_0-\tau(q_0)=0$, we have for $q=q_0$ that
\begin{equation}\label{graph_of_tau_is_below_the_tangent}
\tau(q_0)=\alpha_0q_0\leq\alpha_1q_0-\tau^*(\alpha_1)\iff\tau^*(\alpha_1)\leq(\alpha_1-\alpha_0)q_0.
\end{equation}

Since $\tau(q)$ is concave and $q_1>q_0$, we have $\alpha_1\leq\alpha_0$. Furthermore, by Corollary \ref{tau^*(alpha)_is_nonnegative}, we have $\tau^*(\alpha_1)\geq 0$. From these facts and (\ref{graph_of_tau_is_below_the_tangent}), we obtain that
\begin{equation*}
	\tau^*(q_1)=\alpha_1-\alpha_0=0.
\end{equation*}
Hence, we have $\tau(q_1)=\alpha_0q_1$. Since the points $q_1>q_0$ such that $\tau(q)$ is differentiable at $q_1$ are dense in $(q_0,\infty)$ and $\tau(q)$ is continuous, this equation holds for any $q_1\geq q_0$.
\end{proof}

Here, we begin the proof of the essential part of the main Theorem \ref{the_main_theorem_L^q_dim_Mobius_IFS}.
In the proof, we use the fact that $\widetilde{\tau}(q)\ (q>1)$ is continuous, which is included in Proposition \ref{prop_analyticity_of_zero_of_pressure_function}\footnote{Actually, we can give a short proof only of the continuity, but we do not do it.}. We prove Proposition \ref{prop_analyticity_of_zero_of_pressure_function} in the later Section \ref{subsection_incompatibility}.

\begin{proof}[Proof of the essential part of Theorem \ref{the_main_theorem_L^q_dim_Mobius_IFS}]
We assume that $\A$ is strongly Diophantine. We show that either the following (I) or (II) holds (“only one of them” will be shown in the later Section \ref{subsection_incompatibility}):
\begin{enumerate}
\renewcommand{\labelenumi}{(\Roman{enumi})}
\item we have
\begin{equation*}
\tau(q)=\min\left\{\widetilde{\tau}(q),q-1\right\}\quad\text{for any }q>1,
\end{equation*}
\item there exist $q_0>1$ and $0<\alpha<1$ such that
\begin{equation*}
\tau(q)=
\begin{cases}
\min\left\{\widetilde{\tau}(q),q-1\right\}&\text{if }1<q<q_0,\\
\alpha q&\text{if }q\geq q_0.
\end{cases}
\end{equation*}
\end{enumerate}

First, without using the assumption of the strongly Diophantine condition, we show the trivial bound:
\begin{equation}\label{trivial_bound_for_L^q_spectrum_proof_of_the_main_theorem}
\tau(q)\leq\min\{\widetilde{\tau}(q),q-1\}
\end{equation}
for any $q>1$. Since we have noticed that $\tau(q)\leq q-1$ in Definition \ref{dfn_L^d_dimension}, it is sufficient to show $\tau(q)\leq\widetilde{\tau}(q)$.
We take $x_0\in K$ and fix it. For sufficiently large $m\in\N$, by (\ref{fx0mum_nu_2^-m_L^q_norm_comparing_proof_of_the_main_theorem}), we have
\begin{equation*}
\|\nu^{(m)}\|_q^q\geq\Omega_{\A,q}(1)\sum_{I\in\D_m}\mu_m(f_{x_0}^{-1}I)^q.
\end{equation*}
However, the sum on the right-hand side can be made smaller by breaking it up into atoms of $\mu_m$. Hence, we obtain from the definition of $\mu_m$ that
\begin{equation*}
\sum_{I\in\D_m}\mu_m(f_{x_0}^{-1}I)^q\geq\sum_{i\in\Omega_m}p_i^q.
\end{equation*}
By combining these two inequalities and Lemma \ref{another_description_of_widetildetau_proof_of_the_main_theorem}, we have
\begin{equation*}
\tau(q)=\liminf_{m\to\infty}\left(-\frac{1}{m}\log\|\nu^{(m)}\|_q^q\right)\leq\liminf_{m\to\infty}\left(-\frac{1}{m}\log\sum_{i\in\Omega_m}p_i^q\right)=\widetilde{\tau}(q),
\end{equation*}
and hence obtain (\ref{trivial_bound_for_L^q_spectrum_proof_of_the_main_theorem}).

We write
\begin{equation*}
S=\left\{q>1\left|\ \text{$\tau(q)$ is differentiable at $q$}\right.\right\}.
\end{equation*}
As we have seen right after (\ref{Legendre_transform_at_derivative}), $S$ is dense in $(1,\infty)$.
We take $q\in S$ and assume that
\begin{equation}\label{assumption_tau*alpha_positive_proof_of_the_main_theorem}
\tau^*(\alpha)>0\quad\text{for }\alpha=\tau'(q).
\end{equation}
If $\tau(q)=q-1$, we have nothing to prove. So we also assume $\tau(q)<q-1$. Let $a'=a'(\A)>0$ be the constant given by Lemma \ref{strongly_Diophantine_Omegam_proof_of_the_main_theorem}. Then, we can take a constant $R\in\N$ so large that $a'^m>M2^{-Rm}$ holds for any $m\in\N$. By the above assumptions, we can apply Proposition \ref{the_L^q_morm_of_mum_at_a_finer_scale_proof_of_the_main_theorem} to our $q$ and $R$, and obtain that
\begin{equation}\label{tau_equals_finer_L^q_spectrum_of_mum_proof_of_the_main_theorem}
\tau(q)=\lim_{m\to\infty}\left(-\frac{1}{m}\log\|\mu_m^{(Rm)}\|_q^q\right).
\end{equation}
However, since $\diam\ \xi\leq M2^{-Rm}<a'^m$ for any $\xi\in\D^G_{Rm}$, by Lemma \ref{strongly_Diophantine_Omegam_proof_of_the_main_theorem} and the definition of $\mu_m$, we have that each $\xi\in\D^G_{Rm}$ can contain at most one atom of $\mu_m$, and hence
\begin{equation*}
\|\mu_m^{(Rm)}\|_q^q=\sum_{\xi\in\D^G_{Rm}}\mu_m(\xi)^q=\sum_{i\in\Omega_m}p_i^q
\end{equation*}
for any $m\in\N$. Therefore, we obtain from (\ref{tau_equals_finer_L^q_spectrum_of_mum_proof_of_the_main_theorem}) and Lemma \ref{another_description_of_widetildetau_proof_of_the_main_theorem} that
\begin{equation*}
\tau(q)=\lim_{m\to\infty}\left(-\frac{1}{m}\log\sum_{i\in\Omega_m}p_i^q\right)=\widetilde{\tau}(q).
\end{equation*}

In the above argument, we showed the following: for $q\in S$, if we have (\ref{assumption_tau*alpha_positive_proof_of_the_main_theorem}), then it holds that $\tau(q)=q-1$ or $\widetilde{\tau}(q)$, and hence, together with (\ref{trivial_bound_for_L^q_spectrum_proof_of_the_main_theorem}),
\begin{equation}\label{desired_equation_for_L^q_spectrum_proof_of_the_main_theorem}
\tau(q)=\min\{\widetilde{\tau}(q), q-1\}.
\end{equation}
Therefore, if we assume that (\ref{assumption_tau*alpha_positive_proof_of_the_main_theorem}) holds for any $q\in S$, then we have (\ref{desired_equation_for_L^q_spectrum_proof_of_the_main_theorem}) for any $q\in S$.
Furthermore, $S$ is dense in $(1,\infty)$ and both sides of (\ref{desired_equation_for_L^q_spectrum_proof_of_the_main_theorem}) are continuous in $q>1$. Hence, (\ref{desired_equation_for_L^q_spectrum_proof_of_the_main_theorem}) holds for any $q>1$, and we obtain (I).

We assume that there is $q_1\in S$ such that (\ref{assumption_tau*alpha_positive_proof_of_the_main_theorem}) does not hold for $q=q_1$. We write $\alpha_1=\tau'(q_1)$. Then, by Lemma \ref{tau(q)_coincides_the_line_through_the_origin_proof_of_the_main_theorem}, we have
\begin{equation}\label{tau(q)_coincides_alphaq_qgeqq1_proof_of_the_main_theorem}
\tau(q)=\alpha_1 q\quad\text{for any }q\geq q_1.
\end{equation}
Since $\tau(q)$ is non-decreasing, concave, $\tau(q)\in[0,q-1]\ (q>1)$ and (\ref{tau(q)_coincides_alphaq_qgeqq1_proof_of_the_main_theorem}) holds, we can see that $0\leq\alpha_1<1$.

Furthermore, we show $\alpha_1>0$.
If $\alpha_1=0$, we have $-\tau(q_1)=\lim_{m\to\infty}m^{-1}\log\|\nu^{(m)}\|_{q_1}^{q_1}=0$. We recall the proof of Proposition \ref{existence_of_limit_of_L^q_dim} where we have shown that the sequence $\left(\log(C'\|\nu^{(m)}\|_{q_1}^{q_1})\right)_{m=1}^\infty$ is subadditive for some constant $C'>0$. Then, we have
\begin{equation*}
	\inf_{m\in\N}\frac{1}{m}\log\left(C'\|\nu^{(m)}\|_{q_1}^{q_1}\right)=\lim_{m\to\infty}\frac{1}{m}\log\left(C'\|\nu^{(m)}\|_{q_1}^{q_1}\right)=0.
\end{equation*}
This tells us that
\begin{equation*}
	\|\nu^{(m)}\|_{q_1}^{q_1}=\sum_{I\in\D_m}\nu(I)^{q_1}\geq C'^{-1}
\end{equation*}
for every $m\in\N$. Hence, we have
\begin{equation*}
	C'^{-1}\leq\sum_{I\in\D_m}\nu(I)^{q_1}\leq\max_{I\in\D_m}\nu(I)^{q_1-1}\sum_{I\in\D_m}\nu(I)=\max_{I\in\D_m}\nu(I)^{q_1-1}
\end{equation*}
for every $m\in\N$. This implies that $\nu$ is atomic. We can see from $\nu=\mu{\bm .}\nu$ and $\supp\ \mu=\A$ that the set $X$ of points in $K$ with the maximal $\nu$-mass is finite and $\A$-invariant. However, since the action of $\A$ on $K$ is contracting, this implies that $X$ consists of one point which is fixed by every element of $\A$, and this contradicts that the attractor is not a singleton. Hence, we obtain that $\alpha_1>0$.

If we take another point $q\in S$ such that (\ref{assumption_tau*alpha_positive_proof_of_the_main_theorem}) does not hold for $q$, then we can apply Lemma \ref{tau(q)_coincides_the_line_through_the_origin_proof_of_the_main_theorem} to this $q$ and obtain that $\tau(q')=\alpha q'$ for any $q'\geq q$, where $\alpha=\tau'(q)$. However, this and (\ref{tau(q)_coincides_alphaq_qgeqq1_proof_of_the_main_theorem}) hold simultaneously, and hence we have $\alpha=\alpha_1$.
Therefore, together with (\ref{desired_equation_for_L^q_spectrum_proof_of_the_main_theorem}), we obtain for any $q\in S$ that
\begin{equation}\label{either_desired_or_alpha1q_proof_of_the_main_theorem}
\tau(q)=
\begin{cases}
\min\{\widetilde{\tau}(q),q-1\}&\text{if }\tau^*(\alpha)>0,\\
\alpha_1 q&\text{if }\tau^*(\alpha)=0,
\end{cases}
\quad\text{where }\alpha=\tau'(q),
\end{equation}
and, once we have the latter case, $\tau(q')=\alpha_1q'$ for any $q'\geq q$.

We define
\begin{equation*}
q_0=\inf\left\{q\in S\left|\ \tau^*(\alpha)=0\text{ for }\alpha=\tau'(q)\right.\right\}.
\end{equation*}
We show $q_0>1$. If $q_0=1$, then there is a sequence $(q_k)_{k=1}^\infty\subset S$ such that $q_k\to 1$ as $k\to\infty$ and, for each $k$, $\tau^*(\alpha_k)=0$ for $\alpha_k=\tau'(q_k)$. By (\ref{either_desired_or_alpha1q_proof_of_the_main_theorem}), we have $\tau(q_k)=\alpha_1q_k$ for each $k$. However, $\tau(q_k)\to 0$ and $\alpha_1q_k\to\alpha_1>0$ as $k\to\infty$ and this is a contradiction. Hence, we obtain that $q_0>1$.

Let $q\in S$. If $q<q_0$, then $\tau^*(\alpha)>0$ for $\alpha=\tau'(q)$ and, by (\ref{either_desired_or_alpha1q_proof_of_the_main_theorem}), we have $\tau(q)=\min\{\widetilde{\tau}(q),q-1\}$. If $q>q_0$, then there is $q'\in S$ such that $q>q'$ and $\tau(\alpha')=0$ for $\alpha'=\tau'(q')$. Hence, by (\ref{either_desired_or_alpha1q_proof_of_the_main_theorem}), we have $\tau(q)=\alpha_1q$.
Therefore, we obtain that
\begin{equation*}
\tau(q)=
\begin{cases}
\min\left\{\widetilde{\tau}(q),q-1\right\}&\text{if }1<q<q_0,\\
\alpha_1 q&\text{if }q> q_0.
\end{cases}
\end{equation*}
for any $q\in S$. Since $S$ is dense in $(1,\infty)$ and $\tau(q)$, $\min\{\widetilde{\tau}(q),q-1\}$ and $\alpha_1q$ are continuous in $q>1$, we obtain (II) for $\alpha_1$ in place of $\alpha$. Then, we complete the proof.
\end{proof}

\subsection{Incompatibility of the two cases}\label{subsection_incompatibility}

In this section, we show that the two cases (I) and (II) of the main Theorem 1.8 do not hold simultaneously.
To do this, we use thermodynamic formalism on the full shift space and deduce that the pressure function $\Psi_q(s)$ is analytic in $(q,s)\in\R^2$.

We consider the space $\I^\N$ as the one-sided full shift space and write $\sigma:\I^\N\to\I^\N$ for the shift map. For a real-valued function $\phi:\I^\N\to\R$, we say that $\phi$ is {\it Hölder continuous} if there exist $0<\theta<1$ and $L>0$ such that, for any $n\in\N$,
\begin{equation*}
\omega,\omega'\in\I^\N, \omega|_n=\omega'|_n\implies |\phi(\omega)-\phi(\omega')|\leq L\theta^n,
\end{equation*}
where $\omega|_n=(\omega_1,\dots,\omega_n)\in\I^n$ for $\omega=(\omega_1,\omega_2,\dots)\in\I^\N$. For such $\phi$, the {\it pressure} $P(\phi)$ of $\phi$ is defined by
\begin{equation*}
P(\phi)=\lim_{n\to\infty}\frac{1}{n}\log\sum_{i\in\I^n}\exp\left(\sup_{\omega\in [i]}\left(\sum_{k=0}^{n-1}\phi(\sigma^k\omega)\right)\right)\in\R,
\end{equation*}
where, for $i\in\I^n$, $[i]=\left\{\omega\in\I^\N\left|\ \omega|_n=i\right.\right\}$ is the associated cylinder set.

Here, we define two functions $\varphi,\psi: \I^\N\to\R$ by
\begin{equation*}
\varphi(\omega)=\log p_{\omega_1},\quad\psi(\omega)=-\log A_{\omega_1}'(\pi(\sigma\omega)),\quad \omega\in\I^\N,
\end{equation*}
where $\pi:\I^\N\to K$ is the coding map associated to $\A$ and, for $A\in G$ and $x\in\RP^1$, $A'(x)=(\angle_{v_A^+}\circ A\circ\angle_{u_A^+}^{-1})'(\angle_{u_A^+}(x))$.
It is obvious that $\varphi:\I^\N\to\R$ is Hölder continuous. Furthermore, we have the following.

\begin{lem}\label{Holder_continuity_geometric_potential_proof_of_the_main_theorem}
The function $\psi:\I^\N\to\R$ is Hölder continuous.
\end{lem}

\begin{proof}
Let $n\in\N$ and $\omega,\omega'\in\I^\N$ be such that $\omega|_n=\omega'|_n$. We write $\omega_1=\omega'_1=i\in\I$. By (\ref{derivative_action_of_G_preliminaries}), we have $A_i'(x)\geq \|A_i\|^{-2}\geq\Omega_{\A}(1)$ for any $x\in\RP^1$. Therefore, by the mean value theorem for $\log t\ (t\geq\Omega_\A(1))$, we have
\begin{equation*}
\left|\psi(\omega)-\psi(\omega')\right|=\left|\log A_i'(\pi(\sigma\omega))-\log A_i'(\pi(\sigma\omega'))\right|
\leq O_{\A}(1)\left|A_i'(\pi(\sigma\omega))-A_i'(\pi(\sigma\omega'))\right|.
\end{equation*}
By this estimate and the Lipschitz continuity of $A_i':\RP^1\to\R$, we have
\begin{equation*}
\left|\psi(\omega)-\psi(\omega')\right|\leq O_{\A}(1)d_{\RP^1}(\pi(\sigma\omega),\pi(\sigma\omega')).
\end{equation*}
Since $\sigma\omega|_{n-1}=\sigma\omega'|_{n-1}$ (here we assume $n\geq 2$), if we write $(j_1,\dots,j_{n-1})=\sigma\omega|_{n-1}=\sigma\omega'|_{n-1}\in\I^{n-1}$, we obtain from Corollary \ref{contraction_on_U_by_A_preliminaries} and (\ref{def_of_uniform_hyperbolicity_of_A_preliminaries}) that
\begin{equation*}
d_{\RP^1}(\pi(\sigma\omega),\pi(\sigma\omega'))\leq\frac{C_1\pi}{\|A_{j_1}\cdots A_{j_{n-1}}\|^2}\leq C_1c^{-2}\pi r^{-2(n-1)}.
\end{equation*}
The above two inequalities show the Hölder continuity of $\psi$.
\end{proof}

For the two Hölder continuous functions $\varphi,\psi:\I^\N\to\R$, it is well-known that the function
$\R^2\ni(q,s)\mapsto P(q\varphi+s\psi)\in\R$ is real analytic on $\R^2$ (see \cite{Rue04}). Here, for any $n\in\N$ sufficiently large, $i\in\I^n$ and $\omega\in [i]$, we have
\begin{equation}\label{ergodic_sum_of_varphi_proof_of_the_main_theorem}
q\sum_{k=0}^{n-1}\varphi(\sigma^k\omega)=\log(p_i^q)
\end{equation}
and, by $A_{i_k}(\pi(\sigma^k\omega))=\pi(\sigma^{k-1}\omega)\ (2\leq k\leq n)$,
\begin{align*}
s\sum_{k=0}^{n-1}\psi(\sigma^k\omega)&=-s\log\left(A_{i_1}'(\pi(\sigma\omega))A_{i_2}'(\pi(\sigma^2\omega))\cdots A_{i_n}'(\pi(\sigma^n\omega))\right)\\
&=-s\log\left((A_{i_1}A_{i_2}\cdots A_{i_n})'(\pi(\sigma^n\omega))\right)\\
&=-s\log(A_i'(\pi(\sigma^n\omega))).
\end{align*}
Since $\pi(\sigma^n\omega)\in K$, by Lemma \ref{center_of_expansion_notin_U} and (\ref{derivative_action_of_G_preliminaries}), we also have
\begin{equation*}
A_i'(\pi(\sigma^n\omega))=\Theta_\A\left(\|A_i\|^{-2}\right).
\end{equation*}
From the last two equations, we obtain that
\begin{equation}\label{ergodic_sum_of_psi_proof_of_the_main_theorem}
s\sum_{k=0}^{n-1}\psi(\sigma^k\omega)=\log\left(\|A_i\|^{2s}\right)+O_\A(s).
\end{equation}
By (\ref{ergodic_sum_of_varphi_proof_of_the_main_theorem}) and (\ref{ergodic_sum_of_psi_proof_of_the_main_theorem}), we have
\begin{align*}
P(q\varphi+s\psi)&=
\lim_{n\to\infty}\frac{1}{n}\log\sum_{i\in\I^n}\exp\left(\sup_{\omega\in [i]}\left(q\sum_{k=0}^{n-1}\varphi(\sigma^k\omega)+s\sum_{k=0}^{n-1}\psi(\sigma^k\omega)\right)\right)\\
&=\lim_{n\to\infty}\frac{1}{n}\log\sum_{i\in\I^n}p_i^q\|A_i\|^{2s}.
\end{align*}
This equation and the analyticity stated above tell us the following.

\begin{prop}\label{the_pressure_function_is_analytic_on_R^2_proof_of_the_main_theorem}
The pressure function
\begin{equation*}
\Psi_q(s)=\lim_{n\to\infty}\frac{1}{n}\log\sum_{i\in\I^n}p_i^q\|A_i\|^{2s}
\end{equation*}
is well-defined for any $(q,s)\in\R^2$ and real analytic in $(q,s)\in\R^2$.
\end{prop}

Using Proposition \ref{the_pressure_function_is_analytic_on_R^2_proof_of_the_main_theorem}, we give the proof of the incompatibility of the cases (I) and (II) of Theorem \ref{the_main_theorem_L^q_dim_Mobius_IFS}.

\begin{proof}[Proof of the incompatibility of (I) and (II) of Theorem \ref{the_main_theorem_L^q_dim_Mobius_IFS}]
Assume that the cases (I) and (II) hold simultaneously. Then, there is $q_0>1$ such that
\begin{equation*}
\min\{\widetilde{\tau}(q),q-1\}=\alpha q\quad\text{for any }q\geq q_0.
\end{equation*}
Since $\alpha<1$ and $q-1>\alpha q$ for sufficiently large $q$, this equation says that $\widetilde{\tau}(q)=\alpha q$ for sufficiently large $q$. Hence, by the definition of $\widetilde{\tau}(q)$, we have
\begin{equation}\label{alphaq_is_the_zero_of_the_pressure_function_proof_of_the_main_theorem}
\Psi_q(\alpha q)=0
\end{equation}
for sufficiently large $q$. However, by Proposition \ref{the_pressure_function_is_analytic_on_R^2_proof_of_the_main_theorem}, this fact implies that (\ref{alphaq_is_the_zero_of_the_pressure_function_proof_of_the_main_theorem}) holds for any $q\in\R$. In particular, by letting $q=0$ in (\ref{alphaq_is_the_zero_of_the_pressure_function_proof_of_the_main_theorem}), we have
\begin{equation*}
0=\Psi_0(0)=\lim_{n\to\infty}\frac{1}{n}\log|\I^n|=\log |\I|,
\end{equation*}
but this contradicts $|\I|\geq 2$. Hence, we complete the proof.
\end{proof}

As an appendix, we give a proof of Proposition \ref{prop_analyticity_of_zero_of_pressure_function}.

\begin{proof}[Proof of Proposition \ref{prop_analyticity_of_zero_of_pressure_function}]
For a Hölder continuous function $\phi:\I^\N\to\R$, we write $P_\phi$ for the equilibrium measure for $\phi$ on $\I^\N$. Then, it is known that
\begin{equation}\label{derivative_of_pressure_function_proof_of_the_main_theorem}
\frac{\partial P(q\varphi+s\psi)}{\partial s}=\int_{\I^\N}\psi\ dP_{q\varphi+s\psi}
\end{equation}
for any $(q,s)\in\R^2$ (see \cite{Rue04}). Since $P_{q\varphi+s\psi}$ is $\sigma$-invariant, by (\ref{ergodic_sum_of_psi_proof_of_the_main_theorem}) and (\ref{def_of_uniform_hyperbolicity_of_A_preliminaries}), we have for sufficiently large $n\in\N$ that
\begin{align*}
\int_{\I^\N}\psi\ dP_{q\varphi+s\psi}&=\frac{1}{n}\int_{\I^\N}\sum_{k=0}^{n-1}\psi(\sigma^k\omega)\ dP_{q\varphi+s\psi}(\omega)\\
&=\frac{1}{n}\int_{\I^\N}\left(\log(\|A_{\omega_1}\cdots A_{\omega_n}\|^2)+O_{\A}(1)\right)\ dP_{q\varphi+s\psi}(\omega)\\
&\geq 2\log r+O_{\A}\left(\frac{1}{n}\right).
\end{align*}
By taking $n\gg_\A1$ sufficiently large, we can see that the above value is positive. Hence, by (\ref{derivative_of_pressure_function_proof_of_the_main_theorem}), we have
\begin{equation}\label{partial_derivative_nonzero_proof_of_the_main_theorem}
\frac{\partial P(q\varphi+s\psi)}{\partial s}>0
\end{equation}
for any $(q,s)\in\R^2$.

We notice that, for $q>1$, $s=\widetilde{\tau}(q)$ is the unique $s>0$ such that $\Psi_q(s)=P(q\varphi+s\psi)=0$. Then, by (\ref{partial_derivative_nonzero_proof_of_the_main_theorem}), we can apply the implicit function theorem to $P(q\varphi+s\psi)$ and obtain that $\widetilde{\tau}(q)$ is real analytic in $q>1$.
\end{proof}


\end{document}